\documentclass{daj}
\usepackage{amssymb}
\usepackage{amsthm}
\usepackage{amsfonts}
\usepackage{mathtools}
\usepackage{overpic}
\DeclarePairedDelimiter\ceil{\lceil}{\rceil}

\newcommand{\R}{\mathbb{R}}

\newcommand{\N}{\mathbb{N}}

\newcommand{\tn}{\mathbb{P}}

\newcommand{\calH}{\mathcal{H}}

\newcommand{\calF}{\mathcal{F}}
\newcommand{\calS}{\mathcal{S}}

\newcommand{\calR}{\mathcal{R}}
\newcommand{\calU}{\mathcal{U}}

\newcommand{\Hd}{\dim_{\mathrm{H}}}

\newcommand{\diam}{\operatorname{diam}}

\newcommand{\dist}{\operatorname{dist}}

\theoremstyle{plain}
\newtheorem{thm}[equation]{Theorem}
\newtheorem*{"thm"}{"Theorem"}

\newtheorem{lemma}[equation]{Lemma}

\newtheorem{cor}[equation]{Corollary}
\newtheorem{proposition}[equation]{Proposition}
\newtheorem{question}{Question}
\newtheorem{claim}[equation]{Claim}

\theoremstyle{definition}

\newtheorem{definition}[equation]{Definition}

\newtheorem{notation}[equation]{Notation}

\theoremstyle{remark}
\newtheorem{remark}[equation]{Remark}

\dajAUTHORdetails{%
  title = {On the Hausdorff Dimension of Circular Furstenberg Sets}, 
  author = {Katrin F\"assler, Jiayin Liu, and Tuomas Orponen},
  plaintextauthor = {Katrin F\"assler, Jiayin Liu, Tuomas Orponen},
  plaintexttitle = {On the Hausdorff Dimension of Circular Furstenberg Sets}, 
  keywords = {Furstenberg sets, Hausdorff dimension},
}   

\dajEDITORdetails{%
   year={2024},
   number={18},
   received={2 June 2023},   
   revised={24 January 2024},    
   published={20 December 2024},  
   doi={10.19086/da.127578},       
}   

\begin{document}

\begin{frontmatter}[classification=text]

\title{On the Hausdorff Dimension of Circular Furstenberg Sets} 

\author[fk]{Katrin F\"assler\thanks{Supported by the Academy of Finland via
the project \emph{Singular integrals, harmonic functions,
and boundary regularity in Heisenberg groups}, grant No. 321696}}
\author[jl]{Jiayin Liu\thanks{Supported by the Academy of Finland via
the projects  \emph{Incidences on Fractals}, grant No. 321896 and \emph{Singular integrals, harmonic functions,
and boundary regularity in Heisenberg groups}, grant No. 321696}}
\author[to]{Tuomas Orponen\thanks{Supported by the Academy of Finland via
the project  \emph{Incidences on Fractals}, grant No. 321896}}

\begin{abstract}
For $0 \leq s \leq 1$ and $0 \leq t \leq 3$, a set $F \subset \R^{2}$ is called a \emph{circular $(s,t)$-Furstenberg set} if there exists a family of circles $\mathcal{S}$ of Hausdorff dimension $\Hd \mathcal{S} \geq t$ such that
\begin{displaymath} \Hd (F \cap S) \geq s, \qquad S \in \mathcal{S}. \end{displaymath}
We prove that if $0 \leq t \leq s \leq 1$, then every circular $(s,t)$-Furstenberg set $F \subset \R^{2}$ has Hausdorff dimension $\Hd F \geq s + t$.  The case $s = 1$ follows from earlier work of Wolff on circular Kakeya sets. 
\end{abstract}
\end{frontmatter}

%
%
%

\section{Introduction}

We start by introducing a few key notions. Throughout the paper, we identify families of circles $\mathcal{S}$ with subsets of $\R^{2} \times (0,\infty)$ in the obvious way: the circle $S(x,r)$ with centre $x \in \R^{2}$ and radius $r > 0$ is identified with the point $(x,r) \in \R^{2} \times (0,\infty)$. With this convention, if $E \subset \R^{2} \times (0,\infty)$, then the Hausdorff dimension of the circle family $\mathcal{S} = \{S(x,r) : (x,r) \in E\}$ is defined to be
\begin{displaymath} \Hd \mathcal{S} := \Hd E. \end{displaymath}

\begin{definition}[Circular Furstenberg sets]\label{def:circularFurstenbergSets} Let $0 \leq s \leq 1$ and $0 \leq t \leq 3$. A set $F \subset \R^{2}$ is called a \emph{circular $(s,t)$-Furstenberg set} if there exists a non-empty family of circles $\mathcal{S}$ with $\Hd \mathcal{S} \geq t$ such that $\Hd (F \cap S) \geq s$ for all $S \in \mathcal{S}$.

Equivalently, there exists a non-empty set $E \subset \R^{2} \times (0,\infty)$ with $\Hd E \geq t$, and with the property that $\Hd (F \cap S(x,r)) \geq s$ for all $(x,r) \in E$.  \end{definition}

Our main result is the following:

\begin{thm}\label{main} Let $0 \leq t \leq s \leq 1$. Then, every circular $(s,t)$-Furstenberg set $F \subset \R^{2}$ has Hausdorff dimension $\Hd F \geq s + t$. \end{thm}

\begin{remark} After the first version of this paper was posted on the \emph{arXiv}, Zahl \cite[Theorem 1.12]{2023arXiv230705894Z} proved a significant generalisation of Theorem \ref{main}, which covers much more general "curvy" Furstenberg sets and yields the same lower bound $\Hd F \geq s + t$. \end{remark}

Theorem \ref{main} will be deduced from a more quantitative $\delta$-discretised version, Theorem \ref{FurstenbergThm} below. To state this version, it is convenient to introduce the following subset of the parameter space $\R^{2} \times (0,\infty)$, where the centres are near the origin, and the radii are bounded both from above, and away from zero:
\begin{notation}[The domain $\mathbf{D}$] We write
\begin{equation}\label{def:BB} \mathbf{D} := \{(x,r) \in \R^{2} \times [0,\infty) : |x| \leq \tfrac{1}{4} \text{ and } r \in [\tfrac{1}{2},1]\}. \end{equation}
\end{notation}
A similar normalisation already appears in Wolff's work on circular Kakeya sets, for example \cite{Wolff99}. As long as we restrict attention to circles $S(p)$ with $p \in \mathbf{D}$, his geometric estimates will be available to us, including \cite[Lemma 3.1]{Wolff99}.

The following definition will be ubiquitous in the paper:
\begin{definition} Let $s \geq 0$, $C > 0$, and $\delta \in 2^{-\N}$. A bounded set $P \subset \R^{d}$ is called a \emph{$(\delta,s,C)$-set} if
\begin{displaymath} |P \cap B(x,r)|_{\delta} \leq Cr^{s}|P|_{\delta}, \qquad x \in \R^{d}, \, r \geq \delta. \end{displaymath}
Here, and in the sequel, $|E|_{\delta}$ refers to the number of dyadic $\delta$-cubes intersecting $E$. We also extend the definition to the case where $P$ is a finite family of dyadic $\delta$-cubes: such a family is called a $(\delta,s,C)$-set if the union $\cup P$ is a $(\delta,s,C)$-set in the sense above.  \end{definition}

The following observations are useful to keep in mind about $(\delta,s,C)$-sets. First, if $P$ is a non-empty $(\delta,s,C)$-set, then $|P|_{\delta} \geq C^{-1}\delta^{-s}$. This follows by applying the defining condition at scale $r = \delta$. Second, a $(\delta,s,C)$-set is a $(\delta,t,C)$-set for all $0 \leq t \leq s$.

It turns out that the critical case for Theorem \ref{main} is the case $s = t$: it will suffice to establish a $\delta$-discretised analogue of the theorem in the case $s = t$ (see Theorem \ref{FurstenbergThm} below), and the general case $0 \leq t \leq s$ of Theorem \ref{main} will follow from this. With this in mind, we introduce the following $\delta$-discretised variants of a circular $(s,s)$-Furstenberg sets. In the definition, $\pi_{\R^{3}} \colon \R^{5} \to \R^{3}$ stands for the map $\pi_{\R^{3}}(x_{1},\ldots,x_{5}) = (x_{1},x_{2},x_{3})$.

\begin{definition}\label{d:config}
Let $s \in (0,1]$, $C>0$, and $\delta \in 2^{-\N}$. A \emph{$(\delta,s,C)$-configuration} is a set $\Omega \subset \R^{5}$ such that $P := \pi_{\R^{3}}(\Omega)$ is a non-empty $(\delta,s,C)$-subset of $\mathbf{D}$, and $E(p) := \{v \in \R^{2} : (p,v) \in \Omega\}$ is a non-empty $(\delta,s,C)$-subset of $S(p)$ for all $p \in P$. Additionally, we require that the sets $E(p)$ have constant cardinality: there exists $M \geq 1$ such that $|E(p)| = M$ for all $p \in P$.

If the constant $M$ is worth emphasising, we will call $\Omega$ a $(\delta,s,C,M)$-configuration. Conversely, if the constant $C$ is not worth emphasising, we will talk casually about $(\delta,s)$-configurations (but only in heuristic and informal parts of the paper). \end{definition}

We note that automatically $M \geq \delta^{-s}/C$, since $E(p)$ is a non-empty $(\delta,s,C)$-set, but it may happen that $M$ is much greater than $\delta^{-s}$.

\begin{thm}\label{FurstenbergThm} For every $\kappa > 0$ and $s \in (0,1]$, there exist $\epsilon,\delta_{0} \in (0,\tfrac{1}{2}]$ such that the following holds for all $\delta \in (0,\delta_{0}]$. Let $\Omega$ be a $(\delta,s,\delta^{-\epsilon},M)$-configuration. Then, $|\mathcal{F}|_{\delta} \geq \delta^{\kappa - s}M$, where
\begin{displaymath} \mathcal{F} := \bigcup_{p \in P} E(p). \end{displaymath}
\end{thm}

The proof of Theorem \ref{FurstenbergThm} is based on starting with a $(\delta,s,\delta^{-\epsilon})$-configuration $\Omega$, and refining it multiple times (the required number depends on $\kappa$ and $s$) until the following \emph{total multiplicity function} of the final refinement is uniformly bounded from above.

\begin{definition}[Total multiplicity function]\label{grandMultiplicity} Let $\Omega \subset \R^{5}$ be a bounded set, and let $\delta > 0$. For $w \in \R^{2}$, we write
\begin{equation}\label{grandMF} m_{\delta}(w \mid \Omega) := |\{(p,v) \in \Omega : w \in B(v,\delta)\}|_{\delta}. \end{equation} \end{definition}
The total multiplicity function is called this way, because we will also introduce "partial" multiplicity functions (denoted $m_{\delta,\lambda,t}$) which do not take into account all pairs $(p,v) \in \Omega$, but rather impose certain restrictions on $p$, depending on the parameters $\lambda$ and $t$.

The next theorem contains the technical core of the paper, and it implies Theorem \ref{FurstenbergThm}.

\begin{thm}\label{thm2} For every $\kappa > 0$ and $s \in (0,1]$ there exist $\delta_{0},\epsilon \in (0,\tfrac{1}{2}]$ such that the following holds for all $\delta \in (0,\delta_{0}]$. Let $\Omega \subset \mathbf{D} \times \R^{2}$ be a $(\delta,s,\delta^{-\epsilon})$-configuration with $|P| \leq \delta^{-s - \epsilon}$. Then, there exists a subset $\Omega' \subset \Omega$ such that $|\Omega'|_{\delta} \geq \delta^{\kappa}|\Omega|_{\delta}$, and
\begin{equation}\label{mainIneq} m_{\delta}(w \mid \Omega') \leq \delta^{- \kappa}, \qquad w \in \Omega'. \end{equation}
\end{thm}

\begin{remark}\label{rem3} In practical applications of Theorem \ref{thm2}, it will be important to know that the constant $\epsilon > 0$ stays bounded away from zero as long as $\kappa > 0$ and $s \in (0,1]$ stay bounded away from zero. This is true, and follows from the proof of Theorem \ref{thm2}, where the dependence between $\epsilon$ and $\kappa,s$ is always explicit and effective. Since Theorem \ref{FurstenbergThm} is a consequence of Theorem \ref{thm2}, this remark also applies to Theorem \ref{FurstenbergThm}. \end{remark}

Deducing Theorem \ref{FurstenbergThm} from Theorem \ref{thm2}, and finally Theorem \ref{main} from \ref{FurstenbergThm}, is accomplished in Section \ref{s:discreteToContinuous}.

\subsection{Circular vs. linear Furstenberg sets} The results in this paper should be contrasted with their (known) counterparts regarding \emph{linear $(s,t)$-Furstenberg sets}.

A linear $(s,t)$-Furstenberg set is defined just like a circular $(s,t)$-Furstenberg set, except that the $t$-dimensional family of circles is replaced by a $t$-dimensional family of lines. The main difference between linear and circular Furstenberg sets is that the parameter space of circles is $3$-dimensional, whereas the parameter space of lines is only $2$-dimensional.

This difference makes linear Furstenberg sets substantially simpler: in particular, the analogue of Theorem \ref{main} for linear $(s,t)$-Furstenberg sets is known, see \cite[Theorem A.1]{HSY21} or \cite[Theorem 12]{MR4179019} for two very different proofs, and \cite{MR3973547,MR4002667,MolterRela12,Wolff99} for earlier partial results. Furthermore, any results for circular Furstenberg sets imply their own counterparts for linear Furstenberg sets, simply because the map $z \mapsto 1/z$ takes all lines to circles through $0$. In particular, Theorem \ref{main} gives another -- seriously over-complicated -- proof for \cite[Theorem A.1]{HSY21} and \cite[Theorem 12]{MR4179019}.

Even with Theorem \ref{main} in hand, the theory of circular
Furstenberg sets remains substantially less developed than its
linear counterpart. Theorem \ref{main} is obviously sharp in its
stated range $0 \leq t \leq s \leq 1$, but gives no new
information if $t > s$ (compared to the case $t = s$). In
contrast, it is known that linear $(s,t)$-Furstenberg sets have
Hausdorff dimension $\geq 2s + \epsilon(s,t)$ for $t > s$ (see
\cite{2021arXiv210603338O}). Even stronger results are available
for $t > \min\{1,2s\}$ (see \cite[Theorem
1.6]{2021arXiv211105093F} and \cite{2022arXiv221113363S} for the
current world records). For circular Furstenberg sets, the only
improvement over Theorem \ref{main} is known in the range $t \in
(3s,3]$: in an earlier paper \cite{2022arXiv220401770L}, the
second author proved that every circular $(s,t)$-Furstenberg set
has Hausdorff dimension at least $t/3 + s$, when $s \in (0,1]$ and $t \in (0,3]$ (the result is only stated for $t \in (0,1]$, but the proof actually works for $t \in (0,3]$).

The sharp lower bound for the dimension of linear $(s,t)$-Furstenberg sets is a major open problem: it seems plausible that every linear $(s,t)$-Furstenberg has dimension at least $\min\{(3s + t)/2,s + 1\}$. The case $t = 1$ of the problem was posed by Wolff in \cite[\S 3]{MR1692851} and \cite[Remark 1.5]{Wolff99}. The $(s + 1)$-bound governs the case $s + t \geq 2$, and is already known, see \cite[Theorem 1.6]{2021arXiv211105093F}. The bound $\min\{(3s + t)/2,s + 1\}$ would be sharp if true.

Linear Furstenberg sets can be viewed as special cases of circular Furstenberg sets (as explained above), so at least one cannot hope for something stronger than the lower bound $\min\{(3s + t)/2,s + 1\}$ for circular $(s,t)$-Furstenberg sets. However, it is not clear to us if the optimal lower bounds for linear and circular Furstenberg sets should always coincide. Theorem \ref{main} shows that they do in the range $0 \leq t \leq s \leq 1$.

\begin{remark} After this paper appeared on the \emph{arXiv}, the linear Furstenberg set problem was solved in \cite{2023arXiv230110199O,2023arXiv230808819R}. \end{remark}

\subsection{Relation to previous work} The main challenge in the proof of Theorem \ref{thm2} is to combine the non-concentration hypotheses inherent in $(\delta,s)$-configurations with the techniques of Wolff \cite{MR1800068,MR1473067} developed to treat the case $s = 1$ of Theorem \ref{main}. Our argument is also inspired by the work of Schlag \cite{MR1986697}.

To be accurate with the references, Wolff in \cite[Corollary 5.4]{MR1800068} proved that if $t \in [0,1]$, and $E \subset \R^{2}$ is a Borel set containing circles centred at all points of a Borel set with Hausdorff dimension $\geq t$, then $\Hd E \geq 1 + t$. This is formally weaker than the statement that circular $(1,t)$-Furstenberg sets have dimension $\geq 1 + t$, but the distinction is fairly minor: Wolff's technique is robust enough to deal with circular $(1,t)$-Furstenberg sets. The main novelty in the present paper is to consider the cases $(s,t)$ with $0 \leq t \leq s < 1$.

To illustrate the challenge, consider the case $s = \tfrac{1}{2}$. Let $\Omega = \{(p,v) : p \in P \text{ and } v \in E(p)\}$ be a $(\delta,\tfrac{1}{2})$-configuration. The $(\delta,\tfrac{1}{2})$-set property of the sets $E(p) \subset S(p)$ implies that $|E(p)|_{\delta} \gtrapprox \delta^{-1/2}$ for all $p \in P$. Unfortunately, this information alone is far too weak, because all the circles $S(p)$, $p \in P$, may be tangent to a single rectangle $R \subset \R^{2}$ of dimensions $\delta \times \delta^{1/2}$, and $|R|_{\delta} \sim \delta^{-1/2}$. So, if we only had access to the information $|E(p)|_{\delta} \gtrapprox \delta^{-1/2}$, all the sets $E(p)$ might be contained in $R$. In this case, the resulting "Furstenberg set" $\mathcal{F}$ in \eqref{FurstenbergThm} would have $|\mathcal{F}|_{\delta} \leq |R|_{\delta} \sim \delta^{-1/2}$. In other words, we could hope (at best!) to prove the trivial lower bound
\begin{equation}\label{form173} \dim \mathcal{F} \geq \tfrac{1}{2}, \end{equation}
whereas the "right answer" given by Theorem \ref{main} is $\dim \mathcal{F} \geq 1$. In a previous work \cite{2022arXiv220401770L}, the second author showed that every circular $(s,s)$-Furstenberg set has Hausdorff dimension at least $\max\{4s/3,2s^{2}\}$, and the second bound "$2s^{2}$" matches \eqref{form173} for $s = \tfrac{1}{2}$: this bound indeed follows by applying the techniques of Wolff and Schlag without fully exploiting the non-concentration of the sets $E(p)$. The first bound "$4s/3$" used the non-concentration, but only in a non-sharp "two-ends" manner.

Our proof is also inspired by the very recent work of Pramanik, Yang, and Zahl \cite{2022arXiv220702259P}. In fact, \cite[Section 1.1]{2022arXiv220702259P} is entitled \emph{A Furstenberg-type problem for circles}, and a special case of Theorem \ref{main} follows from \cite[Theorem 1.3]{2022arXiv220702259P}. To describe this case, let $s \in [0,1]$, and let $E \subset \R$ be a set with $\Hd E \geq s$. Let $\mathcal{S}$ be a $t$-dimensional family of circles, with $0 \leq t \leq s$, and write $E_{S} := S \cap (E \times \R)$ for all $S \in \mathcal{S}$. Assume that $\Hd E_{S} = \Hd E \geq s$ for all $S \in \mathcal{S}$. Then
\begin{displaymath} F := \bigcup_{S \in \mathcal{S}} E_{S} \end{displaymath}
is an $(s,t)$-Furstenberg set, and \cite[Theorem 1.3]{2022arXiv220702259P} (with some effort) implies $\Hd F \geq s + t$. In other words, \cite[Theorem 1.3]{2022arXiv220702259P} treats the case of $(s,t)$-Furstenberg sets arising from the specific construction described above. This precursor allowed us to expect Theorem \ref{FurstenbergThm}, but we did not succeed in modifying the argument of \cite{2022arXiv220702259P} to prove it in full generality. Our proof, outlined in the next section, is therefore rather different from \cite{2022arXiv220702259P}.

While the existing literature on circular Furstenberg sets is narrow, there are many more works dealing with various aspects of circular -- or in general: curvilinear -- Kakeya problems. We do not delve into the details or definitions here, but we refer the reader to \cite{MR3946717,2022arXiv221008320C,MR3568105,MR3775465,MR887283,MR1724841,MR4043823,MR4499576,MR2221250,MR3231483} for more information.

\subsection{Ideas of the proof: key concepts and structure}\label{s:outline} When studying circular Kakeya or Furstenberg sets, one needs to understand the geometry of intersecting $\delta$-annuli. If $p = (x,r) \in \R^{2} \times (0,\infty)$ and $\delta > 0$, we write $S^{\delta}(p)$ for the closed $\delta$-annulus around the circle $S(p)$, thus $S^{\delta}(p) = \{w \in \R^{2} : \dist(w,S(p)) \leq \delta\}$.

If $p = (x,r), q = (x',r') \in \mathbf{D} \subset \R^{2} \times (0,\infty)$, what does this intersection $S^{\delta}(p) \cap S^{\delta}(q)$ look like (when non-empty)? Wolff noted that the answer depends on two parameters:
\begin{equation}\label{form154} \lambda := \lambda(p,q) := ||x - x'| - |r - r'|| \quad \text{and} \quad t := t(p,q) := |p - q|. \end{equation}
Notice that "$t$" in \eqref{form154} has a different meaning than the letter "$t$" in $(s,t)$-Furstenberg sets. For the majority of the paper (proofs of Theorems \ref{FurstenbergThm} and \ref{thm2}), we only consider $(s,s)$-Furstenberg sets, so this should not cause confusion. In fact, from now on the letter "$t$" will always refer to the \emph{distance} parameter defined in \eqref{form154}, except for the short proof of Theorem \ref{main} in Section \ref{s:discreteToContinuous} (where the distance parameter is not needed).

Here $\lambda$ is called the \emph{tangency} parameter. If $\lambda(p,q) = 0$, then the circles $S(p),S(q)$ are internally tangent, whereas if $\lambda(p,q) \sim 1$, the circles $S(p),S(q)$ intersect roughly transversally. The intersection $S^{\delta}(p) \cap S^{\delta}(q)$ can be covered by boundedly many \emph{$(\delta,\delta/\sqrt{\lambda t})$-rectangles}. In general, a \emph{$(\delta,\sigma)$-rectangle} is the intersection of a $\delta$-annulus with a disc of radius $\sigma$, thus
\begin{displaymath} R^{\delta}_{\sigma}(p,v) = S^{\delta}(p) \cap B(v,\sigma) \end{displaymath}
for some $v \in S(p)$. If $\delta \leq \sigma \leq \sqrt{\delta}$, a $(\delta,\sigma)$-rectangle looks like a "straight" rectangle of dimensions $\sim \delta \times \sigma$. If $\sigma > \sqrt{\delta}$, then the curvature of the annulus becomes visible, and a $(\delta,\sigma)$-rectangle is a genuinely "curvy" set of thickness $\delta$ and diameter $\sim \sigma$.

When bounding the total multiplicity function $m_{\delta}$ (Definition \ref{grandMultiplicity}), one ends up studying families of $(\delta,\sigma)$-rectangles, for all possible values $\delta \leq \sigma \leq 1$. In some form, this problem appears in all previous works related to circular Kakeya sets, but the manner of formalising it varies. For us, the main new twist is to incorporate the information from the "fractal" sets $E(p) \subset S(p)$.

In addition to the total multiplicity function, we introduce a range of \emph{partial multiplicity functions}. The precise definition is Definition \ref{def:multFunction1}, but we give the idea. For $\delta \leq \lambda \leq t \leq 1$, the partial multiplicity function $m_{\delta,\lambda,t}$ looks like this: for $(p,v) \in \Omega$ (with $p \in P$ and $v \in E(p)$), we write
\begin{displaymath} m_{\delta,\lambda,t}(p,v) := |\{(p',v') \in \Omega^{\delta}_{\sigma} : \lambda(p,p') \sim \lambda, \, t(p,p') \sim t \text{ and } R^{\delta}_{\sigma}(p,v) \cap R^{\delta}_{\sigma}(p',v') \neq \emptyset\}|. \end{displaymath}
Here $\sigma := \delta/\sqrt{\lambda t}$, a common notation in the paper. The set $\Omega^{\delta}_{\sigma}$ is the \emph{$(\delta,\sigma)$-skeleton} of $\Omega$: slightly vaguely, it is a maximal $(\delta \times \sigma)$-separated set inside the original configuration $\Omega$.

It turns out that the total multiplicity function $m_{\delta}$ is bounded from above by the sum of the partial multiplicity functions $m_{\delta,\lambda,t}$, where the sum ranges over dyadic pairs $(\lambda,t)$, $\delta \leq \lambda \leq t \leq 1$. There are $\lesssim (\log(1/\delta))^{2} \leq \delta^{-\kappa}$ such pairs $(\lambda,t)$. So, to prove the upper bound \eqref{mainIneq} for $m_{\delta}$, it suffices to prove it separately for all the partial functions $m_{\delta,\lambda,t}$. This is what we do, see Theorem \ref{thm5}. Bounding $m_{\delta}$ by the sum of the partial functions $m_{\delta,\lambda,t}$ is straightforward, and is accomplished at the end of the paper, in Section \ref{s:thm2Proof}.

The partial multiplicity functions $m_{\delta,\lambda,t}$ have been normalised so that they might potentially satisfy the same bounds as the total multiplicity function (see Theorem \ref{thm2}): after replacing the original $(\delta,s)$-configuration $\Omega$ by a suitable refinement $\Omega'$ (depending on $\lambda$ and $t$), we expect -- and will prove in Theorem \ref{thm5} -- that
\begin{equation}\label{form175} \|m_{\delta,\lambda,t}(\cdot \mid \Omega')\|_{L^{\infty}(\Omega')} \lessapprox 1, \qquad \delta \leq \lambda \leq t \leq 1. \end{equation}
The proof of \eqref{form175} proceeds in a specific order of the triples $(\delta,\lambda,t)$. In order to cope with a given triple $(\delta,\lambda,t)$, we will need to know \emph{a priori} that the triples $(\lambda,\lambda,t)$ and $(\delta,\lambda',t)$ for all $\delta \leq \lambda' < \lambda$ have already been dealt with. More precisely: if we have already found a refinement $\Omega' \subset \Omega$ such that \eqref{form175} holds for all the triples $(\delta,\lambda',t)$ with $\delta \leq \lambda' < \lambda$, and also for the triple $(\lambda,\lambda,t)$, then we are able to refine $\Omega'$ further to obtain \eqref{form175} for $(\delta,\lambda,t)$.

We can now explain a technical challenge we need to overcome: the partial multiplicity function $m_{\delta,\lambda,t}$ counts elements in the $(\delta,\sigma)$-skeleton of $\Omega$, rather than $\Omega$ itself. However, our assumptions on the configuration $\Omega$ were formulated at scale $\delta$ -- recall that $\Omega$ is a $(\delta,s)$-configuration, which meant that both $P$, and the sets $E(p)$, are $(\delta,s)$-sets. In order for \eqref{form175} to be plausible, the property of "being a $(\delta,s)$-configuration" needs to be hereditary: the $(\delta,\sigma)$-skeleton $\Omega^{\delta}_{\sigma}$ of a $(\delta,s)$-configuration $\Omega$ needs to look like a $(\delta,\sigma,s)$-configuration (whatever that precisely means). This is not literally true, but we develop reasonable substitutes for this idea in Section \ref{s:prelimConfigurations}.

We next outline where the "inductive" structure for proving \eqref{form175} stems from. Why do we need information about the triple $(\lambda,\lambda,t)$ in order to handle the triple $(\delta,\lambda,t)$? The reason is one of the main technical results of the paper, Theorem \ref{thm4}. This is a generalisation of Wolff's famous "tangency bound" \cite[Lemma 1.4]{MR1800068}. We sketch the idea of Wolff's result, and our generalisation, in a slightly special case. Namely, we will confine the discussion to the case $t = 1$ to keep the numerology as simple as possible.

In Wolff's terminology, a pair of sets $W,B \subset P \subset \mathbf{D}$ is called \emph{bipartite} if
\begin{displaymath} \dist(W,B) \sim 1. \end{displaymath}
If $p \in W$ and $q \in B$, we have $|p - q| \sim 1$, but the tangency parameter $\lambda(p,q)$ may vary freely in $[0,1]$. If $\lambda(p,q) \sim \lambda \in [0,1]$, recall that the intersection $S^{\delta}(p) \cap S^{\delta}(q)$ can be covered by boundedly many $(\delta,\delta/\sqrt{\lambda})$-rectangles. When bounding the multiplicity function $m_{\delta,\lambda,1}$, the following turns out to be a key question:
\begin{question}\label{q1} What is the maximal cardinality of incomparable $(\delta,\delta/\sqrt{\lambda})$-rectangles which are incident to at least one pair $(p,q) \in W \times B$ with $\lambda(p,q) \sim \lambda$? \end{question}
One of the main results in Wolff's paper \cite{MR1800068} contains the answer in the case $\lambda = \delta$. If $\mathcal{R}_{\delta}$ is a collection of incomparable $(\delta,\sqrt{\delta})$-rectangles incident to at least one pair $(p,q) \in W \times B$ with $\lambda(p,q) \lesssim \delta$, then \cite[Lemma 1.4]{MR1800068} states that
\begin{equation}\label{form210} |\mathcal{R}_{\delta}| \lessapprox (|W||B|)^{3/4} + \text{lesser terms.} \end{equation}
This is a highly non-trivial result. In contrast, the case $\lambda \sim 1$ is trivial: the sharp answer is $|\mathcal{R}_{1}| \lesssim |W||B|$. In this case the $(\delta,\delta/\sqrt{1})$-rectangles are roughly $\delta$-discs, and clearly a generic bipartite pair $W,B$ may generate $\sim |W||B|$ transversal intersections.

Is there a way to "interpolate" between these bounds? One might hope that if $\delta \ll \lambda \ll 1$, then $|\mathcal{R}_{\delta}| \lessapprox (|W||B|)^{\theta(\lambda)}$ for some useful intermediate exponent $\theta(\lambda) \in (\tfrac{3}{4},1)$. Unfortunately, this is not true: if $\lambda \gg \delta$, the best one can say is $|\mathcal{R}_{\lambda}| \lesssim |W||B|$.
\begin{figure}[h!]
\begin{center}
\begin{overpic}[scale = 0.7]{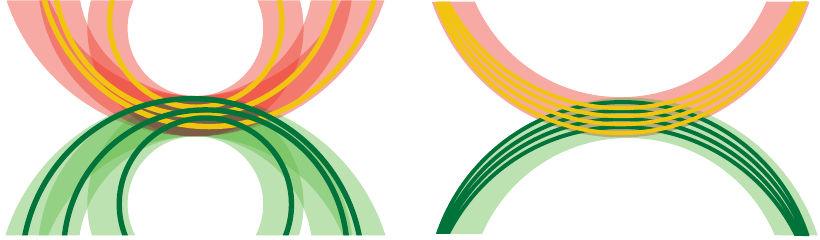}
\put(22,19){$R_{\lambda}$}
\put(74,19){$R_{\lambda}$}
\put(-1,20){$\mathcal{W}$}
\put(-1,5){$\mathcal{B}$}
\end{overpic}
\caption{Scenarios with $|\mathcal{R}_{\lambda}| \sim |W||B|$.}\label{fig6}
\end{center}
\end{figure}

Figure \ref{fig6} shows two slightly different ways in which $|\mathcal{R}_{\lambda}| \sim |W||B|$ can be realised. In both examples, there are two well-separated collections $\mathcal{W},\mathcal{B}$ of (thick, $\lambda$-separated) $\lambda$-annuli, all elements of which are tangent to a common $(\lambda,\sqrt{\lambda})$-rectangle $R_{\lambda}$. (A technical comment: to make the figure clearer, we deliberately draw annuli with external tangencies, although formally all our tangency-counting problems and estimates concern numbers of internal tangencies. The distinction between internal and external tangencies is, however, not relevant for the phenomenon we describe here.)

Inside each annulus in $\mathcal{W}$ (respectively $\mathcal{B})$ pick $X_{\mathcal{W}}$ (respectively $X_{\mathcal{B}}$) thinner $\delta$-annuli, shown in darker colours. This way one gets two well-separated collections $W,B$ of $\delta$-annuli with cardinalities
\begin{displaymath} |W| = |\mathcal{W}| \cdot X_{\mathcal{W}} \quad \text{and} \quad |B| = |\mathcal{B}| \cdot X_{\mathcal{B}}. \end{displaymath}
The picture on the left of Figure \ref{fig6} represents the case $X_{\mathcal{W}} = X_{\mathcal{B}} = 1$, the picture on the right represents the case $|\mathcal{W}| = |\mathcal{B}| = 1$. If the $\delta$-annuli in $W,B$ are chosen appropriately, their pairwise intersections (contained in $R_{\lambda}$) are located at incomparable $(\delta,\delta/\sqrt{\lambda})$-rectangles, say $\mathcal{R}$. (To be more accurate, this can be done as long as the total number of intersections $|W||B|$ does not exceed the total number of incomparable $(\delta,\delta/\sqrt{\lambda})$-rectangles contained in $R_{\lambda}$, roughly $(\lambda/\delta)^{2}$.) Each of the rectangles in $\mathcal{R}$ has type $(\geq 1,\geq 1)$ relative to $(W,B)$. Therefore, $|\mathcal{R}| \sim |W||B|$, provided $|W||B| \leq (\lambda/\delta)^{2}$.

The trivial upper bound $|\mathcal{R}_{\lambda}| \lesssim |W||B|$ is useless for $\lambda \ll 1$, but there is a way to improve it. The examples shown in Figure \ref{fig6} indicate the main obstructions: the high numbers of incomparable $(\delta,\delta/\sqrt{\lambda})$-rectangles are "caused" by either
\begin{itemize}
\item[(a)] a high level of tangency of "parent" annuli of thickness $\lambda$, or
\item[(b)] a high number of "child" $\delta$-annuli contained inside "parent" annuli of thickness $\lambda$.
\end{itemize}
If we stipulate \emph{a priori} bounds on the numbers relevant for problems (a)-(b), we get a non-trivial upper bound for $|\mathcal{R}_{\lambda}|$, which looks like this (see Theorem \ref{thm4} for a precise statement):
\begin{equation}\label{form174} |\mathcal{R}_{\lambda}| \lessapprox (|W||B|)^{3/4} \cdot (X_{\lambda}Y_{\lambda})^{1/2} + \text{lesser terms,} \end{equation}
 Here $X_{\lambda} = \max |P \cap B_{\lambda}|$, where the "$\max$" runs over balls $B_{\lambda} \subset \R^{2} \times (0,\infty)$ of radius $\lambda$, and $Y_{\lambda}$ is an upper bound for how many $\lambda$-annuli can be tangent to any fixed $(\lambda,\sqrt{\lambda})$-rectangle. In fact,
\begin{displaymath} Y_{\lambda}  = \|m_{\lambda,\lambda,1}\|_{L^{\infty}}. \end{displaymath}
In the examples of Figure \ref{fig6}, we have $X_{\lambda} = 1$ and $Y_{\lambda} = |W| = |B| \sim \lambda/\delta$ (left picture) or $X_{\lambda} \sim |W| \sim |B| \sim \lambda/\delta$ and $Y_{\lambda} = 1$ (right picture). In both cases \eqref{form174} only yields the trivial bound, as it should. On the other hand, if we have already established \eqref{form175} for the triple $(\lambda,\lambda,1)$, we can rest assured that $Y_{\lambda} \lessapprox 1$, and \eqref{form174} becomes a useful tool for proving \eqref{form175} for the triple $(\delta,\lambda,1)$ (bounds for the number $X_{\lambda}$ are, more easily, provided by non-concentration conditions on the collections of circles). This explains why our inductive proof of \eqref{form175} needs information about the triples $(\lambda,\lambda,t)$ to handle the triples $(\delta,\lambda,t)$. There is a separate reason why all the triples $(\delta,\lambda',t)$, $\lambda' < \lambda$, need to be treated before the triple $(\delta,\lambda,t)$, but we will not discuss this here: the reason will be revealed around Figure \ref{fig3}.

We have now quite thoroughly explained the structure of the paper, but let us summarise. In the short Section \ref{s:discreteToContinuous}, we first deduce Theorem \ref{FurstenbergThm} from Theorem \ref{thm2}, and then Theorem \ref{main} from Theorem \ref{FurstenbergThm}. Section \ref{s:prelimConfigurations} deals with the question: to what extent is the $(\lambda,\sigma)$-skeleton of a $(\delta,s)$-configuration a $(\lambda,\sigma,s)$-configuration?

Section \ref{s:tangencies} introduces $(\delta,\sigma)$-rectangles properly, and studies their elementary geometric properties. For example, what do we exactly mean by two $(\delta,\sigma)$-rectangles being "incomparable"? The results in Section \ref{s:tangencies} will look familiar to those readers knowledgeable of Wolff's work, but our $(\delta,\sigma)$-rectangles are more general than Wolff's $(\delta,\sqrt{\delta/t})$-rectangles, and in some cases we need more quantitative estimates than those recorded in \cite{MR1800068}.

In Section \ref{lambdalambdat}, we establish the cases $(\lambda,\lambda,t)$ of the estimate \eqref{form175}. The main produce of that section is Theorem \ref{thm3}. The geometric input behind Theorem \ref{thm3} is simply Wolff's estimate \eqref{form210}, and this is why it can be proven before introducing the general $(\delta,\lambda,t)$-version in \eqref{form174}. The proof of \eqref{form174} occupies Section \ref{deltalambdat}.

Finally, Section \ref{s:mainInduction} applies the estimate \eqref{form174} to prove \eqref{form175} in full generality.
 The upper bound for the total multiplicity function $m_{\delta}$ is an easy corollary, and the proof Theorem \ref{thm2}
 is concluded in Section \ref{s:thm2Proof}. In Appendix \ref{app}
 we prove some results from Section \ref{s:comparableRectangles}.

\subsection*{Notation} Some of the notation in this section has already been introduced above, but we gather it here for ease of reference. If $r \in 2^{-\N}$, the notation $|E|_{r}$ refers to the number of dyadic $r$-cubes intersecting $E$. Here $E$ might be a subset of $\R$, $\R^{2}$, or $\R^{3}$. We will only ever consider dyadic cubes in $\R^{3}$ which are subsets of the special region $\mathbf{D}$ introduced in \eqref{def:BB}. Therefore, the notation $\mathcal{D}_{r}$ will always refer to dyadic $r$-cubes contained in $\mathbf{D}$.

In general, we will denote points in $\R^{3}$ (typically in $\mathbf{D}$) by the letters $p,p',q,q'$. Points in $\R^{2}$ are denoted by $v,v',w,w'$.

For $p = (x,r) \in \R^{2} \times (0,\infty)$ (typically $p \in \mathbf{D}$), we write $S(p) = S(x,r)$ for the circle centred at $x$ and radius $r > 0$. The notation $S^{\delta}(p)$ refers to the $\delta$-annulus around $S(p)$, thus $S^{\delta}(p) = \{w \in \R^{2} : \dist(w,S(p)) \leq \delta\}$.

The notation $A \lesssim B$ means that there exists an absolute constant $C \geq 1$ such that $A \leq CB$. The two-sided inequality $A \lesssim B \lesssim A$ is abbreviated to $A \sim B$. If the constant $C$ is allowed to depend on a parameter "$\theta$", we indicate this by writing $A \lesssim_{\theta} B$.

For $\delta \in (0,1]$, the notation $A \lessapprox_{\delta} B$ means
that there exists an absolute constant $C \geq 1$ such that
\begin{displaymath}
A \leq C \cdot\left(1+
\log\left(\tfrac{1}{\delta}\right)^{C}\right) B.
\end{displaymath}
We write $A \thickapprox_{\delta} B$ if simultaneously $A
\lessapprox_{\delta} B$ and $B \lessapprox_{\delta} A$ hold true. If the constant $C$ is allowed to depend on a parameter "$\theta$", we indicate this by writing $A \lessapprox_{\delta,\theta} B$.

Given $p=(x,r) \in \R^{2} \times [0,\infty)$ and $p'=(x',r')$ in $\mathbb{R}^2\times
[0,\infty)$, we write $\Delta(p,p'):= \left||x-x'|-|r-r'|\right|$. This is slightly inconsistent with our notation from \eqref{form154}, but in the sequel we prefer to use the letter "$\Delta$" for this "tangency" parameter.



\section{Proof of Theorem \ref{FurstenbergThm} and Theorem \ref{main}}\label{s:discreteToContinuous}

We first use Theorem \ref{thm2} to prove Theorem \ref{FurstenbergThm}.

\begin{proof}[Proof of Theorem \ref{FurstenbergThm} assuming Theorem \ref{thm2}] Let $\Omega \subset \R^{5}$ be a $(\delta,s,\delta^{-\epsilon},M)$-configuration. Write $P := \pi_{\R^{3}}(\Omega) \subset \mathbf{D}$, and $E(p) = \{v \in \R^{2} : (p,v) \in \Omega\} \subset S(p)$. By replacing $P$ and $E(p)$ by maximal $\delta$-separated subsets, we may assume that $P$, $E(p)$, and $\Omega$ are finite and $\delta$-separated to begin with. Furthermore, $P$ contains a $(\delta,s,\delta^{-2\epsilon})$-subset $\bar{P} \subset P$ of cardinality $|\bar{P}| \leq \delta^{-s}$ by \cite[Lemma 2.7]{2021arXiv210603338O}. Then $\bar{\Omega} := \{(p,v) : p \in \bar{P} \text{ and } v \in E(p)\}$ remains a $(\delta,s,\delta^{-2\epsilon})$-configuration with $|E(p)| \equiv M$. It evidently suffices to prove Theorem \ref{FurstenbergThm} for this sub-configuration, so we may assume that $|P| \leq \delta^{-s}$ to begin with.

With this assumption, we may apply Theorem \ref{thm2} to find a subset $\Omega' \subset \Omega$ with $|\Omega'| \geq \delta^{\kappa}|\Omega| = \delta^{\kappa}M|P|$ and the property
\begin{displaymath} m_{\delta}(w \mid \Omega') \leq \delta^{-\kappa}, \qquad w \in \R^{2}. \end{displaymath}
For $p \in P$, we write $\Omega'(p) := \{v \in \R^{2} : (p,v) \in \Omega'\} \subset E(p)$ (this will become standard notation in the paper).

Let $F'$ be a maximal $\delta$-separated set in
\begin{displaymath} \bigcup_{p \in P} \Omega'(p) \subset \mathcal{F}, \end{displaymath}
where $\mathcal{F}$ appeared in the statement of Theorem \ref{FurstenbergThm}. We claim that $|F'| \geq \delta^{3\kappa - s}M$, if $\delta > 0$ is small enough. This will evidently suffice to prove Theorem \ref{FurstenbergThm}.

First, we notice that $|[\Omega'(p)]_{\delta} \cap F'| \gtrsim |\Omega'(p)|$ for all $p \in P$, where $[A]_{\delta}$ refers to the $\delta$-neighbourhood of $A$. The reason is that if $w \in \Omega'(p)$, then $\dist(w,F') \leq \delta$, and therefore there exists a point $w' \in [\Omega'(p)]_{\delta} \cap F'$ with $|w - w'| \leq \delta$. Moreover, since $\Omega'(p)$ was assumed to be $\delta$-separated, the map $w \mapsto w'$ is at most $C$-to-$1$.  As a consequence of this observation,
\begin{displaymath} \sum_{w \in F'} |\{v \in \Omega'(p) : w \in B(v,\delta)\} \geq |[\Omega'(p)]_{\delta} \cap F'| \gtrsim |\Omega'(p)|. \end{displaymath}
Now,
\begin{align*} \delta^{-\kappa} & \geq \frac{1}{|F'|} \sum_{w \in F'} m_{\delta}(w \mid \Omega') = \frac{1}{|F'|} \sum_{w \in F'} |\{(p,v) \in \Omega' : w \in B(v,\delta)\}|\\
& = \frac{1}{|F'|} \sum_{w \in F'} \sum_{p \in P} |\{v \in \Omega'(p) : w \in B(v,\delta)\}| \gtrsim \frac{1}{|F'|} \sum_{p \in P'} |\Omega'(p)| = \frac{|\Omega'|}{|F'|}. \end{align*}
Now, recalling that $|\Omega'| \geq \delta^{\kappa}M|P| \geq \delta^{\kappa + \epsilon - s}M$, and rearranging, we find $|F'| \geq \delta^{3\kappa - s}M$, assuming $\delta > 0$ small enough. This is what we claimed.  \end{proof}

Now we use Theorem \ref{FurstenbergThm} to prove Theorem \ref{main}. This is virtually the same argument as in the proof of \cite[Lemma 3.3]{HSY21}, but we give the details for the reader's convenience.

\begin{proof}[Proof of Theorem \ref{main}] We may assume that $t > 0$, since every circular $(s,0)$-Furstenberg set has Hausdorff dimension at least $s$ by the non-emptiness of $\mathcal{S}$ in Definition \ref{def:circularFurstenbergSets}. Fix $0 < t \leq s \leq 1$, and let $F \subset \R^{2}$ be a circular $(s, t)$-Furstenberg set with parameter set $E \subset \R^{2} \times (0,\infty)$ satisfying $\Hd E \geq t$. To avoid confusion, we mention already now that the plan is to apply Theorem \ref{FurstenbergThm} with parameter "$t$" in place of "$s$", and with $M \approx \delta^{-s}$ (which is potentially much larger than $\delta^{-t}$).

Translating and scaling $F$, it is easy to reduce to the case $E \subset \mathbf{D}$. Fix $t' \in [t/2,t]$ and $t' \leq s' < s$. Since $\mathcal{H}^{t'}_{\infty}(E) > 0$, there exists $\alpha = \alpha(E,t') >0$ and $E_1 \subset E$ such that $\calH^{t'}_\infty(E_1)> \alpha$,
where
\begin{equation}\label{lbdd11}
  E_1 := \{ p\in E \ | \ \calH^{s'}_\infty(F \cap S(p))> \alpha\}.
\end{equation}
This follows from the sub-additivity of Hausdorff content.

We also fix a parameter $\kappa > 0$, and we apply Theorem \ref{FurstenbergThm} with constants $\kappa$ and $t'$ (as above). The result is a constant $\epsilon(\kappa,t') > 0$. Recalling Remark \ref{rem3}, the constant $\epsilon(\kappa,t') > 0$ stays bounded away from zero for all $t' \in [t/2,t]$. We set
\begin{displaymath} \epsilon := \epsilon(\kappa,t) := \inf_{t' \in [t/2,t]} \epsilon(\kappa,t') > 0. \end{displaymath}

Next, we choose $k_0 = k_{0}(\alpha,\epsilon) = k_0(E,t',\epsilon) \in \N$ satisfying
\begin{equation}\label{para31}
  \alpha>\sum_{k=k_0}^{\infty} \frac{1}{k^2} \quad \text{and} \quad k_{0}^{2} \leq \min\{2^{\epsilon k_{0}}/C,2^{\kappa k_{0}}/C\},
\end{equation}
where $C \geq 1$ is an absolute constant to be determined later. Let $\calU = \{ D(x_i,r_i)\}_{i \in\mathcal{I}}$ be an arbitrary cover of $F$
by dyadic $r_i$-cubes with $r_i \le 2^{-k_0}$ and
$F \cap D(x_i,r_i) \ne \emptyset$ for all $i \in \mathcal{I}$.
For $k \ge k_0$, write
$$ \mathcal{I}_k:= \{i \in \mathcal{I} :  r_i = 2^{-k} \} \quad \mbox{and} \quad F_k:= \{\cup D(x_i,r_i) : i \in \mathcal{I}_k \}.$$
By the pigeonhole principle and \eqref{para31} we deduce that for each $p \in E_1$, there exists $k(p) \ge k_0$ such that
$$\calH^{s'}_\infty( F \cap S(p)\cap F_{k(p)})> k(p)^{-2} .$$
Using pigeonhole principle again we obtain that there exists $k_1 \ge k_0$ such that
\begin{equation}\label{lbdd21}
  \calH^{t'}_\infty(E_2)> k_1^{-2}
\end{equation}
 where $ E_2 := \{p \in E_1 : k(p) = k_1\}.$ By the construction of $E_2$, we have
\begin{displaymath}
\calH^{s'}_\infty(S(p)\cap F_{k_1})\ge \calH^{s'}_\infty(F \cap S(p) \cap F_{k_1})> k_1^{-2}, \qquad p \in E_{2}.
\end{displaymath}
Write $\delta =2^{-k_1}$.
By \eqref{lbdd21} and \cite[Lemma 3.13]{FasslerOrponen14}, we know that there exists a $\delta$-separated $(\delta,t',Ck_{1}^{2})$-set $P \subset E_2$ satisfying $(k_{1}^{-2}/C)\delta^{-t'} \leq |P| \leq \delta^{-t'}$. Since $P \subset E_{2}$, we have
\begin{equation}\label{lbdd31}
  \calH^{s'}_\infty(S(p)\cap F_{k_1})> k_1^{-2}, \qquad p \in P.
\end{equation}
 Applying \cite[Lemma 3.13]{FasslerOrponen14} again to $S(p)\cap F_{k_1}$, $p \in P$, we obtain $\delta$-separated $(\delta,s',Ck_{1}^{2})$-sets $E(p) \subset S(p)\cap F_{k_1}$ such that
\begin{equation*}
  |E(p)| \equiv M \geq (k_{1}^{-2}/C)\delta^{-s'} \stackrel{\eqref{para31}}{\geq} \delta^{\kappa - s'}, \qquad p \in P.
\end{equation*}
By \eqref{para31}, $P$ is a $(\delta,t',\delta^{-\epsilon})$-set, and each $E(p)$ is a $(\delta,s',\delta^{-\epsilon})$-set. Since $s' \geq t'$, the sets $E(p)$ are automatically also $(\delta,t',\delta^{-\epsilon})$-sets. Therefore,
$$ \Omega:= \{(p,v) : p \in P \text{ and } v \in E(p)\} \subset \mathbb{R}^5$$
is a $(\delta,t',\delta^{-\epsilon},M)$-configuration. Recall that $\epsilon \leq \epsilon(\kappa,t')$ by the definition of $\epsilon$. Letting
$$ \calF:= \bigcup_{p \in P} E(p)  $$
and applying Theorem \ref{FurstenbergThm}, we deduce that $|\mathcal{F}|_{\delta} \geq \delta^{\kappa - t'}M \geq \delta^{2\kappa - s'-t'}$.

Since $E(p) \subset F_{k_1}$ for each $p \in P$, we have $\calF \subset F_{k_1}$, which implies
 $$|\mathcal{I}_{k_1}| = |F_{k_1}|_\delta \geq  |\mathcal{F}|_{\delta} \geq \delta^{2\kappa - s'-t'}.$$
 Then
$$ \sum_{i \in \mathcal{I}} r_i^{s'+t'- 2\kappa} \ge \sum_{i \in \mathcal{I}_{k_1}} r_i^{s'+t'- 2\kappa} = \delta^{s'+t'- 2\kappa}|\mathcal{I}_{k_1}| \geq 1. $$
As the covering was arbitrary, we infer that $\Hd F \ge s'+t' - 2\kappa$.
Sending $s' \nearrow s$, $t' \nearrow t$, and $\kappa \searrow 0$, we arrive at the desired result.
\end{proof}


\section{Preliminaries on $(\delta,s)$-configurations}\label{s:prelimConfigurations}

The proof of Theorem \ref{thm2} -- the multiplicity upper bound for $(\delta,s)$-configurations -- will involve considering such configurations at scales $\Delta \gg \delta$. In a dream world, a $(\delta,s)$-configuration would admit a "dyadic" structure which would enable statements of the following kind: (a) the $\Delta$-parents of a $(\delta,s)$-configuration form a $(\Delta,s)$-configuration, and (b) the $\Delta$-parents of a $(\delta,s,C,M)$-configuration form a $(\Delta,s,C',M')$-configuration. Such claims are not only false as stated, but also seriously ill-defined.

To formulate the problems -- and eventually their solutions -- precisely, we introduce notation for dyadic cubes.

\begin{definition}[Dyadic cubes] For $\delta \in 2^{-\N}$, let $\mathcal{D}_{\delta}$ be the family
of dyadic cubes in $\R^{3}$ of side-length $\delta$ which are
contained in the set $\mathbf{D}$. We also write $\mathcal{D}
:= \bigcup_{\delta \in 2^{-\N}} \mathcal{D}_{\delta}$. If $P
\subset \R^{3}$ is an arbitrary set of points, or a family of
cubes, we also write
\begin{displaymath} \mathcal{D}_{\delta}(P) := \{Q \in \mathcal{D}_{\delta} : Q \cap P \neq \emptyset\}. \end{displaymath}
For $p \in \mathbf{D}$, we write $Q_{\delta}(p) \in
\mathcal{D}_{\delta}$ for the unique cube in
$\mathcal{D}_{\delta}$ containing $p$. \end{definition}

We then explain some of the problems we need to overcome. The first one is that if $P \subset \mathbf{D}$ or $P \subset \mathcal{D}_{\delta}$ is a $(\delta,s)$-set, it is not automatic that $P_{\Delta} := \mathcal{D}_{\Delta}(P)$ is a $(\Delta,s)$-set for $\delta < \Delta \leq 1$. This is not too serious: it is well-known that there exists a "refinement" $P' \subset P$ such that $|P'| \approx_{\delta} |P|$, and $P_{\Delta}'$ is a $(\Delta,s)$-set (a proof of this claim will be hidden inside the proof of Proposition \ref{prop6}).

There is another problem of the same nature, which seems more complex to begin with, but can eventually be solved with the same idea. Assume that $\Omega = \{(p,v) : p \in P \text{ and } v \in E(p)\}$ is a $(\delta,s)$-configuration, and $\Delta \gg \delta$. In what sense can we guarantee that some "$\Delta$-net" $\Omega_{\Delta} \subset \Omega$ is a $(\Delta,s)$-configuration? By the fact stated in the previous paragraph, we may start by refining $P \mapsto P'$ such that $P_{\Delta}'$ is a $(\Delta,s)$-set. Then the question becomes: which set $E_{\Delta}(\mathbf{p}) \subset S(\mathbf{p})$ should we associate to each $\mathbf{p} \in P_{\Delta}'$ in such a manner that
\begin{displaymath} \Omega_{\Delta} = \{(\mathbf{p},\mathbf{v}) : \mathbf{p} \in P_{\Delta}' \text{ and } \mathbf{v} \in E_{\Delta}(\mathbf{p})\} \end{displaymath}
is a $(\Delta,s)$-configuration -- which hopefully still has some useful relationship with $\Omega$? This question will eventually be answered in the main result of this section, Proposition \ref{prop6}, but we first need to set up some notation.

For $p=(x,r)\in \mathbb{R}^3_+$ and an arc $I\subset S(p)$, we let
$V(p,I)$ be the (one-sided) cone centred at $x$ and spanned by the
arc $I$. That is,
$$  V(p,I):= \bigcup_{e \in I}\{x+ t(e-x)\}_{t \ge 0}.
$$

\begin{definition}[Dyadic arcs] We introduce a dyadic partition on the circles $S(p)$. If
$\sigma \in 2^{-\N}$ and $p=(x,r) \in \mathbf{D}$, we let
$\mathcal{S}_{\sigma}(p)$ be a partition of $S(p)$ into disjoint
(half-open) arcs of length $2\pi r\sigma$. We also let
$\mathcal{S}(p) := \bigcup_{\sigma \in 2^{-\N}}
\mathcal{S}_{\sigma}(p)$. (We note that for $p = (x,r) \in \mathbf{D}$, always $r \in [\tfrac{1}{2},1]$, so the dyadic $\sigma$-arcs have length comparable to $\sigma$.)
\end{definition}

\begin{remark}\label{dyadicConvention} The notation of dyadic arcs $\mathcal{S}_{\sigma}(p)$ will often be applied with parameters such as $\sigma = \sqrt{\delta/t}$ or $\sigma = \delta/\sqrt{\lambda t}$, which are not dyadic rationals to begin with. In such cases, we really mean $\mathcal{S}_{\bar{\sigma}}(p)$, where $\bar{\sigma} \in 2^{-\N}$ is the smallest dyadic rational with $\sigma \leq \bar{\sigma}$.   \end{remark}

\begin{notation} In the sequel, it will be very common that the letters $p,q,\mathbf{p}$ refer to dyadic cubes instead of points in $\mathbf{D}$. Regardless, we will use the notation $S(p)$, $\mathcal{S}_{\delta}(p)$ and $V(p,I)$. This always refers to the corresponding definitions relative to the centre of $p,q,\mathbf{p}$, which is an element of $\mathbf{D}$. \end{notation}

\begin{lemma} \label{conelem} Let $0 < \delta \leq \Delta \leq 1$ and $0 < \sigma \leq \Sigma \leq
  1$ be dyadic numbers with  $\Delta \leq \Sigma$. Assume
  that $p\in \mathcal{D}_\delta$ and
$\mathbf{p} \in \mathcal{D}_\Delta$ with $p
\subset \mathbf{p}$, and let $v \in \mathcal{S}_{\sigma}(p)$. If $\mathbf{v} \in \mathcal{S}_{\Sigma}(\mathbf{p})$
is such that
\begin{displaymath}
 v \cap V(\mathbf{p},\mathbf{v}) \neq \emptyset,
 \end{displaymath}
then there exists an arc $I_{\mathbf{v}}\subset S(\mathbf{p})$ of
length $\lesssim \Sigma$ such that $v\subset
V(\mathbf{p},I_{\mathbf{v}})$ and $\mathbf{v}\subset
I_{\mathbf{v}}$.
\end{lemma}

For all $p,\mathbf{p}$, and $v$ as in the statement of the lemma,
there exists at least one $\mathbf{v}\in
\mathcal{S}_{\Sigma}(\mathbf{p})$ such that $v \cap
V(\mathbf{p},\mathbf{v}) \ne \emptyset$, simply because
$\mathbb{R}^2= \cup_{\mathbf{v} \in
\mathcal{S}_{\Sigma}(\mathbf{p})} V(\mathbf{p},\mathbf{v})$.

\begin{proof} Without loss of generality, we may assume that $\Sigma \leq
1/12$, say. We denote
$$\mathcal{S}_{\Sigma}(\mathbf{p},v):= \{
\mathbf{v}\in \mathcal{S}_{\Sigma}(\mathbf{p}) : v \cap
V(\mathbf{p},\mathbf{v}) \ne \emptyset\}.$$ Our goal is to bound
the cardinality of $\mathcal{S}_{\Sigma}(\mathbf{p},v)$ uniformly
from above and
 prove
that $I_{\mathbf{v}}$ can be obtained as the union of the arcs in
$\mathcal{S}_{\Sigma}(\mathbf{p},v)$.

 Let $(x,r),(\mathbf{x},\mathbf{r}) \in
 \mathbf{D}$ be the centers of the cubes $p\in \mathcal{D}_{\delta}$ and $\mathbf{p}\in\mathcal{D}_{\Delta}$,
 respectively. By assumption, $\mathbf{r}\geq 1/2$ and $\Delta
 \leq \Sigma\leq 1/12$, so that $\mathbf{r}-3\Delta \gtrsim 1$.
Since $S(p) \subset S^{3\Delta}(\mathbf{p})$ by a simple application of the triangle inequality, we find
\begin{equation}\label{eq:dist_v_center}
\mathrm{dist}(v,\mathbf{x})\geq
\mathrm{dist}(S(p),\mathbf{x})\gtrsim 1.
\end{equation}
Moreover, using also that $r\geq 1/2$, and $\delta \leq 1/12$, it
follows that $\mathbf{x}$ must be contained in the interior of the
disk bounded by $S(p)$.

By the connectedness of $v$ and since $v\cap\{\mathbf{x}\} =
\emptyset$, we find that $ \cup_{
\mathcal{S}_{\Sigma}(\mathbf{p},v)} \mathbf{v}$ is a connected set
in $S(\mathbf{p})$ which implies that
$\mathcal{S}_{\Sigma}(\mathbf{p},v)=\{\mathbf{v}_i\}_{i=1,\cdots,m}$
is a family of adjacent arcs. If $m\in \{1,2\}$, then their union
is obviously an arc $I_{\mathbf{v}}$ of length at most
$4\pi\mathbf{r} \Sigma$ with $v\subset
V(\mathbf{p},I_{\mathbf{v}})$ and $\mathbf{v}\subset
I_{\mathbf{v}}$. Thus we assume from now on that $m\geq 3$.
Letting $\mathbf{v}_i^+$ and $\mathbf{v}_i^-$ be the two endpoints
of the arc $\mathbf{v}_i$, we can arrange the arcs $\mathbf{v}_i
\in \mathcal{S}_{\Sigma}(\mathbf{p},v)$ in such an order that $
\mathbf{v}_i^+ = \mathbf{v}_{i+1}^-$ for all $i=1, \cdots, m-1$.

To conclude the proof of the lemma, it suffices to show that $m$
is bounded from above by a universal constant. As $v\in
\mathcal{S}_{\sigma}(p)$ for $p=(x,r)$, the length $\ell(v)$ of
$v$ is $2\pi r \sigma$. As $\mathbf{x}$ lies inside the disk
bounded by $S(p)$, the set $v \cap V(\mathbf{p},\mathbf{v}_i)$ is
a curve for every $i$. Since $\sigma \leq \Sigma$ and $r\leq 2$,
we have that
$$4\pi  \Sigma \geq \ell(v) = \sum_{i=1}^m\ell(v \cap V(\mathbf{p},\mathbf{v}_i))
 \ge \sum_{i=2}^{m-1}\ell(v \cap V(\mathbf{p},\mathbf{v}_i)).$$
 Thus, the desired upper bound for $m$ will follow, if we manage
 to prove that
 \begin{equation}\label{eq:goal_lengthv}
\ell(v \cap V(\mathbf{p},\mathbf{v}_i))\gtrsim \Sigma, \qquad 2 \leq i \leq m-1.
\end{equation}

Note that
$$  \partial V(\mathbf{p},\mathbf{v}_i) = \{\mathbf{x}\} \cup \{\mathbf{x}+ t (\mathbf{v}_i^+ -\mathbf{x})\}_{t>0}
\cup\{\mathbf{x}+ t (\mathbf{v}_i^- -\mathbf{x})\}_{t>0}, \qquad 1 \leq i \leq m.$$
Write $\mathbf{\bar v}_i^+:=
\{\mathbf{x}+ t (\mathbf{v}_i^+ -\mathbf{x})\}_{t>0}$ and
$\mathbf{\bar v}_i^-:= \{\mathbf{x}+ t (\mathbf{v}_i^-
-\mathbf{x})\}_{t>0}$. We have
$$ \mathbf{\bar v}_i^+ \cap \mathbf{\bar v}_i^- = \emptyset, \qquad 1 \leq i \leq m.$$
Recall that $v \cap \{\mathbf{x}\} = \emptyset$. Then, by the
arrangement of the arcs $\mathbf{v}_i$, we know for $i=2,
\cdots,m-1$, that $v$ must intersect both $\mathbf{\bar v}_i^+$
and $\mathbf{\bar v}_i^-$. Let
$$ x_i^+ \in  \mathbf{\bar v}_i^+ \cap v  \quad \text{and} \quad x_i^- \in \mathbf{\bar v}_i^- \cap v, \qquad  2 \leq i \leq m-1.$$
We claim that
\begin{equation}\label{lengthv}
  |x_i^+ - x_i^-| \gtrsim \Sigma, \qquad 2 \leq i \leq m-1,
\end{equation}
which will yield \eqref{eq:goal_lengthv} and thus conclude the
proof of the lemma.

To prove \eqref{lengthv}, recall that $\mathbf{v}_i$ is an arc of
length $2\pi \mathbf{r} \Sigma$ in $S(\mathbf{p})$. Thus
 $$ \angle(\mathbf{\bar v}_i^+, \mathbf{\bar v}_i^-) = 2 \pi \Sigma   \leq \pi/2. $$
 We have
 \begin{align*}
  |x_i^+ -x_i^-| & \ge \mathrm{dist}(\{x_i^+\} , \mathbf{\bar v}_i^-)
   = \inf\{|x_i^+ - y|: y \in \mathbf{\bar v}_i^- \}
    = |x_i^+ - \mathbf{x}| \sin \angle(\mathbf{\bar v}_i^+, \mathbf{\bar v}_i^-) \\
    & \overset{\eqref{eq:dist_v_center}}{\gtrsim} \frac{ \angle(\mathbf{\bar v}_i^+, \mathbf{\bar v}_i^-)}2  \gtrsim \Sigma,
 \end{align*}
 where for the second inequality we recall that $x_i^+\in v \subset
 S^{3\Delta}(\mathbf{p})$, and we use the fact that $\sin \theta \ge \theta/2$ for all $0 \le \theta \le \pi/2$. The proof is complete.
\end{proof}

Dyadic cubes have the well-known useful property that if $Q,Q' \in
\mathcal{D}$ with $Q \cap Q' \neq \emptyset$, then either $Q
\subset Q'$ or $Q' \subset Q$. For a fixed circle $S(p)$, the
dyadic arcs $\mathcal{S}(p)$ have the same property, but things
get more complicated when we want to compare dyadic arcs in
$\mathcal{S}(p),\mathcal{S}(q)$ for $p \neq q$. The next notation
is designed to clarify this issue.

\begin{notation} Let $0 < \delta \leq \Delta \leq 1$ and $0 < \sigma \leq \Sigma
\leq
  1$ be dyadic numbers.
Assume that $p\in \mathcal{D}_\delta$ and $\mathbf{p} \in
\mathcal{D}_\Delta$ with $p \subset \mathbf{p}$.
For each $v \in \mathcal{S}_{\sigma}(p)$, we write $v \prec \mathbf{v}$ for the unique arc
$\mathbf{v} \in \mathcal{S}_{\Sigma}(\mathbf{p})$ such that the
centre of $v$ is contained in $V(\mathbf{p},\mathbf{v})$.  In particular, $v \cap V(\mathbf{p},\mathbf{v}) \neq \emptyset$. For two pairs $(p,v)$ and $(\mathbf{p},\mathbf{v})$, we write
\begin{displaymath} (p,v) \prec (\mathbf{p},\mathbf{v}) \quad \Longleftrightarrow \quad p \subset \mathbf{p} \text{ and } v \prec \mathbf{v}. \end{displaymath}
We remark that by Lemma \ref{conelem}, if  $\Delta \leq \Sigma$
and $(p,v) \prec (\mathbf{p},\mathbf{v})$, then $v \subset
V(\mathbf{p},I_{\mathbf{v}})$ for an arc $I_{\mathbf{v}}\subset
S(\mathbf{p})$ of length $\sim \Sigma$ with $\mathbf{v}\subset
I_{\mathbf{v}}$.
\end{notation}

The "$\prec$" relation is illustrated in Figure \ref{fig1}. It gives a precise meaning to "dyadic parents" of pairs $(p,v)$ with $p \in \mathcal{D}_{\delta}$ and $v \in S(p)$. We just have to keep in mind that if $(p,v) \prec (\mathbf{p},\mathbf{v})$, then it is not quite true that $v \subset \mathbf{v}$. A good substitute is the inclusion $v \subset V(\mathbf{p},I_{\mathbf{v}})$.

\begin{definition}[Skeleton]\label{def:skeleton} Let $0 < \delta \leq \Delta$ and $0 < \sigma \leq \Sigma$ be dyadic rationals. Assume that $p \in \mathcal{D}_{\delta}$ and $E_{\sigma}(p) \subset \mathcal{S}_{\sigma}(p)$. The \emph{$(\Delta,\Sigma)$-skeleton} of $E_{\sigma}(p)$ is the set
\begin{displaymath} E_{\Sigma}(p) = \{\mathbf{v} \in \mathcal{S}_{\Sigma}(\mathbf{p}) : v \prec \mathbf{v} \text{ for some } v \in E_{\sigma}(p)\}, \end{displaymath}
where $\mathbf{p} \in \mathcal{D}_{\Delta}$ is the unique dyadic cube with $p \subset \mathbf{p}$. (It is important to note that the $(\Delta,\Sigma)$-skeleton of $E_{\sigma}(p)$ is a subset of $\mathcal{S}_{\Sigma}(\mathbf{p})$ instead of $\mathcal{S}_{\Sigma}(p)$. These coincide if $\Delta = \delta$.)

We also need the following version of the definition. Let $P \subset \mathcal{D}_{\delta}$, and assume that we are given a (possibly empty) family $E_{\sigma}(p) \subset \mathcal{S}_{\sigma}(p)$ for all $p \in P$. Write $\Omega = \{(p,v) : p \in P \text{ and } v \in E_{\sigma}(p)\}$. Then, the
\emph{$(\Delta,\Sigma)$-skeleton of $\Omega$} is defined to be
\begin{displaymath} \Omega^{\Delta}_{\Sigma} := \{(\mathbf{p},\mathbf{v}) : \mathbf{p} \in \mathcal{D}_{\Delta}, \, \mathbf{v} \in \mathcal{S}_{\Sigma}(\mathbf{p}), \text{ and } (p,v) \prec (\mathbf{p},\mathbf{v}) \text{ for some } (p,v) \in \Omega\}. \end{displaymath}
In other words, $\Omega^{\Delta}_{\Sigma}$ consists of pairs $(\mathbf{p},\mathbf{v})$ such that $\mathbf{p} \in \mathcal{D}_{\Delta}(P)$, and $\mathbf{v} \in E_{\Sigma}(p)$ for some $p \in P$ with $p \subset \mathbf{p}$. We write
\begin{displaymath} E_{\Sigma}(\mathbf{p}) := \{\mathbf{v} \in \mathcal{S}_{\Sigma}(\mathbf{p}) : (\mathbf{p},\mathbf{v}) \in \Omega^{\Delta}_{\Sigma}\} \quad \text{and} \quad
P_{\Delta} := \{\mathbf{p} \in \mathcal{D}_{\Delta} : E_{\Sigma}(\mathbf{p}) \neq \emptyset\}. \end{displaymath} \end{definition}

\begin{remark} Note that $E_{\Sigma}(\mathbf{p})$ is the union of all the $(\Delta,\Sigma)$-skeletons $E_{\Sigma}(p)$ for all $p \in P$ with $p \subset \mathbf{p}$. Thus, $E_{\Sigma}(\mathbf{p})$ may be rather wild, even if the individual sets $E_{\sigma}(p)$ are nice (say, $(\sigma,s)$-sets). Proposition \ref{prop6} will regardless give us useful information about the sets $E_{\Sigma}(\mathbf{p})$, provided that we are first allowed to prune $\Omega$ (and hence the sets $E_{\sigma}(p)$) slightly. \end{remark}

\begin{figure}[h!]
\begin{center}
\begin{overpic}[scale = 1]{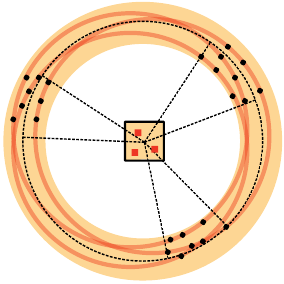}
\put(40,38){$\mathbf{p}$}
\end{overpic}
\caption{The red squares represent the centres of three circles
$S(p_{1}),S(p_{2}),S(p_{3})$, where $p_{1},p_{2},p_{3} \in
\mathcal{D}_{\delta}$. In the figure we have $p_{1},p_{2},p_{3}
\subset \mathbf{p}$ for a certain $\mathbf{p} \in
\mathcal{D}_{\Delta}$, where $\Delta > \delta$. Therefore the red
$\delta$-annuli
$S^{\delta}(p_{1}),S^{\delta}(p_{2}),S^{\delta}(p_{3})$ are
contained in the (yellow) $\Delta$-annulus
$S^{\Delta}(\mathbf{p})$. The black dots on the red circles
represent the sets
$E_{\sigma}(p_{1}),E_{\sigma}(p_{2}),E_{\sigma}(p_{3})$, and the
three longer arcs spanning the cones form the set
$E_{\Sigma}(\mathbf{p}) \subset S(\mathbf{p})$. As shown in the
figure, each pair $(p_{j},v)$ with $v \in E_{\sigma}(p_{j})$
satisfies $(p_{j},v) \prec (\mathbf{p},\mathbf{v})$ for some
$\mathbf{v} \in E_{\Sigma}(\mathbf{p})$.}\label{fig1}
\end{center}
\end{figure}

Let us recap the meaning of $(\delta,s,C,M)$-configurations from Definition \ref{d:config}. These were defined to be sets $\Omega \subset \R^{5}$ such that $P = \pi_{\R^{3}}(\Omega) \subset \mathbf{D}$ is a non-empty $(\delta,s,C)$-set, and $E(p) = \{v \in \R^{2} : (p,v) \in \Omega\}$ is a $(\delta,s,C)$-subset of $S(p)$ for all $p \in P$, satisfying $|E(p)|_{\delta} \equiv M$. We next pose the following dyadic (and slightly generalised) variant of the definition.

\begin{definition}\label{d:config2}
Let $0 < s \leq 1$, $C>0$, and let $0 < \delta  \leq 1$, $0 < \delta,\sigma \leq
1$ be dyadic rationals. A $(\delta,\sigma,s,C,M)$-configuration is a set of the form
\begin{displaymath} \Omega = \{(p,v) : p \in P \text{ and } v \in E_{\sigma}(p)\}, \end{displaymath}
where $P \subset \mathcal{D}_{\delta}$ is a $(\delta,s,C)$-set, and $E_{\sigma}(p) \subset \mathcal{S}_{\sigma}(p)$, for $p \in P$, is a $(\sigma,s,C)$-set of constant cardinality $|E_{\sigma}(p)| \equiv M$. If $\Omega$ is a $(\delta,\sigma,s,C,M)$-configuration for some $M$, we simply say that $\Omega$ is a $(\delta,\sigma,s,C)$-configuration.

 \end{definition}

In the new terminology, the $(\delta,s,C,M)$-configurations from Definition \ref{d:config} correspond to $(\delta,\delta,s,C,M)$-configurations. To be precise, we should distinguish between $(\delta,s,C,M)$-configurations and "dyadic" $(\delta,s,C,M)$-configurations, but we will not do this: in the sequel, the terminology will always refer to the dyadic variant in Definition \ref{d:config2}.

We record the following simple \emph{refinement principle} for $(\delta,\sigma,s,C,M)$-configurations:
\begin{lemma}[Refinement principle]\label{refinement} Let $\Omega$ be a $(\delta,\sigma,s,C)$-configuration, and let $G \subset \Omega$ be a subset with $|G| \geq c|\Omega|$, where $c \in (0,1]$. Then, there exists a $(\delta,\sigma,s,2C/c)$-configuration $\Omega' \subset G$ with $|\Omega'| \geq (c^{2}/4)|\Omega|$. \end{lemma}

\begin{proof} Write $\Omega = \{(p,v) : p \in P \text{ and } v \in E_{\sigma}(p)\}$. For $p \in P$, let $G(p) := \{v \in E_{\sigma}(p) : (p,v) \in G\}$. Note that (with $M := |E_{\sigma}(p)|$), we have
\begin{displaymath} cM|P| = c|\Omega| \leq |G| = \sum_{p \in P} |G(p)| \leq M|\{p : |G(p)| \geq cM/2\}| + cM|P|/2.  \end{displaymath}
It follows that the set $P' := \{p \in P : |G(p)| \geq cM/2\}$ has $|P'| \geq c|P|/2$. For each $p \in P'$, let $E_{\sigma}'(p) \subset G(p)$ be a set with $|E_{\sigma}'(p)| = cM/2 = c|E_{\sigma}(p)|/2$. Now, $P'$ is a $(\delta,s,2C/c)$-set, $E_{\sigma}'(p)$ is a $(\sigma,s,2C/c)$-set for all $p \in P'$, and
\begin{displaymath} \Omega' := \{(p,v) : p \in P' \text{ and } v' \in E_{\sigma}'(p)\} \subset G \end{displaymath}
is the desired $(\delta,\sigma,2C/c)$-configuration with $|\Omega'| = c|P'|M/2 \geq (c^{2}/4)|\Omega|$. \end{proof}

We then arrive at the main result of this section.

\begin{proposition}\label{prop6}
Let  $0 < \delta \leq \Delta \leq 1$ and $0 < \sigma \leq \Sigma
\leq 1$ be dyadic numbers with $\delta \leq \sigma$ and $\Delta \leq \Sigma$. For every $C \geq 1$, there exists a
constant $C' \approx_{\delta} C$ such that the following
holds. If $\Omega$ is a $(\delta,\sigma,s,C)$-configuration, then there exists a
subset $G \subset \Omega$ with $|G| \approx_{\delta} |\Omega|$ whose
$(\Delta,\Sigma)$-skeleton $G^{\Delta}_{\Sigma}$ is a
$(\Delta,\Sigma,s,C')$-configuration with the property
\begin{equation}\label{form40} |\{(p,v) \in G : (p,v) \prec (\mathbf{p},\mathbf{v})\}|
\approx_{\delta} \frac{|\Omega|}{|G^{\Delta}_{\Sigma}|}, \qquad (\mathbf{p},\mathbf{v}) \in G^{\Delta}_{\Sigma}.
\end{equation}
\end{proposition}

\begin{remark}\label{rem1}
Let $0 < \delta \leq \Delta \leq 1$ and $0 < \sigma \leq \Sigma \leq 1$ be dyadic rationals.
Let $\Omega =\{(p,v) : p \in
P \text{ and } v\in E_{\sigma}(p)\}$, as in Proposition \ref{prop6}. We will use the following notation:
\begin{displaymath}
 \mathbf{p} \otimes \mathbf{v} := \{(p,v) \in \Omega : (p,v) \prec (\mathbf{p},\mathbf{v}))\}, \qquad \mathbf{p} \in \mathcal{D}_{\Delta}, \, \mathbf{v} \in \mathcal{S}_{\Sigma}(\mathbf{p}).
 \end{displaymath}
 (So, the sets $\mathbf{p} \otimes \mathbf{v}$ depend on "$\Omega$" even though this is suppressed from the notation). The sets $\mathbf{p} \otimes \mathbf{v}$ are disjoint for distinct $(\mathbf{p},\mathbf{v})$ with
$\mathbf{p}\in \mathcal{D}_{\Delta}$ and $\mathbf{v}\in
\mathcal{S}_{\Sigma}(\mathbf{p})$. Indeed, if $\mathbf{p} \neq \mathbf{p}'$, evidently no pair
$(p,v)$ can lie in $\mathbf{p} \otimes \mathbf{v}$ and
 $\mathbf{p}' \otimes \mathbf{v}'$
  for any $\mathbf{v},\mathbf{v}' \in \mathcal{S}(\mathbf{p})$.
   On the other hand, if $\mathbf{p} = \mathbf{p}'$ and $p \in \mathcal{D}_{\delta}$ with $p \subset \mathbf{p} = \mathbf{p}'$,
   then for each arc $v \in \mathcal{S}_{\sigma}(p)$ we have chosen exactly one arc $\mathbf{v} \in \mathcal{S}_{\Sigma}(\mathbf{p})$
   such that $v \prec \mathbf{v}$. That is, $(p,v) \prec
   (\mathbf{p},\mathbf{v})$ for only one $ (\mathbf{p},\mathbf{v})$.
\end{remark}

To simplify the proof of Proposition \ref{prop6} slightly, we extract the following lemma:
\begin{lemma}\label{lemma9} Let $0 < \delta \leq \Delta \leq 1$ be dyadic rationals, and let $P \subset \mathcal{D}_{\delta}$ be a $(\delta,s,C)$-set. Assume that every set
\begin{displaymath} \mathbf{p} \cap P := \{p \in P : p \subset \mathbf{p}\}, \qquad \mathbf{p} \in P_{\Delta} := \mathcal{D}_{\Delta}(P), \end{displaymath}
has cardinality $|\mathbf{p} \cap P| \in [m,2m]$ for some $m \geq 1$. Then $P_{\Delta}$ is a $(\Delta,s,C')$-set with $C' \sim C$. \end{lemma}

\begin{proof} Let $Q \in \mathcal{D}_{r}$ with $\Delta \leq r \leq 1$. Then,
\begin{displaymath} m \cdot |Q \cap P_{\Delta}| \leq |Q \cap P| \lesssim Cr^{s}|P| \leq 2m \cdot Cr^{s}|P_{\Delta}|. \end{displaymath}
Dividing by "$m$" yields a dyadic version of the $(\Delta,s,C')$-set condition for $P_{\Delta}$. This easily implies the usual $(\Delta,s,C')$-set condition with a slightly worse "$C'$". \end{proof}

We then complete the proof of Proposition \ref{prop6}.

 \begin{proof}[Proof of Proposition \ref{prop6}] In the first part of the proof, we construct certain sets
 $P_{\Delta}\subset \mathcal{D}_{\Delta}$ and $\mathbf{E}(\mathbf{p})\subset \mathcal{S}_{\Sigma}(\mathbf{p})$, $\mathbf{p}\in
 P_{\Delta}$, by pigeonholing, and we define $\bar{\Omega} = \{(\mathbf{p},\mathbf{v}) : p \in P_{\Delta} \text{ and } \mathbf{v} \in
\mathbf{E}(\mathbf{p})\}$. The set $G \subset \Omega$ will be defined as
\begin{equation}\label{form169} G := \bigcup_{(\mathbf{p},\mathbf{v}) \in \bar{\Omega}} \mathbf{p} \otimes \mathbf{v} \subset \Omega. \end{equation}
This implies trivially that $G^{\Delta}_{\Sigma} \subset \bar{\Omega}$. In the second part of the proof, we show that $\bar{\Omega}$ is a $(\Delta,\Sigma,s,C')$-configuration satisfying \eqref{form40}, so in particular $\mathbf{p} \otimes \mathbf{v} \neq \emptyset$ for all $(\mathbf{p},\mathbf{v}) \in \bar{\Omega}$. Therefore also $G^{\Delta}_{\Sigma} \supset \bar{\Omega}$ by definitions, and the proof will be complete.

Write $\Omega = \{(p,v) : p \in P \text{ and } v \in E(p)\}$, where $P \subset \mathcal{D}_{\delta}$ and $E(p) \subset \mathcal{S}_{\delta}(p)$. To construct $P_{\Delta}$, consider initially $P_\Delta^1 := \mathcal{D}_{\Delta}(P)$. Each $\mathbf{p} \in P_\Delta^1$ may contain different
numbers of $\delta$-cubes from $P$, and to fix this we perform our first
pigeonholing. Let $\mathbf{p} \in P_\Delta^1$ and define
$$  \mathcal{D}_{\delta}(\mathbf{p} \cap P) :=\{p \in P : p \subset \mathbf{p}\} \mbox{ and }
 P_{\Delta,i}^{1} := \{\mathbf{p} \in P_\Delta^1 : 2^{i-1} \leq |\mathcal{D}_{\delta}(\mathbf{p} \cap P)| < 2^i \}   $$
for $i \geq 1$. Observing that
$$  |P| = \sum_{i\in \mathbb{N}} \sum_{\mathbf{p} \in P_{\Delta,i}} |\mathcal{D}_{\delta}(\mathbf{p} \cap P)| $$
and noting that $P_{\Delta,i}$ is empty if
$2^{i-1}>|\mathcal{D}_{\delta}|\sim \delta^{-3}$, we conclude by
pigeonholing that there exists $i_0\lessapprox_{\delta} 1$ such that
$$ |P| \approx_\delta  \sum_{\mathbf{p} \in P_{\Delta,i_0}} |\mathcal{D}_{\delta}(\mathbf{p} \cap P)|.$$
For this index $i_0$, we then have
$|P_{\Delta,i_0}|\approx_\delta |P|/2^{i_0}$.

For each $\mathbf{p}=(\mathbf{x},\mathbf{r})\in P^1_{\Delta}$, we
next construct families $\mathcal{S}^j_{\Sigma}(\mathbf{p})$,
$j\in \mathbb{N}$, that will be used for the definition of the
sets $\mathbf{E}(\mathbf{p})$. Again, we use pigeonholing to find
a subset of $\{ \mathbf{p} \otimes \mathbf{v}\colon \mathbf{p} \in
P_{\Delta,i_0},\, \mathbf{v} \in
\calS_\Sigma(\mathbf{p})\}$ of typical cardinality.

First, since each $\mathbf{p}\in P_{\Delta,i_0}$
contains $\sim 2^{i_0}$ cubes $p\in P$, since we
have $|E_\sigma(p)| \equiv M$ for all of them, and since for each such
$p$ and $v\in E_{\sigma}(p)$ there exists a unique $\mathbf{v}
\in \calS_\Sigma(\mathbf{p})$ such that $(p,v) \prec (\mathbf{p},
\mathbf{v} )$, we obtain
\begin{equation}\label{eq:sum_M}
 \sum_{\mathbf{v} \in \calS_\Sigma(\mathbf{p})}|\mathbf{p}
\otimes \mathbf{v}| \sim 2^{i_0}M.
\end{equation}
 Next,
for $j \geq 1$, we define
\begin{equation}\label{sj}
  \calS_\Sigma^j(\mathbf{p}) := \{\mathbf{v} \in \calS_\Sigma(\mathbf{p}):  |\mathbf{p} \otimes \mathbf{v}| \in [2^{j-1}, 2^{j})   \} .
\end{equation}
Since $\delta \leq \sigma$ by assumption, we have $|\mathbf{p} \otimes \mathbf{v}| \lesssim \delta^{-3}\sigma^{-1} \leq \delta^{-4}$. It follows  that $\calS_\Sigma^j(\mathbf{p})= \emptyset$ if $2^j
\gg \delta^{-4}$. Hence, by \eqref{eq:sum_M} and
pigeonholing, there exists $j(\mathbf{p}) \lessapprox_{\delta}1$
such that
\begin{equation*}
  \sum_{\mathbf{v} \in \calS_\Sigma^{j(\mathbf{p})}(\mathbf{p})}|\mathbf{p} \otimes \mathbf{v}| \approx_{\delta} 2^{i_0}M.
\end{equation*}
By a second pigeonholing, since $|P_{\Delta,i_0}|\lesssim\delta^{-3}$,
 there exists $j_0 \lessapprox_{\delta} 1$ and $P_\Delta \subset
P_{\Delta,i_0} \subset P_\Delta^1$ such that
\begin{equation}\label{eq:PDelta}
 |P_\Delta| \approx_{\delta}
|P_{\Delta,i_0}|\approx_{\delta}  \frac{|P|}{2^{i_0}},
\end{equation}
and
\begin{equation}\label{approx1}
  \sum_{\mathbf{v} \in \calS_\Sigma^{j_0}(\mathbf{p})}|\mathbf{p} \otimes \mathbf{v}|
  \approx_{\delta}  2^{i_0}M, \quad \mathbf{p} \in P_\Delta.
\end{equation}
In \eqref{eq:CardSigmaj}, we will see that all the sets
$\calS_\Sigma^{j_0}(\mathbf{p})$, $p\in P_{\Delta}$, have cardinality $\approx_{\delta} 2^{i_{0}}M/2^{j_{0}}$, but the sets $\mathbf{E}(\mathbf{p})$, $\mathbf{p} \in P_{\Delta}$ are required to have exactly the same cardinality. To this end, we define
\begin{equation}\label{eq:MSigma} M_\Sigma:=
\min\{|\calS_\Sigma^{j_0}(\mathbf{p})| : \mathbf{p} \in
P_\Delta\},\end{equation} which satisfies $M_{\Sigma}\geq 1$ by
\eqref{approx1}. For each $\mathbf{p} \in P_\Delta$, we choose
$\mathbf{E}(\mathbf{p})$ to be an arbitrary subset of
$\calS_\Sigma^{j_0}(\mathbf{p})$ of cardinality $
|\mathbf{E}(\mathbf{p})|= M_\Sigma$. Now, as already announced at the start of the proof, we set $\bar{\Omega} := \{(\mathbf{p},\mathbf{v}) : \mathbf{p} \in P_{\Delta} \text{ and } \mathbf{v} \in \mathbf{E}(\mathbf{p})\}$, and we define $G \subset \Omega$ with the formula \eqref{form169}. We record that $|\bar{\Omega}| = |P_{\Delta}|M_{\Sigma}$.

Keeping in mind Remark \ref{rem1} about the disjointness of the sets $\mathbf{p}\otimes \mathbf{v}$, and
using the definition of the sets $\mathbf{E}(\mathbf{p})\subset
\mathcal{S}^{j_0}_{\Sigma}(\mathbf{p})$, we have
\begin{displaymath}
|G| = \sum_{\mathbf{p} \in P_{\Delta}} \sum_{\mathbf{v} \in \mathbf{E}(\mathbf{p})} |\mathbf{p} \otimes \mathbf{v}| \geq |P_{\Delta}| \cdot M_{\Sigma} \cdot
2^{j_0 - 1}
\end{displaymath}
To conclude that $|G| \gtrapprox_{\delta} |\Omega| = |P|M$, it suffices to check that
\begin{equation}\label{eq:goalMSigma}
M_\Sigma \approx_{\delta} \frac{|P|M}{|P_\Delta|2^{j_0}}.
\end{equation} Since $(\mathbf{p} \otimes \mathbf{v}) \cap
(\mathbf{p} \otimes \mathbf{v'}) =\emptyset$ for $\mathbf{v} \ne
\mathbf{v'}$, recalling first \eqref{sj} and \eqref{approx1}, and
then \eqref{eq:PDelta}, we have
\begin{equation}\label{eq:CardSigmaj}
|\calS_\Sigma^{j_0}(\mathbf{p})| \approx_{\delta}
\frac{2^{i_0}M}{2^{j_0}} \approx_{\delta}
\frac{|P|M}{|P_\Delta|2^{j_0}}, \quad \mathbf{p} \in P_\Delta.
\end{equation} Hence \eqref{eq:goalMSigma} holds by the definition
of $M_{\Sigma}$ in \eqref{eq:MSigma},  and therefore
$|G| \approx_{\delta} |\Omega|$, as desired. In retrospect, \eqref{eq:CardSigmaj} also implies $M_{\Sigma} \approx 2^{i_{0}}M/2^{j_{0}}$.

We next verify \eqref{form40}. The definition of
$\calS_\Sigma^{j_{0}}(\mathbf{p})$ in \eqref{sj} results in
\begin{equation}\label{eq:pv}
 |\mathbf{p} \otimes \mathbf{v}| \sim 2^{j_0} \stackrel{\eqref{eq:goalMSigma}}{\approx_{\delta}}
\frac{|P|M}{|P_\Delta|M_\Sigma} = \frac{|\Omega|}{|\bar{\Omega}|}, \quad  \mathbf{p} \in P_\Delta
\mbox{ and } \mathbf{v} \in \mathbf{E}(\mathbf{p}). \end{equation}
By the definition of $G$, we have $(\mathbf{p}\otimes
\mathbf{v}) \cap G = \mathbf{p}\otimes \mathbf{v} \neq \emptyset$ for
$\mathbf{p}\in P_{\Delta}$ and $\mathbf{v}\in
E_{\Sigma}(\mathbf{v})$, and evidently $G^{\Delta}_{\Sigma} = \bar{\Omega}$. Therefore \eqref{form40} follows from \eqref{eq:pv}.

Next we  show that $P_\Delta$ is a $(\Delta,s,C')$-set and
$\mathbf{E}(\mathbf{p})$ is a $(\Sigma,s,C')$-set for all
$\mathbf{p} \in P_\Delta$. This will show that $\bar{\Omega} = G^{\Delta}_{\Sigma}$ is a $(\Delta,\Sigma,s,C',M_{\Sigma})$-configuration, and conclude the proof of the proposition. To verify that $P_\Delta$ is a $(\Delta,s,C')$-set, note that $P' = \bigcup_{\mathbf{p} \in
P_\Delta}\mathcal{D}_{\delta}(\mathbf{p} \cap P)$ has $|P'| \approx_{\delta} |P|$ by \eqref{eq:PDelta}. Therefore $P'$ is a $(\delta,s,C')$-set with $C' \approx_{\delta} C$. But now $P_{\Delta} = \mathcal{D}_{\Delta}(P')$, and every cube in $P_{\Delta}$ contains $\sim 2^{i_{0}}$ elements of $P'$. Therefore, it follows from Lemma \ref{lemma9} that $P_{\Delta}$ is a $(\Delta,s,C')$-set.

It remains to verify that $\mathbf{E}(\mathbf{p})$ is a $(\Sigma,s,C')$-set for
all $\mathbf{p} \in P_\Delta$. Fix $\mathbf{p} \in P_\Delta$, let $\Sigma
\le r \le 1$ be a dyadic number and $\mathbf{v}_r \in
\mathcal{S}_r(\mathbf{p})$. Our goal is to show (and it suffices to show) that $|\{\mathbf{v} \in \mathbf{E}(\mathbf{p}) : \mathbf{v} \subset \mathbf{v}_{r}\}| \lessapprox_{\delta} Cr^{s}M_{\Sigma}$. To this end, we first note that
\begin{equation}\label{eq:goal_v_prec2}
   |\{ v \in E_\sigma(p) : v \prec \mathbf{v}_r\}| \lesssim C r^sM,\quad p \in \mathbf{p} \cap P.
\end{equation}
This follows by observing that all $v \in E_{\sigma}(p)$ with $v \prec \mathbf{v}_{r}$ are contained in $V(\mathbf{p},I_{\mathbf{v}}) \cap S(p)$ by Lemma
\ref{conelem}, and $\diam(V(\mathbf{p},I_{\mathbf{v}}) \cap S(p)) \lesssim r$.

Next, observe that
\begin{equation}\label{eq:SetUnion}
\mathop{\bigcup_{\mathbf{v} \in \mathbf{E}(\mathbf{p})}}_{\mathbf{v} \subset \mathbf{v}_{r}}
\mathbf{p} \otimes \mathbf{v} \subset  \bigcup_{p \in
\mathbf{p} \cap P} \{(p,v) : v \in \mathbf{E}(p) \text{ and } v \prec \mathbf{v}_r\}.
\end{equation}
Thus recalling that
$|\mathbf{p} \otimes \mathbf{v}|\sim  2^{j_0}$ for $\mathbf{v}\in \mathbf{E}(\mathbf{p})\subset
\mathcal{S}^{j_0}_{\Sigma}(\mathbf{p})$, we have
\begin{align*}
  |\{\mathbf{v} \in \mathbf{E}(\mathbf{p}) : \mathbf{v} \subset \mathbf{v}_{r}\}| \cdot 2^{j_0} & \sim
  \Big | \mathop{\bigcup_{\mathbf{v} \in \mathbf{E}(\mathbf{p})}}_{\mathbf{v} \subset \mathbf{v}_{r}} \mathbf{p} \otimes \mathbf{v} \Big| \\
   & \overset{}{\leq} \Big|\bigcup_{p \in
\mathbf{p} \cap P} \{(p,v) : v \in E_\sigma(p) \text{ and } v \prec \mathbf{v}_r\}\Big| \\
   & \leq |\mathbf{p} \cap P| \max_{p \in
\mathbf{p} \cap P} |\{ (p,v) : v \in E_\sigma(p) \text{ and } v \prec \mathbf{v}_r\}|  \\
   & \overset{\eqref{eq:goal_v_prec2}}{\lesssim} 2^{i_0} \cdot Cr^s M.
\end{align*}
To conclude the proof, we recall from \eqref{eq:CardSigmaj} that $M_\Sigma \approx_{\delta}
2^{i_0}M/2^{j_0}$. Therefore,
\begin{displaymath} |\{\mathbf{v} \in \mathbf{E}(\mathbf{p}) : \mathbf{v} \subset \mathbf{v}_{r}\}| \lessapprox_{\delta} Cr^s \cdot \frac{2^{i_0}M}{2^{j_0}} \approx_{\delta} Cr^s M_\Sigma, \qquad \mathbf{v} \in \mathcal{S}_{r}(\mathbf{p}), \end{displaymath}
as desired. We have now proven that $G^{\Delta}_{\Sigma} = \bar{\Omega} = \{(\mathbf{p},\mathbf{v}) : \mathbf{p} \in P_{\Delta} \text{ and } \mathbf{v} \in \mathbf{E}(\mathbf{p})\}$ is a $(\Delta,\Sigma,s,C',M_{\Sigma})$-configuration, as claimed. \end{proof}


\section{Rectangles and geometry}\label{s:tangencies}

The purpose of this section is to gather facts about curvilinear rectangles (that is: pieces of annuli) and their geometry. Similar considerations are present in every paper regarding curvilinear Kakeya problems and its relatives, for example \cite{2022arXiv220702259P,MR1800068,MR1473067,Wolff99}.

\subsection{$(\delta,\sigma)$-rectangles and some basic properties}

\begin{definition}[$(\delta,\sigma)$-rectangle] Let $\delta,\sigma \in (0,1]$. By definition, a $(\delta,\sigma)$-rectangle is a set of the form
\begin{displaymath} R^{\delta}_{\sigma}(p,v) := S^{\delta}(p) \cap B(v,\sigma), \end{displaymath}
where $p \in \mathbf{D}$ and $v \in S(p)$. For $C > 0$, we write $CR^{\delta}_{\sigma}(p,v) := R^{C\delta}_{C\sigma}(p,v)$. \end{definition}

\begin{remark} If the reader is familiar with the terminology of Wolff's paper \cite{MR1800068},
we mention here that Wolff's "$(\delta,t)$-rectangles" are the same as our $(\delta,\sqrt{\delta/t})$-rectangles. While Wolff's notation for these objects is more elegant, the purpose of our terminology is to handle e.g. $(\delta,\delta/\sqrt{\lambda t})$-rectangles without having to introduce further notation.
\end{remark}

For the next lemma, we recall that $\Delta(p,p') = ||x - x'| - |r - r'||$ for $p = (x,r) \in \R^{2} \times (0,\infty)$ and $p' = (x',r') \in \R^{2} \times (0,\infty)$.

\begin{lemma}\label{lemma5} Let $p,q \in \mathbf{D}$ be points with $|p - q| = t$ and $\Delta(p,q) = \lambda$. Then, the intersection $S^{\delta}(p) \cap S^{\delta}(q)$ can be covered by boundedly many $(\delta,\sigma)$-rectangles, where
\begin{displaymath} \sigma := \delta/\sqrt{(\lambda + \delta)(t + \delta)}. \end{displaymath}
Conversely, assume that $v \in S(p) \cap S(q)$. Then, for all $C \geq 1$, we have
\begin{equation}\label{form71} CR_{\sigma}^{\delta}(p,v) \subset S^{C'\delta}(q), \end{equation}
where $C' \lesssim \max\{C,C^{2}\delta/(\lambda + \delta)\} \leq C^{2}$. \end{lemma}

\begin{proof} The first statement is well-known, see for example \cite[Lemma 3.1]{Wolff99}, so we only prove the inclusion \eqref{form71}. Recall that $p,q \in \mathbf{D}$, so the radii of the circles $S(p),S(q)$ are bounded between $\tfrac{1}{2}$ and $1$. For this reason, there is no loss of generality in assuming that $S(p)$ is the unit circle $S(p) = S^{1}$, that the radius of $S(q)$ is $r \in [\tfrac{1}{2},1)$, and that $S(q)$ is centred at a point $z = (x,0)$ with $x > 0$. These are incidentally the same normalisations as in  \cite[Lemma 3.1]{Wolff99}, and our proof is overall very similar to the argument in that lemma. With this notation, we observe that
\begin{displaymath} \lambda = \Delta(p,q) = ||z - 0| - |1 - r|| = |(1 - x) - r| \sim |(1 - x)^{2} - r^{2}|, \end{displaymath}
and $t = |p - q| \sim x + (1 - r)$. Since $S(p) \cap S(q) \neq \emptyset$, we moreover have $x \geq 1 - r$, and therefore $t \sim x$. We may assume that $x \sim t \geq \delta$, since otherwise \eqref{form71} follows from $CR^{\delta}_{\sigma}(p,v) \subset S^{C\delta}(p) \subset S^{C\delta + t}(q) \subset S^{2C\delta}(q)$.

We assume that $v \in S(p) \cap S(q)$. Since $S(p) = S^{1}$, we may therefore write $v = e^{i\theta_{0}}$ for some $\theta_{0} \in (-\pi,\pi]$. Recalling that $z$ is the center of $S(q)$, we have
\begin{equation}\label{form83} 1 - 2 x \cos \theta_{0} + x^{2} = |e^{i\theta_0} - z|^{2} = r^{2} \quad \Longleftrightarrow \quad \cos \theta_{0} = \frac{1 - r^{2} + x^{2}}{2x}. \end{equation}
We further rewrite this as
\begin{displaymath} \cos \theta_{0} = \frac{1 - r^{2} + x^{2}}{2x} = 1 - \frac{r^{2} - (1 - x)^{2}}{2x} =: 1 - h, \end{displaymath}
where $|h| = |r^{2} - (1 - x)^{2}|/(2x) \sim \lambda/t$. We now claim that
\begin{equation}\label{form82a} |\theta - \theta_{0}| \leq C\sigma \sim \frac{C\delta}{\sqrt{(\lambda + \delta)t}} \quad \Longrightarrow \quad ||e^{i\theta} - z|^{2} - r^{2}| \leq C'\delta, \end{equation}
where $C' \sim \max\{C,C^{2}\delta/\lambda\}$. This means that a circular arc on $S(p)$ of length $C\sigma$ around $v = e^{i\theta_{0}}$ is contained in $S^{C'\delta}(q)$. Since every point on $CR^{\delta}_{\sigma}(p,v)$ is within distance $C\delta \leq C'\delta$ from such a circular arc, \eqref{form71} follows immediately.

Revisiting the calculation in \eqref{form83}, the condition on the right hand side of \eqref{form82a} is equivalent to
\begin{displaymath} |1 - 2x \cos \theta + x^{2} - r^{2}| \lesssim C'\delta \quad \Longleftrightarrow \quad |\cos \theta - \cos \theta_{0}| = \left|\cos \theta - \frac{1 - r^{2} + x^{2}}{2x} \right| \leq \frac{C'\delta}{2x}. \end{displaymath}
Moreover, the right hand side here is $\sim C'\delta/t$. To prove that this estimate is valid whenever $|\theta - \theta_{0}| \leq C\sigma$, we note that
\begin{displaymath} 1 = \sin^{2}\theta_{0} + \cos^{2} \theta_{0} = \sin^{2}\theta_{0} + (1 - h)^{2} \quad \Longrightarrow \quad \sin^{2} \theta_{0} = 2h - h^{2}, \end{displaymath}
and therefore
\begin{displaymath} |\cos' \theta_{0}| = |\sin \theta_{0}| = \sqrt{|2h - h^{2}|} \lesssim \sqrt{\lambda/t}. \end{displaymath}
recalling that $|h| \sim \lambda/t \leq 1$. Finally, for all $|\theta - \theta_{0}| \leq C\sigma$, we have
\begin{displaymath} |\cos \theta - \cos \theta_{0}| = \left| \int_{\theta_{0}}^{\theta} \cos' \zeta \, d\zeta \right| \leq \int_{\theta_{0}}^{\theta} |\cos' \zeta - \cos' \theta_{0}| \, d\zeta + |\theta - \theta_{0}| \cdot |\cos' \theta_{0}| =: I_{1} + I_{2}. \end{displaymath}
The term $I_{2}$ is bounded from above by
\begin{displaymath} I_{2} \lesssim C \sigma \cdot \sqrt{\lambda/t} \sim \frac{C\delta}{\sqrt{(\lambda + \delta)t}} \cdot \sqrt{\lambda/t} \sim \frac{C\delta}{t}. \end{displaymath}
Since $\cos' = \sin$ is $1$-Lipschitz, the term $I_{1}$ is bounded from above by
\begin{displaymath} I_{1} \leq \int_{\theta_{0}}^{\theta} |\zeta - \theta_{0}| \, d\zeta = \int_{0}^{|\theta - \theta_{0}|} |\zeta| \, d\zeta \sim |\theta - \theta_{0}|^{2} \leq C^{2}\sigma^{2} = \frac{C^{2}\delta^{2}}{(\lambda + \delta) t}. \end{displaymath}
This completes the proof of \eqref{form82a} with constant $C' \sim \max\{C\delta,C^{2}\delta/(\lambda + \delta)\}$.  \end{proof}

\begin{cor}\label{cor2} Let $p,q \in \mathbf{D}$ be points with $\lambda = \Delta(p,q)$ and $t = |p - q|$. Write $\sigma := \delta/\sqrt{(\lambda + \delta)(t + \delta)}$. Assume that
\begin{displaymath} CR^{\delta}_{\sigma}(p,v) \cap CR^{\delta}_{\sigma}(q,w) \neq \emptyset \end{displaymath}
for some $v \in S(p)$, $w \in S(q)$, and $C \geq 1$. Then, $CR^{\delta}_{\sigma}(p,v) \subset C'R^{\delta}_{\sigma}(q,w)$ for some $C' \lesssim C^{4}$. \end{cor}

\begin{proof} Fix $\mathbf{v} \in CR^{\delta}_{\sigma}(p,v) \cap CR^{\delta}_{\sigma}(q,w)$. Then
\begin{displaymath} \max\{\dist(\mathbf{v},S(p)),\dist(\mathbf{v}),S(q)\}\} \leq C\delta. \end{displaymath}
Consequently, there exist points
\begin{displaymath} v' \in S(p) \cap B(v,C\sigma) \quad \text{and} \quad w' \in S(q) \cap B(w,C\sigma) \end{displaymath}
such that $|v' - \mathbf{v}| \leq C\delta$ and $|w' - \mathbf{v}| \leq C\delta$. Now we shift the circles $S(p)$ and $S(q)$ a little bit so that $\mathbf{v}$ lies in their intersection. The details are as follows. Write $p = (x,r)$, and define $p' = (x',r)$, where $x' = x + (\mathbf{v} - v')$. Thus, $S(p') = S(p) + \mathbf{v} - v'$, and
\begin{displaymath} \mathbf{v} = v' + (\mathbf{v} - v') \in S(p) + (\mathbf{v} - v') = S(p').  \end{displaymath}
We define similarly $q' := (y',r)$, where $q = (y,r)$, and $y' = y + \mathbf{v} - v'$.

With these definitions, $|p - p'| \leq C\delta$ and $|q - q'| \leq C\delta$, and
\begin{displaymath} \mathbf{v} \in S(p') \cap S(q'). \end{displaymath}
Write $\lambda' := \Delta(p',q')$ and $t' := |p' - q'|$, and $\sigma' := \delta/\sqrt{(\lambda' + \delta)(t' + \delta)}$. After a small case chase, it is easy to check that $\sigma \leq AC\sigma'$, where $A \geq 1$ is absolute (the worst case in the inequality occurs if $\lambda \leq t \leq \delta$, but $\lambda' \sim t' \sim C\delta$). It now follows from Lemma \ref{lemma5} that
\begin{displaymath} (AC^{2})R^{\delta}_{\sigma'}(p',\mathbf{v}) \subset S^{C'\delta}(q'), \end{displaymath}
where $C' \lesssim C^{4}$, since $A \geq 1$ is absolute. Finally,
\begin{align*} CR^{\delta}_{\sigma}(p,v) & = S^{C\delta}(p) \cap B(v,C\sigma)\\
& \subset S^{2C\delta}(p') \cap B(\mathbf{v},AC^{2}\sigma')\\
& \subset (AC^{2})R^{\delta}_{\sigma'}(p',\mathbf{v}) \subset S^{C'\delta}(q') \subset S^{2C'\delta}(q). \end{align*}
Since also $CR^{\delta}_{\sigma}(p,v) \subset B(\mathbf{v},2C\sigma) \subset B(w,4C\sigma)$, we have now shown that
\begin{displaymath} CR^{\delta}_{\sigma}(p,v) \subset S^{2C'\delta}(q) \cap B(w,4C\sigma) \subset 2C'R^{\delta}_{\sigma}(q,w), \end{displaymath}
as claimed. \end{proof}


\subsection{Comparable rectangles}\label{s:comparableRectangles}

\begin{definition}\label{def:comparability}
Given a constant $C \geq 1$, we say that two $(\delta,\sigma)$-rectangles
 $R_{1},R_{2}$ are \emph{$C$-comparable} if there exists a third $(\delta,\sigma)$-rectangle
  $R = R^{\delta}_{\sigma}(p,v)$ such that $R_{1},R_{2} \subset CR$.
If no such rectangle $R$ exists, we say that $R_1$ and $R_2$ are
\emph{$C$-incomparable}.
    \end{definition}

The definition of $C$-comparability raises a few questions. Is it necessary to speak about the third rectangle $R$, or is it equivalent to require that $R_{1} \subset CR_{2}$ and $R_{2} \subset CR_{1}$ (up to changing constants)? If this definition is equivalent, is it enough to require the one-sided condition $R_{1} \subset CR_{2}$? The answer to both questions is affirmative, and follows from the next lemma.

\begin{lemma}\label{lemma6} Let $0 < \delta \leq \sigma \leq 1$, and let $R_{1},R_{2}$ be $(\delta,\sigma)$-rectangles satisfying $R_{1} \subset CR_{2}$ for some $C \geq 1$. Then $CR_{2} \subset C'R_{1}$ for some $C' \lesssim C^{5}$. \end{lemma}

\begin{proof} Write $R_{1} = R_{\sigma}^{\delta}(p_{1},v_{1})$ and $R_{2} = R^{\delta}_{\sigma}(p_{2},v_{2})$. Let $t := |p_{1} - p_{2}|$ and $\lambda := \Delta(p_{1},p_{2})$. Write
\begin{displaymath} \bar{\sigma} := C\delta/\sqrt{(\lambda + C\delta)(t + C\delta)}. \end{displaymath}
From Lemma \ref{lemma5}, we know that the intersection $S^{C\delta}(p_{1}) \cap S^{C\delta}(p_{2})$ can be covered by boundedly many $(C\delta,\bar{\sigma})$-rectangles. Since $R_{1} \subset CR_{2} \subset S^{C\delta}(p_{1}) \cap S^{C\delta}(p_{2})$, and $\diam(R_{1}) \gtrsim \sigma$, we may infer that $\bar{\sigma} \gtrsim \sigma$. We set $\bar{C} := \max\{1,C(\sigma/\bar{\sigma})\} \lesssim C$.

It follows from our assumption $R^{\delta}_{\sigma}(p_{1},v_{1}) \subset CR^{\delta}_{\sigma}(p_{2},v_{2})$ that $R^{\delta}_{\sigma}(p_{1},v_{1}) \subset \bar{C}R^{C\delta}_{\bar{\sigma}}(p_{2},v_{2})$, and in particular $\bar{C}R^{C\delta}_{\bar{\sigma}}(p_{1},v_{1}) \cap \bar{C}R^{C\delta}_{\bar{\sigma}}(p_{2},v_{2}) \neq \emptyset$. Therefore, applying Corollary \ref{cor2} at scale $C\delta$ and with constant $\bar{C}$, we get
\begin{displaymath} CR^{\delta}_{\sigma}(p_{2},v_{2}) \subset \bar{C}R^{C\delta}_{\bar{\sigma}}(p_{2},v_{2}) \stackrel{\mathrm{C.\,} \ref{cor2}}{\subset} C'R^{C\delta}_{\bar{\sigma}}(p_{1},v_{1}) \subset S^{CC'\delta}(p_{1}) \end{displaymath}
for some $C' \lesssim \bar{C}^{4} \sim C^{4}$. Finally, since also $CR^{\delta}_{\sigma}(p_{2},v_{2}) \subset B(v_{2},C\sigma) \subset B(v_{1},2C\sigma)$, using that $v_{1} \in R_{1} \subset B(v_{2},C\sigma)$, we may infer that
\begin{displaymath} CR^{\delta}_{\sigma}(p_{2},v_{2}) \subset S^{CC'\delta}(p_{1}) \cap B(v_{1},2C\sigma) \subset CC'R^{\delta}_{\sigma}(p_{1},v_{1}). \end{displaymath}
Since $CC' \lesssim C^{5}$, the proof is complete. \end{proof}

\begin{remark} \label{r:CompRem} Lemma \ref{lemma6} clarifies the (up-to-constants) equivalence of different notions of comparability. If $R_{1},R_{2}$ are $C$-comparable $(\delta,\sigma)$-rectangles according to Definition \ref{def:comparability}, then there exists a third $(\delta,\sigma)$-rectangle $R$ such that $R_{1},R_{2} \subset CR$. But now $R \subset C'R_{2}$ according to Lemma \ref{lemma6}, so $R_{1} \subset C'R_{2}$. By symmetric reasoning, also $R_{2} \subset C'R_{1}$.

Similarly, if we took as our definition the one-sided inclusion $R_{1} \subset CR_{2}$, then Lemma \ref{lemma6} would imply that $R_{2} \subset C'R_{1}$, and consequently we could infer the symmetric condition $R_{1} \subset C'R_{2}$ and $R_{2} \subset C'R_{1}$ (or $R_{1},R_{2} \subset C'R$ with either $R := R_{1}$ or $R := R_{2}$). \end{remark}

We record the following useful corollary of Lemma \ref{lemma6}:

\begin{cor}[Transitivity of comparability]\label{cor1}
For every $C \geq 1$ there exists $C' \lesssim C^{5}$ such that
the following holds. Let $0 < \delta \leq \sigma \leq 1$, and let
$R_{1},R_{2},R_{3}$ be $(\delta,\sigma)$-rectangles such that
$R_{1},R_{2}$ and $R_{2},R_{3}$ are both $C$-comparable. Then
$R_{1},R_{3}$ are $C'$-comparable.
\end{cor}

\begin{proof} Since $R_{1},R_{2}$ and $R_{2},R_{3}$ are $C$-comparable, by definition there exist $(\delta,\sigma)$-rectangles $R_{12}$ and $R_{23}$ such that $R_{1},R_{2} \subset CR_{12}$ and $R_{2},R_{3} \subset CR_{23}$. We may infer from Lemma \ref{lemma6} that
\begin{displaymath} R_{1} \subset CR_{12} \subset C'R_{2} \quad \text{and} \quad R_{3} \subset CR_{23} \subset C'R_{2}, \end{displaymath}
for some $C' \lesssim C^{5}$. This means by definition that $R_{1},R_{3}$ are $C'$-comparable. \end{proof}

Next, given a  family $\calR$ of pairwise $100$-incomparable
$(\delta,\sigma)$-rectangles, for  $A \geq 100$, we
will show that there exists a subfamily $\bar \calR \subset \calR$ consisting of $A$-incomparable rectangles such that $A^{O(1)}|\bar \calR| \geq |\calR|$. This result will be proved in Corollary \ref{AT14}.

Indeed, Corollary \ref{AT14} is a direct consequence of the
following proposition:

\begin{proposition} \label{PYZ}
Let $A \geq 100$ and $\delta \leq \sigma \leq 1$, and let
$\mathcal{R}$ be a family of pairwise $100$-incomparable
$(\delta,\sigma)$-rectangles. Suppose also that there exists a fixed ($\delta,\sigma)$-rectangle
  $\mathbf{R}$ such that the union of the rectangles in $\mathcal{R}$ is contained in $A\mathbf{R}$. Then, $|\mathcal{R}| \lesssim A^{10}$.
\end{proposition}

After somewhat tedious initial reductions, the proof will be virtually the same as the
proof of \cite[Lemma 3.15]{2022arXiv220702259P}. We postpone the details to the Appendix \ref{app}, and only give a short outline here. Since \cite[Lemma 3.15]{2022arXiv220702259P}
was stated for curvilinear rectangles arising as neighborhoods of
arcs of graphs of $C^2(\mathbf{I})$ functions defined on an interval $\mathbf{I}\subset \mathbb{R}$, we need several auxiliary lemmas (see Lemma
\ref{lipgragh}, \ref{cinematic}, \ref{Nincl5}, and \ref{Nincl4} in Appendix \ref{app}) to reduce our proof to a situation similar to \cite[Lemma 3.15]{2022arXiv220702259P}. Then, in the
terminology of \cite{2022arXiv220702259P}, the
$(\delta,\sigma)$-rectangles we need to consider are called
$(\delta,t)$-rectangles with $t = \delta/\sigma^{2}$ provided that
$\sigma \geq \sqrt{\delta}$ (hence $t \leq 1$). Thus in the range
$\sigma \geq \sqrt{\delta}$, our proposition would basically
follow from \cite[Lemma 3.15]{2022arXiv220702259P}. (The
comparison between the different types of rectangles is stated
more precisely in \eqref{fact422}-\eqref{fact422b}.) But we also need to check that the proof works if $\sigma \leq
\sqrt{\delta}$.
 In this range our rectangles are shorter than any of the rectangles literally treated by \cite[Lemma 3.15]{2022arXiv220702259P}.
 The argument we give in Appendix \ref{app} for Proposition \ref{PYZ} reveals, however, that the proof sees no essential difference between these cases. Alternatively, one could treat the case $\sigma \leq \sqrt{\delta}$ separately, relying on the fact that the $(\delta,\sigma)$-rectangles in this range look like "ordinary" or "straight" rectangles.

The following consequence of Proposition \ref{PYZ} is similar in
spirit to \cite[Lemma 3.16]{2022arXiv220702259P}. It is not used
in this section but will be applied later in the proof of Theorem
\ref{thm4}.

\begin{cor} \label{AT14} Let $A \geq 100$ and $\delta \le \sigma \le 1$. Let $\calR$ be a pairwise $100$-incomparable family of $(\delta,\sigma)$-rectangles.
    Then $\calR$ contains a subset $\bar{\mathcal{R}}$ of cardinality $|\bar{\mathcal{R}}| \gtrsim A^{-50}|\mathcal{R}|$
consisting of pairwise
    $A$-incomparable rectangles.
\end{cor}

\begin{proof}
  Let $\bar{\mathcal{R}}$ be the maximal $A$-incomparable subfamily of $\calR$.
That is, $\bar{\mathcal{R}}$  consists of pairwise
$A$-incomparable rectangles, and any element in $\calR$ is $A$-comparable to at least one rectangle in
$\bar{\mathcal{R}}$. For  $R \in \bar{\mathcal{R}}$, we define
  $$ \calR_A(R):= \{ R' \in \calR: R' \subset CA^{5}R\},$$
 where $C \geq 1$ is an absolute constant to be fixed momentarily. By Proposition \ref{PYZ},
  \begin{equation}\label{eq:RAR}
  |\calR_A(R)| \lesssim A^{50}, \quad R \in \bar{\mathcal{R}}.
  \end{equation}
We claim that
 \begin{equation}\label{eq:calRUnion}
 \calR = \bigcup_{R \in \bar{\mathcal{R}}}
 \calR_A(R).\end{equation}
 Once \eqref{eq:calRUnion} has been verified, a combination of \eqref{eq:RAR}-\eqref{eq:calRUnion} shows that $|\calR| \lesssim  |\bar \calR|A^{50}$, and the proof will be complete.

To prove \eqref{eq:calRUnion}, fix $R' \in \mathcal{R}$. Then $R'$ is $A$-comparable to some $R \in \bar{\mathcal{R}}$ by the maximality of $\bar{\mathcal{R}}$. By Remark
    \ref{r:CompRem}, this gives $R' \subset CA^{5}R$, provided that $C > 0$ is a sufficiently large absolute constant. In particular, $R' \in \mathcal{R}_{A}(R)$, as desired.  \end{proof}


\subsection{A slight generalisation of Wolff's tangency counting bound}

The following definition is due to Wolff \cite{MR1800068}.

\begin{definition}[$t$-bipartite pair]\label{def:WolffBipartite} Let $0 < \delta \leq t \leq 1$. A pair of sets $W,B \subset \mathbf{D}$ is called \emph{$t$-bipartite} if both $W,B$ are $\delta$-separated, $\max\{\diam(B),\diam(W)\} \leq t$, and additionally
\begin{displaymath} \dist(B,W) \geq t \quad \text{and} \quad \diam(B \cup W) \leq 100t. \end{displaymath}
\end{definition}

\begin{lemma}\label{lemma7} Let $\delta \leq t \leq 1$, and let $W,B \subset \mathbf{D}$ be a $t$-bipartite pair of sets. Let $C \geq 1$ be a constant, and assume that $p_{1},\ldots,p_{k} \in W$ and $q_{1},\ldots,q_{l} \in B$ are points satisfying
\begin{displaymath} \Delta(p_{i},q_{j}) \leq C\delta, \qquad 1 \leq i \leq k, \, 1 \leq j \leq l. \end{displaymath}
Assume further that there exists a point $v \in \R^{2}$ which lies on all the circles $S(p_{i}),S(q_{j})$.

Write $\Sigma := \sqrt{\delta/t}$. Then, for suitable $C' \sim C$, every $(\delta,\Sigma)$-rectangle $R^{\delta}_{\Sigma}(p_{i},v)$ is contained in every annulus $S^{C'\delta}(p_{m})$ and $S^{C'\delta}(q_{n})$ (where $i$ has no relation to $m,n$).   \end{lemma}

\begin{proof} We will use the inclusion \eqref{form71}. Namely, \eqref{form71} applied with $\lambda := C\delta$ shows immediately that if $(i,j) \in \{1,\ldots,k\} \times \{1,\ldots,l\}$ is a fixed pair, then
\begin{equation}\label{form84} R^{\delta}_{\Sigma}(p_{i},v) \subset S^{C'\delta}(q_{j}) \end{equation}
for some $C' \sim C$. (Note that now $\Sigma \lesssim
\sqrt{C}\delta/\sqrt{\Delta(p_{i},q_{j})|p_{i} - q_{j}|} \sim
\sqrt{C}\sigma$ in the notation of \eqref{form71}, so we may apply
\eqref{form71} with constant $\sim \sqrt{C}$ to obtain
\eqref{form84}.) This already proves that every rectangle
$R^{\delta}_{\Sigma}(p_{i},v)$ is contained in every annulus
$S^{C'\delta}(q_{j})$. What remains is to prove a similar
conclusion about the annuli $S^{C'\delta}(p_{m})$ for $m \neq i$.

To proceed, we observe that \eqref{form84} can immediately be upgraded to
\begin{equation}\label{form85} R_{\Sigma}^{\delta}(p_{i},v) \subset R^{C'\delta}_{\Sigma}(q_{j},v), \end{equation}
simply as a consequence of \eqref{form84} and definitions. Further, if $m \in \{1,\ldots,k\}$, we have
\begin{equation}\label{form86} R^{C'\delta}_{\Sigma}(q_{j},v) \subset S^{C''\delta}(p_{m}), \end{equation}
by another application of the inclusion \eqref{form71} (here still $C'' \sim C' \sim C$). Now, chaining \eqref{form85}-\eqref{form86}, we find $R^{\delta}_{\Sigma}(p_{i},v) \subset S^{C''\delta}(p_{m})$. Combined with \eqref{form84}, this completes the proof. \end{proof}

In this paper, we will need the following slight relaxation of $t$-bipartite pairs:

\begin{definition}[Almost $t$-bipartite pair]\label{def:Bipartite} Let $\delta \leq t \leq 1$. A pair of sets $W,B \subset \mathbf{D}$ is called \emph{$(\delta,\epsilon)$-almost $t$-bipartite} if both $W,B$ are $\delta$-separated, and additionally
\begin{displaymath} \dist(W,B) \geq \delta^{\epsilon}t \quad \text{and} \quad \diam(B \cup W) \leq \delta^{-\epsilon}t. \end{displaymath}
\end{definition}

\begin{definition}[Type]\label{def:WolffType} Let $0 < \delta \leq \sigma \leq 1$, $\epsilon > 0$. Let $W,B \subset \mathbf{D}$ be finite sets. For $m,n \geq 1$, we say that a $(\delta,\sigma)$-rectangle $R \subset \R^{2}$ has \emph{type $(\geq m,\geq n)_{\epsilon}$ relative to $(W,B)$} if $R \subset S^{\delta^{1 - \epsilon}}(p)$ for at least $m$ points $p \in W$, and $R \subset S^{\delta^{1 - \epsilon}}(q)$ at least $n$ points $q \in B$. \end{definition}

Here is a slight variant of \cite[Lemma 1.4]{MR1800068}:

\begin{lemma}\label{lemma8} For every $\epsilon > 0$, there exists $\delta_{0} > 0$ such that the following holds for all $\delta \in (0,\delta_{0}]$. Let $0 < \delta \leq t \leq 1$, and let $W,B \subset \mathbf{D}$ be a $(\delta,\epsilon)$-almost $t$-bipartite pair of sets. Let $\Sigma := \sqrt{\delta/t}$, and let $\mathcal{R}^{\delta}_{\Sigma}$ be a family of pairwise $100$-incomparable $(\delta,\Sigma)$-rectangles of type $(\geq m,\geq n)_{\epsilon}$ relative to $(W,B)$, where $1 \leq m \leq |W|$ and $1 \leq n \leq |B|$. Then,
\begin{equation}\label{form78} |\mathcal{R}^{\delta}_{\Sigma}| \leq \delta^{-C\epsilon} \left(\left(\frac{|W||B|}{mn} \right)^{3/4} + \frac{|W|}{m} + \frac{|B|}{n} \right), \end{equation}
where $C > 0$ is an absolute constant. \end{lemma}

This lemma is the same as \cite[Lemma 1.4]{MR1800068}, except that it allows for constants of form "$\delta^{-\epsilon}$" in both Definition \ref{def:Bipartite} and Definition \ref{def:WolffType}. In \cite[Lemma 1.4]{MR1800068}, the definition of "$t$-bipartite pair" is exactly the one we stated in Definition \ref{def:WolffBipartite}, and  the definition of "type" was defined with a large absolute constant $C_{0} \geq 1$ in place of $\delta^{-\epsilon}$. As it turns out, Lemma \ref{lemma8} can be formally reduced to \cite[Lemma 1.4]{MR1800068} with a little pigeonholing.

\begin{proof}[Proof of Lemma \ref{lemma8}] In this proof, the letter "$C$" will refer to an absolute constant whose value may change from line to line.

We may assume that $\delta^{1 - 3\epsilon} \leq t$, since if
$(W,B)$ is $(\delta,\epsilon)$-almost $t$-bipartite for some $t
\leq \delta^{1 - 3\epsilon}$, then both $W$ and $B$ have
cardinality $\lesssim \delta^{-12\epsilon}$, and it easily follows
that $|\mathcal{R}^{\delta}_{\Sigma}| \leq \delta^{-C\epsilon}$.

 By assumption, we have $\diam(W) \leq \delta^{-\epsilon}t$ and $\diam(B) \leq \delta^{-\epsilon}t$. Therefore, we may decompose both $W$ and $B$ into $r \lesssim (\delta^{-\epsilon}t/\delta^{\epsilon}t)^{3} = \delta^{-6\epsilon}$ subsets $W_{1},\ldots,W_{r}$ and $B_{1},\ldots,B_{r}$ of diameter $\leq \delta^{\epsilon}t$. Now, for each pair $W_{i},B_{j}$, we have
\begin{equation}\label{form77} \dist(W_{i},B_{j}) =: \tau_{ij} \in [\delta^{\epsilon}t,\delta^{-\epsilon}t]. \end{equation}
Each pair $(W_{i},B_{j})$ is $\tau_{ij}$-bipartite in the terminology of Definition \ref{def:WolffBipartite}, since \eqref{form77} holds, and
\begin{displaymath} \max\{\diam(W_{i}),\diam(B_{j})\} \leq \delta^{\epsilon}t \leq \tau_{ij} \quad \text{and} \quad \diam(W_{i} \cup B_{j}) \leq 3\tau_{ij}. \end{displaymath}
Next, notice that if $R \in \mathcal{R}^{\delta}_{\Sigma}$, then there exists (by the pigeonhole principle) at least one pair $(W_{i},B_{j})$ such that $R^{\delta}_{\Sigma}$ has type $(\geq \bar{m},\geq \bar{n})_{\epsilon}$ relative to $(W_{i},B_{j})$, where
\begin{displaymath} \bar{m} := \max\{\delta^{6\epsilon}m,1\} \quad \text{and} \quad \bar{n} := \max\{\delta^{6\epsilon}n,1\}. \end{displaymath}
This means that there exist at least $\bar{m}$ circles $S(p_{1}),\ldots,S(p_{\bar{m}})$ with $p_{k} \in W_{i}$ and and at least $\bar{n}$ circles $S(q_{1}),\ldots,S(q_{\bar{n}})$ with $q_{l} \in B_{j}$ with the property
\begin{equation}\label{form79} R \subset S^{\delta^{1 - \epsilon}}(p_{k}) \quad \text{and} \quad R \subset S^{\delta^{1 - \epsilon}}(q_{l}). \end{equation}
Based on what we just said, we have
\begin{equation}\label{form80} \mathcal{R}^{\delta}_{\Sigma} \subset \bigcup_{i,j} \mathcal{R}^{\delta}_{\Sigma}(i,j) \quad \Longrightarrow \quad |\mathcal{R}^{\delta}_{\Sigma}| \leq \sum_{i,j} |\mathcal{R}^{\delta}_{\Sigma}(i,j)|, \end{equation}
where $\mathcal{R}^{\delta}_{\Sigma}(i,j)$ refers to rectangles of
type $(\geq \bar{m},\geq \bar{n})_{\epsilon}$ relative to
$(W_{i},B_{j})$. Since the  number of pairs $(i,j)$ is $\leq
\delta^{-C\epsilon}$, it suffices to prove \eqref{form78} for each
$\mathcal{R}^{\delta}_{\Sigma}(i,j)$ individually.

Fix $1 \leq i,j \leq r$, and write $\tau := \tau_{ij} \in [\delta^{\epsilon}t,\delta^{-\epsilon}t]$, and also abbreviate (or redefine) $W := W_{i}$ and $B := B_{j}$ and $\mathcal{R}^{\delta}_{\Sigma} := \mathcal{R}^{\delta}_{\Sigma}(i,j)$. Before proceeding further, we deduce information about the "tangency" of $p_{k} \in W$ and $q_{l} \in B$ satisfying \eqref{form79}. Recall that $|p_{k} - q_{l}| \geq \tau \geq \delta^{\epsilon}t$, and note that $\diam(R) \gtrsim \Sigma = \sqrt{\delta/t}$. Then,
\begin{displaymath} \sqrt{\delta/t} \lesssim \diam(R)
 \stackrel{\mathrm{L.\,} \ref{lemma5}}{\lesssim} \frac{\delta^{1 - \epsilon}}{\sqrt{\delta^{\epsilon}t \Delta(p_{k},q_{l})}}, \end{displaymath}
from which we may infer that
\begin{equation}\label{form82} \Delta(p_{k},q_{l}) \lesssim \delta^{1 - 3\epsilon}, \qquad 1 \leq k \leq \bar{m}, \, 1 \leq l \leq \bar{n}. \end{equation}

 For purposes to become apparent in a moment, it would be convenient if $W,B$ were $\delta^{1 - 3\epsilon}$-separated instead of just $\delta$-separated. This can be arranged, at the cost of reducing $\bar{m}$ and $\bar{n}$ slightly. Indeed, we may partition $W$ and $B$ into $\delta^{1 - 3\epsilon}$-separated subsets $W_{1},\ldots,W_{s}$ and $B_{1},\ldots,B_{s}$, where $s \leq \delta^{-9\epsilon}$. Now, arguing as before, every rectangle $R \in \mathcal{R}^{\delta}_{\Sigma}$ has type $(\geq \bar{m}',\geq \bar{n}')$ relative to at least one pair $(W_{i},B_{j})$, where $\bar{m}' := \max\{\delta^{9\epsilon}\bar{m},1\}$ and $\bar{n}' := \max\{\delta^{9\epsilon}\bar{n},1\}$. After repeating the argument at \eqref{form80}, we may focus attention to bounding the number of rectangles associated with a fixed $(W_{i},B_{j})$. Since the passage from $(W,B)$ to $(W_{i},B_{j})$ eventually just affects the absolute constant "$C$" in \eqref{form78}, we now assume that $W,B$ are $\delta^{1 - 3\epsilon}$-separated to begin with, and $\bar{m}' = \bar{m}$ and $\bar{n} = \bar{n}'$.

The improved separation of $W,B$ gives the following benefit: the pair $(W,B)$ is $\tau$-bipartite relative to the scale $\delta^{1 - 3\epsilon}$ in the strong sense of Definition \ref{def:WolffBipartite}. The role of "$\delta$" (or now $\delta^{1 - 3\epsilon}$) is hardly emphasised, but one of the assumptions in Definition \ref{def:WolffBipartite} was that a $\tau$-bipartite set is $\delta$-separated, and the conclusion of \cite[Lemma 1.4]{MR1800068} concerns "type" and "tangency" defined for $\delta$-annuli and $(\delta,\sqrt{\delta/t})$-rectangles. Now, since $W,B$ are $\delta^{1 - 3\epsilon}$-separated, we have access to the conclusion of the same lemma at scale $\delta^{1 -3\epsilon}$.

Now, \cite[Lemma 1.4]{MR1800068} implies that the maximal number of pairwise $100$-incomparable $(\delta^{1 - 3\epsilon},\sqrt{\delta^{1 - 3\epsilon}/\tau})$-rectangles of type $(\geq \bar{m},\geq \bar{n})$ relative to $(W_{i},B_{j})$ is bounded from above by the right hand side of \eqref{form78}. The definition of "type" here is the one which Wolff is using in the statement of \cite[Lemma 1.4]{MR1800068}: a $(\delta^{1 - 3\epsilon},\sqrt{\delta^{1 - 3\epsilon}/\tau})$-rectangle $\bar{R}$ has type $(\geq \bar{m},\geq \bar{n})$ relative to $(W,B)$ if there are $p_{1},\ldots,p_{\bar{m}} \in W$ and $q_{1},\ldots,q_{\bar{n}} \in B$ such that
\begin{equation}\label{form81} \bar{R} \subset S^{C\delta^{1 - 3\epsilon}}(p_{k}) \cap S^{C\delta^{1 - 3\epsilon}}(q_{l}), \qquad 1 \leq k \leq \bar{m}, \, 1 \leq l \leq \bar{n}, \end{equation}
where $C > 0$ is an absolute constant.

What does this conclusion about the rectangles $\bar{R}$ tell us
about the cardinality of $\mathcal{R}^{\delta}_{\Sigma}$? We will
use the $(\delta,\Sigma)$-rectangles in
$\mathcal{R}^{\delta}_{\Sigma}$ to produce a new family
$\bar{\mathcal{R}}$ of pairwise $100$-incomparable $(\delta^{1 -
3\epsilon},\bar{\Sigma})$-rectangles satisfying \eqref{form81},
where $\bar{\Sigma} = \sqrt{\delta^{1 - 3\epsilon}/\tau}$. Then,
we will apply the upper bound for $|\bar{\mathcal{R}}|$ (given by
\cite[Lemma 1.4]{MR1800068}) to conclude the desired estimate for
$|\mathcal{R}^{\delta}_{\Sigma}|$.

Recall from \eqref{form79} that each of our $(\delta,\Sigma)$-rectangles $R \in \mathcal{R}^{\delta}_{\Sigma}$ has type $(\geq \bar{m},\geq \bar{n})_{\epsilon}$ relative to $(W,B)$ in the sense $R \subset S^{\delta^{1 - \epsilon}}(p_{k}) \cap S^{\delta^{1 - \epsilon}}(q_{l})$ for every $1 \leq k \leq \bar{m}$ and $1 \leq l \leq \bar{n}$. As we observed in \eqref{form82}, this implies $\Delta(p_{k},q_{l}) \lesssim \delta^{1 - 3\epsilon}$. Recall that further $|p_{k} - q_{l}| \sim \tau$ for all $1 \leq k \leq \bar{m}$ and $1 \leq l \leq \bar{n}$.

In view of applying Lemma \ref{lemma7}, we would need that the
circles $S(p_{k})$ and $S(q_{l})$ share a common point. This is
not quite true, but it is true for slightly shifted copies of
$S(p_{k})$ and $S(q_{l})$. Namely, take "$v$" to be an arbitrary
point in $R$, for example its centre (writing $R =
R^{\delta}_{\Sigma}(p,v)$ for some $p \in \mathbf{D}$ and $v \in
S(p)$). Now, since $v \in R\subset S^{\delta^{1 -
\epsilon}}(p_{k})$, there exists $\bar{p}_{k} \in
B(p_{k},\delta^{1 - \epsilon})$ such that $v \in S(\bar{p}_{k})$
(see the proof of Corollary \ref{cor2}). Similarly, there exist
points $\bar{q}_{l} \in B(q_{k},\delta^{1 - \epsilon})$, $1 \leq l
\leq \bar{n}$, such that $v \in S(\bar{q}_{l})$. Note that the
crucial hypotheses $\Delta(\bar{p}_{k},\bar{q}_{l}) \lesssim
\delta^{1 - 3\epsilon}$ and $|\bar p_{k} - \bar q_{l}| \sim \tau$
were not violated (since $\tau \geq \delta^{\epsilon}t \geq
\delta^{1 - 2\epsilon}$).

Now, we are in a position to apply Lemma \ref{lemma7} at scale $\delta^{1 - 3\epsilon}$, and with "$\tau$" in place of "$t$". The conclusion is that if we set
\begin{displaymath} \bar{R} := \bar{R}(R) := R^{\delta^{1 - 3\epsilon}}_{\bar{\Sigma}}(p_{1},v), \qquad \bar{\Sigma} := \sqrt{\delta^{1 - 3\epsilon}/\tau}, \end{displaymath}
then \eqref{form81} holds, provided that the constant $C > 0$ is sufficiently large (initially with the points $\bar{p}_{k},\bar{q}_{l}$, but since $|\bar{p}_{k} - p_{k}| \leq \delta^{1 - \epsilon}$ and $|q_{l} - \bar{q}_{l}| \leq \delta^{1 - \epsilon}$, we also get \eqref{form81} as stated). In other words, $\bar{R}$ is a $(\delta^{1 - 3\epsilon},\bar{\Sigma})$-rectangle which has type $(\geq \bar{m},\geq \bar{n})$ relative to $(W,B)$ in the terminology of Wolff.

We have now shown that each rectangle $R \in \mathcal{R}^{\delta}_{\Sigma}$ gives rise to a $(\delta^{1 - 3\epsilon},\bar{\Sigma})$-rectangle
 $\bar{R}(R)$ which has type $(\geq \bar{m},\geq \bar{n})$ relative to $(W,B)$. We also observe that
\begin{equation}\label{form212} R \stackrel{\eqref{form79}}{\subset} S^{\delta^{1 - \epsilon}}(p_{1})
\cap B(v,\Sigma) \subset S^{\delta^{1 - 3\epsilon}}(p_{1}) \cap B(v,\bar{\Sigma}) = \bar{R}(R). \end{equation}
Finally, let $\bar{\mathcal{R}}$ be a maximal pairwise $100$-incomparable subset of
 $\{\bar{R}(R) : R \in \mathcal{R}^{\delta}_{\Sigma}\}$.
 The rectangles in $\bar{\mathcal{R}}$ have type $(\geq \bar{m},\geq \bar{n})$
 relative to $(W,B)$, so $|\bar{\mathcal{R}}|$ satisfies the desired upper bound \eqref{form78} by \cite[Lemma 1.4]{MR1800068}. It remains to show that
\begin{equation}\label{form213} |\mathcal{R}^{\delta}_{\Sigma}| \leq \delta^{-C\epsilon}|\bar{\mathcal{R}}|. \end{equation}
If $R \in \mathcal{R}^{\delta}_{\Sigma}$, then $\bar{R}(R) \sim_{100} \bar{R}$ for some $\bar{R} \in \bar{\mathcal{R}}$.
Combining \eqref{form212} and Lemma \ref{lemma6}, we may infer that $R \subset \bar{R}(R) \subset C\bar{R}$ for some absolute constant $C > 0$.
Therefore, \eqref{form213} will follow if we manage to argue that
\begin{displaymath} |\{R \in \mathcal{R}^{\delta}_{\Sigma} : R \subset C\bar{R}\}| \leq \delta^{-C\epsilon}, \qquad \bar{R} \in \bar{\mathcal{R}}.
 \end{displaymath}
But since the rectangles in $\mathcal{R}^{\delta}_{\Sigma}$ are pairwise $100$-incomparable, this follows immediately from Proposition \ref{PYZ}.
The proof is complete.
\end{proof}


\section{Bounding partial multiplicity functions with high tangency}\label{lambdalambdat}

In this section, we will finally introduce the \emph{partial multiplicity functions} $m_{\delta,\lambda,t}$ mentioned in the proof outline, Section \ref{s:outline} (see Definition \ref{def:multFunction1}). The plan of this section is to prove a desirable upper bound for $m_{\lambda,\lambda,t}$ -- the partial multiplicity function only taking into account incidences of maximal tangency at scale $\lambda$. This will be accomplished in Theorem \ref{thm3}, although most of the work is contained in Proposition \ref{prop3}.

\begin{notation}[$G_{\lambda,t}^{\rho}(\omega)$]\label{not3} Let $0 < \delta \leq \sigma \leq 1$, and let $P \subset \mathcal{D}_{\delta}$, $\{E(p)\}_{p \in P}$ be finite sets, where $E(p) \subset \mathcal{S}_{\sigma}(p)$ for all $p \in P$. Let $\Omega = \{(p,v) : p \in P \text{ and } v \in E(p)\}$. If $G \subset \Omega$ is an arbitrary subset, $\delta \leq \lambda \leq t \leq 1$, and $\rho \geq 1$, we define
\begin{displaymath} G_{\lambda,t}^{\rho}(\omega) := \{(p',v') \in G : t/\rho \leq |p - p'| \leq \rho t \text{ and } \lambda/\rho \leq \Delta(p,p') \leq \rho \lambda\}, \qquad \omega \in \Omega. \end{displaymath}
The distance $|p - p'|$ and $\Delta(p,p')$ are defined relative to the centres of $p,p' \in \mathcal{D}_{\delta}$. If $\lambda \in [\delta,\rho \delta]$ (as in Proposition \ref{prop3} below), we remove the lower bound $\Delta(p,p') \geq \lambda/\rho$ from the definition.
 \end{notation}

\begin{proposition}\label{prop3} For every $\kappa > 0$, and $s \in (0,1]$, there exist $\epsilon = \epsilon(\kappa,s) \in (0,\tfrac{1}{2}]$ and $\lambda_{0} = \lambda_{0}(\epsilon,\kappa,s) > 0$ such that the following holds for all $\lambda \in (0,\lambda_{0}]$. Let $\lambda \leq t \leq 1$ and $\Sigma := \sqrt{\lambda/t}$. Let $\Omega = \{(p,v) : p \in P \text{ and } v \in E(p)\}$ be a $(\lambda,\Sigma,s,\lambda^{-\epsilon})$-configuration (see Definition \ref{d:config2}). Then, there exists a $(\lambda,\Sigma,s,C_{\kappa}\lambda^{-\epsilon})$-configuration $G \subset \Omega$ with $|G| \sim_{\kappa} |\Omega|$ with the property
\begin{equation}\label{form39} |\{\omega' \in G_{\lambda,t}^{\lambda^{-\epsilon}}(\omega) : \lambda^{-\epsilon}R^{\lambda}_{\Sigma}(\omega) \cap \lambda^{-\epsilon}R^{\lambda}_{\Sigma}(\omega') \neq \emptyset\}| \leq \lambda^{s - \kappa}|P|, \qquad \omega \in G. \end{equation}
\end{proposition}

To be precise, $\Sigma$ in Proposition \ref{prop3} refers to the smallest dyadic rational $\bar{\Sigma} \in 2^{-\N}$ with $\Sigma \leq \bar{\Sigma}$, recall Remark \ref{dyadicConvention}. Taking this carefully into account has a small impact on some constants in the proof below, but leave this to the reader.

\begin{proof}[Proof of Proposition \ref{prop3}] Write $M_{\Sigma} := |E(p)|$ for $p \in P$ (this constant is independent of $p \in P$ by Definition \ref{d:config2}). We start by disposing of the special case where $t \leq \lambda^{1 - \kappa/3}$. In this case we claim that $G = \Omega$ works. To see this, note that now $\Sigma = \sqrt{\lambda/t} \geq \lambda^{\kappa/6}$, so $M_{\Sigma} = |E(p)| \leq |\mathcal{S}_{\Sigma}(p)| \lesssim \lambda^{-\kappa/6}$. Furthermore,
\begin{displaymath} \Omega_{\lambda,t}^{\lambda^{-\epsilon}}(p,v) \subset \{(p',v') \in \Omega : p' \in P \cap B(p,\lambda^{1 - \kappa/2})\}, \qquad (p,v) \in \Omega,  \end{displaymath}
assuming that $\epsilon \leq \kappa/6$. Fix $\omega = (p,v) \in \Omega$. Then, for every $p' \in B(p,\lambda^{1 - \kappa/2})$, using the $\Sigma$-separation of $E(p')$, there are $\lesssim \lambda^{-\epsilon}$ possible choices $v' \in E(p')$ such that
\begin{displaymath} \lambda^{-\epsilon}R^{\lambda}_{\Sigma}(\omega) \cap \lambda^{-\epsilon}R^{\lambda}_{\Sigma}(p',v') \neq \emptyset. \end{displaymath}
Consequently,
\begin{align*} |\{\omega' \in \Omega_{\lambda,t}^{\lambda^{-\epsilon}}(\omega) : \lambda^{-\epsilon}R^{\lambda}_{\Sigma}(\omega) \cap \lambda^{-\epsilon}R^{\lambda}_{\Sigma}(\omega') \neq \emptyset\}| & \lesssim \lambda^{-\epsilon} \cdot |P \cap B(p,\lambda^{1 - \kappa/2})|\\
& \lesssim \lambda^{-2\epsilon}\lambda^{(1 - \kappa/2)s}|P| \leq \lambda^{s - 5\kappa/6}|P|. \end{align*}
using the $(\lambda,s,\lambda^{-\epsilon})$-set property of $P$ in the final inequality, as well as $\epsilon \leq \kappa/6$, and $s \leq 1$. We have now proven \eqref{form39} with $G = \Omega$. In the sequel, we may assume that
\begin{equation}\label{form106} t \geq \lambda^{1 - \kappa/3}. \end{equation}

Fix $\epsilon = \epsilon(\kappa,s) > 0$ and $\lambda > 0$ (depending on $\epsilon,\kappa,s$) be so small that
\begin{equation}\label{def:epsilon} A \cdot 18^{\ceil{20/\kappa}}\epsilon < \kappa s \quad \text{and} \quad A^{18^{\ceil{20/\kappa}}}\lambda^{-18^{\ceil{20/\kappa}}\epsilon} < \lambda^{-s}, \end{equation}
where $A \geq 1$ is a suitable absolute constant. We start by defining a sequence of constants
\begin{displaymath} \mathbf{C}_{0} \gg \mathbf{C}_{1} \gg \ldots \gg \mathbf{C}_{h} := \lambda^{-\epsilon}, \end{displaymath}
where  $h = \ceil{20/\kappa}$, and such that $\mathbf{C}_{j} = A\mathbf{C}_{j + 1}^{18}$. Thus,
\begin{equation}\label{form205} \mathbf{C}_{0} \leq A^{18^{\ceil{20/\kappa}}}\lambda^{-18^{\ceil{20/\kappa}}\epsilon} < \lambda^{-s}. \end{equation}

We will abbreviate
\begin{displaymath} n_{j}(\omega \mid G) := |\{\omega' \in G^{\mathbf{C}_{j}}_{\lambda,t}(\omega) : \mathbf{C}_{j}R^{\lambda}_{\Sigma}(\omega) \cap \mathbf{C}_{j}R^{\lambda}_{\Sigma}(\omega') \neq \emptyset\}| \end{displaymath}
for $G \subset \Omega$ and $\omega \in \Omega$. Note that the constants $\mathbf{C}_{j}$ are decreasing functions of "$j$", so $n_{h} \leq n_{h - 1} \leq \ldots \leq n_{0}$. Also, $n_{j}(\omega \mid G)$ is an upper bound for the left hand side of \eqref{form39} for each $0 \leq j \leq h$, since $\mathbf{C}_{j} \geq \lambda^{-\epsilon}$.

We start by recording the "trivial" upper bound
\begin{equation}\label{form45} n_{0}(\omega \mid G) \leq n_{0}(\omega \mid \Omega) \lesssim \mathbf{C}_{0}|P|, \qquad \omega \in \Omega, \, G \subset \Omega. \end{equation}
The first inequality is clear. To see the second inequality, fix $\omega = (p,v) \in \Omega$ and $p' \in P$. Now, if $v' \in \mathcal{S}_{\Sigma}(p')$ is such that
\begin{displaymath} \mathbf{C}_{0}R^{\lambda}_{\Sigma}(p',v') \cap \mathbf{C}_{0}R^{\lambda}_{\Sigma}(\omega) \neq \emptyset, \end{displaymath}
then $|v - v'| \lesssim \mathbf{C}_{0}\Sigma$. But the points $v' \in S_{\Sigma}(p')$ are $\Sigma$-separated, so there are $\lesssim \mathbf{C}_{0}$ possible choices for $v'$, for each $p' \in P$. This gives \eqref{form45}.

The trivial inequality \eqref{form45} tells us that the estimate \eqref{form39} holds automatically with $G = \Omega$ and $\kappa = 2s$, since $\lambda^{-\epsilon} \leq \mathbf{C}_{0} < \lambda^{-s}$ by \eqref{def:epsilon}.

By the previous explanation, if $\kappa > 2s$, there is nothing to prove (we can take $G = \Omega$). Let us then assume that $\kappa \leq 2s$. Then, let
\begin{displaymath} 0 = \kappa_{1} < \kappa_{2} < \ldots < \kappa_{h} = 2s \end{displaymath}
be a $(\kappa s/10)$-dense sequence in $[0,2s]$ (this is why we chose $h = \ceil{20/\kappa}$). We now define a decreasing sequence of sets $\Omega = G_{0} \supset G_{1} \supset \ldots \supset G_{k}$, where $k \leq h$. We set $G_{0} := \Omega$, and in general we will always make sure inductively that $|G_{j + 1}| \geq \tfrac{1}{2}|G_{j}|$ for $j \geq 0$. Note that $n_{0}(\omega \mid G_{0}) \leq \lambda^{-s}|P| = \lambda^{s - \kappa_{h}}|P|$ by \eqref{form45}, for all $\omega \in G_{0}$.

Let us then assume that the sets $G_{0} \supset \ldots \supset G_{j}$ have already been defined. We also assume inductively that $n_{j}(\omega \mid G_{j}) \leq \lambda^{s - \kappa_{h - j}}|P|$ for all $\omega \in G_{j}$. This was true for $j = 0$. Define
\begin{displaymath} H_{j} := \{\omega \in G_{j} : n_{j + 1}(\omega \mid G_{j}) \geq \lambda^{s - \kappa_{h - (j + 1)}}|P|\}, \qquad 0 \leq j \leq k. \end{displaymath}
This is the subset of $G_{j}$ where the lower bound for the multiplicity nearly matches the (inductive) upper bound -- albeit with a slightly different definition of the multiplicity function.  There are two options.
\begin{enumerate}
\item If $|H_{j}| \geq \tfrac{1}{2}|G_{j}|$, then we set $H := H_{j}$ and $k := j$ and the construction of the sets $G_{j}$ terminates.
\item If $|H_{j}| < \tfrac{1}{2}|G_{j}|$, then the set $G_{j + 1} := G_{j} \, \setminus \, H_{j}$ has $|G_{j + 1}| \geq \tfrac{1}{2}|G_{j}|$, and moreover
\begin{displaymath} n_{j + 1}(\omega \mid G_{j + 1}) \leq n_{j + 1}(\omega \mid G_{j}) \leq \lambda^{s - \kappa_{h - (j + 1)}}|P|, \qquad \omega \in G_{j + 1}. \end{displaymath}
In other words, $G_{j + 1}$ is a valid "next set" in our sequence $G_{0} \supset \ldots \supset G_{j + 1}$, and the inductive construction may proceed.
\end{enumerate}

The "hard" case of the proof of Proposition \ref{prop3} occurs when case (1) is reached for some "$j$" with $\kappa_{h - j} > \kappa$. Namely, if case (1) never takes place for such indices "$j$", then we can keep constructing the sets $G_{j}$ until the first index "$j$" where $\kappa_{h - j} < \kappa$. At this stage, the set $G := G_{j}$ satisfies $n_{j}(\omega \mid G) \leq \lambda^{s - \kappa}|P|$ for all $\omega \in G$ (so \eqref{form39} is satisfied because $\mathbf{C}_{j} \geq \lambda^{-\epsilon}$), and since $|G| \geq 2^{-j}|\Omega| \geq 2^{-\ceil{20/\kappa}}|\Omega| \sim_{\kappa} |\Omega|$, the proof is complete. (To be accurate, $G$ is not quite yet a $(\delta,\Sigma,s,C_{\kappa}\lambda^{-\epsilon})$-configuration, but this can be fixed by a single application of Lemma \ref{refinement}).

In fact, we claim that case (1) cannot occur: more precisely, if $\epsilon = \epsilon(\kappa,s) > 0$ is as small as we declared in \eqref{def:epsilon}, then case (1) cannot occur for $\kappa_{h - j} \geq \kappa$. To prove this, we make a counter assumption: case (1) is reached at some index $j \in \{0,\ldots,h\}$ satisfying $\kappa_{h - j} \geq \kappa$. We write $\bar{\kappa} := \kappa_{h - j}$ and
\begin{equation}\label{form91} \kappa_{h - (j + 1)} =: \bar{\kappa} - \zeta, \qquad \text{where }\zeta \leq (\kappa s)/10 \leq (\bar{\kappa}s)/10. \end{equation}
We also set
\begin{displaymath} \bar{G} := G_{j} \quad \text{and} \quad H := H_{j} = \{\omega \in \bar{G} : n_{j + 1}(\omega \mid \bar{G}) \geq \lambda^{s - \bar{\kappa} + \zeta}|P|\}, \end{displaymath}
so that $|H| \geq \tfrac{1}{2}|\bar{G}| \gtrsim_{\kappa} |\Omega|$ by the assumption that case (1) occurred. Finally, we will abbreviate
\begin{equation}\label{form214} n := \lambda^{s - \bar{\kappa} + \zeta}|P|  \end{equation}
in the sequel. Thus, to spell out the definitions, we have $H \subset \bar{G}$, and
\begin{equation}\label{form46} |\{\omega' \in \bar{G}_{\lambda,t}^{\mathbf{C}_{j + 1}}(\omega) : \mathbf{C}_{j + 1}R^{\lambda}_{\Sigma}(\omega) \cap \mathbf{C}_{j + 1}R^{\lambda}_{\Sigma}(\omega') \neq \emptyset\}| \geq n, \qquad \omega \in H. \end{equation}
On the other hand, by the definition of $\bar{G} = G_{j}$, and $\bar{\kappa} = \kappa_{h - j}$, we have
\begin{equation}\label{form47} |\{\omega' \in \bar{G}_{\lambda,t}^{\mathbf{C}_{j}}(\omega) : \mathbf{C}_{j}R^{\lambda}_{\Sigma}(\omega) \cap \mathbf{C}_{j}R^{\lambda}_{\Sigma}(\omega') \neq \emptyset\}| \leq \lambda^{s - \bar{\kappa}}|P| = \lambda^{-\zeta}n, \qquad \omega \in \bar{G}. \end{equation}

We perform a small refinement to $H$. Note that
\begin{displaymath} \sum_{p \in P} |H(p)| = |H| \gtrsim_{\kappa} |\Omega| = M_{\Sigma}|P|, \end{displaymath}
where as usual $H(p) = \{v \in E(p) : (p,v) \in H\}$. Consequently, there exists a subset $\bar{P} \subset P$ of cardinality $|\bar{P}| \gtrsim_{\kappa} |P|$ and a number $\bar{M}_{\Sigma} \gtrsim_{\kappa} M_{\Sigma}$ such that $|H(p)| \geq \bar{M}_{\Sigma}$ for all $p \in \bar{P}$. For each $p \in \bar{P}$, we further pick (arbitrarily) a subset $\bar{H}(p) \subset H(p)$ of cardinality precisely $|\bar{H}(p)| = \bar{M}_{\Sigma}$. Then, we define $\bar{H} := \{(p,v) : p \in \bar{P} \text{ and } v \in \bar{H}(p)\} \subset H$. Note that $|\bar{H}| \sim_{\kappa} |\Omega|$, and now $\bar{H}$ has the additional nice feature compared to $H$ that
\begin{equation}\label{form193} |\bar{H}(p)| = \bar{M}_{\Sigma}, \qquad p \in \bar{P}. \end{equation}

Let $\mathcal{B}$ be a cover of $P$ by balls of radius $\tfrac{1}{4}t/\mathbf{C}_{j + 1}$ such that even the concentric balls of radius $2t\mathbf{C}_{j + 1}$ (that is, the balls $\{8\mathbf{C}_{j + 1}^{2}B : B \in \mathcal{B}\}$) have overlap bounded by $\lambda^{-C(\kappa)\epsilon}$ (this is possible, since $\mathbf{C}_{j} \leq \lambda^{-C(\kappa)\epsilon}$ for all $1 \leq j \leq h$, recall \eqref{form205}). Then, we choose the ball $B(p_{0},\tfrac{1}{4}t/\mathbf{C}_{j + 1}) \in \mathcal{B}$ in such a way that the ratio
\begin{displaymath} \theta := \frac{|\bar{P} \cap B(p_{0},\tfrac{1}{4}t/\mathbf{C}_{j + 1})|}{|P \cap B(p_{0},2\mathbf{C}_{j + 1}t)|} \end{displaymath}
is maximised. We claim that $\theta \geq \lambda^{C(\kappa)\epsilon}$: this follows from the estimate
\begin{displaymath} |\bar{P}| \leq \sum_{B \in \mathcal{B}} |\bar{P} \cap B| \leq \theta \sum_{B \in \mathcal{B}} |P \cap 8\mathbf{C}_{j + 1}^{2}B| \leq \theta \lambda^{-C(\kappa)\epsilon}|P|, \end{displaymath}
and recalling that $|\bar{P}| \gtrsim_{\kappa} |P|$. Now, we set

\begin{equation}\label{form89} W := \bar{P} \cap B(p_{0},\tfrac{1}{4}t/\mathbf{C}_{j + 1}) \quad \text{and} \quad B := P \cap B(p_{0},2\mathbf{C}_{j + 1}t) \, \setminus \, B(p_{0},\tfrac{1}{2}t/\mathbf{C}_{j + 1}), \end{equation}
so that
\begin{equation}\label{form42} |B| \leq |P \cap B(p_{0},2\mathbf{C}_{j + 1}t)| = \theta^{-1}|W| \lesssim_{\kappa} \lambda^{-C(\kappa)\epsilon}|W|. \end{equation}
We also set
\begin{displaymath} \mathbf{W} := \{(p,v) \in \bar{H} : p \in W\} \quad \text{and} \quad \mathbf{B} := \{(p,v) \in \bar{G} : p \in B\}. \end{displaymath}
Let us note that
\begin{equation}\label{form42b}  |\mathbf{W}(p)| = |\{v \in E(p) : (p,v) \in \mathbf{W}\}| \geq |\bar{H}(p)| = \bar{M}_{\Sigma} \sim_{\kappa} M_{\Sigma}, \quad p \in W, \end{equation}
since $W \subset \bar{P}$, recall \eqref{form193}. We now claim that
\begin{equation}\label{form195} \omega \in \mathbf{W} \quad \Longrightarrow \quad \bar{G}_{\lambda,t}^{\mathbf{C}_{j + 1}}(\omega) \subset \mathbf{B}_{\lambda,t}^{\mathbf{C}_{j + 1}}(\omega). \end{equation}
Indeed, fix $\omega = (p,v) \in \mathbf{W}$ and $(p',v') \in \bar{G}^{\mathbf{C}_{j + 1}}_{\lambda,t}(\omega)$. We simply need to show that $p' \in B$, and this follows from $p \in W \subset B(p_{0},\tfrac{1}{4}t/\mathbf{C}_{j + 1})$, and $t/\mathbf{C}_{j + 1} \leq |p - p'| \leq \mathbf{C}_{j + 1}t$, and the triangle inequality:
\begin{displaymath} \tfrac{3}{4}t/\mathbf{C}_{j + 1} \leq |p - p'| - |p_{0} - p| \leq |p_{0} - p'| \leq |p_{0} - p| + |p - p'| \leq 2\mathbf{C}_{j + 1}t. \end{displaymath}
From \eqref{form195}, and since $\mathbf{W} \subset \bar{H} \subset H$, and recalling \eqref{form46}, it follows
\begin{equation}\label{form41}  |\{\beta \in \mathbf{B}_{\lambda,t}^{\mathbf{C}_{j + 1}}(\omega) : \mathbf{C}_{j + 1}R^{\lambda}_{\Sigma}(\omega) \cap \mathbf{C}_{j + 1}R^{\lambda}_{\Sigma}(\beta) \neq \emptyset\}| \geq n > 0, \qquad \omega \in \mathbf{W}. \end{equation}

Next, we consider the rectangles
\begin{displaymath} \mathcal{R}^{\lambda}_{\Sigma} := \{R^{\lambda}_{\Sigma}(\omega) : \omega \in \mathbf{W}\}. \end{displaymath}
To be precise, let $\mathcal{R}^{\lambda}_{\Sigma}$ be the maximal family of pairwise $100$-incomparable $(\lambda,\Sigma)$-rectangles inside the family indicated above. Below, we will denote the $100$-comparability of $R,R'$ by $R \sim_{100} R'$. We now seek to show that every rectangle in $\mathcal{R}^{\lambda}_{\Sigma}$ has a high type relative to the pair $(W,B)$, in the terminology of Definition \ref{def:WolffType}.

To this end, we first define the quantity
\begin{equation}\label{form87} m(R) = |\{\omega \in \mathbf{W} : R \sim_{100} R^{\lambda}_{\Sigma}(\omega)\}|, \end{equation}
The value of $m(R)$ may vary between $1$ and $\leq \lambda^{-4}$, but by pigeonholing, we may find a subset $\bar{\mathcal{R}}^{\lambda}_{\Sigma} \subset \mathcal{R}^{\lambda}_{\Sigma}$ with the property $m(R) \equiv m \in [1,\lambda^{-4}]$ for all $R \in \bar{\mathcal{R}}^{\lambda}_{R}$, and moreover
\begin{equation}\label{form194} \sum_{\omega \in \mathbf{W}} |\{R \in \bar{\mathcal{R}}^{\lambda}_{\Sigma} : R \sim_{100} R^{\lambda}_{\Sigma}(\omega)\}| \gtrapprox_{\lambda} \sum_{\omega \in \mathbf{W}} |\{R \in \mathcal{R}^{\lambda}_{\Sigma} : R \sim_{100} R^{\lambda}_{\Sigma}(\omega)\}|. \end{equation}
Now, we have
\begin{align} |\mathcal{\bar{R}}^{\lambda}_{\Sigma}| & = \frac{1}{m} \sum_{R \in \bar{\mathcal{R}}^{\lambda}_{\Sigma}} \sum_{p \in W} \mathop{\sum_{v \in E(p)}}_{(p,v) \in \mathbf{W}} \mathbf{1}_{\{R \sim_{100} R^{\lambda}_{\Sigma}(p,v)\}} \notag\\
& \stackrel{\eqref{form194}}{\gtrapprox_{\lambda}} \frac{1}{m} \sum_{p \in W} \mathop{\sum_{v \in E(p)}}_{(p,v) \in \mathbf{W}} |\{R \in \mathcal{R}^{\lambda}_{\Sigma} : R \sim_{100} R^{\lambda}_{\Sigma}(p,v)\}| \notag\\
&\label{form48} \stackrel{\eqref{form42b}}{\geq} \frac{|W|\bar{M}_{\Sigma}}{m} \sim_{\kappa} \frac{|W|M_{\Sigma}}{m}.  \end{align}
The second-to-last inequality is true because every rectangle $R^{\lambda}_{\Sigma}(p,v)$ with $(p,v) \in \mathbf{W}$ is $100$-comparable to at least one rectangle in $\mathcal{R}^{\lambda}_{\Sigma}$, by definition of $\mathcal{R}^{\lambda}_{\Sigma}$.

\subsubsection{Proving that $m \lessapprox n$} We next claim that
\begin{equation}\label{form196} m(R) \leq \lambda^{-\zeta} n, \qquad R \in \mathcal{R}^{\lambda}_{\Sigma}, \end{equation}
where $n \geq 1$ was the constant defined in \eqref{form214}. In particular $m \leq \lambda^{-\zeta}n$. The estimate \eqref{form196} will eventually follow from the inductive hypothesis \eqref{form47}, but the details take some work. Let $R^{\lambda}_{\Sigma}(\omega) \in \mathcal{R}^{\lambda}_{\Sigma}$, with $\omega = (p,v) \in \mathbf{W}$. According to \eqref{form41}, there exists at least one element $\beta = (q,w) \in \mathbf{B}^{\mathbf{C}_{j + 1}}_{\lambda,t}(\omega) \subset \bar{G}$ such that
\begin{equation}\label{form199} \mathbf{C}_{j + 1}R^{\lambda}_{\Sigma}(\omega) \cap \mathbf{C}_{j + 1}R^{\lambda}_{\Sigma}(\beta) \neq \emptyset. \end{equation}
We claim that if $\omega' = (p',v') \in \mathbf{W}$ is any element such that $R^{\lambda}_{\Sigma}(\omega) \sim_{100} R^{\lambda}_{\Sigma}(\omega')$, then automatically
\begin{equation}\label{form197} \omega' \in \bar{G}_{\lambda,t}^{\mathbf{C}_{j}}(\beta) \quad \text{and} \quad \mathbf{C}_{j}R^{\lambda}_{\Sigma}(\omega') \cap \mathbf{C}_{j}R^{\lambda}_{\Sigma}(\beta) \neq \emptyset. \end{equation}
This will show that
\begin{displaymath} m(R) \leq |\{\omega' \in \bar{G}_{\lambda,t}^{\mathbf{C}_{j}}(\beta) : \mathbf{C}_{j}R^{\lambda}_{\Sigma}(\beta) \cap \mathbf{C}_{j}R^{\lambda}_{\Sigma}(\omega') \neq \emptyset\}| \stackrel{\eqref{form47}}{\leq} \lambda^{-\zeta}n, \end{displaymath}
as desired. The points $\omega,\omega' \in \mathbf{W}$ and $\beta \in \mathbf{B}$, as above, will be fixed for the remainder of this subsection.

The second claim in \eqref{form197} is easy: since $R^{\lambda}_{\Sigma}(\omega) \sim_{100} R^{\lambda}_{\Sigma}(\omega')$, it follows from Lemma \ref{lemma6} that $R^{\lambda}_{\Sigma}(\omega') \subset AR^{\lambda}_{\Sigma}(\omega) \subset \mathbf{C}_{j + 1}R^{\lambda}_{\Sigma}(\omega)$ for a suitable absolute constant $A \geq 1$. Lemma \ref{lemma6} then yields
\begin{equation}\label{form200} \mathbf{C}_{j + 1}R^{\lambda}_{\Sigma}(\omega) \subset A\mathbf{C}_{j + 1}^{5}R^{\lambda}_{\Sigma}(\omega') \subset \mathbf{C}_{j}R^{\lambda}_{\Sigma}(\omega'). \end{equation}
The second part of \eqref{form197} follows from this inclusion, and \eqref{form199}.

We turn to the first claim in \eqref{form197}. Since $\omega' = (p',v') \in \mathbf{W}$ and $\beta = (q,w) \in \mathbf{B}$, we have $p' \in W$ and $q \in B$, so
\begin{displaymath} t/\mathbf{C}_{j} \leq \tfrac{1}{4}t/\mathbf{C}_{j + 1} \leq |p' - q| \leq 2\mathbf{C}_{j + 1}t \leq \mathbf{C}_{j}t. \end{displaymath}
It therefore only remains to show that $\Delta(p',q) \leq \mathbf{C}_{j}\lambda$. To this end, recall that $\omega = (p,v)$. Then, since $\beta = (q,w) \in \mathbf{B}^{\mathbf{C}_{j + 1}}_{\lambda,t}(\omega)$, we have
\begin{displaymath} \bar{\lambda} := \Delta(p,q) \leq \mathbf{C}_{j + 1}\lambda \quad \text{and} \quad \bar{t} := |p - q| \leq \mathbf{C}_{j + 1}t. \end{displaymath}
Consequently,
\begin{displaymath} \bar{\Sigma} := \lambda/\sqrt{(\bar{\lambda} + \lambda)(\bar{t} + \lambda)} \gtrsim \mathbf{C}_{j + 1}^{-1}\sqrt{\lambda/t} = \mathbf{C}_{j + 1}^{-1}\Sigma,\end{displaymath}
and because of this,
\begin{displaymath} A\mathbf{C}_{j + 1}^{2}R^{\lambda}_{\bar{\Sigma}}(\omega) \cap A\mathbf{C}_{j + 1}^{2}R^{\lambda}_{\bar{\Sigma}}(\beta) \supset \mathbf{C}_{j + 1}R^{\lambda}_{\Sigma}(\omega) \cap \mathbf{C}_{j + 1}R^{\lambda}_{\Sigma}(\beta) \stackrel{\eqref{form199}}{\neq} \emptyset. \end{displaymath}
It now follows from Corollary \ref{cor2} applied at scale $\lambda$ and with constant $C = A\mathbf{C}_{j + 1}^{2}$ that
\begin{equation}\label{form201} R^{\lambda}_{\Sigma}(\omega) \subset \mathbf{C}_{j + 1}'R^{\lambda}_{\Sigma}(\beta) \subset S^{\mathbf{C}_{j + 1}'\lambda}(q), \end{equation}
for some $\mathbf{C}_{j + 1}' \lesssim \mathbf{C}_{j + 1}^{8}$. On the other hand, we saw in \eqref{form200} that
\begin{displaymath} R^{\lambda}_{\Sigma}(\omega) \subset A\mathbf{C}_{j + 1}^{5}R^{\lambda}_{\Sigma}(\omega') \subset S^{A\mathbf{C}^{5}_{j + 1}}(p'), \end{displaymath}
and therefore $R^{\lambda}_{\Sigma}(\omega)$ is contained in the intersection $S^{\mathbf{C}_{j + 1}^{9}\lambda}(q) \cap S^{\mathbf{C}^{9}_{j + 1}}(p')$. But this intersection can be covered by boundedly many discs of radius $\mathbf{C}_{j + 1}^{9}\lambda/\sqrt{\Delta(p',q)|p' - q|}$, which shows that
\begin{displaymath} \sqrt{\frac{\lambda}{t}} = \Sigma \lesssim  \frac{\mathbf{C}_{j + 1}^{18}\lambda}{\sqrt{\Delta(p',q)|p' - q|}}, \end{displaymath}
and rearranging this we find $\Delta(p',q) \lesssim \mathbf{C}_{j + 1}^{12}\lambda$. This proves that $\Delta(p',q) \leq \mathbf{C}_{j}\lambda$, since we chose $\mathbf{C}_{j} = A\mathbf{C}_{j + 1}^{18}$ above \eqref{form205}. We have now shown \eqref{form197}, and therefore \eqref{form196}.

\subsubsection{The type of rectangles in $\bar{\mathcal{R}}^{\lambda}_{\Sigma}$} We claim that that every $R \in \bar{\mathcal{R}}^{\lambda}_{\Sigma}$ has type $(\geq \bar{m},\geq \bar{n})_{\rho}$ relative to $(W,B)$, where
\begin{equation}\label{form204} \bar{m} := \lambda^{\rho}m, \quad \bar{n} \geq \lambda^{\rho}n, \quad \text{and} \quad \rho = 10 \cdot 18^{\ceil{20/\kappa}}\epsilon. \end{equation}
Let us recall from Definition \ref{def:WolffType} what this means: a $(\lambda,\Sigma)$-rectangle $R$ has type $(\geq \bar{m},\geq \bar{n})_{\rho}$ relative to $W,B$ if there exists at least $\bar{m}$ points $\{p_{1},\ldots,p_{\bar{m}}\} \subset W$ and at least $\bar{n}$ points $\{q_{1},\ldots,q_{\bar{n}}\} \subset B$ such that
\begin{equation}\label{form88} R \subset S^{\lambda^{1 - \rho}}(p_{k}) \cap S^{\lambda^{1 - \rho}}(q_{l}), \qquad 1 \leq k \leq \bar{m}, \, 1 \leq l \leq \bar{n}. \end{equation}
To see this, recall that $m(R) \equiv m$ for all $R \in \bar{\mathcal{R}}^{\lambda}_{\Sigma}$, where $m(R)$ was defined in \eqref{form87}: there exist $m$ pairs $\{\omega_{1},\ldots,\omega_{m}\} \subset \mathbf{W}$ such that $R \sim_{100} R^{\lambda}_{\Sigma}(\omega_{j})$. Writing $\omega_{k} = (p_{k},v_{k})$, and using Lemma \ref{lemma6}, this implies
\begin{displaymath} R \subset AR^{\lambda}_{\Sigma}(\omega_{k}) \subset S^{\lambda^{1 - \epsilon}}(p_{k}), \end{displaymath}
where $A \geq 1$ is absolute, and the second inclusion holds for $\lambda > 0$ small enough. This is even better than the first inclusion in \eqref{form88}. There is a small problem: some of the points "$p_{k}$" may be repeated, even though the pairs $\omega_{k} = (p_{k},v_{k}) \in \mathbf{W}$ are distinct. However, for $p_{k} \in \mathbf{D}$ fixed, there are $\lesssim 1$ choices $v_{k} \in E(p)$ such that $R \subset AR^{\lambda}_{\Sigma}(p_{k},v_{k})$ (since $E(p)$ is $\Sigma$-separated), so the number of distinct points "$p_{k}$" is $\gtrsim m$, and certainly $\geq \bar{m}$.

The proof of the second inclusion in \eqref{form88} is similar, but now based on \eqref{form41}: for all $R = R^{\lambda}_{\Sigma}(\omega) \in \mathcal{R}^{\lambda}_{\Sigma}$, there exist $n$ pairs
\begin{displaymath} \{\beta_{1},\ldots,\beta_{n}\} \subset \mathbf{B}^{\mathbf{C}_{j + 1}}_{\lambda,t}(\omega) \quad \text{s.t.} \quad \mathbf{C}_{j + 1}R \cap \mathbf{C}_{j + 1}R^{\lambda}_{\Sigma}(\beta_{l}) \neq \emptyset \text{ for } 1 \leq l \leq n. \end{displaymath}
If we write $\beta_{l} = (q_{l},w_{l})$, then the same argument which we used in \eqref{form201} shows that
\begin{equation}\label{form203} R^{\lambda}_{\Sigma}(\omega) \subset \mathbf{C}_{j + 1}^{9}R^{\lambda}_{\Sigma}(\beta_{l}) \subset S^{\mathbf{C}_{j + 1}^{9}\lambda}(q_{l}) \subset S^{\lambda^{1 - \rho}}(q_{l}), \qquad 1 \leq l \leq n, \end{equation}
using in the final inclusion that
\begin{displaymath} \mathbf{C}_{j + 1}^{9} \leq \mathbf{C}_{0}^{9} \stackrel{\eqref{form205}}{\leq} A^{9 \cdot 18^{\ceil{20/\kappa}}} \cdot \lambda^{-9 \cdot 18^{\ceil{20/\kappa}}\epsilon} \stackrel{\eqref{form204}}{\leq} \lambda^{-\rho}, \end{displaymath}
assuming $\lambda > 0$ small enough (depending on $\epsilon,\kappa$) in the final inequality. This proves the second inclusion in \eqref{form88}. Again, all the "$n$" points $q_{l}$ need not be distinct, but for every fixed $q_{l}$, the first inclusion in \eqref{form203} can hold for $\lesssim \mathbf{C}_{j + 1}^{5} \leq \lambda^{-\rho}$ choices of "$w_{l}$", so $|\{q_{1},\ldots,q_{l}\}| \gtrsim \lambda^{\rho}n$, as desired. This completes the proof of \eqref{form88}.

\subsubsection{Applying Lemma \ref{lemma8}} To find a contradiction, and conclude the proof, we aim to apply Lemma \ref{lemma8} to bound the cardinality of $\bar{\mathcal{R}}^{\lambda}_{\Sigma}$ from above. Notice that, by the definition of $W,B$, see \eqref{form89}, the definition of "$\rho$" at \eqref{form204}, and since $4\mathbf{C}_{j + 1} \leq \lambda^{-\rho}$, the pair $(W,B)$ is $(\lambda,\rho)$-almost $t$-bipartite. In the previous section, we showed that every rectangle $R \in \bar{\mathcal{R}}_{\Sigma}^{\lambda}$ has type $(\geq \bar{m},\geq \bar{n})_{\rho}$ relative to $W,B$. Therefore, Lemma \ref{lemma8} is applicable to $\bar{\mathcal{R}}^{\lambda}_{\Sigma}$. This yields the following inequality (see explanations below it):
\begin{displaymath} \frac{|W| M_{\Sigma}}{m} \stackrel{\eqref{form48}}{\lessapprox_{\lambda}} |\bar{\mathcal{R}}^{\lambda}_{\Sigma}| \lessapprox_{\lambda} \left(\frac{|B||W|}{mn} \right)^{3/4} + \frac{|B|}{m} + \frac{|W|}{n} \lessapprox_{\lambda} \left(\frac{|W|^{2}}{mn} \right)^{3/4} + \lambda^{-\zeta} \frac{|W|}{m}. \end{displaymath}
To make the estimate look neater, we allowed the "$\approx_{\lambda}$" notation hide constants of the form $\lambda^{-C(\kappa)\epsilon}$, where $C(\kappa) \geq 1$ is a constant depending only on $\kappa$. In the second inequality, we are hiding the constant $\lambda^{-C\rho} = \lambda^{-C(\kappa)\epsilon}$ produced by Lemma \ref{lemma8}. In the third inequality, we are hiding the constant $\lambda^{-C(\kappa)\epsilon}$ produced by \eqref{form42}. The factor $\lambda^{-\zeta}$ in the second inequality appears from \eqref{form196}, and it is a good moment to recall from \eqref{form91} that $\zeta \leq (\kappa s)/10$.

We observe immediately that the second term on the right cannot dominate the left hand side for $\epsilon = \epsilon(\kappa,s) > 0$ sufficiently small (the choice in \eqref{def:epsilon} should suffice), and $\lambda = \lambda(\epsilon,\kappa,s) > 0$ sufficiently small: this is because $M_{\Sigma} \geq \lambda^{\epsilon}\Sigma^{-s} = \lambda^{\epsilon}\sqrt{\lambda/t}^{-s} \geq \lambda^{\epsilon - \kappa s/6} \geq \lambda^{-\kappa s/7}$ (using our assumption \eqref{form106}), whereas $\lambda^{-\zeta} \leq \lambda^{-\kappa s/10}$ by \eqref{form91}.

Therefore, the term $|W|^{3/2}/(mn)^{3/4}$ needs to dominate the left hand side. Rearranging this inequality, using again $m \leq \lambda^{-\zeta}n$, recalling that $n = \lambda^{s - \bar{\kappa} + \zeta}|P|$, and finally using the $(\lambda,s,\lambda^{-\epsilon})$-set property of $P$ to bound $|W| \leq \lambda^{-\epsilon}t^{s}|P|$ leads to
\begin{displaymath} M_{\Sigma} \lessapprox_{\lambda} \lambda^{-\zeta/4}n^{-1/2}|W|^{1/2} \leq \lambda^{-s/2 - \zeta + \bar{\kappa}/2} \left(\frac{|W|}{|P|} \right)^{1/2} \lessapprox_{\lambda} \lambda^{-\zeta + \bar{\kappa}/2}\left(\frac{t}{\lambda} \right)^{s/2}. \end{displaymath}
This inequality is impossible for $\epsilon,\lambda > 0$ small enough depending on $\kappa$, since $\bar{\kappa} \geq \kappa$, and $\zeta \leq (\kappa s)/10 \leq \bar{\kappa}/10$ -- and finally because $M_{\Sigma} \equiv |E(p)| \geq \lambda^{\epsilon}\Sigma^{-s} = \lambda^{\epsilon}(t/\lambda)^{s/2}$.

To summarise, we have now shown that the case (1) in the construction of the sequence $\{G_{j}\}$ cannot occur as long as long as $\kappa_{h - j} \geq \kappa$. As we explained below the case distinction, this allows us to set $G := G_{j}$ for the first index satisfying $\kappa_{h - j} < \kappa$. The proof of Proposition \ref{prop3} is complete. \end{proof}

We will use Proposition \ref{prop3} via Theorem \ref{thm3} below. First, as promised at the beginning of this section, we introduce \emph{the partial multiplicity functions}. Compare these with the \emph{total multiplicity function} from Definition \ref{grandMultiplicity}.

 \begin{definition}[Partial multiplicity function]\label{def:multFunction1} Fix $0 < \delta \leq \Delta \leq \lambda \leq t \leq 1$ and $\rho \geq 1$. Let $P \subset \mathcal{D}_{\delta}$, and $E(p) \subset \mathcal{S}_{\delta}(p)$ for all $p \in P$. Write $\Omega = \{(p,v) : p \in P \text{ and } v \in E(p)\}$, and let $\sigma \in 2^{-\N}$ be the smallest dyadic rational larger than $\Delta/\sqrt{\lambda t}$. For $G \subset \Omega$, we define
\begin{displaymath} m_{\Delta,\lambda,t}^{\rho,C}(\omega \mid G) := |\{\omega' \in (G^{\Delta}_{\sigma})_{\lambda,t}^{\rho}(\omega) : CR^{\Delta}_{\sigma}(\omega) \cap CR^{\Delta}_{\sigma}(\omega') \neq \emptyset\}|, \qquad \omega \in G \cup G^{\Delta}_{\sigma}. \end{displaymath}
Here $G^{\Delta}_{\sigma}$ is the $(\Delta,\sigma)$-skeleton of $G$. \end{definition}

\begin{remark} The only interesting parameters "$\Delta$" for us will be $\Delta = \delta$ and $\Delta = \lambda$. If $\Delta = \delta$, we will usually write $\sigma = \Delta/\sqrt{\lambda t} = \delta/\sqrt{\lambda t}$, and for $\Delta = \lambda$, we will instead use the capital letter $\Sigma = \Delta/\sqrt{\lambda t} = \sqrt{\lambda /t}$. Also, to be accurate, the notation $\sigma,\Sigma$ typically refers to the smallest dyadic rational greater than $\delta/\sqrt{\lambda t}$ and $\sqrt{\lambda/t}$, respectively.

Finding a nice notation for the partial multiplicity functions is a challenge, due to the large number of parameters. In addition to the "range" and "constant" parameters $\rho$ and $C$, one could add up to $4$ further parameters: two "skeleton" parameters and two "rectangle" parameters. In practice, however, if the triple $(\Delta,\lambda,t)$ is given, the only useful rectangles are the $(\Delta,\Delta/\sqrt{\lambda t})$-rectangles. This relationship stems from Lemma \ref{lemma5}. So, we have decided against introducing the fourth parameter independently.

 \end{remark}

\begin{thm}\label{thm3} For every $\kappa > 0$ and $s \in (0,1]$, there exist $\epsilon_{0} := \epsilon_{0}(\kappa,s) \in (0,\tfrac{1}{2}]$ and $\delta_{0} = \delta_{0}(\epsilon,\kappa,s) > 0$ such that the following holds for all $\delta \in (0,\delta_{0}]$ and $\epsilon \in (0,\epsilon_{0}]$.

Let $\Omega$ be a $(\delta,\delta,s,\delta^{-\epsilon})$-configuration with $P := \pi_{\R^{3}}(\Omega)$. Fix $\delta \leq \lambda \leq t \leq 1$. Then, there exists a $(\delta,\delta,s,C\delta^{-\epsilon})$-configuration $G \subset \Omega$ such that $C \approx_{\delta,\kappa} 1$, $|G| \approx_{\delta,\kappa} |\Omega|$, and
\begin{equation}\label{form18} m_{\lambda,\lambda,t}^{\delta^{-\epsilon_{0}},\delta^{-\epsilon_{0}}}(\omega \mid G) \leq \delta^{-\kappa}\lambda^{s}|P|_{\lambda}, \qquad \omega \in G^{\lambda}_{\Sigma}. \end{equation}
\end{thm}

\begin{proof} Let us spell out what \eqref{form18} means: for $\Sigma = \sqrt{\lambda/t}$, we should prove that
\begin{displaymath} |\{\omega' \in (G^{\lambda}_{\Sigma})^{\delta^{-\epsilon_{0}}}_{\lambda,t}(\omega) : \delta^{-\epsilon_{0}}R^{\lambda}_{\Sigma}(\omega) \cap \delta^{-\epsilon_{0}}R^{\lambda}_{\Sigma}(\omega') \neq \emptyset\}| \leq \delta^{-\kappa}\lambda^{s}|P|_{\lambda}, \qquad \omega \in G^{\lambda}_{\Sigma}. \end{displaymath}
We first dispose of the case where $\lambda \geq \delta^{\kappa/10}$. In this case we simply take $\epsilon_{0} = \kappa/5$ and $G := \Omega$. Now, the left hand side of \eqref{form18} is bounded from above by
\begin{displaymath} |G^{\lambda}_{\Sigma}| \lesssim \lambda^{-3}\Sigma^{-1} \leq \lambda^{-4} \leq \delta^{-2\kappa/5} = \delta^{-3\kappa/5}\delta^{\epsilon_{0}} \leq \delta^{-3\kappa/5}\lambda^{s}|P|_{\lambda}, \end{displaymath}
using finally the assumption that $P$ is a non-empty $(\delta,s,\delta^{-\epsilon_{0}})$-set.

Let us then assume that $\lambda \leq \delta^{\kappa/10}$. We apply Proposition \ref{prop6} with $\sigma = \delta$ and $\Delta = \lambda$ and $\Sigma = \sqrt{\lambda/t} \geq \sigma$. This produces a subset $G_{0} \subset \Omega$ of cardinality $|G_{0}| \approx_{\delta} |\Omega|$ whose $(\lambda,\Sigma)$-skeleton
\begin{displaymath}  (G_{0})^{\lambda}_{\Sigma} = \{(\mathbf{p},\mathbf{v}) : \mathbf{p} \in P_{\lambda} \text{ and } \mathbf{v} \in \mathbf{E}(\mathbf{p})\} \end{displaymath}
is a $(\lambda,\Sigma,s,C\delta^{-\epsilon})$-configuration with $C \approx_{\delta} 1$ (in particular, this skeleton is a $(\lambda,\Sigma,s,C\delta^{-\epsilon_{0}})$-configuration). Moreover, recall from \eqref{form40} that
\begin{equation}\label{form44} |\{(p,v) \in G_{0} : (p,v) \prec (\mathbf{p},\mathbf{v})\}| \approx_{\delta} \frac{|\Omega|}{|(G_{0})^{\lambda}_{\Sigma}|}, \qquad (\mathbf{p},\mathbf{v}) \in (G_{0})^{\lambda}_{\Sigma}. \end{equation}
It may be worth emphasising a small technical point: we never claimed, and do not claim here either, that $G_{0}$ would be a $(\delta,\delta,s,C\delta^{-\epsilon})$-configuration.

Next, we apply Proposition \ref{prop3} with constants "$\kappa,s$". This produces a constant $\epsilon_{1} := \epsilon_{1}(\kappa,s) > 0$. Note that since $\lambda \leq \delta^{\kappa/10}$ by assumption, we have $C\delta^{-\epsilon_{0}} \leq \lambda^{-20\epsilon_{0}/\kappa}$ for $\delta > 0$ small enough. Therefore, if we choose "$\epsilon_{0}$" presently so small that $20\epsilon_{0}/\kappa < \epsilon_{1}$, we see that $(G_{0})^{\lambda}_{\Sigma}$ is a $(\lambda,\Sigma,s,\lambda^{-\epsilon_{1}})$-configuration. Now, by Proposition \ref{prop3}, there exists a $(\lambda,\Sigma,s,C_{\kappa}\lambda^{-\epsilon_{1}})$-configuration $\mathbf{G} \subset (G_{0})^{\lambda}_{\Sigma}$ with $|\mathbf{G}| \sim_{\kappa} |(G_{0})^{\lambda}_{\Sigma}|$, and the property
\begin{equation}\label{form43} |\{\omega' \in \mathbf{G}_{\lambda,t}^{\lambda^{-\epsilon_{1}}}(\omega) : \lambda^{-\epsilon_{1}}R^{\lambda}_{\Sigma}(\omega) \cap \lambda^{-\epsilon_{1}}R^{\lambda}_{\Sigma}(\omega') \neq \emptyset\}| \leq \lambda^{s - \kappa}|P_{\lambda}|, \qquad \omega \in \mathbf{G}. \end{equation}
Note that $\delta^{-\epsilon_{0}} \leq \lambda^{-\epsilon_{1}}$ by our choices of constants, and $\lambda \geq \delta$, so \eqref{form43} implies
\begin{equation}\label{form43a} |\{\omega' \in \mathbf{G}_{\lambda,t}^{\delta^{-\epsilon_{0}}}(\omega) : \delta^{-\epsilon_{0}}R^{\lambda}_{\Sigma}(\omega) \cap \delta^{-\epsilon_{0}}R^{\lambda}_{\Sigma}(\omega') \neq \emptyset\}| \leq \\\delta^{-\kappa}\lambda^{s}|P|_{\lambda}, \qquad \omega \in \mathbf{G}. \end{equation}
We also used that $|P_{\lambda}| \leq |P|_{\lambda}$. Next, let
\begin{displaymath} G_{1} := \bigcup_{(\mathbf{p},\mathbf{v}) \in \mathbf{G}} \{(p,v) \in G_{0} : (p,q) \prec (\mathbf{p},\mathbf{v})\} = \bigcup_{(\mathbf{p},\mathbf{v}) \in \mathbf{G}} G_{0} \cap (\mathbf{p} \otimes \mathbf{v}). \end{displaymath}
Then $(G_{1})^{\lambda}_{\Sigma} \subset \mathbf{G}$ by definition, so \eqref{form43a} implies \eqref{form18} for $G_{1}$. Moreover, as explained in Remark \ref{rem1}, the sets $\mathbf{p} \otimes \mathbf{v}$ are disjoint, so
\begin{displaymath} |G_{1}| = \sum_{(\mathbf{p},\mathbf{v}) \in \mathbf{G}} |G_{0} \cap (\mathbf{p} \otimes \mathbf{v})| \stackrel{\eqref{form44}}{\approx_{\delta}} |\mathbf{G}| \cdot \frac{|\Omega|}{|(G_{0})^{\lambda}_{\Sigma}|} \sim_{\kappa} |\Omega|. \end{displaymath}
The only problem remaining is that $G_{1}$ may not be a $(\delta,\delta,s,C\delta^{-\epsilon})$-configuration. However, $|G_{1}| \approx_{\delta,\kappa} |\Omega|$, so it follows from the refinement principle (Lemma \ref{refinement}) that there exists a $(\delta,\delta,s,C\delta^{-\epsilon})$-configuration $G \subset G_{1}$ such that $C \approx_{\delta,\kappa} 1$ and $|G| \approx_{\delta,\kappa} |\Omega|$. Now, $G$ continues to satisfy \eqref{form18}, so the proof of Theorem \ref{thm3} is complete.  \end{proof}

\begin{remark} It may be worth remarking that if $G$ is the final $(\delta,\delta,s)$-configuration in the previous theorem, the $(\lambda,\Sigma)$-skeleton $G^{\lambda}_{\Sigma}$ may fail to be a $(\lambda,\Sigma,s)$-configuration. This was not claimed either. It seems generally tricky to ensure that a set $G \subset \Omega$ is simultaneously a $(\delta,\sigma,s)$-configuration and a $(\Delta,\Sigma,s)$-configuration for $\delta \ll \Delta$ and $\sigma \ll \Sigma$.  \end{remark}


\section{An upper bound for incomparable $(\delta,\sigma)$-rectangles}\label{deltalambdat}

\begin{notation}\label{not1} Let $0 < \delta \leq \Delta \leq 1$ and $0 < \sigma \leq \Sigma \leq 1$. Let $p \in \mathcal{D}_{\delta}$, and let $E_{\sigma}(p) \subset \mathcal{S}_{\sigma}(p)$ (recall that the notation $S_{\delta}(p)$ refers to a circle associated to the centre of $p$). We write
\begin{displaymath} \mathcal{E}^{\Delta}_{\Sigma}(p) := \bigcup_{\mathbf{v} \in E_{\Sigma}(p)} R^{\Delta}_{\Sigma}(\mathbf{p},\mathbf{v}) \subset S^{\Delta}(\mathbf{p}), \end{displaymath}
where $\mathbf{p} \in \mathcal{D}_{\Delta}$ is the unique dyadic $\Delta$-cube with $p \subset \mathbf{p}$, and $E_{\Sigma}(p)$ is the $(\Delta,\Sigma)$-skeleton of $E_{\sigma}(p)$, namely $E_{\Sigma}(p) = \{\mathbf{v} \in \mathcal{S}_{\Sigma}(\mathbf{p}) : v \prec \mathbf{v} \text{ for some } v \in E_{\sigma}(p)\}$.  \end{notation}

\begin{lemma}\label{lemma2} Let $C \geq 1$, $0 < \delta \leq \Delta \leq 1$, $0 < \sigma \leq \Sigma \leq 1$. Assume also that $\Delta \leq \Sigma$. Let $p \in \mathcal{D}_{\delta}$, and let $E_{\sigma}(p) \subset \mathcal{S}_{\sigma}(p)$. Then $C\mathcal{E}^{\delta}_{\sigma}(p) \subset C'\mathcal{E}^{\Delta}_{\Sigma}(p)$ for some $C' \sim C$.
\end{lemma}

\begin{figure}[h!]
\begin{center}
\begin{overpic}[scale = 1]{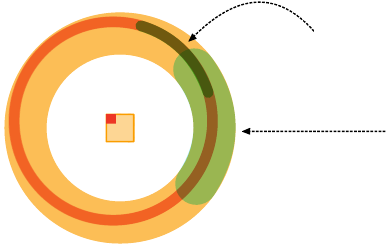}
\put(36,30){$\mathbf{p}$}
\put(23,32){$p$}
\put(82,53){$CR^{\delta}_{\sigma}(p,v)$}
\put(102,29){$C'R^{\Delta}_{\Sigma}(\mathbf{p},\mathbf{v})$}
\end{overpic}
\caption{The rectangles $CR^{\delta}_{\sigma}(p,v)$ and $C'R^{\Delta}_{\Sigma}(\mathbf{p},\mathbf{v})$ in the proof of Lemma \ref{lemma2}.}\label{fig2}
\end{center}
\end{figure}

\begin{proof} The proof is illustrated in Figure \ref{fig2}. The set $\mathcal{E}^{\delta}_{\sigma}(p)$ is a union of the $(\delta,\sigma)$-rectangles $R^{\delta}_{\sigma}(p,v)$ centred at $v \in E_{\sigma}(p)$. Let $R = R^{\delta}_{\sigma}(p,v)$ be one of these rectangles. By the definition of $(\Delta,\Sigma)$-skeleton, there exists $\mathbf{v} \in E_{\Sigma}(p)$ such that $v \prec \mathbf{v}$, or in other words $p \subset \mathbf{p} \in \mathcal{D}_{\Delta}$ and $v \cap V(\mathbf{p},\mathbf{v}) \neq \emptyset$. Since $|p - \mathbf{p}| \leq 2\Delta$ and $\delta \leq \Delta$, we have
\begin{displaymath} S^{C\delta}(p) \subset S^{2C\Delta}(\mathbf{p}). \end{displaymath}
Moreover, it follows from $v \cap V(\mathbf{p},\mathbf{v}) \neq \emptyset$ and $\Delta \leq \Sigma$ that $|v - \mathbf{v}| \leq C'\Sigma$, where $C' > 0$ is absolute. Consequently,
\begin{displaymath} B(v,C\sigma) \subset B(v,C\Sigma) \subset B(\mathbf{v},(C + C')\Sigma). \end{displaymath}
Combining this information, we have
\begin{displaymath} CR_{\sigma}^{\delta}(p,v) = S^{C\delta}(p) \cap B(v,C\sigma) \subset S^{2C\Delta}(\mathbf{p}) \cap B(\mathbf{v},(C + C')\Sigma) = C''R^{\Delta}_{\Sigma}(\mathbf{p},\mathbf{v}), \end{displaymath}
with $C'' := \max\{2C,C + C'\}$. This completes the proof. \end{proof}

We next define a variant of the "type" introduced in Definition \ref{def:WolffType}.

\begin{definition}\label{def:type} Let $0 < \delta \leq \sigma \leq1$, and let $P \subset \mathcal{D}_{\delta}$. For every $p \in P$, let $E(p) \subset \mathcal{S}_{\delta}(p)$. Let $W,B \subset P$ be finite sets. For $\delta \leq \lambda \leq 1$, and $m,n \geq 1$, we say that a $(\delta,\sigma)$-rectangle $R \subset \R^{2}$ has \emph{$\lambda$-restricted type $(\geq m,\geq n)_{\epsilon}$ relative to $(W,B,\{E(p)\})$} if there exists a set $W_{R} \subset W$ of cardinality $|W_{R}| \geq m$, and for every $p \in W_{R}$ a subset $B_{R}(p) \subset B$ of cardinality $|B_{R}(p)| \geq n$ such that the following holds:
\begin{enumerate}
\item $\delta^{\epsilon}\lambda \leq \Delta(p,q) \leq \delta^{-\epsilon}\lambda$ for all $p \in W_{R}$ and all $q \in B_{R}(p)$.
\item $R \subset \delta^{-\epsilon}\mathcal{E}^{\delta}_{\sigma}(p) \cap \delta^{-\epsilon}\mathcal{E}^{\delta}_{\sigma}(q)$ for all $p \in W_{R}$ and all $q \in B_{R}(p)$.
\end{enumerate}
If $\lambda = \delta$, then the requirement in (1) is relaxed to $\Delta(p,q) \leq \delta^{1 - \epsilon}$.
\end{definition}

\begin{remark} The presence of the sets $E(p)$ is a major difference compared to Definition \ref{def:WolffType}, and we will distinguish between these two definitions by using the terminology "...relative to $(W,B)$" in Definition \ref{def:WolffType}, and "...relative to $(W,B,\{E(p)\})$ in Definition \ref{def:type}. We will make sure that there is never a risk of confusion which definition is meant.

Other differences are (obviously) the condition (1) of Definition \ref{def:type}, which is completely absent from Definition \ref{def:WolffType}. A more subtle point is the asymmetry of Definition \ref{def:type}: even if a rectangle has $\lambda$-restricted type $(\geq m,\geq m)_{\epsilon}$ relative to $(W,B,\{E(p)\})$, it need not have $\lambda$-restricted type $(\geq m,\geq m)_{\epsilon}$ relative to $(B,W,\{E(p)\})$. \end{remark}

\begin{thm}\label{thm4} For every $\eta > 0$, there exist $\epsilon = \epsilon(\eta) \in (0,1]$ and $\delta_{0} = \delta_{0}(\eta,\epsilon) \in (0,1]$ such that the following holds for all $\delta \in (0,\delta_{0}]$. Let $0 < \delta \leq \lambda \leq t \leq 1$ be dyadic rationals with $\lambda \leq \delta^{2\epsilon}t$. Let $P \subset \mathcal{D}_{\delta}$ be a set satisfying
\begin{equation}\label{form49} |P \cap \mathbf{p}| \leq X_{\lambda}, \qquad \mathbf{p} \in \mathcal{D}_{\lambda}, \end{equation}
where $X_{\lambda} \in \N$. For every $p \in P$, let $E(p) \subset \mathcal{S}_{\delta}(p)$. Write $\Sigma := \sqrt{\lambda/t}$, and let $\Omega^{\lambda}_{\Sigma}$ be the $(\lambda,\Sigma)$-skeleton of $\Omega = \{(p,v) : p \in P \text{ and } v \in E(p)\}$. Assume that, for some $Y_{\lambda} \in \N$,
\begin{equation}\label{form54} m^{\delta^{-\mathbf{A}\epsilon},\delta^{-\mathbf{A}\epsilon}}_{\lambda,\lambda,t}(\omega \mid \Omega) \leq Y_{\lambda}, \qquad \omega \in \Omega^{\lambda}_{\Sigma}, \end{equation}
where $\mathbf{A} \geq 1$ is a sufficiently large absolute constant, in particular $\mathbf{A}$ is independent of the previous parameters $\delta,\eta,\lambda,t$. Write $\sigma := \delta/\sqrt{\lambda t}$, and let $W,B$ be a $(\delta,\epsilon)$-almost $t$-bipartite pair of subsets of $P$. Let $1 \leq m \leq |W|$ and $1 \leq n \leq |B|$. Let $\mathcal{R}^{\delta}_{\sigma}$ be a collection of pairwise $100$-incomparable $(\delta,\sigma)$-rectangles whose $\lambda$-restricted type relative to $(W,B,\{E(p)\})$ is $(\geq m,\geq n)_{\epsilon}$. Then,
\begin{equation}\label{form61} |\mathcal{R}^{\delta}_{\sigma}| \leq \delta^{-\eta} \left[ \left(\frac{|W||B|}{mn} \right)^{3/4} (X_{\lambda}Y_{\lambda})^{1/2} + \frac{|W|}{m} \cdot X_{\lambda}Y_{\lambda} + \frac{|B|}{n} \cdot X_{\lambda}Y_{\lambda} \right]. \end{equation}
\end{thm}

\begin{remark} It is worth noting that the upper bound \eqref{form54} which is \textbf{assumed} here looks exactly like the upper bound provided by Theorem \ref{thm3}.

Another remark is that \eqref{form61} in the case $\lambda = \delta$ may actually be weaker than Wolff's tangency bound \eqref{form78}. In this case evidently $X_{\lambda} \leq 1$, but it may well happen that $Y_{\lambda} \gg 1$. This is irrelevant for our purposes, since Theorem \ref{thm4} will only be applied in a situation where $Y_{\lambda} \lessapprox 1$. For the interested reader, we mention that the main loss in the proof arises from the estimate \eqref{form95}, which is always unsharp if $M_{\lambda}N_{\lambda} \cdot (m_{\lambda}n_{\lambda}) \gg (\lambda/\delta)^{2}$.   \end{remark}

\begin{proof}[Proof of Theorem \ref{thm4}] We start with the case $m = 1 = n$, and later deal with the general case with a "random sampling" argument. Fix $\eta > 0$. We also choose $\epsilon > 0$ so small that $\sqrt{\epsilon} < c\eta$ for a suitable absolute constant to be determined later (this constant will be determined by the constant in Lemma \ref{lemma8}).

In this proof, "$C$" will refer to an absolute constant whose value may change -- usually increase -- from one line to the next without separate remark. We will also assume, when needed, that "$\delta > 0$ is small enough" without separate remark.

We may assume with no loss of generality that the rectangles in $\mathcal{R}^{\delta}_{\sigma}$ are pairwise $\delta^{-C\epsilon}$-incomparable for a suitable absolute constant $C > 0$, instead of just $100$-incomparable. This is because by Corollary \ref{AT14}, any collection of $100$-incomparable rectangles $\mathcal{R}^{\delta}_{\sigma}$ contains a $\delta^{-C\epsilon}$-incomparable subset $\bar{\mathcal{R}}^{\delta}_{\sigma}$ of cardinality $|\bar{\mathcal{R}}^{\delta}_{\sigma}| \geq \delta^{O(C\epsilon)}|\mathcal{R}^{\delta}_{\sigma}|$, and now it suffices to prove \eqref{form61} for $\bar{\mathcal{R}}^{\delta}_{\sigma}$.

By assumption, every rectangle $R \in \mathcal{R}^{\delta}_{\sigma}$ has $\lambda$-restricted type $(\geq 1, \geq 1)_{\epsilon}$ relative to $(W,B)$. Thus, for every $R \in \mathcal{R}^{\delta}_{\sigma}$ we may associate a pair $(p,q)_{R} \in W \times B$ with the properties
\begin{equation}\label{form51} \delta^{\epsilon}\lambda \leq \Delta(p,q) \leq \delta^{-\epsilon}\lambda \quad \text{and} \quad R \subset \delta^{-\epsilon}\mathcal{E}^{\delta}_{\sigma}(p) \cap \delta^{-\epsilon}\mathcal{E}^{\delta}_{\sigma}(q) \subset S^{\delta^{1 - \epsilon}}(p) \cap S^{\delta^{1 - \epsilon}}(q). \end{equation}
(If $\lambda = \delta$, we only have $\Delta(p,q) \leq \delta^{1 - \epsilon}$.) We record at this point that any fixed pair $(p,q) \in W \times B$ can only be associated to boundedly many rectangles $R \in \mathcal{R}^{\delta}_{\sigma}$:
\begin{equation}\label{form70} |\{R \in \mathcal{R}^{\delta}_{\sigma} : (p,q)_{R} = (p,q)\}| \lesssim 1, \qquad (p,q) \in W \times B. \end{equation}
Indeed, if there exists at least one rectangle $R_{0} \in \mathcal{R}^{\delta}_{\sigma}$ such that $(p,q)_{R_{0}} = (p,q)$, then $|p - q| \geq \delta^{\epsilon}t$ and $\Delta(p,q) \geq \delta^{\epsilon}\lambda$. Under these conditions, Lemma \ref{lemma5} implies that the intersection $S^{\delta^{1 - \epsilon}}(p) \cap S^{\delta^{1 - \epsilon}}(q)$ can be covered by boundedly many $(\delta^{1 - C\epsilon},\delta^{1 - C\epsilon}/\sqrt{\lambda t})$-rectangles, and actually they can be selected to be of the form
\begin{displaymath} \delta^{-C\epsilon}R_{j} := \delta^{-C\epsilon}R^{\delta}_{\sigma}(q,v_{j}), \qquad 1 \leq j \lesssim 1, \end{displaymath}
where each $R_{j}$ is a $(\delta,\sigma)$-rectangle. (We note that this is true also if $\lambda = \delta$, using only $|p - q| \geq \delta^{\epsilon} t$ in that case.) We claim that each rectangle $R \in \mathcal{R}^{\delta}_{\sigma}$ with $(p,q)_{R} = (p,q)$ is $\delta^{-C\epsilon}$-comparable to one of the rectangles $R_{j}$. This will imply \eqref{form70}, because at most one rectangle in $\mathcal{R}_{\sigma}^{\delta}$ can be $\delta^{-C\epsilon}$-comparable to a fixed $R_{j}$: indeed any pair of $(\delta,\sigma)$-rectangles $\delta^{-C\epsilon}$-comparable to $R_{j}$ would be $\lesssim \delta^{-C\epsilon}$-comparable to each other by Corollary \ref{cor1}, contradicting our "without loss of generality" assumption that the rectangles in $\mathcal{R}^{\delta}_{\sigma}$ are $\delta^{-C\epsilon}$-incomparable. Thus, the left hand side of \eqref{form70} is bounded by the number of the rectangles $R_{j}$ (which is $\lesssim 1$).

We then show that every $R \in \mathcal{R}^{\delta}_{\sigma}$ with $(p,q)_{R} = (p,q)$ is $\delta^{-C\epsilon}$-comparable to some $R_{j}$. Namely, if $(p,q)_{R} = (p,q)$, then $R \subset S^{\delta^{1 - \epsilon}}(p) \cap S^{\delta^{1 - \epsilon}}(q)$ by definition, and because $\diam(R) \leq 2\sigma$, it follows that $R \subset \delta^{-\epsilon}R^{\delta}_{\sigma}(q,v)$ for some $v \in S(q)$ (here e.g. $v$ is the closest point on $S(q)$ from the centre of $R$). On the other hand, since the rectangles $\delta^{-C\epsilon}R_{j} = \delta^{-C\epsilon}R^{\delta}_{\sigma}(q,v_{j})$ cover $S^{\delta^{1 - \epsilon}}(p) \cap S^{\delta^{1 - \epsilon}}(q)$, one of them intersects $R$, say $R \cap \delta^{-C\epsilon}R_{j} \neq \emptyset$. Now, it is easy to check that
\begin{displaymath} R_{j} \subset \delta^{-C\epsilon}R^{\delta}_{\sigma}(q,v), \end{displaymath}
and therefore $R,R_{j} \subset \delta^{-C\epsilon}R^{\delta}_{\sigma}(q,v)$. In other words, $R,R_{j}$ are $\delta^{-C\epsilon}$-comparable.

With the proof of \eqref{form70} behind us, we proceed with other preliminaries. Let
\begin{displaymath} \mathbf{p} \in \mathcal{D}_{\lambda}(W) =: \mathcal{W}_{\lambda} \quad \text{and} \quad \mathbf{q} \in \mathcal{D}_{\lambda}(B) =: \mathcal{B}_{\lambda}, \end{displaymath}
and write
\begin{displaymath} \mathcal{R}^{\delta}_{\sigma}(\mathbf{p},\mathbf{q}) := \{R \in \mathcal{R}^{\delta}_{\sigma} : (p,q)_{R} \in (W \cap \mathbf{p}) \times (B \cap \mathbf{q})\}. \end{displaymath}
With this notation, we have
\begin{equation}\label{form64} |R^{\delta}_{\sigma}| \leq \sum_{(\mathbf{p},\mathbf{q}) \in \mathcal{W}_{\lambda} \times \mathcal{B}_{\lambda}} |\mathcal{R}^{\delta}_{\sigma}(\mathbf{p},\mathbf{q})|. \end{equation}
We use the pigeonhole principle to find subsets
\begin{displaymath} \overline{\mathcal{W}}_{\lambda} \subset \mathcal{W}_{\lambda} \quad \text{and} \quad \overline{\mathcal{B}}_{\lambda} \subset \mathcal{B}_{\lambda} \end{displaymath}
with the properties
\begin{equation}\label{form67} \begin{cases} |W \cap \mathbf{p}| \sim M_{\lambda}, & \mathbf{p} \in \overline{\mathcal{W}}_{\lambda}, \\ |B \cap \mathbf{q}| \sim N_{\lambda}, & \mathbf{q} \in \overline{\mathcal{B}}_{\lambda}, \end{cases} \end{equation}
(where $M_{\lambda},N_{\lambda} \in \{1,\ldots,X_{\lambda}\}$ are fixed integers) and such that
\begin{equation}\label{form66} |\mathcal{R}^{\delta}_{\sigma}| \stackrel{\eqref{form64}}{\leq} \sum_{(\mathbf{p},\mathbf{q}) \in \mathcal{W}_{\lambda} \times \mathcal{B}_{\lambda}} |\mathcal{R}^{\delta}_{\sigma}(\mathbf{p},\mathbf{q})| \approx_{\delta} \sum_{(\mathbf{p},\mathbf{q}) \in \overline{\mathcal{W}}_{\lambda} \times \overline{\mathcal{B}}_{\lambda}} |\mathcal{R}^{\delta}_{\sigma}(\mathbf{p},\mathbf{q})|. \end{equation}
It now suffices to show that
\begin{equation}\label{form65} \sum_{(\mathbf{p},\mathbf{q}) \in \overline{\mathcal{W}}_{\lambda} \times \overline{\mathcal{B}}_{\lambda}} |\mathcal{R}^{\delta}_{\sigma}(\mathbf{p},\mathbf{q})| \lessapprox_{\delta} (|W||B|)^{3/4}(X_{\lambda}Y_{\lambda})^{1/2} + |W|(X_{\lambda}Y_{\lambda}) + |B|(X_{\lambda}Y_{\lambda}). \end{equation}
To begin with, we claim that
\begin{equation}\label{form50} |\mathcal{R}^{\delta}_{\sigma}(\mathbf{p},\mathbf{q})| \lesssim M_{\lambda}N_{\lambda}, \qquad \mathbf{p} \in \overline{\mathcal{W}}_{\lambda}, \, \mathbf{q} \in \overline{\mathcal{B}}_{\lambda}.  \end{equation}
This follows from
\begin{displaymath} |\mathcal{R}^{\delta}_{\sigma}(\mathbf{p},\mathbf{q})| = \sum_{(p,q) \in (W \cap \mathbf{p}) \times (B \cap \mathbf{q})} |\{R \in \mathcal{R}^{\delta}_{\sigma} : (p,q)_{R} = (p,q)\}|, \end{displaymath}
and the fact recorded in \eqref{form70} that every term in this sum is $\lesssim 1$.

To proceed estimating \eqref{form66}, notice that we only need to sum over the pairs $(\mathbf{p},\mathbf{q}) \in \overline{\mathcal{W}}_{\lambda} \times \overline{\mathcal{B}}_{\lambda}$ with $\mathcal{R}^{\delta}_{\sigma}(\mathbf{p},\mathbf{q}) \neq \emptyset$. In this case there exists at least one pair $p \in W \cap \mathbf{p}$ and $q \in B \cap \mathbf{q}$ satisfying \eqref{form51}. It follows from Lemma \ref{lemma2} applied with $\Delta = \lambda$ and $\Sigma := \sqrt{\lambda/t} \geq \max\{\lambda,\sigma\}$ that
\begin{equation}\label{form52} \delta^{2\epsilon}t \leq |\mathbf{p} - \mathbf{q}| \leq \delta^{-2\epsilon}t, \quad \Delta(\mathbf{p},\mathbf{q}) \leq \delta^{-2\epsilon}\lambda, \quad \text{and} \quad \delta^{-2\epsilon}\mathcal{E}^{\lambda}_{\Sigma}(\mathbf{p}) \cap \delta^{-2\epsilon}\mathcal{E}^{\lambda}_{\Sigma}(\mathbf{q}) \neq \emptyset. \end{equation}
Here the bounds $|\mathbf{p} - \mathbf{q}| \geq \delta^{2\epsilon}t$ and $\Delta(\mathbf{p},\mathbf{q}) \leq \delta^{-2\epsilon}\lambda$ used our assumption $\lambda \leq \delta^{2\epsilon}t$ (and that $|p - q| \geq \delta^{\epsilon}t$ for some pair $p \in \mathbf{p}$ and $q \in \mathbf{q}$). To spell out the definitions, here
\begin{displaymath} \mathcal{E}_{\Sigma}^{\lambda}(\mathbf{p}) = \bigcup_{\mathbf{v} \in E_{\Sigma}(\mathbf{p})} R^{\lambda}_{\Sigma}(\mathbf{p},\mathbf{v}), \end{displaymath}
where $E_{\Sigma}(\mathbf{p})$ is the $(\lambda,\Sigma)$-skeleton of $E(p)$ (for $p \in P \cap \mathbf{p}$). We record at this point that
\begin{displaymath} \mathbf{p} \in \mathcal{W}_{\lambda} \cup \mathcal{B}_{\lambda} \text{ and } \mathbf{v} \in E_{\Sigma}(\mathbf{p}) \quad \Longrightarrow \quad (\mathbf{p},\mathbf{v}) \in \Omega_{\Sigma}^{\lambda}, \end{displaymath}
where $\Omega_{\Sigma}^{\lambda}$ is the $(\Sigma,\lambda)$-skeleton of $\Omega$.

We write $\mathbf{p} \sim \mathbf{q}$ if $\mathbf{p} \in \overline{\mathcal{W}}_{\lambda}$, $\mathbf{q} \in \overline{\mathcal{B}}_{\lambda}$, and the conditions \eqref{form52} hold. Then, by \eqref{form66} and the preceding discussion
\begin{equation}\label{form58} |\mathcal{R}^{\delta}_{\sigma}| \lessapprox \sum_{\mathbf{p} \sim \mathbf{q}} |\mathcal{R}^{\delta}_{\sigma}(\mathbf{p},\mathbf{q})| \stackrel{\eqref{form50}}{\lesssim} M_{\lambda}N_{\lambda} \cdot |\{(\mathbf{p},\mathbf{q}) \in \overline{\mathcal{W}}_{\lambda} \times \overline{\mathcal{B}}_{\lambda} : \mathbf{p} \sim \mathbf{q}\}|. \end{equation}
To estimate the cardinality $|\{(\mathbf{p},\mathbf{q}) : \mathbf{p} \sim \mathbf{q}\}|$, we will infer from \eqref{form52} that whenever $\mathbf{p} \sim \mathbf{q}$, then $S(\mathbf{p})$ and $S(\mathbf{q})$ are "roughly" tangent to a $(\lambda,\Sigma)$-rectangle, denoted $R^{\lambda}_{\Sigma}(\mathbf{p},\mathbf{q})$, more precisely satisfying
\begin{equation}\label{form94} R^{\lambda}_{\Sigma}(\mathbf{p},\mathbf{q}) \subset \delta^{-C\epsilon}\mathcal{E}^{\lambda}_{\Sigma}(\mathbf{p}) \cap \delta^{-C\epsilon}\mathcal{E}^{\lambda}_{\Sigma}(\mathbf{q}) \end{equation}
for a suitable absolute constant $C \geq 1$. Let us justify why $R^{\lambda}_{\Sigma}(\mathbf{p},\mathbf{q})$ can be found. Since $\delta^{-2\epsilon}\mathcal{E}^{\lambda}_{\Sigma}(\mathbf{p}) \cap \delta^{-2\epsilon}\mathcal{E}^{\lambda}_{\Sigma}(\mathbf{q}) \neq \emptyset$, there first of all exist $\mathbf{v} \in E_{\Sigma}(\mathbf{p})$, $\mathbf{w} \in E_{\Sigma}(\mathbf{q})$, and a point
\begin{displaymath} v \in \delta^{-2\epsilon}R^{\lambda}_{\Sigma}(\mathbf{p},\mathbf{v}) \cap \delta^{-2\epsilon}R^{\lambda}_{\Sigma}(\mathbf{q},\mathbf{w}). \end{displaymath}
Consequently, we may find points $\bar{\mathbf{p}}$ and $\bar{\mathbf{q}}$ with $|\bar{\mathbf{p}} - \mathbf{p}| \leq \delta^{-2\epsilon}\lambda$ and $|\bar{\mathbf{q}} - \mathbf{q}| \leq \delta^{-2\epsilon}\lambda$ such that $v \in S(\bar{\mathbf{p}}) \cap S(\bar{\mathbf{q}})$. Since $\Sigma \geq \lambda$, we also have
\begin{equation}\label{form93} \max\{\dist(v,E_{\Sigma}(\mathbf{p})),\dist(v,E_{\Sigma}(\mathbf{q}))\} \leq \delta^{-2\epsilon}\Sigma. \end{equation}
Now, it follows from \eqref{form52} and the inclusion \eqref{form71} (and noting that $\Sigma \leq \delta^{-\epsilon}\sqrt{\lambda/|\mathbf{p} - \mathbf{q}|}$) that
\begin{displaymath} R^{\lambda}_{\Sigma}(\mathbf{p},\mathbf{q}) := R^{\lambda}_{\Sigma}(\bar{\mathbf{p}},v) \subset S^{\delta^{-C\epsilon}\lambda}(\bar{\mathbf{p}}) \cap S^{\delta^{-C\epsilon}\lambda}(\bar{\mathbf{q}}). \end{displaymath}
Taking also into account \eqref{form93}, we arrive at \eqref{form94}.

Now that we have defined the $(\lambda,\Sigma)$-rectangles $R^{\lambda}_{\Sigma}(\mathbf{p},\mathbf{q})$, we let $\mathcal{R}^{\lambda}_{\Sigma}$ be a maximal collection of pairwise $100$-incomparable rectangles in $\{\mathcal{R}^{\lambda}_{\Sigma}(\mathbf{p},\mathbf{q}) : \mathbf{p} \in \overline{\mathcal{W}}_{\lambda}, \, \mathbf{q} \in \overline{\mathcal{B}}_{\lambda} \text{ and } \mathbf{p} \sim \mathbf{q}\}$. For $R \in \mathcal{R}_{\Sigma}^{\lambda}$, we then write $R \sim (\mathbf{p},\mathbf{q})$ if $\mathbf{p} \sim \mathbf{q}$ and $R \sim_{100} R^{\lambda}_{\Sigma}(\mathbf{p},\mathbf{q})$. With this notation, we may estimate
\begin{equation}\label{form57} |\{(\mathbf{p},\mathbf{q}) \in \overline{\mathcal{W}}_{\lambda} \times \overline{\mathcal{B}}_{\lambda} : \mathbf{p} \sim \mathbf{q}\}| \leq \sum_{R \in \mathcal{R}^{\lambda}_{\Sigma}} |\{(\mathbf{p},\mathbf{q}) \in \overline{\mathcal{W}}_{\lambda} \times \overline{\mathcal{B}}_{\lambda} : R \sim (\mathbf{p},\mathbf{q})\}|,  \end{equation}
since every pair $(\mathbf{p},\mathbf{q})$ with $\mathbf{p} \sim \mathbf{q}$ satisfies $R \sim (\mathbf{p},\mathbf{q})$ for at least one rectangle $R \in \mathcal{R}^{\lambda}_{\Sigma}$.

To estimate \eqref{form57} further, we consider the following slightly \emph{ad hoc} "type" of the rectangles $R \in \mathcal{R}_{\Sigma}^{\lambda}$ relative to the pair $(\overline{\mathcal{W}}_{\lambda},\overline{\mathcal{B}}_{\lambda})$. (This notion will not appear outside this proof.) We say that $R \in \mathcal{R}_{\Sigma}^{\lambda}$ has \emph{type $(m_{\lambda},n_{\lambda})$ relative to $(\overline{\mathcal{W}}_{\lambda},\overline{\mathcal{B}}_{\lambda})$} if the following sets $\mathcal{W}_{\lambda}(R) \subset \overline{\mathcal{W}}_{\lambda}$ and $\mathcal{B}_{\lambda}(R) \subset \overline{\mathcal{B}}_{\lambda}$ have cardinalities $|\mathcal{W}_{\lambda}(R)| = m_{\lambda}$ and $|\mathcal{B}_{\lambda}(R)| = n_{\lambda}$:
\begin{itemize}
\item $\mathcal{W}_{\lambda}(R)$ consists of all $\mathbf{p} \in \overline{\mathcal{W}}_{\lambda}$ such that $R \subset \delta^{-\mathbf{C}\epsilon}\mathcal{E}_{\Sigma}^{\lambda}(\mathbf{p})$.
\item $\mathcal{B}_{\lambda}(R)$ consists of all $\mathbf{q} \in \overline{\mathcal{B}}_{\lambda}$ such that $R \subset \delta^{-\mathbf{C}\epsilon}\mathcal{E}_{\Sigma}^{\lambda}(\mathbf{q})$.
\end{itemize}
Here
\begin{equation}\label{form221} \mathbf{C} := 2C \geq 2, \end{equation}
where "$C$" is the absolute constant from \eqref{form94}. We observe at once that the type of every rectangle $R \in \mathcal{R}^{\lambda}_{\Sigma}$ is $(\geq 1,\geq 1)$ in this terminology, because each $R \in \mathcal{R}^{\lambda}_{\Sigma}$ has the form $R = R^{\lambda}_{\Sigma}(\mathbf{p},\mathbf{q})$ for some $(\mathbf{p},\mathbf{q}) \in \overline{\mathcal{W}}_{\lambda} \times \overline{\mathcal{B}}_{\lambda}$, and \eqref{form94} holds for this pair $(\mathbf{p},\mathbf{q})$.

\begin{remark}\label{rem2} Assume for a moment that $\lambda \leq \delta^{\sqrt{\epsilon}}$. Then, if $R \in \mathcal{R}^{\lambda}_{\Sigma}$ has type $(\geq m_{\lambda},n_{\lambda})$ relative to $(\overline{\mathcal{W}}_{\lambda},\overline{\mathcal{B}}_{\lambda})$ according to the definition above, then $R$ also has type $(\geq m_{\lambda},n_{\lambda})_{\mathbf{C}\sqrt{\epsilon}}$ relative to $(\overline{\mathcal{W}}_{\lambda},\overline{\mathcal{B}}_{\lambda})$ in the sense of Definition \ref{def:WolffType}. This is simply because
\begin{displaymath} \delta^{-\mathbf{C}\epsilon}\mathcal{E}_{\Sigma}^{\lambda}(\mathbf{p}) \subset S^{\delta^{-\mathbf{C}\epsilon}\lambda}(p), \end{displaymath}
and $\delta^{-\mathbf{C}\epsilon} \leq \lambda^{-\mathbf{C}\sqrt{\epsilon}}$ by the temporary assumption $\lambda \leq \delta^{\sqrt{\epsilon}}$. But the "ad hoc" definition here is far more restrictive: it requires $R$ to lie close to the sets $E_{\Sigma}(\mathbf{p})$ and $E_{\Sigma}(\mathbf{q})$.  \end{remark}

We now establish two claims related to our \emph{ad hoc} notion of type:

\begin{claim}\label{c1} If $R \in \mathcal{R}_{\Sigma}^{\lambda}$ has type $(m_{\lambda},n_{\lambda})$ relative to $(\overline{\mathcal{W}}_{\lambda},\overline{\mathcal{B}}_{\lambda})$, then
\begin{equation}\label{form60} |\{(\mathbf{p},\mathbf{q}) \in \overline{\mathcal{W}}_{\lambda} \times \overline{\mathcal{B}}_{\lambda} : R \sim (\mathbf{p},\mathbf{q})\}| \leq m_{\lambda}n_{\lambda}. \end{equation}
\end{claim}

\begin{proof} Let $\mathcal{W}'_{\lambda}(R) \subset \overline{\mathcal{W}}_{\lambda}$ be the subset of all those $\mathbf{p} \in \overline{\mathcal{W}}_{\lambda}$ such that $R \sim (\mathbf{p},\mathbf{q})$ for at least one $\mathbf{q} \in \overline{\mathcal{B}}_{\lambda}$. Define $\mathcal{B}'_{\lambda}(R)$ similarly, interchanging the roles of $\overline{\mathcal{W}}_{\lambda}$ and $\overline{\mathcal{B}}_{\lambda}$. Evidently
\begin{displaymath} |\{(\mathbf{p},\mathbf{q}) \in \overline{\mathcal{W}}_{\lambda} \times \overline{\mathcal{B}}_{\lambda} : R \sim (\mathbf{p},\mathbf{q})\}| \leq |\mathcal{W}'_{\lambda}(R)||\mathcal{B}'_{\lambda}(R)|. \end{displaymath}
It remains to show that $\mathcal{W}_{\lambda}'(R) \subset \mathcal{W}_{\lambda}(R)$ and $\mathcal{B}_{\lambda}'(R) \subset \mathcal{B}_{\lambda}(R)$. To see this, fix $\mathbf{p} \in \mathcal{W}_{\lambda}'(R)$. By definition, there exists $\mathbf{q} \in \overline{\mathcal{B}}_{\lambda}$ such that $R \sim (\mathbf{p},\mathbf{q})$. This means that $R$ is $100$-comparable to the rectangle $R^{\lambda}_{\Sigma}(\mathbf{p},\mathbf{q})$ which satisfies \eqref{form94}. According to Lemma \ref{lemma6}, there exists an absolute constant $C > 0$ such that
\begin{displaymath} R \subset CR^{\lambda}_{\Sigma}(\mathbf{p},\mathbf{q}) \stackrel{\eqref{form94}}{\subset} C\delta^{-C\epsilon}\mathcal{E}^{\lambda}_{\Sigma}(\mathbf{p}) \cap C\delta^{-\epsilon}\mathcal{E}^{\lambda}_{\Sigma}(\mathbf{q}) \subset \delta^{-\mathbf{C}\epsilon}\mathcal{E}^{\lambda}_{\Sigma}(\mathbf{p}). \end{displaymath}
This shows that $\mathbf{p} \in \mathcal{W}_{\lambda}(R)$ by definition. Thus $\mathcal{W}_{\lambda}'(R) \subset \mathcal{W}_{\lambda}(R)$. The proof of the inclusion $\mathcal{B}_{\lambda}'(R) \subset \mathcal{B}_{\lambda}(R)$ is similar. \end{proof}

\begin{claim}\label{claim1} Assume that $R \in \mathcal{R}_{\Sigma}^{\lambda}$ has type $(m_{\lambda},n_{\lambda})$ relative to $(\overline{\mathcal{W}}_{\lambda},\overline{\mathcal{B}}_{\lambda})$, and assume that the constant "$\mathbf{A}$" in \eqref{form54} satisfies $\mathbf{A} \geq 3(\mathbf{C} + 1)$, where $\mathbf{C}$ is the absolute constant determined at \eqref{form221}. Then
\begin{equation}\label{form59} \max\{m_{\lambda},n_{\lambda}\} \leq Y_{\lambda}. \end{equation}
\end{claim}

This is where the absolute constant $\mathbf{A}$ in the statement of Theorem \ref{thm4} is determined.

\begin{proof}[Proof of Claim \ref{claim1}] Write $R = \mathcal{R}^{\lambda}_{\Sigma}(\mathbf{p},\mathbf{q})$ with $\mathbf{p} \in \overline{\mathcal{W}}_{\lambda}$, $\mathbf{q} \in \overline{\mathcal{B}}_{\lambda}$, and $\mathbf{p} \sim \mathbf{q}$. Then, enumerate $\mathcal{W}_{\lambda}(R) = \{\mathbf{p}_{1},\ldots,\mathbf{p}_{m_{\lambda}}\}$. Now $\delta^{2\epsilon}t \leq |\mathbf{p}_{j} - \mathbf{q}| \leq \delta^{-2\epsilon}t$ for all $1 \leq j \leq m_{\lambda}$, and moreover
\begin{displaymath} R \subset \delta^{-\mathbf{C}\epsilon}\mathcal{E}_{\Sigma}^{\lambda}(\mathbf{p}_{j}) \cap \delta^{-\mathbf{C}\epsilon}\mathcal{E}_{\Sigma}^{\lambda}(\mathbf{q}), \qquad 1 \leq j \leq m_{\lambda}. \end{displaymath}
Unraveling this inclusion, there exist $\mathbf{w} \in E_{\Sigma}(\mathbf{q})$, and for each $1 \leq j \leq m_{\lambda}$ some $\mathbf{v}_{j} \in E_{\Sigma}(\mathbf{p}_{j})$ such that
\begin{equation}\label{form206} R \subset \delta^{-\mathbf{C}\epsilon}R^{\lambda}_{\Sigma}(\mathbf{p}_{j},\mathbf{v}_{j}) \cap \delta^{-\mathbf{C}\epsilon}R^{\lambda}_{\Sigma}(\mathbf{q},\mathbf{w}). \end{equation}
We now claim that
\begin{equation}\label{form207} (\mathbf{p}_{j},\mathbf{v}_{j}) \in (\Omega^{\lambda}_{\Sigma})^{\delta^{-\mathbf{A}\epsilon}}_{\lambda,t}(\mathbf{q},\mathbf{w}), \qquad 1 \leq j \leq m_{\lambda}, \end{equation}
if $\mathbf{A} \geq 3(\mathbf{C} + 1)$. Noting that $\omega := (\mathbf{q},\mathbf{w}) \in \Omega^{\lambda}_{\Sigma}$, this will prove that
\begin{displaymath} m_{\lambda} \leq |\{\omega' \in (\Omega_{\Sigma}^{\lambda})_{\lambda,t}^{\delta^{-\mathbf{A}\epsilon}}(\omega) : \delta^{-\mathbf{A}\epsilon}R^{\lambda}_{\Sigma}(\omega') \cap \delta^{-\mathbf{A}\epsilon}R^{\lambda}_{\Sigma}(\omega) \neq \emptyset\}| = m^{\delta^{-\mathbf{A}\epsilon},\delta^{-\mathbf{A}\epsilon}}_{\lambda,\lambda,t}(\omega \mid \Omega) \stackrel{\eqref{form54}}{\leq} Y_{\lambda}, \end{displaymath}
and the upper bound $|\mathcal{B}_{\lambda}(R)| = n_{\lambda} \leq Y_{\lambda}$ can be established in a similar fashion.

Regarding \eqref{form207}, we already know that $\delta^{\mathbf{A}\epsilon}t \leq |\mathbf{p}_{j} - \mathbf{q}| \leq \delta^{-\mathbf{A}\epsilon}t$, provided that $\mathbf{A} \geq 2$. So, it remains to show that $\Delta(\mathbf{p}_{j},\mathbf{q}) \leq \delta^{-\mathbf{A}\epsilon}\lambda$. But \eqref{form206} implies that
\begin{displaymath} R \subset S^{\delta^{-\mathbf{C}\epsilon}\lambda}(\mathbf{p}_{j}) \cap S^{\delta^{-\mathbf{C}\epsilon}\lambda}(\mathbf{q}), \qquad 1 \leq j \leq m_{\lambda}. \end{displaymath}
The set $R \in \mathcal{R}^{\lambda}_{\Sigma}$ has $\diam(R) \sim \Sigma = \sqrt{\lambda/t}$, and on the other hand Lemma \ref{lemma5} implies that the intersection of the two annuli above can be covered by boundedly many discs of radius
\begin{displaymath} \delta^{-\mathbf{C}\epsilon}\lambda/\sqrt{\Delta(\mathbf{p}_{j},\mathbf{q})|\mathbf{p}_{j} - \mathbf{q}|} \leq \delta^{-(\mathbf{C} + 1)\epsilon}\lambda/\sqrt{\Delta(\mathbf{p}_{j},\mathbf{q}) \cdot t}.\end{displaymath}
This shows that $\sqrt{\lambda/t} \lesssim \delta^{-(\mathbf{C} + 1)\epsilon}\lambda/\sqrt{\Delta(\mathbf{p}_{j},\mathbf{q}) \cdot t}$, and rearranging gives $\Delta(\mathbf{p}_{j},\mathbf{q}) \lesssim \delta^{-2(\mathbf{C} + 1)\epsilon}\lambda$. This completes the proof of \eqref{form207}, and the lemma.
  \end{proof}

Each rectangle $R \in \mathcal{R}_{\Sigma}^{\lambda}$ has some type $(m_{\lambda},n_{\lambda})$ relative to $(\overline{\mathcal{W}}_{\lambda},\overline{\mathcal{B}}_{\lambda})$, with $1 \leq m_{\lambda},n_{\lambda} \lesssim \lambda^{-3}$. By pigeonholing, we may find a subset $\bar{\mathcal{R}}_{\Sigma}^{\lambda} \subset \mathcal{R}_{\Sigma}^{\lambda}$ such that every rectangle $R \in \bar{\mathcal{R}}_{\Sigma}^{\lambda}$ has type between $(m_{\lambda},n_{\lambda})$ and $(2m_{\lambda},2n_{\lambda})$ for some $m_{\lambda},n_{\lambda} \geq 1$, and moreover
\begin{equation}\label{form63} \sum_{R \in \mathcal{R}^{\lambda}_{\Sigma}} |\{(\mathbf{p},\mathbf{q}) \in \overline{\mathcal{W}}_{\lambda} \times \overline{\mathcal{B}}_{\lambda} : R \sim (\mathbf{p},\mathbf{q})\}| \approx_{\delta} \sum_{R \in \bar{\mathcal{R}}^{\lambda}_{\Sigma}} |\{(\mathbf{p},\mathbf{q}) \in \overline{\mathcal{W}}_{\lambda} \times \overline{\mathcal{B}}_{\lambda} : R \sim (\mathbf{p},\mathbf{q})\}|. \end{equation}
When we now combine \eqref{form58} with \eqref{form57}, then \eqref{form63}, and finally \eqref{form60}, we find
\begin{equation}\label{form95} |\mathcal{R}_{\sigma}^{\delta}| \lessapprox_{\delta} M_{\lambda}N_{\lambda} \cdot (m_{\lambda}n_{\lambda}) \cdot |\bar{\mathcal{R}}_{\Sigma}^{\lambda}|. \end{equation}
To conclude the proof of \eqref{form61} from here, we consider separately the "main" case $\lambda \leq \delta^{\sqrt{\epsilon}}$, and the "trivial" case $\lambda \geq \delta^{\sqrt{\epsilon}}$. In the trivial case, we simply apply the following uniform estimates:
\begin{displaymath} \max\{m_{\lambda},n_{\lambda}\} \leq \lambda^{-C} \leq \delta^{-C\sqrt{\epsilon}} \quad \text{and} \quad |\bar{\mathcal{R}}_{\Sigma}^{\lambda}| \leq \lambda^{-C} \leq \delta^{-C\sqrt{\epsilon}} \end{displaymath}
Consequently, using also $M_{\lambda} \leq \min\{|W|,X_{\lambda}\}$ and $N_{\lambda} \leq \min\{|B|,X_{\lambda}\}$, we get
\begin{displaymath} |\mathcal{R}^{\delta}_{\sigma}| \lessapprox_{\delta} \delta^{-3C\sqrt{\epsilon}}(M_{\lambda}N_{\lambda}) \leq \delta^{-3C\sqrt{\epsilon}}(|W||B|)^{3/4}X_{\lambda}^{1/2}. \end{displaymath}
This is even better than the case $m = 1 = n$ of \eqref{form61}, assuming $3C\sqrt{\epsilon} \leq \eta$.

Assume then that $\lambda \leq \delta^{\sqrt{\epsilon}}$. In this case, as pointed out in Remark \ref{rem2}, the family $\bar{\mathcal{R}}^{\lambda}_{\Sigma}$ consists of $(\lambda,\Sigma)$-rectangles of type $(\geq m_{\lambda},\geq n_{\lambda})_{\mathbf{C}\sqrt{\epsilon}}$ relative to $(\overline{\mathcal{W}}_{\lambda},\overline{\mathcal{B}}_{\lambda})$, in the sense of Definition \ref{def:WolffType}. Furthermore, the pair $(\overline{\mathcal{W}}_{\lambda},\overline{\mathcal{B}}_{\lambda})$ is $(\lambda,\mathbf{C}\sqrt{\epsilon})$-almost $t$-bipartite by \eqref{form52}, and since $\delta^{-2\epsilon} \leq \lambda^{-\mathbf{C}\sqrt{\epsilon}}$. Consequently, by Lemma \ref{lemma8}, we have
\begin{equation}\label{form62} |\bar{\mathcal{R}}_{\Sigma}^{\lambda}| \leq \lambda^{-O(\sqrt{\epsilon})} \left[ \left(\frac{|\overline{\mathcal{W}}_{\lambda}||\overline{\mathcal{B}}_{\lambda}|}{m_{\lambda}n_{\lambda}}\right)^{3/4} + \frac{|\overline{\mathcal{W}}_{\lambda}|}{m_{\lambda}} + \frac{|\overline{\mathcal{B}}_{\lambda}|}{n_{\lambda}} \right]. \end{equation}
In particular, we may choose $\epsilon = \epsilon(\eta) > 0$ so small that $\lambda^{-O(\sqrt{\epsilon})} \leq \delta^{-\eta}$.

The estimate \eqref{form62} is not yet the same as the case $m = 1 = n$ of \eqref{form61}. To reach \eqref{form61} from here, we consider separately the cases where the first, second, or third terms in \eqref{form62} dominate. In all cases, we will use (recall \eqref{form67}) that
\begin{displaymath} |\overline{\mathcal{W}}_{\lambda}| \lesssim \frac{|W|}{M_{\lambda}} \quad \text{and} \quad |\overline{\mathcal{B}}_{\lambda}| \lesssim \frac{|B|}{N_{\lambda}} \quad \text{and} \quad \max\{M_{\lambda},N_{\lambda}\} \leq X_{\lambda}. \end{displaymath}
Now, if the first ("main") term in \eqref{form62} is the largest, then (omitting the factor $\lambda^{-O(\sqrt{\epsilon})}$ for notational simplicity, and combining \eqref{form95} with \eqref{form62})
\begin{align*} |\mathcal{R}_{\sigma}^{\delta}| & \lessapprox_{\delta} M_{\lambda}N_{\lambda} \cdot (m_{\lambda}n_{\lambda}) \cdot \left(\frac{|\overline{\mathcal{W}}_{\lambda}||\overline{\mathcal{B}}_{\lambda}|}{m_{\lambda}n_{\lambda}}\right)^{3/4}\\
& \lesssim (M_{\lambda}N_{\lambda})^{1/4} \cdot (m_{\lambda}n_{\lambda})^{1/4} \cdot (|W||B|)^{3/4} \stackrel{\eqref{form59}}{\leq} (X_{\lambda}Y_{\lambda})^{1/2} \cdot (|W||B|)^{3/4}. \end{align*}
This is what we desired in \eqref{form61} (case $m = n = 1$).

Assume next that the second term in \eqref{form62} dominates. Then,
\begin{displaymath} |\mathcal{R}_{\sigma}^{\delta}| \lessapprox_{\delta} M_{\lambda}N_{\lambda} \cdot (m_{\lambda}n_{\lambda}) \cdot \frac{|\overline{\mathcal{W}}_{\lambda}|}{m_{\lambda}} = N_{\lambda} \cdot n_{\lambda} \cdot |W| \lesssim X_{\lambda}Y_{\lambda}|W|. \end{displaymath}
Similarly, if the third term in \eqref{form62} dominates, we get $|\mathcal{R}_{\sigma}^{\delta}| \lessapprox_{\delta} X_{\lambda}Y_{\lambda}|B|$. This concludes the proof of \eqref{form61} in the case $m = 1 = n$.

We then, finally, consider the case of general $1 \leq m \leq |W|$ and $1 \leq n \leq |B|$. This is morally the random sampling argument from \cite[Lemma 1.4]{MR1800068}, but the details are more complicated due to our asymmetric definition of "$\lambda$-restricted type". Fix a large absolute constant $A \geq 1$ (to be determined soon; this constant has no relation to the constant $\mathbf{A}$ introduced in Claim \ref{claim1}). Let $\overline{W} \subset W$ be the subset obtained by keeping every element of $W$ with probability $A/m$. Define the random subset $\bar{B} \subset B$ in the same way, keeping every element of $B$ with probability $A/n$. However, if $m \leq 2A$, we keep all the elements of $W$, and if $n \leq 2A$, we keep all the elements of $B$. We assume in the sequel that $\min\{m,n\} \geq 2A$ and leave the converse special cases to the reader (the case $\max\{m,n\} < 2A$ is completely elementary, but to understand what to do in the case $m < 2A \leq n$, we recommend first reading the argument below, and then thinking about the small modification afterwards.)

The underlying probability space is $\{0,1\}^{|W|} \times \{0,1\}^{|B|} =: \Lambda$. The pairs $(\omega,\beta) \in \Lambda$ are in $1$-to-$1$ correspondence with subset-pairs $\overline{W} \times \bar{B} \subset W \times B$, and we will prefer writing "$(\overline{W},\bar{B}) \in \Lambda$" in place of "$(\omega,\beta) \in \Lambda$". We denote by $\mathbb{P}$ the probability which corresponds to the explanation in the previous paragraph: thus, the probability of a sequence $(\omega,\beta)$ equals
\begin{displaymath} \mathbb{P}\{(\omega,\beta)\} = (\tfrac{A}{m})^{|\{\omega_{i} = 1\}|}(1 - \tfrac{A}{m})^{|\{\omega_{i} = 0\}|}(\tfrac{A}{n})^{|\{\beta_{j} = 1\}|}(1 - \tfrac{A}{n})^{|\{\beta_{j} = 0\}|}. \end{displaymath}
The most central random variables will be $|\overline{W}|$ and $|\bar{B}|$, formally
\begin{displaymath} |\overline{W}|(\omega,\beta) := |\{1 \leq i \leq |W| : \omega_{i} = 1\}| \quad \text{and} \quad |\bar{B}|(\omega,\beta) := |\{1 \leq j \leq |B| : \beta_{j} = 1\}| \end{displaymath}
In expectation $\mathbb{E}|\overline{W}| = A|W|/m$ and $\mathbb{E}|\bar{B}| = A|B|/n$. By Chebychev's inequality, the probability that either $|\overline{W}| \geq 4A|W|/m$ or $|\bar{B}| \geq 4A|B|/n$ is at most $\tfrac{1}{2}$. We let $\Lambda' \subset \Lambda$ be sequences in $(\omega,\beta) \in \Lambda$ for which $|\overline{W}(\omega,\beta)| \leq 4A|W|/m$ and $|\bar{B}(\omega,\beta)| \leq 4A|B|/n$. As we just said, $\mathbb{P}(\Lambda') \geq \tfrac{1}{2}$.

Let $\mathcal{R}^{\delta}_{\sigma}(\overline{W},\bar{B}) \subset \mathcal{R}^{\delta}_{\sigma}$ be the subset which has $\lambda$-restricted type $(\geq 1,\geq 1)_{\epsilon}$ relative to $(\overline{W},\bar{B})$. We claim that there exists $(\overline{W},\bar{B}) \in \Lambda'$ such that
\begin{equation}\label{form69} |\mathcal{R}_{\sigma}^{\delta}| \leq 4|\mathcal{R}^{\delta}_{\sigma}(\overline{W},\bar{B})|. \end{equation}
To see this, fix $R \in \mathcal{R}_{\sigma}^{\delta}$, and recall the definition of $\lambda$-restricted type $(\geq m,\geq n)_{\epsilon}$ relative to $(W,B)$. There exists a set $W_{R} \subset W$ with $|W_{R}| \geq m$, and for each $p \in W_{R}$ a subset
\begin{equation}\label{form215} B(p) \subset B \quad \text{with} \quad |B(p)| \geq n, \end{equation}
such that $\delta^{\epsilon}\lambda \leq \Delta(p,q) \leq \delta^{-\epsilon}\lambda$ for all $p \in W_{R}$ and $q \in B(p)$, and $R \subset \delta^{-\epsilon}\mathcal{E}_{\sigma}^{\delta}(p) \cap \delta^{-\epsilon}\mathcal{E}_{\sigma}^{\delta}(q)$ for all $p \in W_{R}$ and $q \in B_{R}(p)$. We claim that for any $c > 0$, we have
\begin{equation}\label{form96} \mathbb{P}(\{\exists \text{ at least one pair } (p,q) \in \overline{W} \times \bar{B} \text{ such that } p \in W_{R} \text{ and } q \in B(p)\}) \geq 1 - c, \end{equation}
assuming that the constant "$A$" is chosen large enough, depending only on $c$. Before attempting this, we prove something easier: $\mathbb{P}(\{\overline{W} \cap W_{R} \neq \emptyset\}) \geq 1 - c$. For each $p \in W_{R}$ fixed, we have
\begin{displaymath} \mathbb{P}(\{p \notin \overline{W}\}) = 1 - \frac{A}{m}. \end{displaymath}
Moreover, these events are independent when $p \in W_{R}$ (or even $p \in W$) varies. Therefore,
\begin{equation}\label{form98} \mathbb{P}(\{\overline{W} \cap W_{R} = \emptyset\}) = \prod_{p \in W_{R}} \mathbb{P}(\{p \notin \overline{W}\}) = (1 - \tfrac{A}{m})^{|W_{R}|} \leq \left((1 - \tfrac{A}{m})^{m/A} \right)^{A}. \end{equation}
Since $m \geq 2A$, the right hand side is bounded from above by $\rho^{A}$ for some (absolute) $\rho < 1$, and in particular the probability is $< c$ as soon as $\rho^{A} < c$.

To proceed towards \eqref{form96}, we partition the event $\{\overline{W} \cap W_{R} \neq \emptyset\}$ into a union of events of the form $\{\overline{W} \cap W_{R} = H\}$, where $H \subset W_{R}$ is a fixed non-empty subset. Clearly the events $\{\overline{W} \cap W_{R} = H\}$ and $\{\overline{W} \cap W_{R} = H'\}$ are disjoint for distinct (not necessarily disjoint) $H,H' \subset W_{R}$. For every $\emptyset \neq H \subset W_{R}$, we designate a point $p_{H} \in H$ in an arbitrary manner. For example, we could enumerate the points in $W_{R}$, and $p_{H} \in H$ could be the point with the lowest index in the enumeration. Then, for $H \subset W_{R}$ fixed, we consider the event $\{\bar{B} \cap B(p_{H}) \neq \emptyset\}$, where $B(p_{H}) \subset B$ is the set from \eqref{form215}. Since $\mathbb{P}(\{q \notin \bar{B}\}) = 1 - A/n$, and $|B(p_{H})| \geq n$, a calculation similar to the one on line \eqref{form98} shows that
\begin{equation}\label{form97} \mathbb{P}(\{\bar{B} \cap B(p_{H}) \neq \emptyset\}) \geq 1 - \rho^{A} > 1 - c, \qquad \emptyset \neq H \subset W_{R}, \end{equation}
assuming that $\rho^{A} < c$. Furthermore, we notice that for $\emptyset \neq H \subset W_{R}$ fixed,
\begin{displaymath} \mathbb{P}(\{\overline{W} \cap W_{R} = H\} \cap \{\bar{B} \cap B(p_{H}) \neq \emptyset\}) = \mathbb{P}(\{\overline{W} \cap W_{R} = H\})\mathbb{P}(\{\bar{B} \cap B(p_{H}) \neq \emptyset\}). \end{displaymath}
From a probabilistic point of view, this is because the events $\{\bar{B} \cap B(p_{H}) \neq \emptyset\}$ and $\{\overline{W} \cap W_{R} = H\}$ are independent. From a measure theoretic point of view, the set $\{\overline{W} \cap W_{R} = H\} \cap \{\bar{B} \cap B(p_{H}) \neq \emptyset\} \subset \{0,1\}^{|W|} \times \{0,1\}^{|B|} = \Lambda$ can be written as a product set. Now, we may estimate as follows:
\begin{align*} \sum_{\emptyset \neq H \subset W_{R}} & \mathbb{P}(\{\bar{B} \cap B(p_{H}) \neq \emptyset\} \cap \{\overline{W} \cap W_{R} = H\})\\
& \stackrel{\eqref{form97}}{\geq} (1 - c) \sum_{\emptyset \neq H \subset W_{R}} \mathbb{P}(\{\overline{W} \cap W_{R} = H\})\\
& = (1 - c) \cdot \mathbb{P}(\{\overline{W} \cap W_{R} \neq \emptyset\}) \geq (1 - c)^{2}. \end{align*}
On the other hand, the events we are summing over on the far left are disjoint, and their union is contained in the event shown in \eqref{form96}. This proves \eqref{form96} with $(1 - c)^{2}$ in place of $(1 - c)$, which is harmless.

Let $G_{R} \subset \Lambda$ be the "good" event from \eqref{form96}. Note that if $(\overline{W},\bar{B}) \in G_{R}$, then $R$ has restricted $\lambda$-type $(\geq 1,\geq 1)_{\epsilon}$ relative to $(\overline{W},\bar{B})$ -- indeed this is due to the pair $(p,q) \in \overline{W} \times \bar{B}$ with $p \in W_{R}$ and $q \in B(p)$ whose existence is guaranteed by the definition of $(\overline{W},\bar{B}) \in G_{R}$. Thus $R \in \mathcal{R}^{\delta}_{\sigma}(\overline{W},\bar{B})$ (defined above \eqref{form69}) whenever $(\overline{W},\bar{B}) \in G_{R}$. This implies that
\begin{displaymath} \int_{\Lambda'} |\mathcal{R}^{\delta}_{\sigma}(\overline{W},\bar{B})| \, d\tn(\overline{W},\bar{B}) = \sum_{R \in \mathcal{R}^{\delta}_{\sigma}} \mathbb{P}(\Lambda' \cap \{R \in \mathcal{R}^{\delta}_{\sigma}(\overline{W},\bar{B})\}) \geq \sum_{R \in \mathcal{R}^{\delta}_{\sigma}} \mathbb{P}(\Lambda' \cap G_{R}). \end{displaymath}
Finally, recall that $\mathbb{P}(\Lambda') \geq 1/2$ and $\mathbb{P}(G_{R}) \geq 1 - c$. In particular, if we choose $c < 1/4$ (and thus finally fix "$A$" sufficiently large), then the integral above is bounded from below by $|\mathcal{R}^{\delta}_{\sigma}|/4$. This proves the existence of $(\overline{W},\bar{B}) \in \Lambda'$ such that \eqref{form69} holds.

Finally, since every $R \in \mathcal{R}^{\delta}_{\sigma}(\overline{W},\bar{B}) =: \bar{\mathcal{R}}_{\sigma}^{\delta}$ has $\lambda$-restricted type $(\geq 1,\geq 1)_{\epsilon}$ relative to $(\overline{W},\bar{B})$, the first part of the proof implies
\begin{displaymath} |\mathcal{R}^{\delta}_{\sigma}| \leq 4|\bar{\mathcal{R}}^{\delta}_{\sigma}| \lesssim \delta^{-\eta} \left[ (|\overline{W}||\bar{B}|)^{3/4}(X_{\lambda}Y_{\lambda})^{1/2} + |\overline{W}|(X_{\lambda}Y_{\lambda}) + |\bar{B}|(X_{\lambda}Y_{\lambda}) \right]. \end{displaymath}
Since $(\overline{W},\bar{B}) \in \Lambda'$, we have $|\overline{W}| \leq 4A|W|/m$ and $|\bar{B}| \leq 4A|B|/n$. Noting that "$A$" is an absolute constant, the upper bound matches \eqref{form61}, and the proof is complete. \end{proof}


\section{Proof of Theorem \ref{thm2}}\label{s:mainInduction}

In this section we finally prove Theorem \ref{thm2}. In fact, we will prove a stronger statement concerning the partial multiplicity functions $m_{\delta,\lambda,t}$, see Theorem \ref{thm5} below. Theorem \ref{thm2} will finally be deduced from Theorem \ref{thm5} in Section \ref{s:thm2Proof}.

Recall Notation \ref{not3}. We will need the following slight generalisation, where the ranges of the "distance" and "tangency" parameters can be specified independently of each other.

\begin{definition}[$G_{\lambda,t}^{\rho_{\lambda},\rho_{t}}(\omega)$]\label{def:lambdaTNeighbourhood} Let $\delta \leq \lambda \leq t \leq 1$, and $G \subset \Omega = \{(p,v) : p \in P \text{ and } v \in E(p)\}$. For $\rho_{\lambda},\rho_{t} \geq 1$ and $\omega = (p,v) \in \Omega$, we write
\begin{displaymath} G^{\rho_{\lambda},\rho_{t}}_{\lambda,t}(\omega) := \{(p',v') \in G : \lambda/\rho_{\lambda} \leq \Delta(p,p') \leq \rho_{\lambda}\lambda \text{ and } t/\rho_{t} \leq |p - p'| \leq \rho_{t}t\}. \end{displaymath}
Similarly, for $Q \subset P \subset \mathbf{D}$, we will also write
\begin{displaymath} Q_{\lambda,t}^{\rho_{\lambda},\rho_{t}}(p) := \{q \in Q : \lambda/\rho_{\lambda} \leq \Delta(p,q) \leq \rho_{\lambda}\lambda \text{ and } t/\rho_{t} \leq |p - q| \leq \rho_{t}t\}. \end{displaymath}
Thus, the former notation concerns pairs, and the latter points. The correct interpretation should always be clear from the context (whether $G \subset \Omega$ or $Q \subset P$).

Whenever $\delta \leq \lambda \leq \delta \rho_{\lambda}$, we modify both definitions so that the two-sided condition $\lambda/\rho_{\lambda} \leq \Delta(p,q) \leq \rho_{\lambda}\lambda$ is replaced by the one-sided condition $\Delta(p,q) \leq \rho_{\lambda}\lambda$.
\end{definition}

\begin{notation}\label{not2} Thankfully, we can most often (not always) use the definitions in the cases $\rho_{\lambda} = \rho = \rho_{t}$. In this case, we abbreviate $G^{\rho_{\lambda},\rho_{t}}_{\lambda,t} =: G^{\rho}_{\lambda,t}$.
\end{notation}

\begin{definition}[$m_{\delta,\lambda,t}^{\rho_{\lambda},\rho_{t},C}$]\label{def:multFunction} Fix $0 < \delta \leq \lambda \leq t \leq 1$ and $\rho_{\lambda},\rho_{t} \geq 1$. Let $\Omega = \{(p,v) : p \in P \text{ and } v \in E(p)\}$ as usual, and write $\sigma := \delta/\sqrt{\lambda t}$. For any set $G \subset \Omega$, we define
\begin{displaymath} m_{\delta,\lambda,t}^{\rho_{\lambda},\rho_{t},C}(\omega \mid G) := |\{\omega' \in (G^{\delta}_{\sigma})_{\lambda,t}^{\rho_{\lambda},\rho_{t}} : CR^{\delta}_{\sigma}(\omega) \cap CR^{\delta}_{\sigma}(\omega') \neq \emptyset\}|, \qquad \omega \in G. \end{displaymath}
Here $G^{\delta}_{\sigma}$ is the $(\delta,\sigma)$-skeleton of $G$. \end{definition}

\begin{notation}\label{not:shorthand} Consistently with Notation \ref{not2}, in the case $\rho_{\lambda} = \rho = \rho_{t}$ we abbreviate
\begin{displaymath} m^{\rho_{\lambda},\rho_{t},C}_{\delta,\lambda,t} =: m^{\rho,C}_{\delta,\lambda,t}. \end{displaymath}
The full generality of the notation will only be needed much later, and we will remind the reader at that point. \end{notation}

\begin{thm}\label{thm5} For every $\kappa \in (0,\tfrac{1}{2}]$ and $s \in (0,1]$, there exist $\epsilon = \epsilon(\kappa,s) > 0$ and $\delta_{0} = \delta_{0}(\epsilon,\kappa,s) > 0$ such that the following holds for all $\delta \in (0,\delta_{0}]$. Let $\Omega = \{(p,v) : p \in P \text{ and } v \in E(p)\}$ be a $(\delta,\delta,s,\delta^{-\epsilon})$-configuration with
\begin{equation}\label{sizeP} |P| \leq \delta^{-s - \epsilon}, \end{equation}
Then, there exists a subset $G \subset \Omega$ of cardinality $|G| \geq \delta^{\kappa}|\Omega|$ such that the following holds simultaneously for all $\delta \leq \lambda \leq t \leq 1$:
\begin{equation}\label{form99} m_{\delta,\lambda,t}^{\delta^{-\epsilon},\kappa^{-1}}(\omega \mid G) \leq \delta^{-\kappa}, \qquad \omega \in G. \end{equation}
\end{thm}

Theorem \ref{thm2} will be easy to derive from Theorem \ref{thm5}. The details are in Section \ref{s:thm2Proof}. Theorem \ref{thm5} will be proven by a sequence of successive refinements to the initial configuration $\Omega$. Every refinement will take care of the inequality \eqref{form99} for one fixed pair $(\lambda,t)$, but the refinements will need to be performed in an appropriate order, as we will discuss later. After a large but finite number of such refinements, we will be able to check that \eqref{form99} holds for all $\delta \leq \lambda \leq t \leq 1$ simultaneously.

\begin{notation} Throughout this section, we allow the implicit constants in the "$\approx_{\delta}$" notation to depend on the constants $\kappa,s$ and $\epsilon = \epsilon(\kappa,s)$ in Theorem \ref{thm5} (the choice of $\epsilon$ is explained in Section \ref{s:constants}). Thus, the notation $A \lessapprox_{\delta} B$ means that $A \leq C(\log(1/\delta))^{C}B$, where $C = C(\epsilon,\kappa,s) > 0$. In particular, if $\delta > 0$ is small enough depending on $\epsilon,\kappa,s$, the inequality $A \lessapprox_{\delta} B$ implies $A \leq \delta^{-\epsilon}B$. \end{notation}

\subsection{Choice of constants}\label{s:constants} We explain how $\epsilon$ in Theorem \ref{thm5} depends on $\kappa,s$. Let $\epsilon_{\mathrm{max}} = \epsilon_{\mathrm{max}}(\kappa,s) > 0$ be an auxiliary constant, which (informally) satisfies $\epsilon \ll \epsilon_{\mathrm{max}} \ll \kappa$. Precisely, the constant $\epsilon_{\mathrm{max}}$ is determined by the following two requirements:
\begin{itemize}
\item Let $\mathbf{A}$ be the absolute constant from Theorem \ref{thm4}. We require $\epsilon_{\mathrm{max}}$ to be so small that if Theorem \ref{thm3} is applied with parameters $\bar{\kappa} = \kappa s/100$ and $s$, then $\mathbf{A}\epsilon_{\mathrm{max}} \leq \epsilon_{0}(\bar{\kappa},s)$, where the $\epsilon_{0}(\bar{\kappa},s)$ is the constant produced by Theorem \ref{thm3}.
\item We apply Theorem \ref{thm4} with constant $\eta = \kappa s/100$, and we require that $\epsilon_{\mathrm{max}} \leq \epsilon(\eta)$ (where $\epsilon(\eta)$ is the constant produced by Theorem \ref{thm4}).
\item We require that $\epsilon_{\mathrm{max}} < c\kappa s$ for a small absolute constant $c > 0$, whose size will be determined later.
\end{itemize}

The relationship between the "final" $\epsilon$ in Theorem \ref{thm5}, and the constant $\epsilon_{\mathrm{max}}$ fixed above, is the following, for a suitable absolute constant $C > 0$:
\begin{equation}\label{form133} C \cdot 10^{100/\kappa}\epsilon \leq \epsilon_{\mathrm{max}}. \end{equation}
As stated in Theorem \ref{thm5}, the threshold $\delta_{0} > 0$ may depend on all the parameters $\epsilon,\kappa,s$. We do not attempt to track the dependence explicitly, and often we will state inequalities of (e.g.) the form "$C \leq \delta^{-\epsilon}$" under the implicit assumption that $\delta > 0$ is small enough, depending on $\epsilon$. Here, we only explicitly record that $\delta_{0} > 0$ is taken so small that
\begin{equation}\label{form184a} CA(\epsilon,\kappa)^{C/\epsilon} \leq \delta_{0}^{-\epsilon_{\mathrm{max}}}, \end{equation}
where $A(\epsilon,\kappa) \geq 1$ is a constant depending only on $\kappa$, and $C \geq 1$ is absolute.

\subsection{The case $t \approx \lambda$}\label{s:largeLambda} In the "main" argument for Theorem \ref{thm5}, we will need to assume that $t \geq \delta^{-\kappa/10}\lambda$. The opposite case $t \leq \delta^{-\kappa/10}\lambda$ is elementary, and we handle it straight away. So, fix $\delta \leq \lambda \leq t \leq 1$ with $t \leq \delta^{-\kappa/10}\lambda$.
\begin{claim} There exists a $(\delta,\delta,s,4\delta^{-\epsilon})$-configuration $G \subset \Omega$ (depending on $\lambda,t$) of cardinality $|G| \geq |\Omega|/16$ such that \eqref{form99} holds with $\epsilon := \kappa/100$. \end{claim}
We record that our assumption $t \leq \delta^{-\kappa/10}\lambda$ implies
\begin{equation}\label{form191} \sigma = \delta/\sqrt{\lambda t} \geq \delta^{\kappa/20}(\delta/\lambda) \geq \delta^{\kappa/5}(\delta/\lambda). \end{equation}
To save a little space, we abbreviate $\mathbf{R}(p,v) := \kappa^{-1}R^{\delta}_{\sigma}(p,v)$. We also write $M := |E(p)|$ for the common cardinality of the sets $E(p)$, $p \in P$. With this notation, we estimate as follows (the final estimate will be justified carefully below the computation):
\begin{align} \frac{1}{|P|} \sum_{p \in P} & \frac{1}{M} \sum_{v \in E(p)} |\{(p',v') \in (\Omega^{\delta}_{\sigma})^{\delta^{-\epsilon}}_{\lambda,t}(p,v) : \mathbf{R}(p,v) \cap \mathbf{R}(p',v') \neq \emptyset\}| \notag\\
&\label{form159} \leq \frac{1}{|P|M} \sum_{p \in P} \sum_{p' \in P^{\delta^{-\epsilon}}_{\lambda,t}(p)} |\{(v,v') \in E(p) \times \mathcal{S}_{\sigma}(p') : \mathbf{R}(p,v) \cap \mathbf{R}(p',v') \neq \emptyset\}|\\
&\label{form128} \leq \delta^{s - \kappa/2}|P|.  \end{align}
We justify the final estimate. The easiest part is
\begin{equation}\label{form158} |P^{\delta^{-\epsilon}}_{\lambda,t}(p)| \leq |P \cap B(p,\delta^{-\epsilon}t)| \leq |P \cap B(p,\delta^{-\epsilon - \kappa/10}\lambda)| \leq \delta^{-2\epsilon - \kappa/10}\lambda^{s}|P|, \end{equation}
using the $(\delta,s,\delta^{-\epsilon})$-set property of $P$. A slightly more elaborate argument is needed to estimate the number of pairs $(v,v')$ appearing in \eqref{form159} for $(p,p')$ fixed. Fix $(p,p') \in P \times P$ with $p' \in P^{\delta^{-\epsilon}}_{\lambda,t}(p)$: thus $|p - p'| \geq \delta^{\epsilon}t \geq \delta^{\epsilon}\lambda$ and $\Delta(p,p') \geq \delta^{\epsilon}\lambda$. Lemma \ref{lemma5} implies that the intersection
\begin{equation}\label{form157} S^{\delta/\kappa}(p) \cap S^{\delta/\kappa}(p') \end{equation}
can be covered by boundedly many discs of radius
\begin{displaymath} \frac{\delta/\kappa}{\sqrt{(\Delta(p,p') + \delta/\kappa)(|p - p'| + \delta/\kappa)}} \leq \frac{\delta/\kappa}{\sqrt{(\delta^{\epsilon}\lambda)(\delta^{\epsilon} \lambda)}} \leq \delta^{-2\epsilon}(\delta/\lambda) =: r. \end{displaymath}
(Here we assumed that $\delta > 0$ is small enough in terms of $\epsilon,\kappa$.) Let $\{B(z_{i},r)\}_{i = 1}^{C}$ be an enumeration of these discs. Now, if $\mathbf{R}(p,v) \cap \mathbf{R}(p',v') \neq \emptyset$, then both $v,v'$ must lie at distance $\leq 2r$ from one of these discs (the intersection $\mathbf{R}(p,v) \cap \mathbf{R}(p',v')$ is contained in the intersection \eqref{form157}, and $\diam(\mathbf{R}) \leq \sigma/\kappa \leq \delta/(\lambda \kappa) \ll r$). On the other hand,
\begin{displaymath} |E(p) \cap B(z_{i},r)| \leq \delta^{-\epsilon}r^{s}M \leq \delta^{-3\epsilon}(\delta/\lambda)^{s}M \quad \text{and} \quad |\mathcal{S}_{\sigma}(p') \cap B(z_{i},4r)| \leq \delta^{-\kappa/4}, \end{displaymath}
where the first inequality used the $(\delta,s,\delta^{-\epsilon})$-set property of $E(p)$, and the second inequality used \eqref{form191}, along with the $\sigma$-separation of $\mathcal{S}_{\sigma}(p')$. This shows that
\begin{displaymath} |\{(v,v') \in E(p) \times \mathcal{S}_{\sigma}(p') : \mathbf{R}(p,v) \cap \mathbf{R}(p',v') \neq \emptyset\}| \leq \delta^{-4\epsilon - \kappa/4}(\delta/\lambda)^{s}M. \end{displaymath}
When this upper bound is plugged into \eqref{form159}, then combined with \eqref{form158}, we find \eqref{form128}.

To conclude the proof, notice that the left hand side of \eqref{form128} is in fact the expectation of the random variable
\begin{displaymath} \omega \mapsto m^{\delta^{-\epsilon},\kappa^{-1}}_{\delta,\lambda,t}(\omega \mid \Omega), \end{displaymath}
relative to normalised counting measure on $\Omega$. By Chebychev's inequality, there exists a set $G \subset \Omega$ with $|G| \geq \tfrac{1}{2}|\Omega|$ such that
\begin{displaymath} m_{\delta,\lambda,t}^{\delta^{-\epsilon},\kappa^{-1}}(\omega \mid G) \leq m_{\delta,\lambda,t}^{\delta^{-\epsilon},\kappa^{-1}}(\omega \mid \Omega) \leq \delta^{s - 3\kappa/4}|P| \leq \delta^{-\kappa}, \qquad \omega \in G, \end{displaymath}
using the assumption \eqref{sizeP} that $|P| \leq \delta^{-s - \epsilon}$ in the final inequality. Finally, we replace "$G$" by a slightly smaller $(\delta,\delta,s,4\delta^{-\epsilon})$-configuration by applying Lemma \ref{refinement} with $c = \tfrac{1}{2}$.

\subsection{Uniform sets} We start preparing for the proof of Theorem \ref{thm5} (the case of pairs $(\lambda,t)$ with $t \geq \delta^{-\kappa/10}\lambda$) with a few auxiliary definitions and results which allow us to find -- somewhat -- regular subsets inside arbitrary finite sets $P \subset \mathbf{D}$.

 \begin{definition}\label{def:uniformity} Let $n \geq 1$, and let
\begin{displaymath} \delta = \Delta_{n} < \Delta_{n - 1} < \ldots < \Delta_{1} \leq \Delta_{0} = 1 \end{displaymath}
be a sequence of dyadic scales.  We say that a set $P\subset \mathbf{D}$ is \emph{$\{\Delta_j\}_{j=1}^n$-uniform} if there is a sequence $\{N_j\}_{j=1}^n$ such that $|\mathcal{D}_{\Delta_{j}}(P \cap \mathbf{p})| = |P\cap \mathbf{p}|_{\Delta_{j}} = N_j$ for all $j\in \{1,\ldots,n\}$ and all $\mathbf{p} \in \mathcal{D}_{\Delta_{j - 1}}(P)$. As usual, we extend this definition to $P \subset\mathcal{D}_{\delta}$ (by applying it to $\cup P$).
\end{definition}

The following lemma allows us to find $\{\Delta_{j}\}_{j = 1}^{n}$-uniform subsets inside general finite sets. The result is a special case of \cite[Lemma 7.3]{2021arXiv210603338O}, which works for more general sequences $\{\Delta_{j}\}_{j = 1}^{m}$ than the sequence $\{2^{-jT}\}_{j = 1}^{m}$ treated in Lemma \ref{l:uniformization}.

\begin{lemma}\label{l:uniformization}
Let $P\subset \mathbf{D}$, $m,T \in \N$, and $\delta := 2^{-mT}$. Let also $\Delta_{j} := 2^{-jT}$ for $0 \leq j \leq m$, so in particular $\delta = \Delta_{m}$. Then, there there is a $\{\Delta_j\}_{j=1}^{m}$-uniform set $P'\subset P$ such that
\begin{equation}\label{form8}
|P'|_\delta \ge  \left(4T \right)^{-m} |P|_\delta. \end{equation}
In particular, if $\epsilon > 0$ and $T^{-1}\log (4T) \leq \epsilon$, then $|P'|_{\delta} \geq \delta^{\epsilon}|P|_{\delta}$.
\end{lemma}

\begin{proof} The inequality \eqref{form8} follows by inspecting the short proof of \cite[Lemma 7.3]{2021arXiv210603338O}. The "in particular" claim follows by noting that
\begin{displaymath} (4T)^{-m} = 2^{-m \log (4T)} = 2^{-mT \cdot (T^{-1}\log (4T))} = \delta^{T^{-1} \log(4T)}. \end{displaymath}
This completes the proof. \end{proof}

\subsection{Initial regularisation for the proof of Theorem \ref{thm5}}\label{s:regularisation} We denote the "given" $(\delta,\delta,s,\delta^{-\epsilon})$-configuration in Theorem \ref{thm5} by
\begin{displaymath} \Omega_{0} = \{(p,v) : p \in P_{0} \text{ and } v \in E_{0}(p)\}, \end{displaymath}
where $P_{0} \subset \mathcal{D}_{\delta}$ is a non-empty $(\delta,s,\delta^{-\epsilon})$-set, and $E_{0}(p) \subset \mathcal{S}_{\delta}(p)$ is a $(\delta,s,\delta^{-\epsilon})$-set of cardinality $M \geq 1$ (for every $p \in P_{0}$). The purpose of this section is to perform an initial pruning to $\Omega_{0}$, that is, to find a $(\delta,\delta,s,\delta^{-2\epsilon})$-configuration
\begin{displaymath} \Omega = \{(p,v) : p \in P \text{ and } v \in E(p)\} \subset \Omega_{0}, \end{displaymath}
where $P$ is a $(\delta,s,\delta^{-2\epsilon})$-set, each $E(p)$ is a $(\delta,s,\delta^{-2\epsilon})$-set with constant cardinality, and $|\Omega| \approx_{\delta} |\Omega_{0}|$. The subset $\Omega$ will have additional useful regularity properties compared to $\Omega_{0}$. After we are finished constructing $\Omega$, we will focus on finding the "final" set $G$ (as in Theorem \ref{thm5}) inside $\Omega$, instead of $\Omega_{0}$.

 There is no loss of generality in assuming that $\delta = 2^{-mT}$ for some $m \geq 1$, and some $T \geq 1$ whose size depends on $\epsilon$ (and therefore eventually $\kappa$). We start by applying Lemma \ref{l:uniformization} to the sequence
\begin{displaymath} \lambda_{j} := 2^{-jT}, \qquad 0 \leq j \leq m. \end{displaymath}
Provided that $T^{-1}\log(4T) \leq \epsilon$, the result is a $\{\lambda_{j}\}_{j = 1}^{m}$-uniform subset $P_{0}' \subset P_{0}$ with cardinality $|P_{0}'| \geq \delta^{\epsilon}|P_{0}|$. In particular, $P_{0}'$ is a $(\delta,s,\delta^{-2\epsilon})$-set. We define $\Omega_{0}' := \{(p,v) : p \in P_{0}' \text{ and } v \in E_{0}(p)\}$. Then $|\Omega_{0}'| \geq \delta^{\epsilon}|\Omega_{0}|$. From this point on, the proof will see no difference between $\Omega_{0},P_{0}$ and $\Omega_{0}',P_{0}'$, so we assume that $P_{0} = P_{0}'$ and $\Omega_{0} = \Omega_{0}'$ to begin with -- or in other words that $P_{0}$ is $\{\lambda_{j}\}_{j = 1}^{m}$-uniform for $\lambda_{j} = 2^{-jT}$, $1 \leq j \leq m$. In particular, the "branching numbers"
\begin{displaymath} N_{j} := |P_{0} \cap \mathbf{p}|_{\lambda_{j}}, \qquad \mathbf{p} \in \mathcal{D}_{2^{T}\lambda_{j}}(P_{0}), \, 1 \leq j \leq m, \end{displaymath}
are well-defined (that is, independent of "$\mathbf{p}$").

We have slightly overshot our target: the argument above shows that $P_{0}$ may be assumed to be $\{2^{-jT}\}_{j = 1}^{m}$-uniform. We only need something weaker. Let $\epsilon > 0$ be so small that the requirement \eqref{form133} is met. Let $\Lambda \subset [\delta,1]$ be a finite set of cardinality $|\Lambda| \sim 1/\epsilon$ which is \emph{multiplicatively $\delta^{-\epsilon/2}$-dense} in the following sense: if $\lambda \in [\delta,1]$ is arbitrary, then there exists $\underline{\lambda} \in \Lambda$ with $\underline{\lambda} \leq \lambda \leq \delta^{-\epsilon/2}\underline{\lambda}$. If $\delta > 0$ is so small that $2^{T} \leq \delta^{-\epsilon}$, we may (and will) choose $\Lambda \subset \{2^{-jT}\}_{j = 1}^{m} = \{\lambda_{j}\}_{j = 1}^{m}$. We agree that $\{\delta,1\} \in \Lambda$, and for every $\lambda \in \Lambda \, \setminus \, \{1\}$, we denote by $\hat{\lambda} \in \Lambda$ the smallest element of $\Lambda$ with $\hat{\lambda} > \lambda$.

Since $\Lambda \subset \{2^{-jT}\}_{j = 1}^{m}$, the set $P_{0}$ is automatically $\Lambda$-uniform: the number
\begin{equation}\label{form100} N_{\lambda} := |P_{0} \cap \mathbf{p}|_{\lambda}, \qquad \mathbf{p} \in \mathcal{D}_{\hat{\lambda}}(P_{0}), \, \lambda \in \Lambda \, \setminus \, \{1\}, \end{equation}
is independent of the choice of $\mathbf{p} \in \mathcal{D}_{\hat{\lambda}}(P_{0})$. From this point on, the uniformity with respect to the denser sequence $\{2^{-jT}\}_{j = 1}^{m}$ will no longer be required. From \eqref{form100}, it follows that also the number
\begin{equation}\label{form102} X_{\lambda} := |P_{0} \cap \mathbf{p}|_{\delta} = |P_{0}|/|P_{0}|_{\lambda}, \qquad \mathbf{p} \in \mathcal{D}_{\lambda}(P_{0}), \, \lambda \in \Lambda, \end{equation}
is independent of the choice of $\mathbf{p} \in \mathcal{D}_{\lambda}(P_{0})$ (since $X_{\lambda}$ is the product of the numbers $N_{\lambda'}$ for $\lambda' \in \Lambda$ with $\lambda' < \lambda$, recalling that $\delta \in \Lambda$ by definition).

Next, for every $\lambda \in \Lambda$ fixed, we associate a finite set $\mathcal{T}(\lambda) \subset [\lambda,1]$ of cardinality $|\mathcal{T}(\lambda)| \sim 1/\epsilon$ which is multiplicatively $\delta^{-\epsilon/2}$-dense on the interval $[\lambda,1]$ in the same sense as above: if $t \in [\lambda,1]$ is arbitrary, then there exists $\underline{t} \in \mathcal{T}(\lambda)$ such that $\underline{t} \leq t \leq \delta^{-\epsilon/2}\underline{t}$. For later technical convenience, it will be useful to know that the sets
\begin{equation}\label{form153} \Lambda(t) := \{\lambda \in \Lambda : t \in \mathcal{T}(\lambda)\}, \qquad t \in \mathcal{T} := \bigcup_{\lambda \in \Lambda} \mathcal{T}(\lambda), \end{equation}
are multiplicatively $\delta^{-\epsilon/2}$-dense in $[\delta,t]$. This can be accomplished by choosing both the $\lambda$'s and the $t$'s from some "fixed" multiplicatively $\delta^{-\epsilon/2}$-dense sequence in $[\delta,1]$, for example $\{\delta,\delta^{1 - \epsilon/2},\delta^{-\epsilon},\ldots,1\}$.

We order the pairs $(\lambda,t)$ with $\lambda \in \Lambda$ and $t \in \mathcal{T}(\lambda)$ arbitrarily. The total number of pairs is $\lesssim \epsilon^{-2}$. Then, we apply Theorem \ref{thm3} with constant $\kappa s/100$ to the first pair $(\lambda_{1},t_{1})$. If $\epsilon_{\mathrm{max}} > 0$ is sufficiently small (as small as we stated in Section \ref{s:constants}), and since $\epsilon \leq \epsilon_{\mathrm{max}} \leq \mathbf{A}\epsilon_{\mathrm{max}} \leq \epsilon_{0}(\bar{\kappa},s)$, Theorem \ref{thm3} provides us with a $(\delta,\delta,s,C\delta^{-\epsilon})$-configuration $G \subset \Omega_{0}$ such that $C \approx_{\delta} 1$, $|G| \approx_{\delta} |\Omega_{0}|$, and
\begin{equation}\label{form101} m_{\lambda_{1},\lambda_{1},t_{1}}^{\delta^{-\mathbf{A}\epsilon_{\mathrm{max}}},\delta^{-\mathbf{A}\epsilon_{\mathrm{max}}}}(\omega \mid G) \leq \delta^{-\kappa s/100}\lambda_{1}^{s}|P_{0}|_{\lambda_{1}}, \qquad \omega \in G^{\lambda_{1}}_{\Sigma_{1}}, \end{equation}
where $\Sigma_{1} = \sqrt{\lambda_{1}/t_{1}}$, and $\mathbf{A} \geq 1$ is the constant from Theorem \ref{thm4}.

Assume that we have already found a sequence of $(\delta,\delta,s,C_{j}\delta^{-\epsilon})$-configurations $G =: G_{1} \supset G_{2} \supset \ldots G_{j}$, where $C_{j} \approx_{\delta,j} 1$ and $|G_{j}| \approx_{\delta,j} |\Omega_{0}|$, and \eqref{form101} holds for $G_{j}$ relative to the pair $(\lambda_{j},t_{j})$ (with $\Sigma_{j} = \sqrt{\lambda_{j}/t_{j}}$). We reapply Theorem \ref{thm3} to $\Omega_{j} := G_{j}$, and the pair $(\lambda_{j + 1},t_{j + 1})$. This is legitimate, since $j \lesssim \epsilon^{-2}$, and the constant $C_{j}\delta^{-\epsilon}$ is smaller than the threshold $\delta^{-\epsilon_{\mathrm{max}}}$ required to apply Theorem \ref{thm3} with constant "$\kappa s/100$" (by our choice of "$\epsilon$"). Thus, Theorem \ref{thm3} outputs a $(\delta,\delta,s,C_{j + 1}\delta^{-\epsilon})$-configuration $G_{j + 1} \subset G_{j}$ satisfying \eqref{form101} for the pair $(\lambda_{j + 1},t_{j + 1})$, and with $|G_{j + 1}| \approx_{\delta,j + 1} |\Omega_{0}|$.

After Theorem \ref{thm3} has been applied in this "successive" manner to all the pairs $(\lambda,t)$ with $\lambda \in \Lambda$ and $t \in \mathcal{T}(\lambda)$,  we arrive at a final $(\delta,\delta,s,C_{\epsilon}\delta^{-\epsilon})$-configuration
\begin{equation}\label{form220} \Omega = \{(p,v) : p \in P \text{ and } v \in E(p)\}, \end{equation}
where $C_{\epsilon} \approx_{\delta} 1$, $|\Omega| \approx_{\delta} |\Omega_{0}|$, and $\Omega$ satisfies simultaneously a version of \eqref{form101} for all the pairs $(\lambda_{j},t_{j})$. In particular, we note that $|P| \approx_{\delta} |P_{0}|$ and $|E(p)| \approx_{\delta} M$ for all $p \in P$. Therefore, $P,E(p)$ remain $(\delta,s,C_{\epsilon}\delta^{-\epsilon})$-sets with $C_{\epsilon} \approx_{\delta} 1$.

\begin{remark} It is worth comparing the accomplishment \eqref{form101} with the ultimate goal \eqref{form99} in Theorem \ref{thm5}. Roughly speaking, we have now tackled the cases $(\lambda,\lambda,t)$ of \eqref{form99} (with the caveat that this has only been done for the pairs $(\lambda,t)$ with $\lambda \in \Lambda$ and $t \in \mathcal{T}(\lambda)$). \end{remark}

\subsection{Proof of Theorem \ref{thm5}}\label{s:final} We just finished constructing the $(\delta,\delta,s,C_{\epsilon}\delta^{-\epsilon})$-configuration $\Omega = \{(p,v) : p \in P \text{ and } v \in E(p)\} \subset \Omega_{0}$ with $|\Omega| \approx_{\delta} |\Omega_{0}|$ which satisfies property \eqref{form101} (with $G = \Omega$) for all $\lambda \in \Lambda$ and $t \in \mathcal{T}(\lambda)$. We record this once more:
\begin{equation}\label{form130} m^{\delta^{-\mathbf{A}\epsilon_{\mathrm{max}}},\delta^{-\mathbf{A}\epsilon_{\mathrm{max}}}}_{\lambda,\lambda,t}(\omega \mid \Omega) \leq \delta^{-\kappa s/100}\lambda^{s}|P_{0}|_{\lambda}, \quad \omega \in \Omega^{\lambda}_{\Sigma}, \end{equation}
for every $\lambda \in \Lambda$ and $t \in \mathcal{T}(\lambda)$, where $\Sigma = \sqrt{\lambda/t}$.
\begin{remark} At this point, we remind the reader that the left hand side of \eqref{form130} is shorthand notation for
\begin{displaymath} m_{\lambda,\lambda,t}^{\delta^{-\mathbf{A}\epsilon_{\mathrm{max}}},\delta^{-\mathbf{A}\epsilon_{\mathrm{max}}},\delta^{-\mathbf{A}\epsilon_{\mathrm{max}}}}(\omega \mid \Omega),\end{displaymath}
recall Notation \ref{not:shorthand}. Soon we will need the full generality of the notation $m^{\rho_{\lambda},\rho_{t},C}_{\delta,\lambda,t}$. \end{remark}

The main step towards proving Theorem \ref{thm5} for every pair $(\lambda,t)$ with $\delta \leq \lambda \leq t \leq 1$ is to prove it for the (finitely many) pairs $(\lambda,t)$ with $\lambda \in \Lambda$ and $t \in \mathcal{T}(\lambda)$. Write
\begin{displaymath} \mathcal{T} := \bigcup_{\lambda \in \Lambda} \mathcal{T}(\lambda) \subset [\delta,1], \end{displaymath}
and for every $t \in \mathcal{T}$, let $\Lambda(t) := \{\lambda \in \Lambda : t \in \mathcal{T}(\lambda)\} \subset [\delta,t]$. Recall from (around) \eqref{form153}  that $\Lambda(t)$ is multiplicatively $\delta^{-\epsilon/2}$-dense in $[\delta,t]$. This will be used in the form of the corollary that $\Lambda(t)$ is multiplicatively $\delta^{-\epsilon/2}$-dense in $[\delta,\max \Lambda(t)]$.

\begin{proposition}\label{prop8} For every fixed $t \in \mathcal{T}$ and $\lambda \in \Lambda(t)$, there exists a $(\delta,\delta,s,C_{\epsilon}\delta^{-\epsilon})$-configuration $G \subset \Omega$ (depending on $\lambda,t$), such that $|G| \approx_{\delta} |\Omega|$, and
\begin{equation}\label{form110} m_{\delta,\lambda,t}^{\delta^{-\epsilon},C\kappa^{-1}}(\omega \mid G) \leq \delta^{- \kappa}, \qquad \omega \in G, \end{equation}
where $C > 0$ is an absolute constant to be determined in the proof of Proposition \ref{prop7}.\end{proposition}
We will prove Proposition \ref{prop8} in such a way that the various configurations "$G$" will form a nested sequence. So, once the proposition has been established for all pairs $(\lambda,t)$ with $t \in \mathcal{T}$ and $\lambda \in \Lambda(t)$, then the "last" set $G$ will satisfy \eqref{form110} for all pairs $t \in \mathcal{T}$ and $\lambda \in \Lambda(t)$ simultaneously. We start with the easiest cases where $\lambda \approx t$. The value of the constant "$C_{\epsilon}$" will change many times during the proof, but it will always remain $C_{\epsilon} \approx_{\delta} 1$.

\subsubsection*{Pairs $(\lambda,t)$ with $t \leq \delta^{-\kappa/10}\lambda$} Let $\lambda \in \Lambda$ and $t \in \mathcal{T}$ with $t \leq \delta^{-\kappa/10}\lambda$. In this case we apply the claim proved in Section \ref{s:largeLambda}: the conclusion is that there exists a $(\delta,\delta,s,4C_{\epsilon}\delta^{-\epsilon})$-configuration $G_{1} \subset \Omega$ satisfying \eqref{form110} for the fixed pair $(\lambda,t)$. (To be perfectly accurate, one needs to apply the proof of the claim with constant $C\kappa$ in place of $\kappa$.) Next, we simply repeat the argument inside $G_{1}$, and for all the pairs $(\lambda,t) \in \Lambda \times \mathcal{T}$ with $t \leq \delta^{-\kappa/10}\lambda$, in arbitrary order. This involves refining $\Omega$ at most $\lesssim \epsilon^{-2}$ times, so the final product of this argument remains a $(\delta,\delta,s,C_{\epsilon}\delta^{-\epsilon})$-configuration.

Before launching to the main argument -- treating the cases $\lambda \leq \delta^{\kappa/10}t$ -- we use \eqref{form110} to complete the proof of Theorem \ref{thm5}.

\begin{proposition}\label{prop7} Assume that \eqref{form110} holds for simultaneously for all $(\lambda,t) \in \Lambda \times \mathcal{T}$. Then, if the absolute constant $C > 0$ is large enough, we have
\begin{equation}\label{form132} m^{\delta^{-\epsilon/2},\kappa^{-1}}_{\delta,\lambda,t}(\omega \mid G) \leq \delta^{-2\kappa}, \qquad \omega \in G \end{equation}
simultaneously for all $\delta \leq \lambda \leq t \leq 1$ (not necessarily from $\Lambda \times \mathcal{T}$).  \end{proposition}

\begin{proof} Let $\delta \leq \lambda \leq t \leq 1$. Let $\underline{\lambda} \in \Lambda$ and $\underline{t} \in \mathcal{T}(\lambda)$ be elements with $\underline{\lambda} \leq \lambda \leq \delta^{-\epsilon/2}\underline{\lambda}$ and $\underline{t} \leq t \leq \delta^{-\epsilon/2}\underline{t}$. Recall that
\begin{displaymath} m^{\delta^{-\epsilon/2},\kappa^{-1}}_{\delta,\lambda,t}(\omega \mid G) = |\{\omega' \in (G^{\delta}_{\sigma})_{\lambda,t}^{\delta^{-\epsilon/2}}(\omega) : \kappa^{-1}R^{\delta}_{\sigma}(\omega) \cap \kappa^{-1}R^{\delta}_{\sigma}(\omega') \neq \emptyset\}|, \qquad \omega \in G, \end{displaymath}
where $G_{\sigma}^{\delta}$ is the $(\delta,\sigma)$-skeleton of $G$ (with $\sigma = \delta/\sqrt{\lambda t}$). An unpleasant technicality is that $\bar{\sigma} = \delta/\sqrt{\underline{\lambda}\underline{t}} \in [\sigma,\delta^{-\epsilon/2}\sigma]$ might be a little different from $\sigma$, so elements of $G^{\delta}_{\sigma}$ are not automatically elements of $G^{\delta}_{\bar{\sigma}}$. However, for every $\omega' = (q,w) \in G^{\delta}_{\sigma}$, we may pick $\bar{\omega}' = (q,\bar{w}) \in G^{\delta}_{\bar{\sigma}}$ with $(q,w) \prec (q,\bar{w})$, and in particular $|w - \bar{w}| \leq C\bar{\sigma}$ for an absolute constant $C \geq 1$. Then, it is straightforward to check that
\begin{equation}\label{form164} \omega' \in (G^{\delta}_{\sigma})^{\delta^{-\epsilon/2}}_{\lambda,t}(\omega) \quad \Longrightarrow \quad \bar{\omega}' \in (G^{\delta}_{\bar{\sigma}})^{\delta^{-\epsilon}}_{\underline{\lambda},\underline{t}}(\omega), \qquad \omega \in G, \end{equation}
and
\begin{equation}\label{form165} \kappa^{-1}R^{\delta}_{\sigma}(\omega) \cap \kappa^{-1}R^{\delta}_{\sigma}(\omega') \neq \emptyset \quad \Longrightarrow \quad C\kappa^{-1}R^{\delta}_{\bar{\sigma}}(\omega) \cap C\kappa^{-1}R^{\delta}_{\bar{\sigma}}(\bar{\omega}') \neq \emptyset. \end{equation}
The implication \eqref{form165} follows from the inclusion $\kappa^{-1}R^{\delta}_{\sigma}(\omega') \subset C\kappa^{-1}R^{\delta}_{\bar{\sigma}}(\bar{\omega}')$ (note that $\bar{\sigma} \geq \sigma$). Regarding \eqref{form164}, it is worth noting that the implication is even true in the special case $\lambda \leq \delta^{1 - \epsilon/2}$ (recall Definition \ref{def:lambdaTNeighbourhood}) since in that case $\underline{\lambda} \leq \delta^{1 - \epsilon}$.

Finally, observe that the map $\omega' \mapsto \bar{\omega}'$ is at most $\delta^{-\epsilon}$-to-$1$: if $(q,w_{1}),(q,w_{2}),\ldots,(q,w_{N}) \in G^{\delta}_{\sigma}$ are distinct, and $\bar{\omega}' = (q,\bar{w})$ is the image of them all, then $|w_{i} - w_{j}| \gtrsim N\sigma$ for some $1 \leq i \neq j \leq N$, and on the other hand $\max\{|\bar{w} - w_{i}|,|\bar{w} - w_{j}|\} \lesssim \bar{\sigma} \leq \delta^{-\epsilon/2}\sigma$.

Combining this with \eqref{form164}-\eqref{form165}, we find
\begin{displaymath} m^{\delta^{-\epsilon/2},\kappa^{-1}}_{\delta,\lambda,t}(\omega \mid G) \leq \delta^{-\epsilon} m^{\delta^{-\epsilon},C\kappa^{-1}}_{\delta,\underline{\lambda},\underline{t}}(\omega \mid G) \stackrel{\eqref{form110}}{\leq} \delta^{-\kappa - \epsilon}, \qquad \omega \in G. \end{displaymath}
This proves \eqref{form132}, since $\epsilon \leq \kappa$ (by the choices in Section \ref{s:constants}). \end{proof}

\subsection{Proof of Proposition \ref{prop8}} We then arrive at the core of the proof of Theorem \ref{thm5}.

\subsubsection{Structure of the proof Proposition \ref{prop8}}\label{s:structure} Very much like in Section \ref{s:regularisation}, we will enumerate the pairs $(\lambda,t)$ with $t \in \mathcal{T}$, and $\lambda \in \Lambda(t) \cap [\delta,\delta^{\kappa/10}t]$, and we will construct a decreasing sequence of $(\delta,\delta,s)$-configurations $G_{1} \supset G_{2} \supset \ldots$ such that $G_{j}$ satisfies \eqref{form110} for the pair $(\lambda_{j},t_{j})$ -- and therefore automatically for all pairs $(\lambda_{i},t_{i})$ with $1 \leq i \leq j$. We will show inductively that $|G_{j}| \approx_{\delta} |\Omega|$.

In contrast to Section \ref{s:regularisation}, this time the ordering of the pairs $(\lambda_{j},t_{j})$ matters. We will do this as follows. We enumerate the elements of $\mathcal{T}$ arbitrarily. Then, if $t_{j} \in \mathcal{T}$ is fixed, we enumerate the pairs $(\lambda,t_{j})$ with $\lambda \in \Lambda(t_{j}) \cap [\delta,\delta^{\kappa/10}t_{j}]$ in increasing order. Thus, the first pair is $(\delta,t_{j})$, the second one $(\delta^{1 - \epsilon/2},t_{j})$, and so on. This has the crucial benefit that when we are in the process of proving \eqref{form110} for a fixed pair $(\lambda,t_{j})$, we may already assume that (the current) $G$ satisfies \eqref{form110} for all pairs $(\lambda',t_{j})$ with $\lambda' \in \Lambda(t_{j})$ and $\lambda' < \lambda$.

\subsubsection{Setting up the induction} We will then begin to implement the strategy outlined above. Fix $t := t_{j} \in \mathcal{T}$ arbitrarily, and for the remainder of the proof. We enumerate $\Lambda(t) \cap [\delta,\delta^{\kappa/10}t]$ in increasing order, with the abbreviation $|\Lambda| := |\Lambda(t) \cap [\delta,\delta^{\kappa/10}t]|$:
\begin{equation}\label{form170} \delta = \lambda_{1} < \lambda_{2} < \ldots < \lambda_{|\Lambda|} \leq \delta^{\kappa/10}t. \end{equation}
For each index $1 \leq l \leq |\Lambda|$ we also define a constant $C_{l} \geq 1$ in such a way that the sequence $C_{1} > C_{2} > \ldots > C_{|\Lambda|} \geq 1$ is very rapidly decreasing, more precisely
\begin{equation}\label{form129} C_{l + 1} = A(\epsilon,\kappa)^{-1}C_{l}, \qquad 1 \leq l < |\Lambda|
\end{equation}
for a suitable constant $A(\epsilon,\kappa) \geq 1$, depending only on $\kappa$, and to be determined later, precisely right after \eqref{form138}. To complete the definition of the sequence $\{C_{l}\}$, we specify its smallest (last) element:
\begin{equation}\label{form183} C_{|\Lambda|} := C\kappa^{-1}, \end{equation}
where $\kappa > 0$ is the parameter given in Theorem \ref{thm5}, and $C > 0$ is the absolute constant from \eqref{form110}. With these definitions, and noting that $|\Lambda| \leq C/\epsilon$ for an absolute constant $C > 0$, we have
\begin{equation}\label{form184} C_{1} = A(\epsilon,\kappa)^{|\Lambda|}C_{|\Lambda|} \leq CA(\epsilon,\kappa)^{C/\epsilon}\kappa^{-1} \stackrel{\eqref{form184a}}{\leq} \delta^{-\epsilon_{\mathrm{max}}}, \qquad \delta \in (0,\delta_{0}]. \end{equation}
We will prove the following by induction on $k \in \{1,\ldots,|\Lambda|\}$: there exists a decreasing sequence of $(\delta,\delta,s,C_{\epsilon}\delta^{-\epsilon})$-configurations $G_{1} \supset \ldots \supset G_{k}$ such that $|G_{l}| \approx_{\delta} |\Omega|$ for all $1 \leq l \leq k$, and such that the following slightly stronger version of \eqref{form110} holds:
\begin{equation}\label{form111} m^{C_{l}\delta^{-\epsilon},C_{l}}_{\delta,\lambda_{l},t}(\omega \mid G_{l}) \leq \delta^{- \kappa}, \qquad \omega \in G_{l}, \, 1 \leq l \leq k. \end{equation}
Once we have accomplished this for $k = |\Lambda|$, we set $G := G_{|\Lambda|}$. Then \eqref{form110} holds for $G$ (by \eqref{form183}), and for all pairs $(\lambda,t)$ with $\lambda \in \Lambda(t)$. After this, we may repeat the same procedure for all $t \in \mathcal{T}$ in arbitrary order (but always working inside the configurations we have previously constructed). This will complete the proof of Proposition \ref{prop8}.

\begin{remark} Notice that the constants "$C_{l}$" in \eqref{form111} decrease (rapidly) as $l$ increases. The idea is that we can prove \eqref{form111} with index "$k + 1$" and the smaller constant $C_{k + 1}$, provided that we already have \eqref{form111} for all $1 \leq l \leq k$, and the much larger constants $C_{l} \gg C_{k + 1}$. \end{remark}

\subsubsection{The case $k = 1$} This case is a consequence of \eqref{form130} applied with $\lambda = \lambda_{1} = \delta$, with $G_{1} := \Omega$. Note that in this case $\sigma = \delta/\sqrt{\lambda t} = \sqrt{\lambda/t} = \Sigma$, so \eqref{form130} with $\lambda = \delta$ (and our fixed $t \in \mathcal{T}$) can be rewritten as
\begin{equation}\label{form208} m^{\delta^{-\mathbf{A}\epsilon_{\mathrm{max}}},\delta^{-\mathbf{A}\epsilon_{\mathrm{max}}}}_{\delta,\delta,t}(\omega \mid \Omega) \leq \delta^{-s\kappa/100}\lambda^{s}|P_{0}|_{\delta} \leq \delta^{-\kappa}, \qquad \omega \in \Omega^{\delta}_{\Sigma}. \end{equation}
This is actually much stronger than what we need in \eqref{form111}, since $\epsilon < \epsilon_{\mathrm{max}}$, and $C_{1} \sim_{\epsilon,\kappa} 1$. One small point of concern is that \eqref{form111} is a statement about $\omega \in G_{1} = \Omega$, whereas \eqref{form208} deals with $\omega \in \Omega^{\delta}_{\Sigma} = \Omega^{\delta}_{\sigma}$. This is not a problem thanks to the following elementary lemma, which will also be useful later:
\begin{lemma}\label{lemma10} Let $0 < \delta \leq \lambda \leq t \leq 1$ and $\rho_{\lambda},\rho_{t},C \geq 1$. Let $G \subset \Omega$ and $\omega \in G$. Let $\bar{\omega} \in G^{\delta}_{\sigma}$ be the parent of $\omega$ in the $(\delta,\sigma)$-skeleton $G^{\delta}_{\sigma}$, where $\sigma = \delta/\sqrt{\lambda t}$ as usual. Then,
\begin{equation}\label{form209} m_{\delta,\lambda,t}^{\rho_{\lambda},\rho_{t},C/A}(\bar{\omega} \mid G) \leq m_{\delta,\lambda,t}^{\rho_{\lambda},\rho_{t},C}(\omega \mid G) \leq m_{\delta,\lambda,t}^{\rho_{\lambda},\rho_{t},AC}(\bar{\omega} \mid G),\end{equation}
where $A \geq 1$ is absolute.  \end{lemma}

In particular, \eqref{form208} for $\omega \in \Omega^{\delta}_{\Sigma}$ implies \eqref{form111} for all $\omega \in \Omega$, at the cost of replacing the second $\delta^{-\mathbf{A}\epsilon_{\mathrm{max}}}$ by $\delta^{-\mathbf{A}\epsilon_{\mathrm{max}}}/A$ (which is still much bigger than $C_{1} \sim_{\epsilon,\kappa} 1$).

\begin{proof}[Proof of Lemma \ref{lemma10}] We only prove the upper bound, since the lower bound is established in a similar fashion. Let us spell out the quantities in \eqref{form209}:
\begin{displaymath} m_{\delta,\lambda,t}^{\rho_{\lambda},\rho_{t},C}(\omega \mid G) = |\{\omega' \in (G^{\delta}_{\sigma})^{\rho_{\lambda},\rho_{t}}_{\lambda,t}(\omega) : CR^{\delta}_{\sigma}(\omega) \cap CR^{\delta}_{\sigma}(\omega') \neq \emptyset\}| \end{displaymath}
and
\begin{displaymath} m_{\delta,\lambda,t}^{\rho_{\lambda},\rho_{t},AC}(\bar{\omega} \mid G) = |\{\omega' \in (G^{\delta}_{\sigma})^{\rho_{\lambda},\rho_{t}}_{\lambda,t}(\bar{\omega}) : ACR^{\delta}_{\sigma}(\bar{\omega}) \cap ACR^{\delta}_{\sigma}(\omega') \neq \emptyset\}|. \end{displaymath}
The crucial observation is that if the point $\omega \in G$ is written as $\omega = (p,v)$, then the parent $\bar{\omega} = (p,\mathbf{v})$, where $|v - \mathbf{v}| \lesssim 1$, and the "$p$-component" remains unchanged. In particular,
\begin{displaymath} \omega' \in (G^{\delta}_{\sigma})^{\rho_{\lambda},\rho_{t}}_{\lambda,t}(\omega) \quad \Longleftrightarrow \quad \omega' \in (G^{\delta}_{\sigma})^{\rho_{\lambda},\rho_{t}}_{\lambda,t}(\bar{\omega}), \end{displaymath}
since these inclusions only concern the $p$-components of $\omega,\omega',\bar{\omega}$. Therefore, \eqref{form209} boils down to the observation
\begin{displaymath} CR^{\delta}_{\sigma}(\omega) \cap CR^{\delta}_{\sigma}(\omega') \neq \emptyset \quad \Longrightarrow \quad ACR^{\delta}_{\sigma}(\bar{\omega}) \cap CR^{\delta}_{\sigma}(\omega') \neq \emptyset, \end{displaymath}
which follows from $CR^{\delta}_{\sigma}(\omega) \subset ACR^{\delta}_{\sigma}(\bar{\omega})$ (for $A \geq 1$ sufficiently large).  \end{proof}

\subsubsection{Cases $1 < k + 1 \leq |\Lambda|$} We then assume that the $(\delta,\delta,s,C_{\epsilon}\delta^{-\epsilon})$-configurations $G_{1} \supset \ldots \supset G_{k}$ have already been constructed for some $1 \leq k < |\Lambda|$. We next explain how to construct the set $G_{k + 1}$. To be precise, our task is to construct a $(\delta,\delta,s,C_{\epsilon}\delta^{-\epsilon})$-configuration $G_{k + 1} \subset G_{k}$ with the properties $C_{\epsilon} \approx_{\delta} 1$, $|G_{k + 1}| \approx_{\delta} |G_{k}|$, and
\begin{equation}\label{form115} m_{\delta,\lambda_{k + 1},t}^{C_{k + 1}\delta^{-\epsilon},C_{k + 1}}(\omega \mid G_{k + 1}) \leq \delta^{- \kappa}, \qquad \omega \in G_{k + 1}. \end{equation}
We abbreviate
\begin{equation}\label{abbreviations} \lambda := \lambda_{k + 1} \quad \text{and} \quad \sigma := \delta/\sqrt{\lambda_{k + 1}t} \end{equation}
 for the duration of this argument. We write $G_{k} = \{(p,v) : p \in P_{k} \text{ and } v \in G_{k}(p)\}$ with $|G_{k}(p)| \equiv M_{k}$ for all $p \in P_{k}$. Here $|P_{k}|\approx_{\delta} |P|$ and $M_{k} \approx_{\delta} M$ since $|G_{k}| \approx_{\delta} |\Omega| = M|P|$.

Note that the multiplicity function appearing in \eqref{form115} counts elements in the $(\delta,\sigma)$-skeleton of $G_{k + 1}$. It would be desirable to know that $|E_{\sigma}(p)| \equiv M_{\sigma}$ is a constant independent of $p \in P_{k}$, where
\begin{displaymath} E_{\sigma}(p) = \{\mathbf{v} \in \mathcal{S}_{\sigma}(p) : v \prec \mathbf{v} \text{ for some } v \in G_{k}(p)\} \end{displaymath}
is the $(\delta,\sigma)$-skeleton of $G_{k}(p)$. This may not be true to begin with, but may be accomplished with a small pruning, as follows. For each $(p,\mathbf{v}) \in (G_{k})^{\delta}_{\sigma}$, let
\begin{displaymath} M(p,\mathbf{v}) = |\{(p',v) \in G_{k} : (p',v) \prec (p,\mathbf{v})\}| = |\{v \in G_{k}(p) : v \prec \mathbf{v}\}|. \end{displaymath}
The second equation follows from $p \in \mathcal{D}_{\delta}$ (that is, $(p',v) \prec (p,\mathbf{v})$ implies $p' = p$). Now, for each $p \in P_{k}$ fixed, we pigeonhole an integer $M(p) \geq 1$ and a subset $E_{\sigma}'(p) \subset (G_{k})^{\delta}_{\sigma}(p)$ such that $M(p) \leq M(p,\mathbf{v}) \leq 2M(p)$ for all $\mathbf{v} \in E_{\sigma}'(p)$, and further
\begin{equation}\label{form135} |\{v \in G_{k}(p) : v \prec \mathbf{v} \text{ for some } \mathbf{v} \in E_{\sigma}'(p)\}| \approx_{\delta} |G_{k}(p)| = M_{k}. \end{equation}
It follows that $M(p) \cdot |E_{\sigma}'(p)| \approx_{\delta} M_{k} \approx_{\delta} M$ for all $p \in P_{k}$. Next, we pigeonhole an integer $M_{\sigma} \geq 1$, and a subset $\bar{P}_{k} \subset P_{k}$ such that $M_{\sigma} \leq |E_{\sigma}'(p)| \leq 2M_{\sigma}$ for all $p \in \bar{P}_{k}$, and $|\bar{P}_{k}| \approx_{\delta} |P_{k}|$. With this definition, let
\begin{displaymath} \bar{G} := \{(p,v) : p \in \bar{P}_{k}, \, v \in G_{k}(p), \text{ and } v \prec \mathbf{v} \text{ for some } \mathbf{v} \in E_{\sigma}'(p)\}. \end{displaymath}
Thus, the $(\delta,\sigma)$-skeleton of $\bar{G}$ is $\bar{G}^{\delta}_{\sigma} = \{(p,\mathbf{v}) : p \in \bar{P}_{k} \text{ and } \mathbf{v} \in E_{\sigma}'(p)\}$, and for each $p \in \bar{P}_{k}$, the $(\delta,\sigma)$-skeleton of $\bar{G}(p)$ is $\bar{G}^{\delta}_{\sigma}(p) = E_{\sigma}'(p)$, which has constant cardinality $M_{\sigma}$ (up to a factor of $2$). To simplify notation, we denote in the sequel $E_{\sigma}(p) := E_{\sigma}'(p)$ for $p \in \bar{P}_{k}$. Note that
\begin{displaymath} |\bar{G}| = \sum_{p \in \bar{P}_{k}} \sum_{\mathbf{v} \in E_{\sigma}(p)} |\{v \in G_{k}(p) : v \prec \mathbf{v}\}| \stackrel{\eqref{form135}}{\approx_{\delta}} |\bar{P}_{k}|M_{k} \approx_{\delta} |P_{k}|M_{k} = |G_{k}|. \end{displaymath}
To summarise, the procedure above has reduced $G_{k}$ to a subset $\bar{G} \subset G_{k}$ of size $|\bar{G}| \approx_{\delta} |G_{k}|$, and further we have gained the following properties:
\begin{equation}\label{form139} |\bar{G}_{\sigma}^{\delta}(p)| = |E_{\sigma}(p)| \in [M_{\sigma},2M_{\sigma}], \qquad p \in \bar{P}_{k}, \end{equation}
and
\begin{equation}\label{form140} |\{v \in G_{k}(p) : v \prec \mathbf{v}\}| = M(p) \approx_{\delta} M/M_{\sigma}, \qquad p \in \bar{P}_{k}, \, \mathbf{v} \in E_{\sigma}(p). \end{equation}
We also record for future reference that
\begin{equation}\label{form160} M_{\sigma} \sim |E_{\sigma}(p)| \gtrapprox_{\delta} \delta^{\epsilon}\sigma^{-s} = \delta^{\epsilon} \left( \frac{\sqrt{\lambda t}}{\delta} \right)^{s}, \end{equation}
since $\bar{G}(p)$ is a non-empty $(\delta,s,C_{\epsilon}\delta^{-\epsilon})$-set. (It follows from \eqref{form139}-\eqref{form140} that $|\bar{G}(p)| \approx_{\delta} M$ for all $p \in \bar{P}_{k}$, but $\bar{G}$ may fail to be a $(\delta,\delta,s,C\delta^{-\epsilon})$-configuration in the strict sense that the sets $|\bar{G}(p)|$ have equal cardinality. This will not be needed, so we make no attempt to prune back this property. The moral here is that the set $G_{k}$ can be completely forgotten: we will only need $\bar{G} \subset G_{k}$ in the sequel, and the rough constancy of $|\bar{G}^{\delta}_{\sigma}(p)|$.)

We then begin the construction of the set $G_{k + 1} \subset \bar{G}$. This argument requires another induction, in fact very similar to the one  we saw during the proof of Proposition \ref{prop3}. This is not too surprising, given that the "base case" $\delta = \lambda$, or in other words $k = 1$, of \eqref{form111} followed directly from Proposition \ref{prop3}. To reduce confusion with indices, the letters "$k,l$" will from now on refer to the sets in the sequence $G_{1},\ldots,G_{k}$ already constructed in our "exterior" induction and we will use letters "$i,j$" are reserved for the "interior" induction required to construct $G_{k + 1}$.

\begin{remark} It may be worth noting that the "exterior" induction runs $\sim 1/\epsilon$ times, whereas the "interior" induction below runs only $\ceil{20/\kappa} \sim 1/\kappa$ times. This is significant, because it is legitimate to increase (say: double) the constant "$\epsilon$" roughly $1/\kappa$ times and still rest assured that the resulting final constant is $\lesssim 2^{1/\kappa}\epsilon \leq \epsilon_{\mathrm{max}}$ (a small number). In contrast, it would not be legitimate to double the constant "$\epsilon$" roughly $1/\epsilon$ times in the "exterior" induction.  \end{remark}

We start by setting $h := \ceil{20/\kappa}$, and defining the auxiliary sequence of exponents
\begin{equation}\label{form145} 100\epsilon < \epsilon_{h} < \epsilon_{h - 1} < \ldots < \epsilon_{0} < \epsilon_{\mathrm{max}}/100, \end{equation}
where $\epsilon_{j} < \epsilon_{j - 1}/10$ for all $1 \leq j \leq h$. This choice of the sequence $\{\epsilon_{j}\}$ is possible thanks to the relation between the constants "$\epsilon$" and "$\epsilon_{\mathrm{max}}$" explained in Section \ref{s:constants}. Namely, in \eqref{form133} we required that
\begin{displaymath} C\cdot 10^{100/\kappa}\epsilon \leq \epsilon_{\mathrm{max}}. \end{displaymath}
In addition to the exponents $\{\epsilon_{j}\}$, we also define an auxiliary sequence of constants $\{\mathbf{C}_{j}\}$:
\begin{equation}\label{form138} C_{k + 1} \ll \mathbf{C}_{h} \ll \mathbf{C}_{h - 1} \ll \ldots \ll \mathbf{C}_{0} \ll C_{k}. \end{equation}
The necessary rate of decay for the sequence $\{\mathbf{C}_{j}\}$ turns out to be of the form $A\mathbf{C}_{j + 1}^{5} \leq \mathbf{C}_{j}$ for an absolute constant $A \geq 1$. There are $h = \ceil{20/\kappa}$ constants in the sequence, so the sequence $\{\mathbf{C}_{j}\}$ can be found, satisfying \eqref{form138}, since $C_{k} = A(\epsilon,\kappa)C_{k + 1}$ by \eqref{form129}. This is the requirement which determines the size of the constant $A(\epsilon,\kappa)$. It may worth remarking that the constant $A(\epsilon,\kappa)$ necessarily depends on both $\epsilon$ and $\kappa$. This is because the index "$k$" in $C_{k},C_{k + 1}$ ranges in $\{1,\ldots,C/\epsilon\}$ for an absolute constant $C \geq 1$, so $C_{k + 1}$ depends on both $\epsilon,\kappa$. Given the requirement for the constants $\mathbf{C}_{j}$ stated below \eqref{form138}, we see that the size of the multiplicative gap $A(\epsilon,\kappa) = C_{k}/C_{k + 1}$ also depends on both $\epsilon,\kappa$.

Recall that our goal is to define the next set "$G_{k + 1}$" satisfying \eqref{form115}. To do so (as in the proof of Proposition \ref{prop3}), we consider an auxiliary sequence of sets $\bar{G} = \mathbf{G}_{0} \supset \mathbf{G}_{1} \supset \ldots \supset \mathbf{G}_{j}$. Finally, we will set $G_{k + 1} := \mathbf{G}_{j}$ for a suitable member of this auxiliary sequence (or in fact a slight refinement of $\mathbf{G}_{j}$).

Recalling from \eqref{abbreviations} that $\sigma = \delta/\sqrt{\lambda t}$, and $\lambda = \lambda_{k + 1}$, and writing
\begin{equation}\label{def:rho} \rho_{j} := \mathbf{C}_{j}\delta^{-\epsilon}, \end{equation}
we will abbreviate
\begin{align}\label{form171} m_{j}(\omega \mid \mathbf{G}) & := m^{\delta^{-\epsilon_{j},\rho_{j}},\mathbf{C}_{j}}_{\delta,\lambda,t}(\omega \mid \mathbf{G}) \notag\\
& = |\{\omega' \in (\mathbf{G}^{\delta}_{\sigma})^{\delta^{-\epsilon_{j}},\rho_{j}}_{\lambda,t}(\omega) : \mathbf{C}_{j}R^{\delta}_{\sigma}(\omega) \cap \mathbf{C}_{j}R^{\delta}_{\sigma}(\omega') \neq \emptyset\}| \end{align}
for $\mathbf{G} \subset \bar{G}$ and $\omega \in \bar{G}$. We recall that the constant $\delta^{-\epsilon_{j}}$ refers to the range of the tangency parameter "$\lambda$", and the constant $\rho_{j}$ refers to the range of the distance parameter "$t$". It is worth noting that
\begin{displaymath} \mathbf{C}_{j + 1} \leq \mathbf{C}_{j} \quad \text{and} \quad \rho_{j + 1} \leq \rho_{j} \quad \text{and} \quad \delta^{-\epsilon_{j + 1}} \leq \delta^{-\epsilon_{j}}, \end{displaymath}
so $m_{h} \leq m_{h - 1} \leq \ldots \leq m_{0}$. It is also worth noting that since $\epsilon_{j} > 10\epsilon$, the "tangency" range $\delta^{-\epsilon_{j}}$ is very much larger than the "distance" range $\rho_{j} \sim_{\epsilon,\kappa} \delta^{-\epsilon}$, assuming that $\delta > 0$ is sufficiently small in terms of $\epsilon,\kappa$.

We start by recording the "trivial" upper bound
\begin{equation}\label{form114} m_{0}(\omega \mid \mathbf{G}) \leq m_{0}(\omega \mid \Omega_{0}) \lesssim \mathbf{C}_{0}\delta^{-s - \epsilon}, \qquad \omega \in \bar{G}, \, \mathbf{G} \subset \bar{G}, \end{equation}
which has nothing to do with the parameters $\delta^{-\epsilon_{0}},\rho_{0}$, and only has to do with the constant $\mathbf{C}_{0} \sim_{\epsilon,\kappa} 1$. The first inequality is clear. To see the second inequality, fix $\omega = (p,v) \in \bar{G}$ and $(p',v') \in (\Omega_{0})_{\sigma}^{\delta} \subset \{(q,w) : q \in P_{0} \text{ and } w \in \mathcal{S}_{\sigma}(q)\}$ such that
\begin{displaymath} \mathbf{C}_{0}R^{\delta}_{\sigma}(p',v') \cap \mathbf{C}_{0}R^{\delta}_{\sigma}(p,v) \neq \emptyset. \end{displaymath}
Then $v' \in \mathcal{S}_{\sigma}(p')$ and $|v' - v| \lesssim \mathbf{C}_{0}\sigma$. But $\mathcal{S}_{\sigma}(p')$ is $\sigma$-separated, so this can only happen for $\lesssim \mathbf{C}_{0}$ choices of $v'$. This gives \eqref{form114}, recalling that $|P_{0}| \leq \delta^{-s - \epsilon}$ by assumption \eqref{sizeP}.

The trivial inequality \eqref{form114} tells us that the estimate \eqref{form115} holds automatically with $G_{k + 1} = \bar{G}$ and $\kappa = 2s$ (with room to spare), assuming that $\delta,\epsilon > 0$ is chosen so small that $\mathbf{C}_{0} \leq \delta^{-\epsilon} \leq \delta^{-s/2}$. So, we may assume that $0 < \kappa \leq 2s$. Let $0 = \kappa_{1} < \kappa_{2} < \ldots < \kappa_{h} = 2s$ be a $(\kappa s/10)$-dense sequence in $[0,2s]$. Thus $h \leq 20/\kappa$. As already hinted above, we now define a decreasing sequence of sets $\bar{G} = \mathbf{G}_{0} \supset \mathbf{G}_{1} \supset \ldots \supset \mathbf{G}_{l}$, where $l \leq h$. We set $\mathbf{G}_{0} := \bar{G}$, and in general we will assume inductively that $|\mathbf{G}_{j + 1}| \geq \tfrac{1}{2}|\mathbf{G}_{j}|$ for $j \geq 0$ (whenever $\mathbf{G}_{j},\mathbf{G}_{j + 1}$ have been defined). Note that $m_{0}(\omega \mid \mathbf{G}_{0}) \leq \delta^{-2s} = \delta^{- \kappa_{h}}$ by \eqref{form114}, for all $\omega \in \mathbf{G}_{0}$, provided that $\delta > 0$ is small enough.

Let us then assume that the sets $\mathbf{G}_{0} \supset \ldots \supset \mathbf{G}_{j}$ have already been defined. We also assume inductively that
\begin{equation}\label{form137} m^{\delta^{-\epsilon_{j}},\rho_{j},\mathbf{C}_{j}}_{\delta,\lambda,t}(\omega \mid \mathbf{G}_{j}) = m_{j}(\omega \mid \mathbf{G}_{j}) \leq \delta^{-\kappa_{h - j}}, \qquad \omega \in \mathbf{G}_{j}. \end{equation}
This is true by \eqref{form114} for $j = 0$, as we observed above. Define
\begin{displaymath} \mathbf{H}_{j} := \{\omega \in \mathbf{G}_{j} : m_{j + 1}(\omega \mid \mathbf{G}_{j}) \geq \delta^{- \kappa_{h - (j + 1)}}\}. \end{displaymath}
Note that $\kappa_{h - (j + 1)} < \kappa_{h - j}$. So, $\mathbf{H}_{j}$ is the subset of $\mathbf{G}_{j}$ where the lower bound for the $(j + 1)^{st}$ multiplicity nearly matches the (inductive) upper bound on the $j^{th}$ multiplicity.

There are two options.
\begin{enumerate}
\item If $|\mathbf{H}_{j}| \geq \tfrac{1}{2}|\mathbf{G}_{j}|$, then we set $\mathbf{H} := \mathbf{H}_{j}$, and the construction of the sets $\mathbf{G}_{j}$ terminates. We will see that this case cannot occur as long as $\kappa_{h - j} > \kappa$.
\item If $|\mathbf{H}_{j}| < \tfrac{1}{2}|\mathbf{G}_{j}|$, then the set $\mathbf{G}_{j + 1} := \mathbf{G}_{j} \, \setminus \, \mathbf{H}_{j}$ has $|\mathbf{G}_{j + 1}| \geq \tfrac{1}{2}|\mathbf{G}_{j}|$, and moreover
\begin{displaymath} m_{j + 1}(\omega \mid \mathbf{G}_{j + 1}) \leq m_{j + 1}(\omega \mid \mathbf{G}_{j}) \leq \delta^{-\kappa_{h - (j + 1)}}, \qquad \omega \in \mathbf{G}_{j + 1}. \end{displaymath}
In other words, $\mathbf{G}_{j + 1}$ is a valid "next set" in our sequence $\mathbf{G}_{0} \supset \ldots \supset \mathbf{G}_{j + 1}$, and the inductive construction may proceed.
\end{enumerate}

If (and since) case (1) does not occur for indices $j \geq 0$ with $\kappa_{h - j} > \kappa$, we can keep constructing the sets $\mathbf{G}_{j}$ until the first index "$j$" where $\kappa_{h - j} \leq \kappa$. At this stage, the set
\begin{equation}\label{form163a} G_{k + 1} := \mathbf{G}_{j} \end{equation}
satisfies $m_{j}(\omega \mid G_{k + 1}) \leq \delta^{-\kappa}$ for all $\omega \in G_{k + 1}$ by the inductive assumption \eqref{form137}. This implies \eqref{form115}, since $C_{k + 1} \leq \mathbf{C}_{j}$ by \eqref{form138}. Moreover, $|G_{k + 1}| \geq 2^{-j}|\bar{G}| \geq 2^{-20/\kappa}|\bar{G}| \approx_{\delta} |G_{k}|$, so $G_{k + 1}$ is a valid "next set" in the sequence $\{G_{k}\}$. To be precise, we still need to apply Lemma \ref{refinement}, and thereby refine $G_{k + 1}$ (as in \eqref{form163a}) to a $(\delta,\delta,s,C_{\epsilon}\delta^{-\epsilon})$-configuration of cardinality $\approx_{\delta} |G_{k}|$. This will complete the definition of $G_{k + 1}$.

Thus, to complete the construction of the sequence $\{G_{k}\}$, and the proof of Theorem \ref{thm5}, it suffices to verify that the "hard" case (1) cannot occur for any $j \geq 0$ such that $\kappa_{h - j} > \kappa$. To prove this, we make a counter assumption:
\medskip

\textbf{Counter assumption:} Case (1) occurs at some index $j \in \{0,\ldots,h\}$ with $\kappa_{h - j} > \kappa$.

\subsubsection{Deriving a contradiction}\label{s1} The overall strategy is similar to the one we have already encountered in the proofs of Proposition \ref{prop3} and Theorem \ref{thm4}. We will use the counter assumption to produce a "large" collection of incomparable $(\delta,\sigma)$-rectangles, each of which has a high ($\lambda$-restricted) type relative to a certain $(\delta,\epsilon_{\mathrm{max}})$-almost $t$-bipartite pair $(W,B)$ of subsets of $P$. Eventually, the existence of these rectangles will contradict the upper bound established in Theorem \ref{thm4}. The hypothesis \eqref{form54} of Theorem \ref{thm4} will be valid thanks to our previous refinements, specifically \eqref{form130}.

We write $\bar{\kappa} := \kappa_{h - j}$ and (recalling the $(\kappa s)/10$-density of the sequence $\{\kappa_{j}\}$),
\begin{displaymath} \kappa_{h - (j + 1)} =: \bar{\kappa} - \zeta, \qquad \text{where }\zeta \leq (\kappa s)/10 \leq (\bar{\kappa}s)/10. \end{displaymath}
We also abbreviate
\begin{displaymath} \mathbf{G} := \mathbf{G}_{j} \quad \text{and} \quad \mathbf{H} := \mathbf{H}_{j} = \{\omega \in \mathbf{G} : m_{j + 1}(\omega \mid \mathbf{G}) \geq \delta^{- \bar{\kappa} + \zeta}\} \subset \mathbf{G}, \end{displaymath}
and we recall that $|\mathbf{H}| \geq \tfrac{1}{2}|\mathbf{G}| \approx_{\delta} |G_{k}| \approx_{\delta} M|P|$ by the assumption that we are in case (1). Finally, we will abbreviate
\begin{equation}\label{def:n} n := \delta^{- \bar{\kappa} + \zeta}. \end{equation}
To spell out the definition of "$m_{j + 1}$" (recall \eqref{form171}), we have
\begin{equation}\label{form116} |\{\omega' \in (\mathbf{G}^{\delta}_{\sigma})^{\delta^{-\epsilon_{j + 1}},\rho_{j + 1}}_{\lambda,t}(\omega) : \mathbf{C}_{j + 1}R^{\delta}_{\sigma}(\omega) \cap \mathbf{C}_{j + 1}R^{\delta}_{\sigma}(\omega') \neq \emptyset\}| \geq n, \qquad \omega \in \mathbf{H}. \end{equation}
On the other hand, by the inductive assumption \eqref{form137} applied to $\mathbf{G} = \mathbf{G}_{j}$, and recalling that $\bar{\kappa} = \kappa_{h - j}$, we have
\begin{equation}\label{form117} |\{\omega' \in (\mathbf{G}^{\delta}_{\sigma})_{\lambda,t}^{\delta^{-\epsilon_{j}},\rho_{j}}(\omega) : \mathbf{C}_{j}R^{\delta}_{\sigma}(\omega) \cap \mathbf{C}_{j}R^{\delta}_{\sigma}(\omega') \neq \emptyset\}| \leq \delta^{-\bar{\kappa}} = \delta^{-\zeta}n, \quad \omega \in \mathbf{G}. \end{equation}
The numerology is not particularly important yet, but it is crucial that a certain lower bound for $m_{j + 1}(\cdot \mid \mathbf{G})$ holds in a large subset $\mathbf{H} \subset \mathbf{G}$, whereas a nearly matching upper bound for $m_{j}(\omega \mid \mathbf{G})$ holds for all $\omega \in \mathbf{G}$. Achieving this "nearly extremal" situation was the reason to define the sequence $\{\mathbf{G}_{j}\}$.

\begin{remark}\label{rem5} In fact, we will need \eqref{form116}-\eqref{form117} for $\omega \in \mathbf{H}^{\delta}_{\sigma}$ and $\omega \in \mathbf{G}^{\delta}_{\sigma}$ instead of $\omega \in \mathbf{H}$ and $\omega \in \mathbf{G}$, respectively. This is easily achieved, at the cost of changing the constants a little. Indeed, if $A \geq 1$ is a sufficiently large absolute constant, then \eqref{form116}-\eqref{form117} imply
\begin{displaymath} m_{\delta,\lambda,t}^{\delta^{-\epsilon_{j + 1}},\rho_{j + 1},A\mathbf{C}_{j + 1}}(\omega \mid \mathbf{G}) \geq n, \qquad \omega \in \mathbf{H}^{\delta}_{\sigma}, \end{displaymath}
and
\begin{displaymath} m_{\delta,\lambda,t}^{\delta^{-\epsilon_{j}},\rho_{j},\mathbf{C}_{j}/A}(\omega \mid \mathbf{G}) \leq \delta^{-\zeta}n, \qquad \omega \in \mathbf{G}^{\delta}_{\sigma}. \end{displaymath}
These inequalities follow from Lemma \ref{lemma10}. It will be important that the constant $\mathbf{C}_{j}$ is substantially larger than $\mathbf{C}_{j + 1}$, but we can arrange this so (recall the definition \eqref{form138}) that even $\mathbf{C}_{j}/A \gg A\mathbf{C}_{j + 1}$. To avoid burdening the notation with further constants, we will assume from now on that \eqref{form116}-\eqref{form117} hold as stated for $\omega \in \mathbf{H}^{\delta}_{\sigma}$ and $\omega \in \mathbf{G}^{\delta}_{\sigma}$, respectively. \end{remark}

The set $\mathbf{H} \subset \bar{G}$ may have lost the uniformity property \eqref{form139} at scale $\sigma$. That is, we no longer know that all the $(\delta,\sigma)$-skeletons $\mathbf{H}^{\delta}_{\sigma}(p) \subset \bar{G}_{\sigma}^{\delta}(p)$, for $p \in \bar{P}_{k}$, have roughly constant cardinality (let alone $M_{\sigma}$). (Recall that the set $\bar{P}_{k} \subset P_{k}$ was defined below \eqref{form135}.) We resuscitate this property by a slight pruning of $\mathbf{H}$. Note that
\begin{equation}\label{form141} M|P| \approx_{\delta} |\mathbf{H}| = \sum_{p \in \bar{P}_{k}} \sum_{\mathbf{v} \in \mathbf{H}^{\delta}_{\sigma}(p)} |\{v \in \mathbf{H}(p) : v \prec \mathbf{v}\}|. \end{equation}
By pigeonholing, choose a number $\bar{M}_{\sigma} \geq 1$, and a subset $\bar{P} \subset \bar{P}_{k}$ with the properties $\bar{M}_{\sigma} \leq |\mathbf{H}^{\delta}_{\sigma}(p)| \leq 2\bar{M}_{\sigma}$ for all $p \in \bar{P}$, and such that the quantity on the right hand side of \eqref{form141} is only reduced by a factor of $\approx_{\delta} 1$ when replacing $\bar{P}_{k}$ by $\bar{P}$. Thus,
\begin{equation}\label{form142} M|P| \approx_{\delta} \sum_{p \in \bar{P}} \sum_{\mathbf{v} \in \mathbf{H}^{\delta}_{\sigma}(p)} |\{v \in \mathbf{H}(p) : v \prec \mathbf{v}\}| \leq |\bar{P}| \cdot 2\bar{M}_{\sigma} \cdot \max_{\mathbf{v}} |\{v \in \mathbf{H}(p) : v \prec \mathbf{v}\}|. \end{equation}
Here the "$\max$" runs over all $\mathbf{v} \in \mathbf{H}^{\delta}_{\sigma}(p)$, with all possible $p \in \bar{P}$. Here $\mathbf{H}(p) \subset G_{k}(p)$ and $p \in \bar{P}_{k}$, so we see from \eqref{form140} that the "$\max$" is bounded by $\lessapprox_{\delta} M/M_{\sigma}$. Since evidently $\bar{M}_{\sigma} \leq 2M_{\sigma}$, we may now deduce that $\bar{M}_{\sigma} \approx_{\delta} M_{\sigma}$ and $|\bar{P}| \approx_{\delta} |\bar{P}_{k}|$. At this point we define $\bar{\mathbf{H}} := \{(p,v) \in \mathbf{H} : p \in \bar{P}\}$. Then it follows from \eqref{form142} that $|\bar{\mathbf{H}}| \approx_{\delta} |\mathbf{H}| \approx_{\delta} M|P|$, and moreover
\begin{equation}\label{form143} |\bar{\mathbf{H}}^{\delta}_{\sigma}(p)| = |\mathbf{H}^{\delta}_{\sigma}(p)| \sim \bar{M}_{\sigma} \approx_{\delta} M_{\sigma}, \qquad p \in \bar{P}. \end{equation}

\subsubsection{Finding a $t$-bipartite pair} Next, we proceed to find a $(\delta,\epsilon_{\mathrm{max}})$-almost $t$-bipartite pair of subsets of $P$, very much like in the proof of Proposition \ref{prop3}. Let $\mathcal{B}$ be a cover of $P$ by balls of radius $t/(4\rho_{j + 1})$ such that the concentric balls of radius $2\rho_{j + 1}t$ (that is, the balls $\{8\rho_{j + 1}^{2}B : B \in \mathcal{B}\}$) have overlap bounded by $O(\rho_{j + 1}) = O_{\epsilon}(\delta^{-\epsilon}) \leq \delta^{-\epsilon_{\mathrm{max}}}$ (recall that $\rho_{j + 1} = \mathbf{C}_{j + 1}\delta^{-\epsilon}$). Then, we choose a ball $B(p_{0},t/(4\rho_{j + 1})) \in \mathcal{B}$ in such a way that the ratio
\begin{displaymath} \theta := \frac{|\bar{P} \cap B(p_{0},t/(4\rho_{j + 1}))|}{|P \cap B(p_{0},2\rho_{j + 1}t)|} \end{displaymath}
is maximised. Here $\bar{P} \subset \bar{P}_{k} \subset P_{k}$ is the subset of cardinality $|\bar{P}| \approx_{\delta} |P_{k}| \approx_{\delta} |P|$ we just found above, recall \eqref{form143}. We claim that $\theta \gtrapprox_{\delta} \delta^{\epsilon_{\mathrm{max}}}$: this follows immediately from the estimate
\begin{displaymath} |\bar{P}| \leq \sum_{B \in \mathcal{B}} |\bar{P} \cap B| \leq \theta \sum_{B \in \mathcal{B}} |P \cap 8\rho_{k + 1}^{2}B| \leq \theta \delta^{-\epsilon_{\mathrm{max}}}|P|, \end{displaymath}
and since $|\bar{P}| \approx_{\delta} |P|$. Now, we set

\begin{equation}\label{form121} W := \bar{P} \cap B(p_{0},t/(4\rho_{j + 1})) \quad \text{and} \quad B := P \cap B(p_{0},2\rho_{j + 1}t) \, \setminus \, B(p_{0},t/(2\rho_{j + 1})), \end{equation}
so that
\begin{equation}\label{form120} |B| \leq |P \cap B(p_{0},2\rho_{j + 1}t)| = \theta^{-1}|W| \lessapprox_{\delta} \delta^{-\epsilon_{\mathrm{max}}}|W|. \end{equation}
We record at this point that
\begin{equation}\label{form152} \dist(W,B) \geq \tfrac{1}{4}t/\rho_{j + 1} \geq \delta^{\epsilon_{\mathrm{max}}}t \quad \text{and} \quad \diam(W \cup B) \leq 4\rho_{j + 1}t \leq \delta^{-\epsilon_{\mathrm{max}}}t, \end{equation}
so the pair $(W,B)$ is $(\delta,\epsilon_{\mathrm{max}})$-almost $t$-bipartite, independently of "$j$" or "$k$". This will be needed in an upcoming application of Theorem \ref{thm4}.

We then set
\begin{equation}\label{def:fatW} \mathbf{W} := \{(p,v) \in \bar{\mathbf{H}}^{\delta}_{\sigma} : p \in W\} \quad \text{and} \quad \mathbf{B} := \{(p,v) \in \mathbf{G} : p \in B\}. \end{equation}
We note that the "angular" components of $\mathbf{W}$ have separation $\sigma$, but the angular components of $\mathbf{B}$ are $\delta$-separated; this is not a typo. Let us note that
\begin{equation}\label{form119}  \quad |\mathbf{W}(p)| = |\bar{\mathbf{H}}^{\delta}_{\sigma}(p)|_{\sigma} \stackrel{\eqref{form143}}{\sim} \bar{M}_{\sigma} \approx_{\delta} M_{\sigma}, \quad p \in W. \end{equation}
(For this purpose, it was important to choose $W \subset \bar{P}$.) Also, it follows from definitions of $W,B$ that if $p \in W$, and $q \in P$ is arbitrary with $t/\rho_{j + 1} \leq |p - q| \leq \rho_{j + 1}t$, then $q \in B$. Consequently,
\begin{displaymath} \omega \in \mathbf{W} \quad \Longrightarrow \quad (\mathbf{G}^{\delta}_{\sigma})_{\lambda,t}^{\delta^{-\epsilon_{j + 1}},\rho_{j + 1}}(\omega) \subset (\mathbf{B}^{\delta}_{\sigma})_{\lambda,t}^{\delta^{-\epsilon_{j + 1}},\rho_{j + 1}}(\omega). \end{displaymath}
For this inclusion to be true, it is important that in the definition of "$B$" we take into account all points in $\bar{P}_{k}$, and not only the refinement $\bar{P}$. Now this is certainly true, because we are even taking along all the points in $P$.  From this, and since $\mathbf{W} \subset \mathbf{H}^{\delta}_{\sigma}$, and recalling \eqref{form116}, it follows
\begin{equation}\label{form118}  |\{\beta \in (\mathbf{B}^{\delta}_{\sigma})^{\delta^{-\epsilon_{j + 1}},\rho_{j + 1}}_{\lambda,t}(\omega) : \mathbf{C}_{j + 1}R^{\delta}_{\sigma}(\omega) \cap \mathbf{C}_{j + 1}R^{\delta}_{\sigma}(\beta) \neq \emptyset\}| \geq n > 0, \qquad \omega \in \mathbf{W}. \end{equation}
We also used the reduction explained in Remark \ref{rem5} that we may assume \eqref{form116} to hold for all $\omega \in \mathbf{H}^{\delta}_{\sigma}$. Without this reduction, \eqref{form118} would instead hold with constant "$C\mathbf{C}_{j + 1}$".

\subsubsection{The rectangles $\mathcal{R}^{\delta}_{\sigma}$} We will produce a family of $100$-incomparable $(\delta,\sigma)$-rectangles with high $\lambda$-restricted type relative to $(W,B)$. This will place us in a position to apply Theorem \ref{thm4}. Consider the $(\delta,\sigma)$-rectangles $\{R^{\delta}_{\sigma}(\omega) : \omega \in \mathbf{W}\}$, and let
\begin{displaymath} \mathcal{R}^{\delta}_{\sigma} \subset \{R^{\delta}_{\sigma}(\omega) : \omega \in \mathbf{W}\}. \end{displaymath}
be a maximal family of pairwise $100$-incomparable elements. Some rectangles in $\mathcal{R}^{\delta}_{\sigma}$ may arise as $R^{\delta}_{\sigma}(\omega)$ for multiple distinct $\omega \in \mathbf{W}$. We quantify this by considering
\begin{equation}\label{form122} m(R) = |\{\omega \in \mathbf{W} : R \sim_{100} R^{\delta}_{\sigma}(\omega)\}|, \qquad R \in \mathcal{R}^{\delta}_{\sigma}, \end{equation}
where "$\sim_{100}$" refers to $100$-comparability. We note that since every $R \in \mathcal{R}^{\delta}_{\sigma}$ satisfies $R \sim_{100} R^{\delta}_{\sigma}(\omega)$ for some $\omega \in \mathbf{W}$, we have $m(R) \geq 1$ (and $m(R) \leq |\mathbf{W}| \lesssim \delta^{-4}$). By pigeonholing, we may find a subset $\bar{\mathcal{R}}^{\delta}_{\sigma} \subset \mathcal{R}^{\delta}_{\sigma}$ with the property $m(R) \equiv m \in [1,C\delta^{-4}]$ for all $R \in \bar{\mathcal{R}}^{\delta}_{\sigma}$, and moreover
\begin{displaymath} \sum_{\omega \in \mathbf{W}} |\{R \in \bar{\mathcal{R}}^{\delta}_{\sigma} : R \sim_{100} R^{\delta}_{\sigma}(\omega)\}| \approx_{\delta} \sum_{\omega \in \mathbf{W}} |\{R \in \mathcal{R}^{\delta}_{\sigma} : R \sim_{100} R^{\delta}_{\sigma}(\omega)\}|. \end{displaymath}
Now, we have
\begin{align} |\mathcal{\bar{R}}^{\delta}_{\sigma}| & = \frac{1}{m} \sum_{R \in \bar{\mathcal{R}}_{\sigma}^{\delta}} m(R) = \frac{1}{m} \sum_{R \in \bar{\mathcal{R}}^{\delta}_{\sigma}} \sum_{p \in W} \sum_{\mathbf{v} \in \mathbf{W}(p)} \mathbf{1}_{\{R \sim_{100} R^{\delta}_{\sigma}(p,\mathbf{v})\}} \notag\\
&\label{form123}  \approx_{\delta} \frac{1}{m} \sum_{p \in W} \sum_{\mathbf{v} \in \mathbf{W}(p)} |\{R \in \mathcal{R}^{\delta}_{\sigma} : R \sim_{100} R^{\delta}_{\sigma}(p,\mathbf{v})\}| \stackrel{\eqref{form119}}{\gtrapprox_{\delta}} \frac{|W|M_{\sigma}}{m}. \end{align}
(The final lower bound would not necessarily hold for $\bar{\mathcal{R}}_{\sigma}^{\delta}$, since every rectangle $R^{\delta}_{\sigma}(p,\mathbf{v})$, $(p,\mathbf{v}) \in \mathbf{W}$, is not necessarily $100$-comparable to at least one rectangle from $\bar{\mathcal{R}}_{\sigma}^{\delta}$.)

\subsubsection{Proving that $m \lesssim n$} Recall the constant $n = \delta^{-\bar{\kappa} + \zeta}$ from \eqref{def:n}. We next claim that
\begin{equation}\label{form127} m(R) \lesssim_{\epsilon,\kappa} \delta^{-\zeta} n, \qquad R \in \mathcal{R}^{\delta}_{\sigma}, \end{equation}
and in particular $m \lesssim_{\epsilon,\kappa} \delta^{-\zeta}n$. This inequality is analogous to \eqref{form196} in the proof of Proposition \ref{prop3}, but the argument here will be a little harder: now we will finally need the inductive information \eqref{form111} regarding the higher levels of tangency $\lambda_{l}$ for $1 \leq l \leq k$.

Let $R = R^{\delta}_{\sigma}(p,v) \in \mathcal{R}^{\delta}_{\sigma}$, with $(p,v) \in \mathbf{W}$. According to \eqref{form118}, there exists at least one
\begin{equation}\label{def:beta} \beta = (q,w) \in (\mathbf{B}^{\delta}_{\sigma})^{\delta^{-\epsilon_{j + 1}},\rho_{j + 1}}_{\lambda,t}(p,v) \subset (\mathbf{G}^{\delta}_{\sigma})^{\delta^{-\epsilon_{j + 1}},\rho_{j + 1}}_{\lambda,t}(p,v) \end{equation}
such that $\mathbf{C}_{j + 1}R^{\delta}_{\sigma}(p,v) \cap \mathbf{C}_{j + 1}R^{\delta}_{\sigma}(\beta) \neq \emptyset$. We first claim that if $\omega' = (p',v') \in \mathbf{W} \subset \mathbf{G}^{\delta}_{\sigma}$ is any element such that $R^{\delta}_{\sigma}(p,v) \sim_{100} R^{\delta}_{\sigma}(\omega')$, then automatically
\begin{equation}\label{form124} t/\rho_{j} \leq \tfrac{1}{4}t/\rho_{j + 1} \leq |p' - q| \leq 4 \rho_{j + 1}t \leq \rho_{j} t\quad \text{and} \quad \mathbf{C}_{j}R^{\delta}_{\sigma}(\omega') \cap \mathbf{C}_{j}R^{\delta}_{\sigma}(\beta) \neq \emptyset \end{equation}
The first property follows from the separation \eqref{form121} of the sets $W,B$, and noting that $\rho_{j} = \mathbf{C}_{j}\delta^{-\epsilon} \geq 4\mathbf{C}_{j + 1}\delta^{-\epsilon} = 4\rho_{j + 1}$ (recall \eqref{def:rho}).

For the second property, note that since $R^{\delta}_{\sigma}(p,v) \sim_{100} R^{\delta}_{\sigma}(\omega')$, we have $R^{\delta}_{\sigma}(\omega') \subset AR^{\delta}_{\sigma}(p,v) \subset \mathbf{C}_{j + 1}R^{\delta}_{\sigma}(p,v)$ for some absolute constant $A > 0$ according to Lemma \ref{lemma6}, and by a second application of the same lemma,
\begin{displaymath} \mathbf{C}_{j + 1}R^{\delta}_{\sigma}(p,v) \subset \mathbf{C}_{j + 1}'R^{\delta}_{\sigma}(\omega') \end{displaymath}
for some $\mathbf{C}_{j + 1}' \lesssim \mathbf{C}_{j + 1}^{5}$. In particular, $\mathbf{C}_{j + 1}R^{\delta}_{\sigma}(p,v) \subset \mathbf{C}_{j}R^{\delta}_{\sigma}(\omega')$, recalling from \eqref{form138} the rapid decay of the sequence $\{\mathbf{C}_{j}\}$. The second part of \eqref{form124} follows from this inclusion, recalling that $\mathbf{C}_{j + 1}R(p,v) \cap \mathbf{C}_{j + 1}R^{\delta}_{\sigma}(\beta) \neq \emptyset$.

Let us recap: we have now shown that for $\omega' \in \mathbf{W}$ with $R^{\delta}_{\sigma}(p,v) \sim_{100} R^{\delta}_{\sigma}(\omega')$, the conditions \eqref{form124} hold relative to the fixed pair $\beta \in \mathbf{B}^{\delta}_{\sigma}$ (determined by $R$). This gives an inequality of the form
\begin{equation}\label{form125} m(R) \leq |\{\omega' \in (\mathbf{G}^{\delta}_{\sigma})^{\rho_{j}}_{t}(\beta) : \mathbf{C}_{j}R^{\delta}_{\sigma}(\beta) \cap \mathbf{C}_{j}R^{\delta}_{\sigma}(\omega') \neq \emptyset\}|, \end{equation}
where the (non-standard) notation $(\mathbf{G}^{\delta}_{\sigma})^{\rho_{j}}_{t}(\beta)$ refers to those pairs $(p',v')$ such that $t/\rho_{j} \leq |p' - q| \leq \rho_{j}t$. In particular, we have no information -- yet -- about the tangency parameter $\Delta(p',q)$. This almost brings us into a position to apply \eqref{form117}, except for one problem: \eqref{form117} only gives an upper bound for the cardinality of elements
\begin{displaymath} \omega' \in (\mathbf{G}^{\delta}_{\sigma})^{\delta^{-\epsilon_{j}},\rho_{j}}_{\lambda,t}(\beta). \end{displaymath}
To benefit directly from this upper bound, we should be able to add the information
\begin{equation}\label{form147} \delta^{\epsilon_{j}}\lambda \leq \Delta(p',q) \leq \delta^{-\epsilon_{j}}\lambda \end{equation}
to the properties \eqref{form124}. This is a delicate issue: it follows from the choice of $\beta = (q,w)$ in \eqref{def:beta} that we have excellent two-sided control for $\Delta(p,q)$. Regardless, it is only possible to obtain the upper bound for $\Delta(p',q)$ required by \eqref{form147}, given the information that $R^{\delta}_{\sigma}(p,v) \sim_{100} R^{\delta}_{\sigma}(p',v')$. The lower bound may seriously fail: the circles $S(p'),S(q)$ may be much more tangent than the circles $S(p),S(q)$, see Figure \ref{fig3}. This problem will be circumvented by applying our inductive hypothesis. Before that, we however prove the upper bound: we claim that if the properties \eqref{form124} hold, then the upper bound in \eqref{form147} holds. This will be a consequence of Corollary \ref{cor2}.
\begin{figure}[h!]
\begin{center}
\begin{overpic}[scale = 0.9]{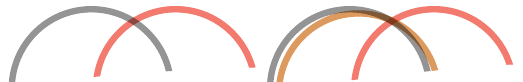}
\put(-3,8){\tiny{$S(q)$}}
\put(35,8){\tiny{$S(p)$}}
\put(60,7){\tiny{$S(p')$}}
\end{overpic}
\caption{The failure of the lower bound in \eqref{form147} in a case where $\lambda \approx 1 \approx t$, thus $\sigma = \delta/\sqrt{\lambda t} \approx \delta$. The black and red annuli $S^{\delta}(p),S^{\delta}(q)$ on the left intersect in a $(\delta,\delta)$-rectangle $R = R^{\delta}_{\delta}(p,v)$. On the right, the rectangle $R$ is evidently $100$-comparable to a $(\delta,\delta)$-rectangle $R' = R^{\delta}_{\delta}(p',v') \subset S^{\delta}(p')$, but nevertheless $\Delta(p',q) \approx \delta \ll \lambda$.}\label{fig3}
\end{center}
\end{figure}

Let $p',p,q$ be as in \eqref{form124}. Thus $R^{\delta}_{\sigma}(p',v') \sim_{100} R^{\delta}_{\sigma}(p,v)$, where the point $(p,v) \in \mathbf{W}$ satisfied
\begin{equation}\label{form178} \mathbf{C}_{j + 1}R^{\delta}_{\sigma}(p,v) \cap \mathbf{C}_{j + 1}R^{\delta}_{\sigma}(q,w) \neq \emptyset, \end{equation}
and
\begin{equation}\label{form179} \delta^{\epsilon_{j + 1}}\lambda \leq \Delta(p,q) \leq \delta^{-\epsilon_{j + 1}}\lambda \quad \text{and} \quad t/\rho_{j} \leq |p - q| \leq \rho_{j}t. \end{equation}
These conditions place us in a position to apply Corollary \ref{cor2}. We write $\bar{\lambda} := \Delta(p,q)$ and $\bar{t} := |p - q|$, and $\bar{\sigma} := \delta/\sqrt{(\bar{\lambda} + \delta)(\bar{t} + \delta)}$. If follows from the upper bounds in \eqref{form179}, and since $\rho_{j} = \mathbf{C}_{j}\delta^{-\epsilon} \leq \delta^{-\epsilon_{j + 1}}$, that
\begin{equation}\label{form188} \sigma \lesssim \delta^{-\epsilon_{j + 1}}\bar{\sigma}. \end{equation}
After this observation, it follows from \eqref{form178} that
\begin{equation}\label{form189} \mathbf{A}_{j + 1}R^{\delta}_{\bar{\sigma}}(p,v) \cap \mathbf{A}_{j + 1}R^{\delta}_{\bar{\sigma}}(q,w) \neq \emptyset \end{equation}
for some $\mathbf{A}_{j + 1} \lesssim \mathbf{C}_{j + 1}\delta^{-\epsilon_{j + 1}} \lesssim_{\epsilon,\kappa} \delta^{-\epsilon_{j + 1}}$. Using \eqref{form188} again, we may also choose the constant $\mathbf{A}_{j + 1}$ (under the same size constraint) so large that
\begin{displaymath} AR^{\delta}_{\sigma}(p,v) \subset \mathbf{A}_{j + 1}R^{\delta}_{\bar{\sigma}}(p,v), \end{displaymath}
where $A \geq 1$ is an absolute constant to be specified momentarily. Now, according to Corollary \ref{cor2}, \eqref{form189} implies
\begin{equation}\label{form198} AR^{\delta}_{\sigma}(p,v) \subset \mathbf{A}_{j + 1}R^{\delta}_{\bar{\sigma}}(p,v) \subset \mathbf{A}_{j + 1}'R^{\delta}_{\bar{\sigma}}(q,w) \end{equation}
for some $\mathbf{A}_{j + 1}' \lesssim \mathbf{A}_{j + 1}^{4} \lesssim_{\epsilon,\kappa} \delta^{-4\epsilon_{j + 1}}$. Finally, since $R^{\delta}_{\sigma}(p',v') \sim_{100} R^{\delta}_{\sigma}(p,v)$, we have
\begin{displaymath} R^{\delta}_{\sigma}(p',v') \stackrel{\mathrm{L.\,} \ref{lemma6}}{\subset} AR^{\delta}_{\sigma}(p,v)  \subset \mathbf{A}_{j + 1}'R^{\delta}_{\bar{\sigma}}(q,w) \subset S^{\mathbf{A}_{j + 1}'\delta}(q). \end{displaymath}
Trivially also $R^{\delta}_{\sigma}(p',v') \subset S^{\mathbf{A}_{j + 1}'\delta}(p')$, so $R^{\delta}_{\sigma}(p',v') \subset S^{\mathbf{A}_{j + 1}'\delta}(q) \cap S^{\mathbf{A}_{j + 1}'\delta}(p')$. This implies
\begin{displaymath} \frac{\mathbf{A}_{j + 1}'\delta}{\sqrt{\Delta(p',q)|p' - q|}} \stackrel{\mathrm{L.\,} \ref{lemma5}}{\gtrsim} \diam(R^{\delta}_{\sigma}(p',v')) \sim \sigma. \end{displaymath}
Recalling that $\mathbf{A}_{j + 1}' \lesssim_{\epsilon,\kappa} \delta^{-4\epsilon_{j + 1}}$, this can be rearranged to
\begin{displaymath} \Delta(p',q) \lesssim_{\epsilon,\kappa} (\delta^{1 - 4\epsilon_{j + 1}}/\sigma)^{2}|p' - q|^{-1} = \delta^{-8\epsilon_{j + 1}}\lambda \cdot (t/|p' - q|) \lesssim_{\epsilon} \delta^{-9\epsilon_{j + 1}}\lambda. \end{displaymath}
In the final inequality we used that $p' \in W$ and $q \in B$, so $|p' - q| \geq t/\rho_{j} \gtrsim_{\epsilon} \delta^{\epsilon_{j + 1}}t$. In \eqref{form145}, the sequence $\{\epsilon_{j}\}$ was chosen to be so rapidly decreasing that $\epsilon_{j} > 10\epsilon_{j + 1}$. Therefore, if $\delta > 0$ is small enough, the inequality above implies the upper bound claimed in \eqref{form147}.

Recalling also \eqref{form125}, we have now shown that
\begin{equation}\label{form125a} m(R) \leq |\{\omega' \in (\mathbf{G}^{\delta}_{\sigma})^{\delta^{-\epsilon_{j}},\rho_{j}}_{\leq \lambda,t}(\beta) : \mathbf{C}_{j}R^{\delta}_{\sigma}(\beta) \cap \mathbf{C}_{j}R^{\delta}_{\sigma}(\omega') \neq \emptyset\}|, \end{equation}
where the "$\leq \lambda$" symbol refers to the fact that we only have guaranteed the upper bound in \eqref{form147}, but not a matching lower bound.

As noted above, the matching lower bound $\Delta(p',q) \gtrapprox \lambda$ may be false. However, recall from \eqref{abbreviations} that $\lambda = \lambda_{k + 1}$, and that the sequence $\{\lambda_{l}\}_{l = 1}^{k}$ is multiplicatively $\delta^{-\epsilon}$-dense (or even $\delta^{-\epsilon/2}$-dense) on the interval $[\delta,\delta^{\epsilon}\lambda] \subset [\delta,\lambda_{k}]$. Therefore, we are either in the happy case of the $2$-sided bound
\begin{equation}\label{form180} \delta^{\epsilon_{j}}\lambda \leq \Delta(p',q) \leq \delta^{-\epsilon_{j}}\lambda, \end{equation}
or otherwise $\Delta(p',q) < \delta^{\epsilon_{j}}\lambda \leq \delta^{\epsilon}\lambda$, and we can find an index $1 \leq l \leq k$ such that
\begin{equation}\label{form181} \delta^{\epsilon}\lambda_{l}/C_{l} \leq \lambda_{l} \leq \Delta(p',q) \leq \delta^{-\epsilon}\lambda_{l} \leq C_{l}\delta^{-\epsilon}\lambda_{l}. \end{equation}
In fact, there is a small gap in this argument: if $\Delta(p',q) < \delta$, then we cannot guarantee \eqref{form181} for any $1 \leq l \leq k$. To fix this, we modify \eqref{form181} so that in the case $l = 1$, only the upper bound is claimed. With this convention, the index $l \in \{1,\ldots,k\}$ satisfying \eqref{form181} can always be found whenever $\Delta(p',q) < \delta^{\epsilon_{j}}\lambda$.

One of the two cases \eqref{form180}-\eqref{form181} is "typical" in the following sense.  Since $\Delta(p',q) \leq \delta^{-\epsilon_{j}}\lambda$ for all pairs $(p',q)$ appearing in \eqref{form125a}, by the pigeonhole principle there exist $\bar{m}(R) \sim_{\epsilon} m(R)$ pairs $\omega_{1},\ldots,\omega_{\bar{m}(R)} \in \mathbf{W}$ with first components $p_{1},\ldots,p_{\bar{m}(R)} \in W$, and a fixed index $1 \leq l \leq k + 1$, such that
\begin{equation}\label{form126} \delta^{\epsilon}\lambda_{l}/C_{l} \leq \Delta(p_{i},q) \leq C_{l}\delta^{-\epsilon}\lambda_{l}, \qquad 1 \leq i \leq \bar{m}(R), \end{equation}
for some fixed $1 \leq l \leq k + 1$. In the case $l = k + 1$, the constant "$\epsilon$" in \eqref{form126} needs to be replaced by $\epsilon_{j}$, recalling the alternatives \eqref{form180}-\eqref{form181}. In the case $l = 1$, the two-sided inequality in \eqref{form126} has to be replaced by the one-sided inequality $\Delta(p_{i},q) \leq C_{1}\delta^{1 - \epsilon}$.


A subtle point is that even though the pairs $\omega_{1},\ldots,\omega_{\bar{m}(R)}$ are distinct, the first components $p_{1},\ldots,p_{\bar{m}(R)}$ need not be. However, they "almost" are: for $p \in W$ fixed, there can only exist $\lesssim \mathbf{C}_{j}$ choices $v \in \mathcal{S}_{\sigma}(p)$ such that $\mathbf{C}_{j}R^{\delta}_{\sigma}(p,v) \cap \mathbf{C}_{j}R^{\delta}_{\sigma}(\beta) \neq \emptyset$ (as in \eqref{form124}). Thus,
\begin{displaymath} |\{p_{1},\ldots,p_{\bar{m}(R)}\}| \gtrsim \mathbf{C}_{j}^{-1}m(R). \end{displaymath}
For this argument, it was important that the "angular" components of the pairs in $\mathbf{W}$ are elements in $\mathcal{S}_{\sigma}(p)$, recall \eqref{def:fatW}. For notational convenience, we will assume in the sequel that the points $p_{1},\ldots,p_{\bar{m}(R)}$ are distinct, and we will trade this information for the weaker estimate $\bar{m}(R) \gtrsim_{\epsilon} \mathbf{C}_{j}^{-1}m(R)$ (this is harmless, since $\mathbf{C}_{j} \lesssim_{\epsilon,\kappa} 1$).

Now, we have two separate cases to consider. First, if $l = k + 1$, then $\lambda_{l} = \lambda$, and we have $\delta^{\epsilon_{j}}\lambda \leq \Delta(p_{i},q) \leq \delta^{-\epsilon_{j}}\lambda$ for all $1 \leq i \leq \bar{m}(R)$. In this case
\begin{displaymath} m(R) \lesssim_{\epsilon} \bar{m}(R) \leq |\{\omega' \in (\mathbf{G}^{\delta}_{\sigma})_{\lambda,t}^{\delta^{-\epsilon_{j}},\rho_{j}}(\beta) : \mathbf{C}_{j}R^{\delta}_{\sigma}(\beta) \cap \mathbf{C}_{j}R^{\delta}_{\sigma}(\omega') \neq \emptyset\}| \stackrel{\eqref{form117}}{\leq} \delta^{-\zeta}n, \end{displaymath}
using that $\beta \in \mathbf{B}^{\delta}_{\sigma} \subset \mathbf{G}^{\delta}_{\sigma}$ (recall also Remark \ref{rem5} where we explained why \eqref{form117} may be assumed to hold for $\beta \in \mathbf{G}^{\delta}_{\sigma}$, not just $\beta \in \mathbf{G}$). This proves \eqref{form127} in the case $l = k + 1$.

Assume finally that $1 \leq l \leq k$. Then, according to \eqref{form126} we have
\begin{equation}\label{form148} m(R) \lesssim_{\epsilon} \bar{m}(R) \leq |\{\omega' \in (\mathbf{G}^{\delta}_{\sigma})_{\lambda_{l},t}^{C_{l}\delta^{-\epsilon},\rho_{j}}(\beta) : \mathbf{C}_{j}R^{\delta}_{\sigma}(\beta) \cap \mathbf{C}_{j}R^{\delta}_{\sigma}(\omega') \neq \emptyset\}|. \end{equation}
We note that $\rho_{j} = \mathbf{C}_{j}\delta^{-\epsilon} \leq C_{l}\delta^{-\epsilon}$ by the choice of the intermediate constants $\{\mathbf{C}_{j}\}$, see \eqref{form138}, so the inequality \eqref{form148} implies
\begin{equation}\label{form185} m(R) \lesssim_{\epsilon} |\{\omega' \in (\mathbf{G}^{\delta}_{\sigma})_{\lambda_{l},t}^{C_{l}\delta^{-\epsilon}}(\beta) : \mathbf{C}_{j}R^{\delta}_{\sigma}(\beta) \cap \mathbf{C}_{j}R^{\delta}_{\sigma}(\omega') \neq \emptyset\}|. \end{equation}
(This remains true as stated also in the special case $l = 1$: in this case \eqref{form126} had to be replaced by the one-sided inequality $\Delta(p_{i},q) \leq C_{1}\delta^{-\epsilon}\lambda_{1} = C_{1}\delta^{1 - \epsilon}$, but this implies $\omega_{i} = (p_{i},v_{i}) \in (\mathbf{G}_{\sigma}^{\delta})^{C_{1}\delta^{-\epsilon}}_{\lambda_{1},t}(\beta)$ for $\lambda_{1} = \delta$, see the last line of Definition \ref{def:lambdaTNeighbourhood}).

The right hand side looks deceptively like $m_{\delta,\lambda_{l},t}^{C_{l}\delta^{-\epsilon},C_{l}}(\omega \mid \mathbf{G})$ (note also that $\mathbf{C}_{j} \leq C_{l}$), and since $\mathbf{G} \subset G_{l}$, the inductive hypothesis \eqref{form111} now appears to show that
\begin{displaymath} m(R) \lesssim_{\epsilon} \delta^{-\kappa} \stackrel{\eqref{def:n}}{\leq} \delta^{-\zeta}n, \end{displaymath}
as desired, using here that $\bar{\kappa} = \kappa_{h - j} > \kappa$ by our counter assumption. There is still a small gap in this argument: the definition of $m^{C_{l}\delta^{-\epsilon},C_{l}}_{\delta,\lambda_{l},t}$ counts elements in the $(\delta,\sigma_{l})$-skeleton of $\mathbf{G}$ with $\sigma_{l} = \delta/\sqrt{\lambda_{l}t} > \sigma$, rather than the $(\delta,\sigma)$-skeleton appearing on the right hand side of \eqref{form185}.

This is easy to fix. The solution is to first use the (distinct!) points $p_{1},\ldots,p_{\bar{m}(R)}$ found in \eqref{form126} to produce a collection of pairs $\bar{\omega}_{1},\ldots,\bar{\omega}_{\bar{m}(R)} \in \mathbf{G}^{\delta}_{\sigma_{l}}$. Indeed, for every $1 \leq i \leq \bar{m}(R)$, we know from \eqref{form124} that there corresponds a pair $\omega_{i} = (p_{i},v_{i}) \in \mathbf{G}^{\delta}_{\sigma}$ such that
\begin{equation}\label{form182} \mathbf{C}_{j}R^{\delta}_{\sigma}(p_{i},v_{i}) \cap \mathbf{C}_{j}R^{\delta}_{\sigma}(\beta) \neq \emptyset. \end{equation}
For every $1 \leq i \leq \bar{m}(R)$, choose $\bar{\omega}_{i} := (p_{i},\mathbf{v}_{i}) \in \mathbf{G}_{\sigma_{l}}^{\delta}$ with $(p_{i},v_{i}) \prec (p_{i},\mathbf{v}_{i})$. Note that the pairs $\bar{\omega}_{1},\ldots,\bar{\omega}_{\bar{m}(R)}$ are all distinct, since the "base" points $p_{1},\ldots,p_{\bar{m}(R)}$ are distinct. Further, it follows from \eqref{form182}, combined with
\begin{displaymath} A\mathbf{C}_{j} \stackrel{\eqref{form138}}{\leq} C_{k} \leq C_{l} \quad \Longrightarrow \quad \mathbf{C}_{j}R^{\delta}_{\sigma}(p_{i},v_{i}) \subset C_{l}R^{\delta}_{\sigma_{l}}(p_{i},\mathbf{v}_{i}) = C_{l}R^{\delta}_{\sigma_{l}}(\bar{\omega}_{i}), \end{displaymath}
(here $A \geq 1$ is a sufficiently large absolute constant) that
\begin{equation}\label{form186} C_{l}R^{\delta}_{\sigma_{l}}(\bar{\omega}_{i}) \cap C_{l}R^{\delta}_{\sigma_{l}}(\beta) \neq \emptyset, \qquad 1 \leq i \leq \bar{m}(R). \end{equation}
(The deduction from \eqref{form182} to \eqref{form186} looks superficially similar to the deduction of the second claim in \eqref{form124}, but now the situation is much simpler, because $(p_{i},v_{i})$ and $(p_{i},\mathbf{v}_{i})$ have the same "$p_{i}$".) We note that the tangency and distance parameters of the pairs $((p_{i},v_{i}),\beta)$ and $(\bar{\omega}_{i},\beta)$ are exactly the same, since the "base point" $p_{i}$ remained unchanged. Consequently, by \eqref{form185} and \eqref{form186}, we have
\begin{align} \bar{m}(R) & \leq |\{\bar{\omega}' \in (\mathbf{G}^{\delta}_{\sigma_{l}})_{\lambda_{l},t}^{C_{l}\delta^{-\epsilon}}(\beta) : C_{l}R^{\delta}_{\sigma_{l}}(\beta) \cap C_{l}R^{\delta}_{\sigma_{l}}(\bar{\omega}') \neq \emptyset\}| \notag\\
&\label{form163} = m^{C_{l}\delta^{-\epsilon},C_{l}}_{\delta,\lambda_{l},t}(\beta \mid \mathbf{G}) \stackrel{\eqref{form111}}{\leq} \delta^{-\kappa} \stackrel{\eqref{def:n}}{\leq} \delta^{-\zeta}n. \end{align}
We have finally proven \eqref{form127}.

\subsubsection{The type of the rectangles $R \in \bar{\mathcal{R}}^{\delta}_{\sigma}$} We next claim that every rectangle $R \in \bar{\mathcal{R}}^{\delta}_{\sigma}$ has $\lambda$-restricted type $(\geq \bar{m},\geq \bar{n})_{\epsilon_{\mathrm{max}}}$ relative to $(W,B,\{E(p)\})$, where $\bar{m} := \delta^{\epsilon_{\mathrm{max}}}m$ and $\bar{n} := \delta^{\epsilon_{\mathrm{max}}}n$. Recall from Definition \ref{def:type} what this means. Given $R \in \bar{\mathcal{R}}^{\delta}_{\sigma}$, we should find a subset $W_{R} \subset W$ with $|W_{R}| \geq \bar{m}$, and the following property: for every $p \in W_{R}$, there exists a subset $B_{R}(p) \subset B$ of cardinality $|B_{R}(p)| \geq \bar{n}$ satisfying
\begin{equation}\label{form150} \delta^{\epsilon_{\mathrm{max}}}\lambda \leq \Delta(p,q) \leq \delta^{-\epsilon_{\mathrm{max}}}\lambda \quad \text{and} \quad R \subset \delta^{-\epsilon_{\mathrm{max}}}\mathcal{E}^{\delta}_{\sigma}(p) \cap \delta^{-\epsilon_{\mathrm{max}}}\mathcal{E}^{\delta}_{\sigma}(q) \end{equation}
for all $p \in W_{R}$ and $q \in B_{R}(p)$. If $\lambda = \delta$, the first requirement in \eqref{form150} is relaxed to $\Delta(p,q) \leq \delta^{-\epsilon_{\mathrm{max}}}\lambda$.

\begin{remark} In \eqref{form150}, the definition of the sets $\mathcal{E}^{\delta}_{\sigma}(p),\mathcal{E}^{\delta}_{\sigma}(q)$ involves the $(\delta,\sigma)$-skeletons of $E(p)$ and $E(q)$. We emphasise that these sets are not the "original" sets $E_{0}(p),E_{0}(q)$ given in Theorem \ref{thm5} (recall the notation from Section \ref{s:regularisation}), but rather the subsets found at the end of Section \ref{s:regularisation}, see \eqref{form220}. This is important, since the upper bound \eqref{form130} will be needed in a moment. \end{remark}

 To begin finding $W_{R}$ and $B_{R}(p)$ for $p \in W_{R}$, recall that $m(R) = m$ for all $R \in \bar{\mathcal{R}}^{\delta}_{\sigma}$. This mean that there exists a set $\mathbf{W}_{R} \subset \mathbf{W}$ of $m$ pairs $\{\omega_{i}\}_{i = 1}^{m} = \{(p_{i},v_{i})\}_{i = 1}^{m}$ such that $R \sim_{100} R^{\delta}_{\sigma}(\omega_{i})$ for all $1 \leq i \leq m$. While the pairs $\omega_{i}$ are all distinct, the first components $p_{i}$ need not be. This issue is similar to the one we encountered below \eqref{form126}, and the solution is also the same: for every $p_{i}$ fixed, there can only be $\lesssim 1$ possibilities for $v \in \mathcal{S}_{\sigma}(p_{i})$ such that $R \sim_{100} R^{\delta}_{\sigma}(p_{i},v)$. Therefore the number of distinct elements in $W_{R} := \{p_{1},\ldots,p_{m}\}$ is $\gtrsim m$, and certainly $|W_{R}| \geq \delta^{\epsilon_{\mathrm{max}}}m = \bar{m}$. To remove ambiguity, for each distinct point $p_{i} \in W_{R}$, we pick a single element $v \in \mathcal{S}_{\sigma}(p_{i})$ such that $(p_{i},v) \in \mathbf{W}_{R}$, and we restrict $\mathbf{W}_{R}$ to this subset without changing notation.

Next, fix $p \in W_{R}$. Let $v \in \mathcal{S}_{\sigma}(p)$ be the unique element such that $\omega = (p,v) \in \mathbf{W}_{R} \subset \mathbf{W}$. Recall from \eqref{form118}  that
\begin{displaymath} |\{\beta \in (\mathbf{B}^{\delta}_{\sigma})^{\delta^{-\epsilon_{j + 1}},\rho_{j + 1}}_{\lambda,t}(\omega) : \mathbf{C}_{j + 1}R^{\delta}_{\sigma}(\omega) \cap \mathbf{C}_{j + 1}R^{\delta}_{\sigma}(\beta) \neq \emptyset\}| \geq n. \end{displaymath}
Thus, there exists a collection $\{\beta_{i}\}_{i = 1}^{n} = \{(q_{i},w_{i})\}_{i = 1}^{n} \subset \mathbf{B}^{\delta}_{\sigma}$ of pairs such that
\begin{equation}\label{form187} \mathbf{C}_{j + 1}R^{\delta}_{\sigma}(\omega) \cap \mathbf{C}_{j + 1}R^{\delta}_{\sigma}(q_{i},w_{i}) \neq \emptyset, \end{equation}
and
\begin{equation}\label{form151} \delta^{\epsilon_{j + 1}}\lambda \leq \Delta(p,q_{i}) \leq \delta^{-\epsilon_{j + 1}}\lambda \quad \text{and} \quad \delta^{\epsilon_{\mathrm{max}}}t \leq |p - q_{i}| \leq \delta^{-\epsilon_{\mathrm{max}}}t \end{equation}
for all $1 \leq i \leq n$. In the estimates for $|p - q_{i}|$, we already plugged in $\rho_{j + 1} = \mathbf{C}_{j + 1}\delta^{-\epsilon} \leq \delta^{-\epsilon_{\mathrm{max}}}$, assuming $\delta > 0$ small enough.

Once more, the $q_{i}$-components of the pairs $\{\beta_{i}\}$ need not all be distinct, but they almost are, by the following familiar argument: for each $q_{i}$, there can correspond $\lesssim \mathbf{C}_{j + 1}$ distinct choices $w \in \mathcal{S}_{\sigma}(q_{i})$ such that \eqref{form187} holds. Therefore, $B_{R}(p) := \{q_{1},\ldots,q_{n}\} \subset B$ has $\gtrsim \mathbf{C}_{j + 1}^{-1}n$ distinct elements, and certainly $|B_{R}(p)| \geq \bar{n}$.

Let us finally check the conditions \eqref{form150} for $p \in W_{R}$ and $q \in B_{R}(p)$. The tangency constraint follows readily from \eqref{form151}, and noting that $\epsilon_{j + 1} \leq \epsilon_{\mathrm{max}}$. So, it remains to check the inclusion in \eqref{form150}. Fix $p \in W_{R}$ and $q \in B_{R}(p)$. By definition, $p \in W_{R}$ means that $R \sim_{100} R^{\delta}_{\sigma}(\omega)$ for some $\omega = (p,v) \in \mathbf{W}_{R} \subset \mathbf{G}^{\delta}_{\sigma}$, and in particular $v \in E_{\sigma}(p)$ (the $(\delta,\sigma)$-skeleton of $E(p)$). Next, $q \in B_{R}(p)$ means that there exists $\beta = (q,w) \in \mathbf{B}^{\delta}_{\sigma}$ (in particular $w \in E_{\sigma}(q)$) such that \eqref{form187}-\eqref{form151} hold. We now claim that
\begin{equation}\label{form190} R \subset \delta^{-\epsilon_{\mathrm{max}}}R^{\delta}_{\sigma}(\omega) \cap \delta^{-\epsilon_{\mathrm{max}}}R^{\delta}_{\sigma}(\beta) \subset \delta^{-\epsilon_{\mathrm{max}}}\mathcal{E}^{\delta}_{\sigma}(p) \cap \delta^{-\epsilon_{\mathrm{max}}}\mathcal{E}^{\delta}_{\sigma}(q). \end{equation}
This is a consequence of Corollary \ref{cor2}, and the argument is extremely similar to the one we recorded below \eqref{form178}-\eqref{form179}. We just sketch the details. Applying Corollary \ref{cor2} with $\bar{\sigma} := \delta/\sqrt{(\Delta(p,q) + \delta)(|p - q| + \delta)}$, it follows from the non-empty intersection \eqref{form187} that
\begin{displaymath} AR^{\delta}_{\sigma}(\omega) \subset \mathbf{A}_{j + 1}R^{\delta}_{\sigma}(\beta), \end{displaymath}
where $A \geq 1$ is a suitable absolute constant, and $\mathbf{A}_{j + 1} \lesssim_{\epsilon} \delta^{-O(\epsilon_{j + 1})}$ (compare with \eqref{form198}). Next, from $R \sim_{100} R^{\delta}_{\sigma}(\omega)$, we simply deduce that $R \subset AR^{\delta}_{\sigma}(\omega)$. Since $\max\{A,\mathbf{A}_{j + 1}\} \leq \delta^{-\epsilon_{\mathrm{max}}}$ for $\delta > 0$ small enough, the inclusion \eqref{form190} follows.

We have now proven that every rectangle $R \in \bar{\mathcal{R}}^{\delta}_{\sigma}$ has $\lambda$-restricted type $(\geq \bar{m},\geq \bar{n})_{\epsilon_{\mathrm{max}}}$ relative to $(W,B,\{E(p)\})$.

\subsubsection{Applying Theorem \ref{thm4}} The constant $\epsilon_{\mathrm{max}} = \epsilon_{\mathrm{max}}(\kappa,s) > 0$ was chosen (recall Section \ref{s:constants}) in such a way that Theorem \ref{thm4} holds with constant $\eta = \kappa s/100$. Therefore, we may apply the theorem as soon as we have checked that its hypotheses are valid. At the risk of over-repeating, we will apply Theorem \ref{thm4} to the space $\Omega = \{(p,v) : p \in P \text{ and } v \in E(p)\} \subset \Omega_{0}$ constructed during the "initial regularisation" in Section \ref{s:regularisation}. Crucially, we recall that $\Omega$ satisfies the upper bounds \eqref{form130} for all $\lambda \in \Lambda$ and $t \in \mathcal{T}(\lambda)$. This means that the hypothesis \eqref{form54} of Theorem \ref{thm4} is valid with constant $Y_{\lambda} = \delta^{-\kappa s/100}\lambda^{s}|P_{0}|_{\lambda}$.

We also recall from Section \ref{s:regularisation} that our set $P_{0}$ is $\Lambda$-uniform (without loss of generality), and at \eqref{form102} we denoted $X_{\lambda} := |P_{0} \cap \mathbf{p}|_{\delta} = |P_{0}|/|P_{0}|_{\lambda}$ for $\mathbf{p} \in \mathcal{D}_{\lambda}(P_{0})$, and $\lambda \in \Lambda$.

We have now verified the hypotheses of Theorem \ref{thm4}. Recall from the previous section that every rectangle $R \in \bar{\mathcal{R}}^{\delta}_{\sigma}$ has type $(\geq \bar{m},\geq \bar{n})_{\epsilon_{\mathrm{max}}}$ relative to $(W,B,\{E(p)\})$. Therefore, we may infer from Theorem \ref{thm4} that
\begin{displaymath} \frac{|W| M_{\sigma}}{m} \stackrel{\eqref{form123}}{\lessapprox_{\delta}} |\bar{\mathcal{R}}^{\delta}_{\sigma}| \leq \delta^{-\kappa s/100} \left[ \left(\frac{|W||B|}{\bar{m}\bar{n}} \right)^{3/4} (X_{\lambda}Y_{\lambda})^{1/2} + \frac{|W|}{\bar{m}} \cdot X_{\lambda}Y_{\lambda} + \frac{|B|}{\bar{n}} \cdot X_{\lambda}Y_{\lambda}\right]. \end{displaymath}
Here
\begin{displaymath} X_{\lambda}Y_{\lambda} \leq (|P_{0}|/|P_{0}|_{\lambda}) \cdot (\delta^{-\kappa s/100}\lambda^{s}|P_{0}|_{\lambda}) = \delta^{- \kappa s/100}\lambda^{s}|P_{0}|. \end{displaymath}
We also recap from \eqref{form127} that $m \lesssim_{\epsilon,\kappa} \delta^{-\zeta}n \leq \delta^{-\zeta - \epsilon_{\mathrm{max}}}\bar{n}$ (where $\zeta < \kappa s/10$), and from \eqref{form120} that $|B| \leq \delta^{-2\epsilon_{\mathrm{max}}}|W|$. Recalling from \eqref{def:n} that $n \geq \delta^{- \kappa + \zeta}$, and from \eqref{sizeP} that $|P_{0}| \leq \delta^{-s - \epsilon}$, we may rearrange and simplify the estimate above to the form
\begin{align} M_{\sigma} & \leq \delta^{-\kappa s/100 - \zeta - O(\epsilon_{\mathrm{max}})} \left[ |W|^{1/2} \cdot \delta^{\kappa/2} \cdot (\delta^{- \kappa s/100}\lambda^{s}|P_{0}|)^{1/2} + \delta^{- \kappa s/100}\lambda^{s}|P_{0}| \right] \notag\\
&\label{form161} \leq \delta^{-\kappa s/100 - \kappa s/10 - O(\epsilon_{\mathrm{max}})} \left[ |W|^{1/2} \cdot (\lambda/\delta)^{s/2} \cdot \delta^{\kappa/2 - \kappa s/100} + \delta^{-\kappa s/100}(\lambda/\delta)^{s}\right]. \end{align}
To derive a contradiction from this estimate, recall from \eqref{form160} that
\begin{equation}\label{form162} M_{\sigma} \geq \delta^{2\epsilon} \left(\frac{\sqrt{\lambda t}}{\delta} \right)^{s} \geq \delta^{2\epsilon - \kappa s/5} \left(\frac{\lambda}{\delta} \right)^{s}. \end{equation}
The second inequality follows from our restriction to pairs $(\lambda,t)$ with $\lambda \leq \delta^{\kappa/10}t$ (recall \eqref{form170}, and that $\lambda = \lambda_{k + 1}$). These inequalities show that the second term in \eqref{form161} cannot dominate the left hand side, provided that $\epsilon_{\mathrm{max}}$ is chosen small enough in terms of $\kappa,s$, and finally $\delta > 0$ is sufficiently small in terms of all these parameters.

To produce a contradiction with the counter assumption formulated above Section \ref{s1}, it remains to show that the first term in \eqref{form161} cannot dominate $M_{\sigma}$. Since $P_{0}$ is a $(\delta,s,\delta^{-\epsilon})$-set, and $W \subset P \subset P_{0}$ is contained in a ball of radius $t$, we have $|W| \leq \delta^{-\epsilon}t^{s}|P_{0}| \leq \delta^{-\epsilon}(t/\delta)^{s}$. Therefore, the first term in \eqref{form161} is bounded from above by
\begin{displaymath} \delta^{\kappa/2 - \kappa s (1/50 + 1/10) - O(\epsilon_{\mathrm{max}})}\left(\frac{\sqrt{\lambda t}}{\delta} \right)^{s} \leq \delta^{\kappa s/5}\left(\frac{\sqrt{\lambda t}}{\delta} \right)^{s}, \end{displaymath}
provided that $\epsilon_{\mathrm{max}} > 0$ is small enough in terms of $\kappa,s$. Evidently, the number above is smaller than the lower bound for $M_{\sigma}$ recorded in \eqref{form162}, provided that $\epsilon,\epsilon_{\mathrm{max}},\delta > 0$ are small enough in terms of $\kappa,s$. We have now obtained the desired contradiction.

To summarise, we have now shown that case (1) in the construction of the sequence $\{\mathbf{G}_{j}\}$ cannot occur as long as $\kappa_{h - j} > \kappa$. As explained at and after \eqref{form163a}, this shows that we may define $G_{k + 1} := \mathbf{G}_{j}$ for a suitable index "$j$", and this set $G_{k + 1}$ satisfies \eqref{form115}. This completes the proof of Proposition \ref{prop8}, then the proof of Proposition \ref{prop7}, and finally the proof of Theorem \ref{thm5}.


\subsection{Deriving Theorem \ref{thm2} from Theorem \ref{thm5}}\label{s:thm2Proof} It clearly suffices to prove Theorem \ref{thm2} for all $\kappa \in (0,c]$, where $c > 0$ is a small absolute constant to be determined later. Fix $\kappa \in (0,c]$, and let $\epsilon = \epsilon(\kappa,s) > 0$ be so small that Theorem \ref{thm5} holds with constants $\kappa,s$.

Let $\Omega = \{(p,v) : p \in P \text{ and } v \in E(p)\}$ be a $(\delta,s,C)$-configuration, as in Theorem \ref{thm2}. There is no \emph{a priori} assumption in Theorem \ref{thm2} that the sets $P,E(p)$ are $\delta$-separated, but it is easy to reduce matters to that case; we leave this to the reader, and in fact we assume that $P \subset \mathcal{D}_{\delta}$ and $E(p) \subset \mathcal{S}_{\delta}(p)$ for all $p \in P$.

To prove Theorem \ref{thm2}, we need to find a subset $G \subset \Omega$ satisfying $|G| \geq \delta^{\kappa}|\Omega|$, and
\begin{equation}\label{form172} m_{\delta}(w \mid G) \lessapprox_{\delta} \delta^{-\kappa}, \qquad w \in \R^{2}. \end{equation}
We start by applying Theorem \ref{thm5} to $\Omega$ to find the subset $G \subset \Omega$ of cardinality $|G| \geq \delta^{\kappa}|\Omega|$. By the choice of "$\epsilon$" above, we then have
\begin{equation}\label{form168} m^{\delta^{-\epsilon},\kappa^{-1}}_{\delta,\lambda,t}(\omega \mid G) \leq \delta^{-\kappa}, \qquad \omega \in G. \end{equation}
We claim that if the absolute constant "$c > 0$" is chosen small enough (thus $\kappa^{-1} \geq c^{-1} > 0$ is sufficiently large), then \eqref{form168} implies that
\begin{equation}\label{form167} |\{(p',v') \in G : v' \in B(v,2\delta)\}| = m_{2\delta}((p,v) \mid G) \lessapprox_{\delta} \delta^{-\kappa}, \qquad (p,v) \in G. \end{equation}
Let us quickly check that this implies \eqref{form172} for all $w \in \R^{2}$. Indeed, if $m_{\delta}(w \mid G) > 0$, then there exists at least one pair $(p,v) \in G$ such that $w \in B(v,\delta)$. Now, it is easy to see from the definitions that $m_{\delta}(w \mid G) \leq m_{2\delta}((p,v) \mid G)$.

The idea for proving \eqref{form167} is to bound the total multiplicity function $m_{2\delta}$ from above by a suitably chosen partial multiplicity function $m_{\delta,\lambda,t}$. Fix $(p,v) = \omega \in G$. Then,
\begin{displaymath} m_{2\delta}(\omega \mid G) \leq \sum_{\lambda \leq t} m_{2\delta}(\omega \mid G_{\lambda,t}^{\delta^{-\epsilon}}(\omega)), \end{displaymath}
where $G_{\lambda,t}^{\delta^{-\epsilon}}(\omega) = \{(p',v') \in G : \delta^{\epsilon}\lambda \leq \Delta(p,p') \leq \delta^{-\epsilon}\lambda \text{ and } \delta^{\epsilon}t \leq |p - p'| \leq \delta^{-\epsilon}t\}$ as in Definition \ref{def:lambdaTNeighbourhood}, and the sum runs over some multiplicatively $\delta^{-\epsilon}$-dense sequences of $\delta \leq \lambda \leq t \leq 1$ (or even all dyadic values, this is not important here). In particular, there exists a fixed pair $(\lambda,t)$, depending on $\omega$, such that
\begin{displaymath} m_{2\delta}(\omega \mid G) \lessapprox_{\delta} m_{2\delta}(\omega \mid G_{\lambda,t}^{\delta^{-\epsilon}}(\omega)) = |\{(p',v') \in G^{\delta^{-\epsilon}}_{\lambda,t}(\omega) : v' \in B(v,2\delta)\}|. \end{displaymath}
Let $\{\omega_{j}\}_{j = 1}^{N} = \{(p_{j},v_{j})\}_{j = 1}^{N} \subset G^{\delta^{-\epsilon}}_{\lambda,t}(\omega)$ be an enumeration of the pairs on the right hand side. The points $\{p_{1},\ldots,p_{N}\}$ may not all be distinct. However, note that if $p_{i}$ is fixed, there are $\lesssim 1$ options $v' \in \mathcal{S}_{\delta}(p_{i})$ such that $v' \in B(v,2\delta)$ (since $v$ is fixed). Therefore, there is a subset of $\sim N$ pairs among $\{(p_{j},v_{j})\}$ such that the points $p_{j}$ are all distinct. Restricting attention to this subset if necessary, we assume that all the points $p_{j}$ are distinct.

Write $\sigma := \delta/\sqrt{\lambda t}$. For every index $j \in \{1,\ldots,N\}$, choose $(p_{j},\mathbf{v}_{j}) \in G^{\delta}_{\sigma}$ (the $(\delta,\sigma)$-skeleton of $G$) such that $(p_{j},v_{j}) \prec (p_{j},\mathbf{v}_{j})$. Automatically
\begin{displaymath} (p_{j},\mathbf{v}_{j}) \in (G^{\delta}_{\sigma})^{\delta^{-\epsilon}}_{\lambda,t}(\omega), \qquad 1 \leq j \leq N, \end{displaymath}
since the point "$p_{j}$" remained unchanged. We also note that $|v_{j} - \mathbf{v}_{j}| \lesssim \sigma$, and the pairs $(p_{j},\mathbf{v}_{j})$ are distinct because the points $p_{j}$ are. We claim that
\begin{equation}\label{form166} \kappa^{-1}R^{\delta}_{\sigma}(p_{j},\mathbf{v}_{j}) \cap \kappa^{-1}R^{\delta}_{\sigma}(\omega) \neq \emptyset, \qquad 1 \leq j \leq N, \end{equation}
provided that $\kappa \leq c$, and $c > 0$ is sufficiently small. Indeed, fix $1 \leq j \leq N$, and recall that $v_{j} \in B(v,2\delta)$. This immediately shows that $v_{j} \in 2R^{\delta}_{\sigma}(p,v) = 2R^{\delta}_{\sigma}(\omega)$, since $\sigma \geq \delta$. On the other hand, $v_{j} \in S(p_{j})$, and $|v_{j} - \mathbf{v}_{j}| \lesssim \sigma$, so also $v_{j} \in CR^{\delta}_{\sigma}(p_{j},\mathbf{v}_{j})$ for some absolute constant $C \geq 1$. Now, \eqref{form166} holds for all $\kappa^{-1} \geq c^{-1} \geq \max\{2,C\}$.

We have now shown that
\begin{displaymath} m_{2\delta}(\omega \mid G) \lessapprox_{\delta} N \leq |\{\omega' \in (G^{\delta}_{\sigma})^{\delta^{-\epsilon}}_{\lambda,t}(\omega) : \kappa^{-1}R^{\delta}_{\sigma}(\omega') \cap \kappa^{-1}R^{\delta}_{\sigma}(\omega) \neq \emptyset\}| = m^{\delta^{-\epsilon},\kappa^{-1}}_{\delta,\lambda,t}(\omega \mid G). \end{displaymath}
Recalling \eqref{form168}, this proves \eqref{form167}, and consequently Theorem \ref{thm2}.

\appendix

\section{Proof of Proposition \ref{PYZ}} \label{app}
We complete the proof of Proposition \ref{PYZ} in this appendix. For the reader's convenience, we recall the statement of Proposition \ref{PYZ} here.

\bigskip \noindent
\textbf{Proposition 63.} {\it Let $A \geq 100$ and $\delta \leq \sigma \leq 1$, and let
$\mathcal{R}$ be a family of pairwise $100$-incomparable
$(\delta,\sigma)$-rectangles. Suppose also that there exists a fixed ($\delta,\sigma)$-rectangle
  $\mathbf{R}$ such that the union of the rectangles in $\mathcal{R}$ is contained in $A\mathbf{R}$. Then, $|\mathcal{R}| \lesssim A^{10}$.}

 \bigskip

As mentioned in Section \ref{s:comparableRectangles}, we first need several auxiliary definitions and lemmas.

\begin{definition}\label{d:Graph}
We denote by $\pi_L\colon \mathbb{R}^2 \to L$ the orthogonal
projection onto a $1$-dimensional subspace $L$ in $\mathbb{R}^2$.
If $\mathbf{I}\subset L$ is a fixed segment, $p\in \mathbb{R}^3$
and $v\in S(p)$ are such that $\pi_L(v)\in \mathbf{I}$, then we
denote by $\Gamma_{\mathbf{I},p,v}$ the connected component of
$\pi_L^{-1}(\mathbf{I})\cap S(p)$ containing $v$.
\end{definition}

The set $\Gamma_{\mathbf{I},p,v}$ need not be a graph over
$\mathbf{I}$ in general. However, given  a rectangle $\mathbf{R}$
and  a family $\mathcal{R}$ of $(\delta,\sigma)$-rectangles as in
Proposition \ref{PYZ} with a suitable upper bound on $\sigma$, we
now show how to select a subfamily
$\mathcal{R}_{\ast}\subset\mathcal{R}$ with $|\mathcal{R}^{\ast}| \geq |\mathcal{R}|/2$ such that both
 $A\mathbf{R}$, and the rectangles in  $\mathcal{R}_{\ast}$, look like
 neighborhoods of $2$-Lipschitz graphs over a fixed line $L$. By
a ``$2$-Lipschitz graph over $L$'' we mean the graph of a
$2$-Lipschitz function defined on a subset of $L$. In the argument below, we abbreviate $R^{\delta}_{\sigma}(p,v) =: R(p,v)$.

\begin{lemma} \label{lipgragh}
  Let $A \geq 1$, $\delta \le \sigma\leq A\sigma \leq \sigma_0:=1/600$.
   Assume that $\mathcal{R}$ is a finite family of $(\delta,\sigma)$-rectangles,
   all contained in $A\mathbf{R}$, where $\mathbf{R}=R(\mathbf{p},\mathbf{v})$
   is another $(\delta,\sigma)$-rectangle. Then there exists a $1$-dimensional subspace $L \subset \R^{2}$,
   an interval $\mathbf{I} \subset L$
   and a subfamily $\calR_\ast \subset \calR$ with $|\calR_\ast|\ge |\calR|/2$ such that
  \begin{enumerate}
    \item\label{i:lipgraph} $\pi_L(A\mathbf{R}) \subset \mathbf{I}$ and $\Gamma_{\mathbf{I},\mathbf{p},\mathbf{v}}$ is a $2$-Lipschitz graph over $\mathbf{I}$;
    \item\label{ii:lipgraph} for each $R(p,v) \in \calR_\ast$:\\ $\pi_L(R) \subset
    \mathbf{I}$
    and
   $\Gamma_{\mathbf{I},p,v}$ is a $2$-Lipschitz graph over
$\mathbf{I}$.
  \end{enumerate}
\end{lemma}

 \begin{proof} First, we find the subspace $L$ and the subfamily $\calR_\ast
\subset \calR$. Let
\begin{displaymath}
J(\mathbf{p},\mathbf{v}) := S(\mathbf{p}) \cap
B(\mathbf{v},\tfrac{1}{100}) \quad \text{and} \quad J(p,v) := S(p)
\cap B(v,\tfrac{1}{100}).
\end{displaymath} These are arcs  on the circles
$S(\mathbf{p})$ and $S(p)$ which contain the ``core arcs''
$S(\mathbf{p}) \cap B(\mathbf{v},A\sigma)$ and $S(p) \cap
B(v,\sigma)$ of the rectangles $A\mathbf{R}$ and $R$ respectively.
We claim that $L$ can be chosen from one of the three lines
\begin{displaymath}
L_{1} = \mathrm{span}\{(1,0)\},\quad L_{2} =
\mathrm{span}\{(1,\sqrt{3})\},\quad L_{3} = \mathrm{span}\{(-1,\sqrt{3})\}
\end{displaymath}
such that $J(\mathbf{p},\mathbf{v})$ and $J(p,v)$ are
$2$-Lipschitz graphs over $L$ for at least half of the rectangles
$R(p,v)\in\mathcal{R}$. The idea is that the arc
$J(\mathbf{p},\mathbf{v})$ (resp. $J(p,v)$) is individually a
$2$-Lipschitz graph over any line which is sufficiently far from
perpendicular to (any tangent of) that arc. For
$J(\mathbf{p},\mathbf{v})$ (resp. $J(p,v)$), this is true for at
least two of the lines among $\{L_{1},L_{2},L_{3}\}$. We give some
details to justify this claim.

For every circle $S(x,r)$ and every line $L$, there exists a
segment $I$ of length $4r/\sqrt{5}$, centered at $\pi_L(z)$, such
that the two components of $\pi_L^{-1}(I)\cap S(x,r)$ are
$2$-Lipschitz graphs over $I$; see the explanation around
\eqref{eq:VertDiff}. The constant ``$1/100$'' in the definition of
$J(p,v)$ has been chosen so small that for each $v\in S(p)$, there
are two choices of lines $L_i$ such that $\pi_{L_i}(J(p,v))$ is
contained in the segment on $L_i$ over which the corresponding arc
of $S(p)$ is a $2$-Lipschitz graph. This also uses the fact that
we are only considering parameters $p=(x,r)\in \mathbf{D}$, so
that $r\geq 1/2$.

For instance, if $p=((0,0),r)$ and
$v=(-\frac{\sqrt{3}}{2}r,\tfrac{1}{2}r)$, then $J(p,v)$ is clearly
a $2$-Lipschitz graph over the line $L_2$, which is perpendicular
to the direction of $v$, but $J(p,v)$ is also a $2$-Lipschitz
graph over the horizontal line $L_1$ since
\begin{displaymath}
\pi_{L_1}(J(p,v))\subset
\pi_{L_1}(B(v,1/100))=[-\tfrac{\sqrt{3}}{2}r-\tfrac{1}{100},-\tfrac{\sqrt{3}}{2}r+\tfrac{1}{100}]\subset
[-\tfrac{2}{\sqrt{5}}r,\tfrac{2}{\sqrt{5}}r].
\end{displaymath}
(By the same argument $J(p,v)$ is a $2$-Lipschitz graph over $L_1$
for  any $v=(r\cos \varphi,r\sin \varphi)$ with $\varphi \in
[\pi/6,5\pi/6]$).

Without loss of generality, we assume in the following that
$J(\mathbf{p},\mathbf{v})$ is a $2$-Lipschitz graph over $L_1$ and
$L_2$. For $1 \le i \leq 2$, define
$$  \calR_{i} := \{R(p,v) \in \calR : J(p,v) \mbox{ is a $2$-Lipschitz graph over $L_i$}\}. $$
We have $|\calR| \le |\calR_{1}| + |\calR_{2}|$. Hence if $|\calR_{1}| \ge |\calR|/2$, we choose
$L=L_1$ and $\calR_\ast=\calR_{1}$. Otherwise, we choose $L=L_2$ and $\calR_\ast=\calR_{2}$. For an
illustration, see Figure \ref{fig5}.
\begin{figure}
\begin{center}
\begin{overpic}[scale = 0.7]{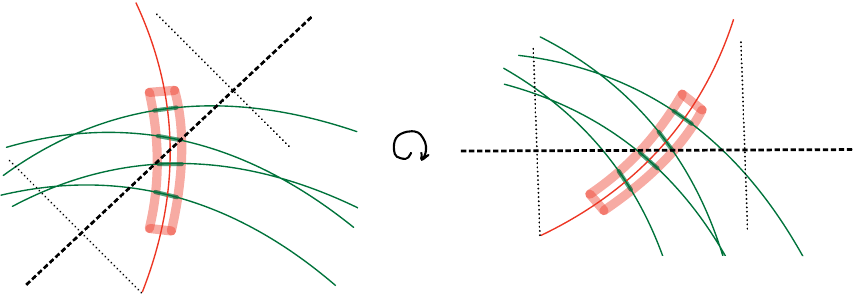}
\put(32,25){$L$} \put(6,27){$J(p,v)$} \put(24,2){$J(q,w)$}
\put(12,7){$\mathbf{I}$} \put(65,13){$\mathbf{I}$} \put(90,13){$L
= \R$}
\end{overpic}
\caption{Finding the line $L$. The fat red rectangle represents
$A\mathbf{R}$, and the smaller green rectangles inside $A\mathbf{R}$ represent the
rectangles $R' \in \mathcal{R}$.}\label{fig5}
\end{center}
\end{figure}
We assume with no loss of generality that $L = L_1 =
\mathrm{span}\{(1,0)\}$ and $\calR_\ast=\calR_{1}$.
We abbreviate $\pi := \pi_{L}$ and identify $L$ with $\mathbb{R}$
via $(x_1,0)\mapsto x_1$. Next, we note that
\begin{equation}\label{gragh1}
\mathbf{I} := [\pi(\mathbf{v})-\tfrac{1}{600},\pi(\mathbf{v})+\tfrac{1}{600}]\subset \pi(J(\mathbf{p},\mathbf{v})) \cap \bigcap_{R(p,v) \in \mathcal{R}_{\ast}} \pi(J(p,v)).
\end{equation}
This follows easily from the fact that $|v - \mathbf{v}| \leq 1/600$ and that $\pi$ restricted to $J(\mathbf{p},\mathbf{v})$ and $J(p,v)$ is $2$-Lipschitz; we omit the details. Since $J(\mathbf{p},\mathbf{v}),J(p,v)$ are $2$-Lipschitz graphs over $L$, the inclusion \eqref{gragh1} shows that $\Gamma_{\mathbf{I},\mathbf{p},\mathbf{v}},\Gamma_{\mathbf{I},p,v}$ are $2$-Lipschitz graphs over the segment $\mathbf{I}$. Moreover, it is clear that
\begin{displaymath} \pi(R) \subset \pi(A\mathbf{R}) \subset \pi(B(\mathbf{v},\tfrac{1}{600})) = \mathbf{I}, \qquad R \in \mathcal{R}_{\ast}. \end{displaymath}
This completes the proof of the lemma. \end{proof}

We will apply (a corollary of) Lemma \ref{lipgragh} to the
rectangle $\mathbf{R}$ in Proposition \ref{PYZ}. For that purpose
we may assume without loss of generality that the line $L$ given
by Lemma \ref{lipgragh} is $\mathrm{span}\{(1,0)\}$, and we restrict
the following discussion to this case. This convention leaves for
each graph $\Gamma_{\mathbf{I},p,v}$ two possibilities: it is
contained either on an `upper' or on a `lower' half-circle. For
$p=(x,r)=(x_1,x_2,r) \in \mathbb{R}^2\times (0,\infty)$, we write
the circle $S(p)$ as the union of two graphs over $L$ as follows
   $$S(p)=S(x,r) =\{ (y_1,y_2) \in \mathbb{R}^2 : (y_1 -x_1)^2 + (y_2- x_2)^2 = r^2  \}=S_{+}(p)\cup S_{-}(p),$$
where
\begin{displaymath}
S_{\pm}(p)=\left\{(y_1,\pm \sqrt{r^2 - (y_1 -x_1)^2} + x_2)\colon
y_1\in [x_1-r,x_1+r]\right\}.
\end{displaymath}
Now for $p=(x_1,x_2,r)\in \mathbb{R}^2\times (0,\infty)$, we
define
\begin{equation}\label{eq:f_p}
f_{p,\pm}(\theta):=\pm \sqrt{r^2 - (\theta -x_1)^2} + x_2,\quad
\theta \in [x_1-r,x_1+r].
\end{equation}
We record for any $\theta \in (x_1-r,x_1+r)$,
\begin{equation}\label{eq:f_deriv}
  f_{p,\pm}'(\theta) = \mp\frac{\theta-x_1}{\sqrt{r^2-(\theta-x_1)^2}}
   \quad \mbox{ and } \quad f_{p,\pm}''(\theta) = \mp\frac{r^2}{[r^2-(\theta-x_1)^2]^{3/2}}.
   \end{equation}
The functions $f_{p,\pm}$ are $2$-Lipschitz on
$[x_1-\tfrac{2}{\sqrt{5}}r,x_1+\tfrac{2}{\sqrt{5}}r]$, and this is
the largest interval with that property. At the endpoints of it,
the corresponding function values are
\begin{equation}\label{eq:VertDiff}
f_{p,\pm}(x_1-\tfrac{2r}{\sqrt{5}})=f_{p,\pm}(x_1+\tfrac{2r}{\sqrt{5}})=x_2\pm
\tfrac{r}{\sqrt{5}}.
\end{equation}
The tangents to $S(p)$ in the respective points on $S(p)$ have
precise slopes $+2$ or $-2$.

For simplicity, we denote $\pi=\pi_L: (y_1,y_2)\mapsto y_1$.
Assume that $\mathbf{I}\subset L$ is an interval and consider
$p\in \mathbf{D}$ and $v\in S(p)$. If the arc
$\Gamma_{\mathbf{I},p,v}$ introduced in Definition \ref{d:Graph}
is a graph over $L$, then
   \begin{equation}\label{eq:UpperLower}
   \mbox{either}\quad  \Gamma_{\mathbf{I},p,v}= \pi^{-1}(\mathbf{I}) \cap S_{+}(p) \quad
   \mbox{or} \quad \Gamma_{\mathbf{I},p,v}= \pi^{-1}(\mathbf{I}) \cap
   S_{-}(p)
   \end{equation}
and $\Gamma_{\mathbf{I},p,v}$ is the graph of
${f_{p,+}}|_{\mathbf{I}}$ or ${f_{p,-}}|_{\mathbf{I}}$,
respectively. We may
   in the following assume that the rectangles  $R(p,v)\in \mathcal{R}_{\ast}$  given by Lemma \ref{lipgragh}  all
yield functions of the same type, either all associated to upper
half-circles, or all associated to lower half-circles. This type may not however be the same as for the
rectangle $\mathbf{R}$, cf.\ Figure \ref{fig5}.

\begin{lemma}\label{cinematic}
  Under the assumptions of Lemma \ref{lipgragh} (with
  $L=\mathrm{span}\{(1,0)\}$), there exists a subset
  $\mathcal{R}_{\ast}\subset \mathcal{R}$ with
  $|\mathcal{R}_{\ast}|\geq |\mathcal{R}|/4$ such that the
  conclusions \eqref{i:lipgraph}-\eqref{ii:lipgraph} hold and
  additionally, either $\Gamma_{\mathbf{I},p,v}= \pi_1^{-1}(\mathbf{I}) \cap
  S_{+}(p)$ for all $R(p,v)\in \mathcal{R}_{\ast}$, or  $\Gamma_{\mathbf{I},p,v}= \pi_1^{-1}(\mathbf{I}) \cap
  S_{-}(p)$ for all $R(p,v)\in \mathcal{R}_{\ast}$.
\end{lemma}

\begin{proof} Observation \eqref{eq:UpperLower} shows that the additional property
can  be arranged by discarding at most half of the elements in the
original family $\mathcal{R}_{\ast}$ given by Lemma
\ref{lipgragh}.
\end{proof}

Even with this additional assumption in place, the family
$\mathcal{R}_{\ast}$ is not quite of the same form as the families
of graph neighborhoods considered in \cite{2022arXiv220702259P},
but it is also not too different. For arbitrary $\eta>0$ and
subinterval $I\subset \mathbf{I}$ in the domain of $f_{p,\pm}$, we
define the \emph{vertical $\eta$-neighborhood}
   $$f^{\eta}_{p,\pm}(I):= \left\{(y_1,y_2) \in I\times\mathbb{R} : f_{p,\pm}(y_1)- \eta \le y_2 \le f_{p,\pm}(y_1)+\eta\right \}.$$
Moreover, for any $\eta\in (0,1/200]$, if $ f_{p,\pm}:\mathbf{I}
\to \mathbb{R}$ is $2$-Lipschitz, then
 \begin{equation}\label{fact421b}
\pi^{-1}(I) \cap S^{\eta}(p) \subset f^{4\eta}_{p,+}(I)\cup
f^{4\eta}_{p,-}(I).
   \end{equation}
Here, the upper bound on $\eta$ ensures that the points on $S(p)$
which are $\eta$-close to points in $\pi^{-1}(I) \cap S^{\eta}(p)$
lie in the part of the graph of $f_{p,\pm}$ where the Lipschitz
constant is small enough for the inclusion \eqref{fact421b} to
hold. In particular, if
   $R=R(p,v)$ is a rectangle with $I=\pi(R)\subset \mathbf{I}$ and $\eta=\delta<1/200$, and
   if $f_p\in \{f_{p,+},f_{p,-}\}$ is such that $\Gamma_{\mathbf{I},p,v}$ is
   the graph of $f_p$, then the inclusion in
   \eqref{fact421b} yields
   \begin{equation}\label{fact422}
R \subset f^{4\delta}_{p}(\pi(R))
   \end{equation}
   since  $ R \subset \pi^{-1}(\pi(R))\cap S^{\delta}(p)$. A
   priori,  \eqref{fact421b} only yields  $R\subset f^{4\delta}_{p,+}(\pi(R))\cup
f^{4\delta}_{p,-}(\pi(R))$, but the conditions $\delta<1/200$,
$\pi(R)\subset \mathbf{I}$  and the assumptions on
$\Gamma_{\mathbf{I},p,v}$ ensure that either $R\subset
f^{4\delta}_{p,+}(\pi(R))$ or
    $R\subset f^{4\delta}_{p,-}(\pi(R))$.

We will also need an opposite inclusion for enlarged rectangles.
Let $\delta\leq \sigma$, $R=R(p,v)$ with
$\pi(R)\subset \mathbf{I}$ and $f_{p}:\mathbf{I}\to\mathbb{R}$  be
$2$-Lipschitz with graph equal to $\Gamma_{\mathbf{I},p,v}$. Then,
for any $C\geq 1$, if $I\subset \mathbf{I}$ is an interval with
$|I|\leq C\sigma$ and such that
 $\pi(R)\subset I$, then
\begin{equation}\label{fact422b}
f^{C\delta}_{p}(I)\subset 4C R.
   \end{equation}
The inclusion $f^{C\delta}_{p}(I)\subset S^{4C\delta}(p)$ is
clear. To prove that also $f^{C\delta}_{p}(I)\subset
B(v,4C\sigma)$, consider an arbitrary point
  $y=(y_1,y_2)\in  f^{C\delta}_{p}(I)$. Since $\pi(R)\subset
  I\subset \mathbf{I}$ and $\Gamma_{\mathbf{I},p,v}$ is the graph of $f_p$, there exists $\theta\in I$ such that
  $v=(\theta,f_{p}(\theta))$ and using the $2$-Lipschitz continuity of $f_p$ on $\mathbf{I}\supset I$, we can estimate
  \begin{align*}
  |y-v|\leq
  |y_2-f_{p}(y_1)|+|(y_1,f_{p}(y_1))-(\theta,f_{p}(\theta))|\leq
  C\delta + \sqrt{5}|y_1-\theta|\leq C \delta
  +3|I|\leq 4C \sigma,
  \end{align*}
  concluding the proof of \eqref{fact422b}.
 In order to apply arguments that were stated in
     \cite{2022arXiv220702259P} for certain $C^2$ functions, we need
     a preliminary result about the behavior of $p\mapsto
     f_{p,\pm}$ with respect to the $C^2(\mathbf{I})$-norm.

\begin{lemma}\label{Nincl5} There exists an absolute constant
$K\geq 1$ such that for all $p,p'\in \mathbf{D}$, if
$\mathbf{I}\subset \mathbb{R}$ is an interval so that
$f_{p,+},f_{p',+}:\mathbf{I} \to \mathbb{R}$ are $2$-Lipschitz,
then
\begin{equation}\label{eq:f_p_dist}
\|f_{p,+}-f_{p',+}\|_{C^2(\mathbf{I})}\leq K|p-p'|.
\end{equation}
\end{lemma}
The corresponding result for the pair $(f_{p,-},f_{p',-})$ is also true, but not needed.
\begin{proof} We abbreviate  $f_p=f_{p,+}$ for $p=(x_1,x_2,r)$.
The norm $\|f_p\|_{C^2(\mathbf{I})}$ is uniformly bounded for all
$p$ and $\mathbf{I}$ as in the statement of the lemma. Indeed,
since $f_p$ is assumed to be $2$-Lipschitz on $\mathbf{I}$, we
have $ f_p(\theta)\in  [x_2+\tfrac{r}{\sqrt{5}},x_2+r]$,  $\theta
\in \mathbf{I}$, by the discussion around \eqref{eq:VertDiff} and
hence
\begin{equation}\label{eq:rBound}
 \tfrac{r}{\sqrt{5}}\leq \sqrt{r^2-(\theta-x_1)^2}\leq r
 \end{equation} for all
$\theta\in \mathbf{I}$. Since $p\in \mathbf{D}$, this yields a
uniform upper bound for $\|f_p\|_{C^2(\mathbf{I})}$, recalling
 the expressions stated in \eqref{eq:f_p}--\eqref{eq:f_deriv} for $f_p$ and its
derivatives. Thus it suffices to prove \eqref{eq:f_p_dist} under
the assumption that $|p-p'|\leq 1/400$.

For arbitrary $p,p'\in  \mathbb{R}^2 \times (0,\infty)$, we have
\begin{equation}\label{eq:annulus_incl}
S(p')\subset S^{2|p-p'|}(p).
\end{equation}
In particular,
\begin{displaymath}
(\theta,f_{p'}(\theta))\in \pi^{-1}(\mathbf{I})\cap S^{2|p-p'|}(p)
 \overset{\eqref{fact421b}}{\subset}
 f^{8|p-p'|}_{p,+}(\mathbf{I})\cup   f^{8|p-p'|}_{p,-}(\mathbf{I}),\quad \theta \in \mathbf{I}.
\end{displaymath}
 Our
upper bound $|p-p'| \leq 1/400$ and the assumption $p,p'\in
\mathbf{D}$ rule out the possibility that
$(\theta,f_{p'}(\theta))\in f^{8|p-p'|}_{p,-}(\mathbf{I})$.
Indeed, by  \eqref{eq:VertDiff}, we know on the one hand that
\begin{displaymath}
f_{p'}(\theta)\in [x_2'+\tfrac{r'}{5},x_2'+r'].
\end{displaymath}
On the other hand, again by \eqref{eq:VertDiff}, if
$(\theta,f_{p'}(\theta))\in f^{8|p-p'|}_{p,-}(\mathbf{I})$, then
necessarily
\begin{displaymath}f_{p'}(\theta)\in
\left[x_2-r-8|p-p'|,x_2-\tfrac{r}{\sqrt{5}}+8|p-p'|\right].
\end{displaymath}
The two inclusions are compatible only if
\begin{displaymath}
x_2'+\tfrac{r'}{5} \leq x_2-\tfrac{r}{\sqrt{5}}+8|p-p'|,
\end{displaymath}
or in other words, if
\begin{displaymath}
8|p-p'|\geq x_2'-x_2 + \tfrac{r+r'}{\sqrt{5}}.
\end{displaymath}
Since this implies that $9|p-p'|\geq 1/\sqrt{5}$, it is
impossible. Thus we conclude that
 \begin{displaymath}
 (\theta,f_{p'}(\theta))\in
 f^{8|p-p'|}_{p}(\mathbf{I}),\quad \theta \in \mathbf{I}.
 \end{displaymath}
  In particular,  it follows
   \begin{equation}\label{eq:SupNormf}
\|f_{p'}-f_{p}\|_{\infty}:=\sup_{\theta \in \mathbf{I}}
|f_{p'}(\theta)
   -f_{p}(\theta)|\leq 8|p-p'|.
   \end{equation}
   We write again in coordinates $p=(x_1,x_2,r)$. The estimate \eqref{eq:rBound}, established at the beginning of the proof,
   combined with
\eqref{eq:SupNormf}, the assumption $p,p'\in \mathbf{D}$,
 and a direct computation gives
   $$  \|f_{p'}'-f_{p}'\|_{\infty}  \lesssim |p-p'|, \quad  \|f_{p'}''-f_{p}''\|_{\infty}  \lesssim |p-p'|,  $$
  with uniform implicit constants. Together with
   \eqref{eq:SupNormf}, this concludes the proof. \end{proof}

To prove Proposition \ref{PYZ}, we have to deal with rectangles
that are $100$-incomparable in the sense of Definition
\ref{def:comparability}. We now record a simple consequence of
this property that will be easier to apply when working with the
`graph neighborhood rectangles'.

\begin{lemma}\label{Nincl4} Let $0<\delta\leq \sigma\leq 1/200$ and
assume that $R=R(p,v), R'=R(p',v')$ are
$100$-incomparable $(\delta,\sigma)$-rectangles with $p,p'\in
\mathbf{D}$. Suppose further that there exists an interval
$\mathbf{I}$ such that $\Gamma_{\mathbf{I},p,v}\subset S_{+}(p)$,
$\Gamma_{\mathbf{I},p',v'}\subset S_{+}(p')$, $\pi(R)\cup
\pi(R')\subset \mathbf{I}$ and so that $f_{p,+}$ and $f_{p',+}$
are $2$-Lipschitz on $\mathbf{I}$.

Then, if $R(p,v)
\cap R(p',v') \ne \emptyset$,
   there exists a point $\theta \in \pi(R(p,v)\cup R(p',v'))$ such that
   $$|f_{p,+}(\theta)-f_{p',+}(\theta)| > 20\delta.$$
\end{lemma}

\begin{proof} We denote $I:=\pi(R(p,v)\cup R(p',v'))$ and observe
that this is an interval since $R(p,v) \cap R(p',v') \ne
\emptyset$. By assumption $I\subset \mathbf{I}$. To prove the
lemma, we argue by contradiction and assume for all $\theta \in
I$,
  \begin{equation}\label{eq:Vert8Delta}
  |f_{p,+}(\theta)-f_{p',+}(\theta)| \le 20\delta.
  \end{equation} This implies that
  \begin{equation}\label{eq:R(p',v')Incl}
  R(p',v') \overset{\eqref{fact422}}{\subset} f^{4\delta}_{p',+}(I) \overset{\eqref{eq:Vert8Delta}}{\subset}
  f^{24\delta}_{p,+}(I).
  \end{equation}
Since
  \begin{displaymath}
  |I|=|\pi(R(p,v)\cup R(p',v'))|\leq |\pi(B(v,\sigma))| +
  |\pi(B(v',\sigma))|\leq 4\sigma,
  \end{displaymath}
we can use \eqref{fact422b} to conclude from
\eqref{eq:R(p',v')Incl} that $R(p',v')\subset 100 R(p,v)$, contradicting the $100$-incomparability of $R(p,v)$ and $R(p',v')$. This concludes the proof.
\end{proof}

We are now in a position to show Proposition \ref{PYZ}:

\begin{proof}[Proof of Proposition \ref{PYZ}]
Let $\mathbf{R}=R(\mathbf{p},\mathbf{v})$ be a fixed
$(\delta,\sigma)$-rectangle as in the statement of the
proposition. Since every $(\delta,\sigma)$-rectangle
$R=R(p,v)\subset A\mathbf{R}$ is contained in
$S^{A\delta}(\mathbf{p},\mathbf{v})\cap S^{A\delta}(p,v)$, and
since $S^{A\delta}(\mathbf{p},\mathbf{v})\cap S^{A\delta}(p,v)$
can be covered by boundedly many
$(A\delta,\sqrt{A\delta/|p-\mathbf{p}|})$-rectangles according to
Lemma \ref{lemma5}, it follows that
\begin{displaymath}
 \sigma
  \lesssim \sqrt{\tfrac{A\delta}{|p - \mathbf{p}|}}.
  \end{displaymath} This holds
  in particular for all $R=R(p,v)\in \mathcal{R}$.
Hence defining
\begin{displaymath}
P_\mathbf{R}:=\{ p \in \mathbf{D}\colon R(p,v)\in\mathcal{R}
\text{ for some }v\in S(p)\},
\end{displaymath} we know that there exists a universal constant
$\mathbf{C}>0$ such that
\begin{equation}\label{eq:P_in_Ball}
P_{\mathbf{R}} \subset
B\big(\mathbf{p},\mathbf{C}\tfrac{A\delta}{\sigma^2}\big) \subset
\mathbb{R}^3.
\end{equation}
We make one more observation about the family $\mathcal{R}$, which
will show in particular that it is finite. Namely, if $R(p,v)\in
\mathcal{R}$, then
\begin{equation}\label{eq:UpperBoundClose}
|\{R(p',v')\in \mathcal{R}\colon |p-p'|\leq \delta\}|\lesssim A.
\end{equation}
Indeed, let $R(p_{1},v_{1}),\ldots,R(p_{n},v_{n})$ be a listing of the rectangles on the left. Then $v_{i} \in A\mathbf{R} \cap S^{3\delta}(p)$ for all $1 \leq i \leq n$. Note that $\diam(A\mathbf{R}) \sim A\sigma$. Now, if $n \geq CA$ for a suitable absolute constant $C \geq 1$, we may find two elements $v_{i},v_{j}$ with $|v_{i} - v_{j}| \leq 10\sigma$. But since $|p_{i} - p_{j}| \leq 2\delta$, it would follow that the rectangles $R(p_{i},v_{i})$ and $R(p_{j},v_{j})$ are $100$-comparable, contrary to our assumption. This proves \eqref{eq:UpperBoundClose}.

We divide the remaining proof into two cases according to the size
of $\sigma$, using the threshold $\sigma_0=1/600$ from Lemma
\ref{lipgragh}. The first case, where $\sigma$ is close to $1$,
will follow roughly speaking because the rectangles in
$\mathcal{R}$ are so curvy that their containment in a common
rectangle $A\mathbf{R}$ forces $P_{\mathbf{R}}$ to be contained in
a $\sim_A \delta$ ball. The second case falls under the regime
where the assumptions of Lemmas \ref{lipgragh} and \ref{cinematic}
are satisfied, and we can work with rectangles that are
essentially neighborhoods of graphs over a fixed line.

\medskip
\textbf{Case 1 ($A^{-1}\sigma_0 < \sigma \le 1$).} Inserting the
lower bound for $\sigma$ into \eqref{eq:P_in_Ball}, we
find that there exists a universal constant $\mathbf{C}>0$
(possibly larger than before) such that
$$P_{\mathbf{R}} \subset B(\mathbf{p},\mathbf{C}A^3\delta).$$
Hence, $P_{\mathbf{R}}$ can be covered by $N\lesssim
(\mathbf{C}A^3)^3$ balls $B_1,\ldots,B_N$ of radius $\delta/2$. By
\eqref{eq:UpperBoundClose}, for every $i=1,\ldots,N$, there are $\lesssim A$ rectangles $R(p,v)\in \mathcal{R}$ with $p\in
B_i$. We deduce that
\begin{displaymath}
|\mathcal{R}| \lesssim (\mathbf{C}A^3)^3 \,\max_{i\in
\{1,\ldots,N\}} |\{R(p,v)\in \mathcal{R}\colon p\in
P_{\mathbf{R}}\cap B_i\}|\lesssim (\mathbf{C}A^3)^3 \cdot A \sim
A^{10}.
\end{displaymath}

\medskip
\textbf{Case 2 ($\sigma \leq A^{-1}\sigma_0$).} Let now
$\mathcal{R}_{\ast}\subset \mathcal{R}$ be the subfamily given by
Lemma \ref{cinematic}. Without loss of generality we may assume
that for every $R(p,v)\in \mathcal{R}_{\ast}$, we have
$\Gamma_{\mathbf{I},p,v}\subset S_{+}(p)$. To implement the
approach from   the proof of \cite[Lemma
3.15]{2022arXiv220702259P}, we need one more reduction to ensure
that the rectangles $R(p,v)$
 we consider
 give rise to functions $f_{p,+}$ that are sufficiently close
to each other in $C^2(\mathbf{I})$-norm. Using  Lemma
\ref{Nincl5}, this can be ensured if the parameters $p$ are
sufficiently close in $\mathbf{D}$. By \eqref{eq:P_in_Ball}, and
recalling $\mathrm{diam}\,\mathbf{D}\leq 2$, we know already that
\begin{equation}\label{eq:P_in_D}
P_{\mathbf{R,\ast}}:=\{ p\in \mathbf{D}\colon \text{ there is }
v\in S(p)\text{ with }R(p,v)\in\mathcal{R}_{\ast}\}\subset
B(\mathbf{p},\mathbf{A}t),
\end{equation}
where  $\mathbf{A} \lesssim A$ and
$$t:= \min\{\delta/\sigma^{2},2\}.$$
On the other hand, by \eqref{eq:UpperBoundClose}, we also know
that for each $p \in P_{\mathbf{R},\ast}$, there are at most
$\lesssim A$ many $v \in S(p)$ such that $R(p,v) \in
\mathcal{R}_{\ast}$. As a result,
\begin{equation}\label{eq:P_large}
|P_{\mathbf{R},\ast}| \gtrsim A^{-1}|\mathcal{R}_{\ast}|.
\end{equation}
Combining \eqref{eq:P_in_D} and \eqref{eq:P_large}, we may choose
a ball
\begin{equation}\label{incl9}
 B_0 \subset B(\mathbf{p},\mathbf{A}t)
\end{equation}
of radius $\frac{t}{2K}$, where $K\geq 1$ is the constant from
Lemma \ref{Nincl5},
 such that
\begin{equation}\label{eq:P_large2}
|P_{\mathbf{R},\ast} \cap B_0| \gtrsim
\mathbf{A}^{-3}|P_{\mathbf{R},\ast}|.
\end{equation}
We define a further subfamily
$$ \calR_\ast^\circ:=\{R(p,v)\in \mathcal{R}_{\ast} : p \in P_{\mathbf{R},\ast} \cap B_0\}.$$
Hence by \eqref{eq:P_large} and \eqref{eq:P_large2}
\begin{displaymath}
|\calR_\ast^\circ| \geq |P_{\mathbf{R},\ast} \cap B_0| \gtrsim
A^{-4}|\mathcal{R}_{\ast}|.
\end{displaymath}
Thus if we manage to show that $ |\calR_\ast^\circ| \lesssim A^3$, we can deduce that
\begin{displaymath}
A^{-4}|\mathcal{R}_{\ast}|\lesssim |\calR_\ast^\circ| \lesssim
A^{3} \quad \Longrightarrow \quad |\mathcal{R}_{\ast}|  \lesssim
{A}^{4}A^{3} \lesssim A^{10}.
\end{displaymath}
This will conclude the proof since $|\mathcal{R}_{\ast}|\sim
|\mathcal{R}|$ by Lemma \ref{cinematic}.

It remains to prove that $|\calR_\ast^\circ| \lesssim A^3$. Applying Corollary \ref{Nincl5}, we deduce that
\begin{equation}\label{tball}
   \|f_i-f_j\|_{C^2(\mathbf{I})} \le t \qquad p_i,p_j \in B_0,
\end{equation}
where $f_i:=f_{p_i,+}$ and $f_j:=f_{p_j,+}$. Following the
argument in \cite[Lemma 3.15]{2022arXiv220702259P}, we will show
that
\begin{equation}\label{form216} |\{R \in \calR_\ast^\circ : z \in R\}| \lesssim A, \qquad z \in \R^{2}. \end{equation}
This will give
\begin{displaymath}
|\calR_\ast^\circ| \cdot \delta \sigma \lesssim \int_{A\mathbf{R}}
\sum_{R \in \calR_\ast^\circ} \mathbf{1}_{R} \lesssim A \cdot
\mathrm{Leb}(A\mathbf{R}) \lesssim A^{3}\delta \sigma,
\end{displaymath} as desired.

To prove \eqref{form216}, fix $z = (\theta_{0},y_{0}) \in \R^{2}$
which is contained in, say, $N$ pairwise $100$-incomparable
$(\delta,\sigma)$-rectangles $R_{j} \in \calR_\ast^\circ$, for $1
\leq j \leq N$. The claim is that $N \lesssim A$. Note that
$\pi(R_j)$ necessarily contains the point $\theta_0 + \sigma/3$ or
$\theta_0 - \sigma/3$, and we can bound individually the
cardinality of the two subfamilies of $\{R_j:\, j=1,\ldots,N\}$
where one of the two options occur. Thus let us assume in the
following without loss of generality that $\theta_0 + \sigma/3 \in
\pi(R_j)$ for all $j$.

To show our claim, it suffices to establish the following two
inequalities:
\begin{equation}\label{form217}
|f_{i}'(\theta_{0}) - f_{j}'(\theta_{0})| \leq 100A \cdot
(\delta/\sigma), \qquad 1 \leq i,j \leq N,
\end{equation}
and
\begin{equation}\label{form218} |f_{i}'(\theta_{0}) - f_{j}'(\theta_{0})| \geq \delta/\sigma, \qquad 1 \leq i \neq j \leq N. \end{equation}
The first inequality will be based on
 the assumption that the rectangles in $\mathcal{R}$ are contained in $A\mathbf{R}$, and the second inequality uses the $100$-incomparability of the rectangles in $\calR_\ast^\circ$.

We give one argument that takes care both of the short rectangles
($\sigma \leq \sqrt{\delta}$),
 and the long rectangles ($\sigma \geq \sqrt{\delta}$) treated in  \cite{2022arXiv220702259P}.
Recalling the $C^2(\mathbf{I})$ bound \eqref{tball}, we have
\begin{equation}\label{form219}
 \|f_{i} -
f_{j}\|_{C^2(\mathbf{I})} \le t=\min\{\delta/\sigma^2,2\}.
\end{equation}
We apply this to prove \eqref{form217}. Let us denote $h := f_{i}
- f_{j}$, and let us assume to the contrary that $|h'(\theta_{0})|
> 100A \cdot (\delta/\sigma)$. Then, using \eqref{form219}, for all
 $\theta \in \pi(R_i)\cup
\pi(R_j)$ with $|\theta - \theta_{0}| \leq \sigma$, we have
\begin{displaymath}
|h'(\theta)| \geq |h'(\theta_{0})| - \|h'' \|_{\infty}|\theta -
\theta_{0}| \geq 100A \cdot (\delta/\sigma) -
\min\{\delta/\sigma^2,2\} \sigma > 99A \cdot (\delta/\sigma),
\end{displaymath}
using  $A \geq 1$. By \eqref{fact422} and the assumption that the
rectangles $R_{j}$ all intersect at $(\theta_{0},y_{0})$ and
$\theta_0\in \pi(R_j)\subset \mathbf{I}$, we have $|h(\theta_{0})|
\leq 8\delta$. We will combine this information with the lower
bound for $|h'|$ on the interval $\pi(R_i)\cup \pi(R_j)$ to reach
a contradiction with the assumption that $R_i\cup R_j \subset A
\mathbf{R}$. Recall that $\theta_0+\sigma/3 \in \pi(R_i) \cap \pi(R_{j})$. Then,
\begin{displaymath}
|h(\theta_{0} + \sigma/3)| \geq
|h(\theta_0+\sigma/3)-h(\theta_0)|-8\delta\geq 99A \cdot
(\delta/\sigma) \cdot \sigma/3 -  8\delta \geq 33A \cdot \delta -
8\delta > 25A\delta.
\end{displaymath}
But this is not consistent with the assumption that
\begin{displaymath} \{(\theta_{0} + \sigma/3,f_{i}(\theta_{0} + \sigma/3)),(\theta_{0} + \sigma/3,f_{j}(\theta_{0} + \sigma/3))\}
\subset R_{i} \cup R_{j} \subset A\mathbf{R}, \end{displaymath}
noting that the ``vertical'' thickness of $A\mathbf{R}$ is at most
$8A\delta$ since $A\mathbf{R} \subset
f_{\mathbf{p},+}^{4A\delta}(\pi(A\mathbf{R}))$ or  $A\mathbf{R}
\subset f_{\mathbf{p},-}^{4A\delta}(\pi(A\mathbf{R}))$ according
to \eqref{fact422}.

The proof of \eqref{form218} is similar. This time we make the
counter assumption that $|h'(\theta_{0})| < \delta/\sigma$. The
assumption $\theta_0 \in \pi(R_i\cap R_j)$ implies that
$\pi(R_i\cup R_j)$ is an interval contained in
$[\theta_{0}-2\sigma,\theta_{0}+2\sigma]$. Using \eqref{form219},
as above, this leads to the following estimate
\begin{displaymath}
|h'(\theta)|  \leq
|h'(\theta_0)|+\|h'\|_{\infty}|\theta-\theta_0|<
\tfrac{\delta}{\sigma} +
\min\left\{\tfrac{\delta}{\sigma^2},2\right\}2\sigma \leq
3\delta/\sigma, \qquad \theta \in \pi(R_i\cup R_j).
\end{displaymath} Finally, since $|h(\theta_{0})| \leq 8 \delta$, we
deduce from the preceding estimate that
\begin{displaymath}
|h(\theta)|\leq 8\delta+ (3\delta/\sigma) \cdot 2\sigma =
14\delta, \qquad \theta\in \pi(R_i\cup R_j).
\end{displaymath} This inequality
contradicts  Lemma \ref{Nincl4} and shows that the
counter-assumption cannot hold. This completes the proof of
\eqref{form218}, and thus the proof of Proposition \ref{PYZ}.
\end{proof}


\section*{Acknowledgments} 
We would like to thank the anonymous reviewers for reading the manuscript carefully, and for making many helpful suggestions.

\bibliographystyle{amsplain}


\begin{dajauthors}
\begin{authorinfo}[fk]
  Katrin F\"assler\\
  Department of Mathematics and Statistics\\ University of Jyv\"askyl\"a,
P.O. Box 35 (MaD)\\
FI-40014 University of Jyv\"askyl\"a\\
Finland\\
  katrin\imagedot{}s\imagedot{}fassler\imageat{}jyu\imagedot{}fi \\
\end{authorinfo}
\begin{authorinfo}[jl]
  Jiayin Liu\\
  Department of Mathematics and Statistics\\ University of Jyv\"askyl\"a,
P.O. Box 35 (MaD)\\
FI-40014 University of Jyv\"askyl\"a\\
Finland\\
  jiayin\imagedot{}mat\imagedot{}liu\imageat{}jyu\imagedot{}fi \\
\end{authorinfo}
\begin{authorinfo}[to]
  Tuomas Orponen\\
  Department of Mathematics and Statistics\\ University of Jyv\"askyl\"a,
P.O. Box 35 (MaD)\\
FI-40014 University of Jyv\"askyl\"a\\
Finland\\
  tuomas\imagedot{}t\imagedot{}orponen\imageat{}jyu\imagedot{}fi\\
\end{authorinfo}
\end{dajauthors}

\end{document}